\documentclass{amsart}
\usepackage{amssymb}
\usepackage{mathtools,mathrsfs,dsfont,amsrefs}
\usepackage{graphicx}
\usepackage{color}
\usepackage{enumitem}
\usepackage{hyperref}

\newtheorem{thm}{Theorem}[section]
\newtheorem{conj}[thm]{Conjecture}
\newtheorem{cor}[thm]{Corollary}
\newtheorem{lem}[thm]{Lemma}
\newtheorem{prop}[thm]{Proposition}
\theoremstyle{definition}
\newtheorem{dfn}[thm]{Definition}
\newtheorem{ass}[thm]{Assumption}

\newtheorem{rmk}[thm]{Remark}
\newtheorem*{ntn}{Notation}

\numberwithin{equation}{section}

\newcommand{\mcA}{\mathcal{A}}
\newcommand{\mcB}{\mathcal{B}}
\newcommand{\mcC}{\mathcal{C}}

\newcommand{\mcE}{\mathcal{E}}
\newcommand{\mcF}{\mathcal{F}}

\newcommand{\mcH}{\mathcal{H}}

\newcommand{\mcJ}{\mathcal{J}}

\newcommand{\mcR}{\mathcal{R}}

\newcommand{\mcX}{\mathcal{X}}

\newcommand{\bbG}{\mathbb{G}}

\newcommand{\bbN}{\mathbb{N}}
\newcommand{\bbQ}{\mathbb{Q}}
\newcommand{\bbR}{\mathbb{R}}

\newcommand{\bbZ}{\mathbb{Z}}

\newcommand{\msfE}{\mathsf{E}}

\newcommand{\mean}[2]{\left\langle#1\right\rangle_{#2}}
\newcommand{\trinorm}[1]{{\left\vert\kern-0.25ex\left\vert\kern-0.25ex\left\vert #1
    \right\vert\kern-0.25ex\right\vert\kern-0.25ex\right\vert}}
\newcommand{\indicator}[1]{\mathds{1}_{#1}} 
\newcommand{\compl}[1]{#1^{c}}
\newcommand{\closure}[1]{\overline{#1}}

\newcommand{\abs}[1]{\left\lvert#1\right\rvert}
\newcommand{\bord}[1]{\partial #1}

\newcommand{\norm}[1]{\left\lVert#1\right\rVert}
\newcommand{\diam}{\mathop{\operatorname{diam}}}
\newcommand{\dist}{\mathop{\operatorname{dist}}}

\newcommand{\supp}{\mathop{\operatorname{supp}}}

\def\Xint#1{\mathchoice
{\XXint\displaystyle\textstyle{#1}}%
{\XXint\textstyle\scriptstyle{#1}}%
{\XXint\scriptstyle\scriptscriptstyle{#1}}%
{\XXint\scriptscriptstyle\scriptscriptstyle{#1}}%
\!\int}
\def\XXint#1#2#3{{\setbox0=\hbox{$#1{#2#3}{\int}$}
\vcenter{\hbox{$#2#3$}}\kern-.5\wd0}}

\def\mint{\Xint-}

\newcommand{\GSC}{\mathrm{GSC}}

\begin{document}

\title[Construction of energies on Sierpi\'{n}ski carpets]{Construction of $p$-energy and associated energy measures on Sierpi\'{n}ski carpets}

\author{Ryosuke Shimizu}
\address{Graduate School of Informatics, Kyoto University, Yoshida-Honcho, Sakyo-ku, Kyoto 606-8501, Japan}
\curraddr{}
\email{r.shimizu@acs.i.kyoto-u.ac.jp}
\thanks{The author was supported in part by JSPS KAKENHI Grant Number JP20J20207.}



\date{March 26, 2023}

\dedicatory{}


\begin{abstract}
  We establish the existence of a scaling limit $\mathcal{E}_p$ of discrete $p$-energies on the graphs approximating a generalized Sierpi\'{n}ski carpet for $p > d_{\text{ARC}}$, where $d_{\text{ARC}}$ is the Ahlfors regular conformal dimension of the underlying generalized Sierpi\'{n}ski carpet.
  Furthermore, the function space $\mathcal{F}_{p}$ defined as the collection of functions with finite $p$-energies is shown to be a reflexive and separable Banach space that is dense in the set of continuous functions with respect to the supremum norm.
  In particular, $(\mathcal{E}_2, \mathcal{F}_2)$ recovers the canonical regular Dirichlet form constructed by Barlow and Bass \cite{BB89} or Kusuoka and Zhou \cite{KZ92}.
  We also provide $\mathcal{E}_{p}$-energy measures associated with the constructed $p$-energy and investigate its basic properties like self-similarity and chain rule.
\end{abstract}

\maketitle


\section{Introduction}\label{sec:intro}
On Euclidean spaces, the \emph{nonlinear potential theory} is built on the theory of the $(1, p)$-Sobolev spaces $W^{1, p}$ and the $p$-energy $\int \abs{\nabla f}^{p}\,dx$.
The main aim of this paper is to construct and study $p$-energies on Sierpi\'{n}ski carpets as a prototype of nonlinear potential theory on complicated metric spaces like ``fractals'' (see also \cite{KM21+}*{Problem 7.6}).
There has been significant progress on ``analysis and probability'' on complicated spaces beyond Euclidean spaces over the last several decades.
The earlier works are the constructions of diffusion processes, which is called the \emph{Brownian motions}, on self-similar sets in 1980s and 1990s.
(For details and precise history of ``analysis on fractals'', see the ICM survey of Kumagai \cite{Kum14} for example.)
In particular, a class of self-similar sets called \emph{generalized Sierpi\'{n}ski carpets} (see Figure \ref{fig.SCs}), is one of the successful examples. 
In this introduction, we restrict to the case of \emph{the standard Sierpi\'{n}ski carpet} (the left in Figure \ref{fig.SCs}), SC for short, for simplicity. 
The first Brownian motion on the SC was given by Barlow and Bass in \cite{BB89}, where they obtained the Brownian motion as a scaling limit of Brownian motions on Euclidean regions approximating the SC.
From an analytic viewpoint, the result of Barlow and Bass gives $2$-energy $\mcE_2$ and the associated $(1,2)$-``Sobolev'' space $\mcF_2$, namely \emph{regular Dirichlet form} on the SC.
Recall that a tuple of $2$-energy $\int\abs{\nabla f}^{2}\,dx$ (on $L^{2}(\bbR^{N}, dx)$) and $(1, 2)$-Sobolev space $W^{1, 2}$ is a typical example of regular Dirichlet forms, which corresponds to the classical Brownian motion on $\bbR^{N}$.
Although it is difficult to define the gradient $\nabla f$ on the SC, we can say that a suitable $2$-energy ``$\int\abs{\nabla f}^{2}\,dx$'' exists on the SC.
Later, Kusuoka and Zhou \cite{KZ92} gave an alternative construction of a regular Dirichlet form as a scaling limit of discrete $2$-energies on a series of graphs approximating the SC as shown in Figure \ref{fig.SCapp}.
Our work gives a ``canonical'' construction of $p$-energy $\mcE_p$ and the associated $(1,p)$-``Sobolev'' space $\mcF_p$ on the SC, which play the same roles as the pair of $\int\abs{\nabla f}^{p}\,dx$ and the Sobolev space $W^{1, p}$, by extending and simplifying the method of Kusuoka and Zhou.


\begin{figure}[tb]\centering
	\includegraphics[height=80pt]{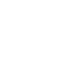}\hspace*{9pt}
	\includegraphics[height=80pt]{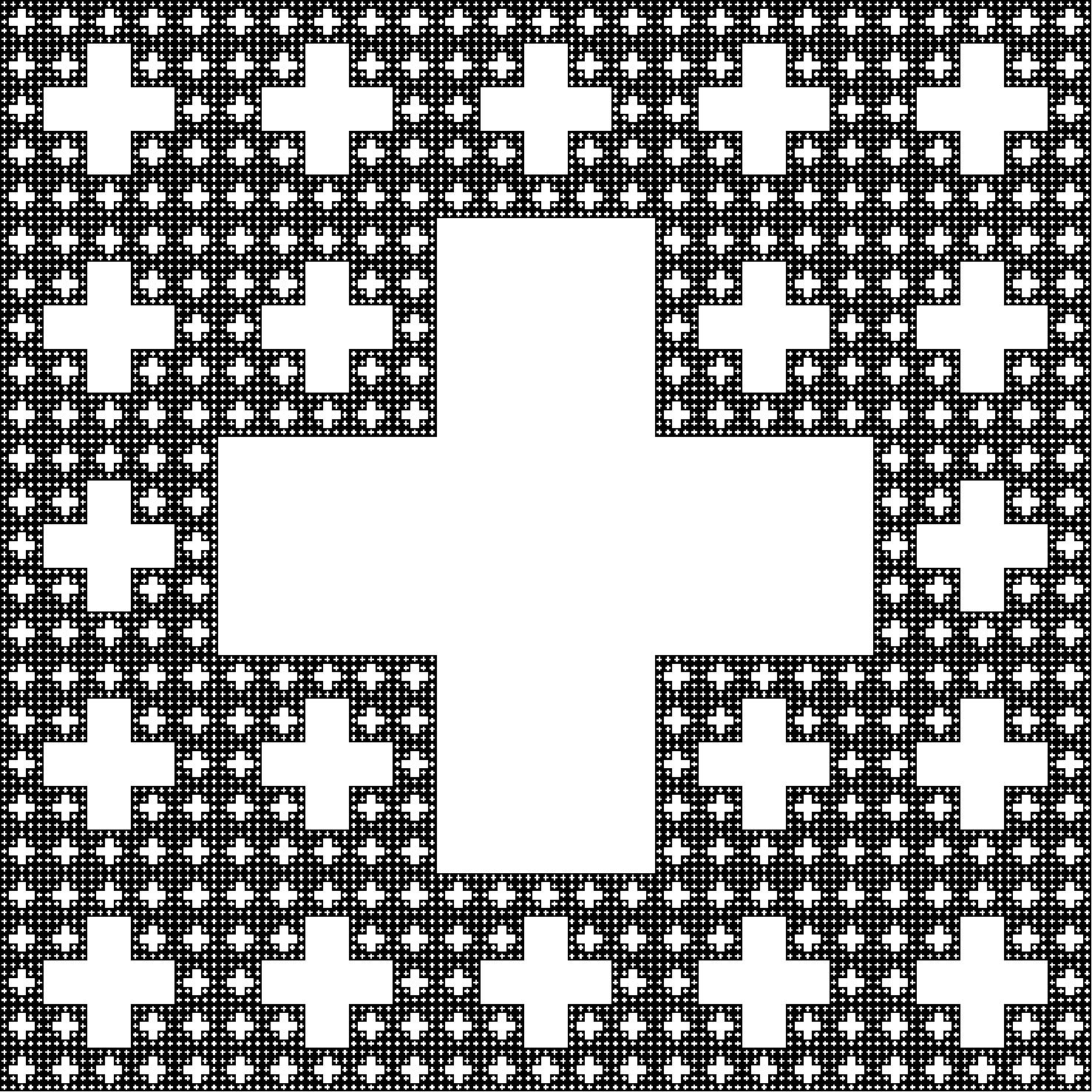}\hspace*{8pt}
	\includegraphics[height=80pt]{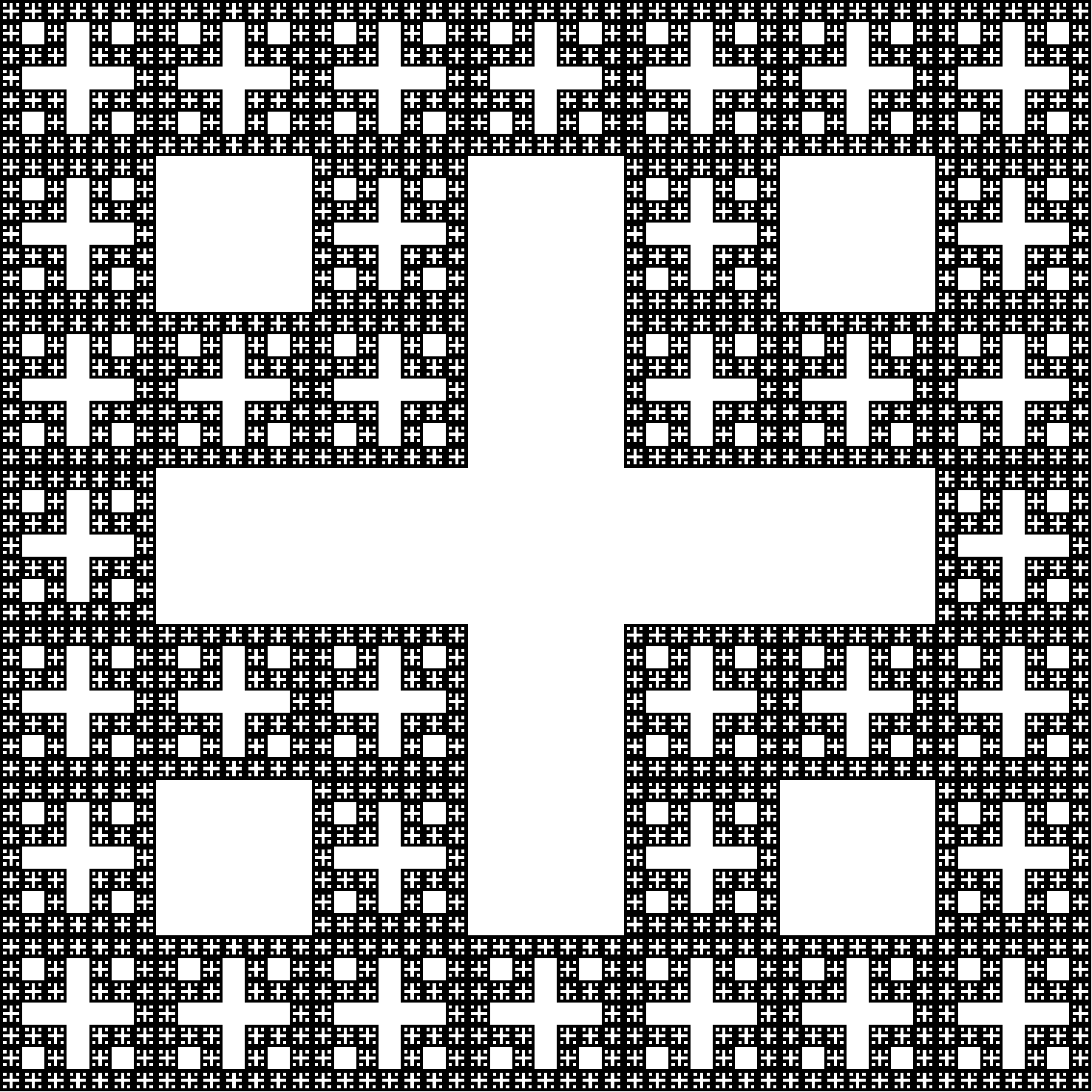}\hspace*{8pt}
	\includegraphics[height=80pt]{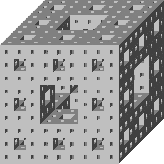}
	\caption{Sierpi\'{n}ski carpet (left), two other generalized Sierpi\'{n}ski
	carpets and Menger sponge (right)}\label{fig.SCs}
\end{figure}

Let us describe briefly our strategy to construct $(\mcE_p, \mcF_p)$ on the SC.
We write $(K, d, \mu)$ to denote the SC as a metric measure space, that is, $K$ is the Sierpi\'{n}ski carpet, $d$ is the Euclidean metric of $\bbR^{2}$ and $\mu$ is the $\dim_{\text{H}}(K, d)$-dimensional Hausdorff measure on $(K, d)$, where $\dim_{\text{H}}(K, d) = \log{8}/\log{3}$ is the Hausdorff dimension of $(K, d)$.
Let $\{ G_{n} \}_{n \ge 1}$ be a series of finite graphs approximating the SC whose edge set is denoted by $E_n$ (see Figure \ref{fig.SCapp} and Definition \ref{dfn.appSC}).
Then discrete $p$-energy $\mcE_{p}^{G_{n}}$ on $G_{n}$ is
\[
\mcE_{p}^{G_{n}}(f) = \frac{1}{2}\sum_{(x,y) \in E_{n}}\abs{M_{n}f(x) - M_{n}f(y)}^{p},
\]
where $M_{n}$ is a discretization operator from $L^{p}(K, \mu)$ to $\bbR^{G_{n}}$ (see Section \ref{sec:preli} for the precise definition).
To obtain an appropriate non-trivial limit of discrete $p$-energies, some renormalization is necessary (see \cite{Bar13} for example).
We will see that the behavior of $\mcR_{p}^{(n)}$ defined as
\begin{equation*}
\mcR_{p}^{(n)} \coloneqq \left(\inf\left\{ \mcE_{p}^{G_{n}}(M_{n}f) \;\middle|\;
    \begin{array}{c}
    \text{$f \in L^{p}(K, \mu)$ with $M_{n}f \equiv 0$ on the left side of $G_{n}$} \\
    \text{and $M_{n}f \equiv 1$ on   the right side} \\
    \end{array}
    \right\}\right)^{-1}
\end{equation*}
gives us the proper renormalization constant of discrete $p$-energies $\mcE_{p}^{G_{n}}$.
In fact, for $p = 2$, Barlow and Bass \cite{BB90} have proved that there exist constants $\rho_{2} > 0$ (the so-called \emph{resistance scaling factor}) and $C \ge 1$ such that
\begin{equation}\label{intro.2res}
    C^{-1}\rho_{2}^{n} \le \mcR_{2}^{(n)} \le C\rho_{2}^{n}, \quad n \in \bbN.
\end{equation}
What Kusuoka and Zhou have shown is that, roughly speaking, the Dirichlet form $(\mcE_{2}, \mcF_{2})$ on the SC is obtained as
\[
\mcF_{2} = \biggl\{ f \in L^{2}(K, \mu) \biggm| \sup_{n \ge 1}\rho_{2}^{n}\mcE_{2}^{G_{n}}(M_{n}f) < \infty \biggr\}
\]
and $\mcE_{2}(f) = \lim_{k \to \infty}\rho_{2}^{n_{k}}\mcE_{2}^{G_{n_k}}(M_{n_k}f)$ for some subsequence $\{ n_k \}_{k \ge 1}$.

\begin{figure}[t]
    \centering
    \includegraphics[height=85pt]{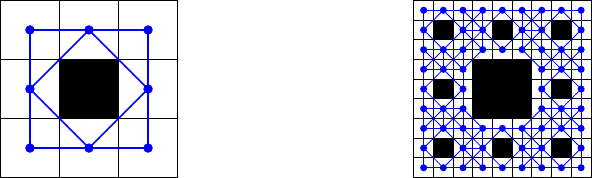}
    \caption{Graphical approximation $\{ G_{n} \}_{n \ge 1}$ of the SC (This figure draws $G_{1}$ and $G_{2}$ in blue)}
    \label{fig.SCapp}
\end{figure}

By using $p$-combinatorial modulus, which is one of fundamental tools in ``quasiconformal geometry'', Bourdon and Kleiner \cite{BK13} have generalized \eqref{intro.2res}, i.e. they have ensured the existence of a constant $\rho_{p} > 0$ such that
\begin{equation}\label{intro.pres}
    C^{-1}\rho_{p}^{n} \le \mcR_{p}^{(n)} \le C\rho_{p}^{n}, \quad n \in \bbN.
\end{equation}
Then our $(1,p)$-``Sobolev'' space $\mcF_{p}$ equipped with the norm $\norm{\,\cdot\,}_{\mcF_{p}}$ is defined by
\[
\mcF_{p} = \biggl\{ f \in L^{p}(K, \mu) \biggm| \sup_{n \ge 1}\rho_{p}^{n}\mcE_{p}^{G_{n}}(M_{n}f) < \infty \biggr\},
\]
and
\[
\norm{f}_{\mcF_{p}} = \norm{f}_{L^{p}} + \left(\sup_{n \ge 1}\rho_{p}^{n}\mcE_{p}^{G_{n}}(M_{n}f)\right)^{1/p}.
\]
Under the following assumption (see Assumption \ref{assum.lowdim}):
\begin{equation}\label{intro.assum}
    \text{$1 < p < \infty$ and $\rho_{p} > 1$},
\end{equation}
we will prove that $\mcF_{p}$ is continuously embedded in the H\"{o}lder space:
\[
\mcC^{0, \theta_{p}} = \biggl\{ f\colon K\to \bbR \biggm| \sup_{x \neq y \in K}\frac{\abs{f(x) - f(y)}}{d(x, y)^{\theta_{p}}} < \infty \biggr\},
\]
where $\theta_{p} \coloneqq \log\rho_{p}/p\log{3}$ (Theorem \ref{thm.embed}).
This embedding result is very powerful.
Indeed, we will deduce the \emph{closedness}, i.e. $(\mcF_{p}, \norm{\,\cdot\,}_{\mcF_{p}})$ is a Banach space, and the \emph{regularity}, i.e. $\mcF_{p}$ is dense in $\mcC(K) = \{ f\colon K \to \bbR \mid \text{$f$ is continuous} \}$ with the sup norm, from this embedding (see Theorems \ref{thm.close} and \ref{thm.core}).

Moreover, the \emph{separability} of $(\mcF_{p}, \norm{\,\cdot\,}_{\mcF_{p}})$ will be deduced from the \emph{reflexivity} of $\mcF_{p}$ (Theorems \ref{thm.newnorm} and \ref{thm.sep}).
Thanks to the separability, one easily sees that, by the diagonal procedure, a subsequential limit $\lim_{k \to \infty}\rho_{p}^{n_k}\mcE_{p}^{G_{n_k}}(M_{n_k}f)$ exists for all $f \in \mcF_{p}$.
Our final object $\mcE_{p}$ called the $p$-energy on the SC will be constructed through these subsequential limits\footnote{To construct ``canonical'' $p$-energy $\mcE_{p}$ on the SC, we need to follow some additional procedures as shown in the work of Kusuoka and Zhou. In this paper, we will introduce new graphs $\{ \bbG_{n} \}_{n \ge 1}$ and consider discrete $p$-energies on them to get a ``good'' $p$-energy. These procedures are described in Section \ref{sec:constr}. See Theorem \ref{thm.main1.2} for the meaning of canonical $p$-energies.}.

The assumption \eqref{intro.assum} is essential for the continuous embedding of $\mcF_{p}$ in the H\"{o}lder space $\mcC^{0, \theta_{p}}$ and has a close connection with the \emph{Ahlfors regular conformal dimension} $\dim_{\textup{ARC}}(K, d)$ which is defined by
\begin{equation}\label{intro.ARCdim}
    \dim_{\textup{ARC}}(K, d) = \inf\left\{ \alpha \;\middle|\;
    \begin{array}{c}
    \text{there exists a metric $\rho$ on $K$ which is} \\
    \text{quasisymmetric to $d$ and $\alpha$-Ahlfors regular} \\
    \end{array}
    \right\}.
\end{equation}
(For the precise definitions of \emph{Ahlfors regularity} and being \emph{quasisymmetric}, see \eqref{eq.AR} and Definition \ref{dfn.ARCdim}.)
Indeed, by results of Carrasco Piaggio \cite{CP13} and Kigami \cite{Kig20}, the condition \eqref{intro.assum} is equivalent to
\begin{equation}\label{intro.lowdim}
    p > \dim_{\textup{ARC}}(K, d).
\end{equation}
We expect that this condition \eqref{intro.lowdim} represents a ``low-dimensional'' phase.
More precisely, we regard the H\"{o}lder embedding $\mcF_{p} \hookrightarrow \mcC^{0,\theta_{p}}$ as a generalization of the classical Sobolev embedding (a consequence of Morrey's inequality).
For this reason, we naturally arrive at the following conjecture.

\begin{conj}\label{conj.crit-dim}
    $\dim_{\textup{ARC}}(K, d) = \inf\{ p \mid \text{$\mcF_{p}$ is embedded in a subset of $\mcC(K)$} \}$.
\end{conj}

To show this conjecture, what we need is the regularity of $\mcF_{p}$, i.e. the density of $\mcF_{p} \cap \mcC(K)$ in $\mcC(K)$ with the sup norm, for $p \le \dim_{\textup{ARC}}(K, d)$ (see \cite{BB99} for $p = 2$).
This is a big open problem for future work.

Besides our ``Sobolev spaces'' $\mcF_{p}$, there has already been an established theory of ``Sobolev spaces on metric measure spaces'' based on the notion of \emph{upper gradients}, which is a counter part of $\abs{\nabla f}$ introduced by Heinonen and Koskela in \cite{HK98}.
We refer to \cites{HKST, Hei} for details.
From the viewpoint of this theory, our $(1, p)$-``Sobolev'' space $\mcF_{p}$ can be seen as a \emph{fractional Korevaar--Schoen Sobolev space}.
Indeed, we will give the following representation of $\mcF_{p}$ (Theorem \ref{thm.main2}):
\begin{equation}\label{intro.LB}
    \mcF_{p} = \left\{ f \in L^{p}(K, \mu) \;\middle|\; \varlimsup_{r \downarrow 0}\int_{K}\mint_{B_{d}(x, r)}\frac{\abs{f(x) - f(y)}^{p}}{r^{\beta_{p}}}\,d\mu(y)d\mu(x) < \infty \right\},
\end{equation}
where $\beta_{p} = \log(8\rho_{p})/\log{3}$.
When $p = 2$, this result is well-known (see \cites{GHL03, Kum00, KS05} for example) and the parameter $\beta_{2}$ is called the \emph{walk dimension}.
For detailed expositions of $\beta_{2}$, see \cites{GHL03, Kum14, KS05} for example.
If $\beta_{p} = p$, then the expression \eqref{intro.LB} coincides with (a slight modification of) the Korevaar--Schoen $(1, p)$-Sobolev space \cites{KS93,KM98}.
However, it is well-known that a strict inequality $\beta_{2} > 2$ holds on the SC (see \cite{BB99}*{Proposition 5.1} or \cite{Kaj20+}).
This phenomenon suggests that the existing theory of ``Sobolev spaces on metric measure spaces'' do not give any non-trivial $(1,p)$-Sobolev spaces on the SC\footnote{It is also well-known that the \emph{Newtonian $(1, p)$-Sobolev space} on the SC becomes $L^{p}(K, \mu)$ due to the lack of plenty rectifiable curves in the SC. See \cite{MT}*{Proposition 4.3.3} and \cite{HKST}*{Proposition 7.1.33} for example.}.
This is one of the reasons why we try to provide an alternative theory of $(1, p)$-``Sobolev'' space and $p$-energy on the SC.

Another major objective of this paper is the $\mcE_{p}$-energy measures associated with $p$-energy $\mcE_{p}$.
In terms of a Dirichlet form $(\mcE_{2}, \mcF_{2})$, $\mcE_{2}$-energy measure of a function $f \in \mcF_{2}$ is defined as the unique Borel measure $\mu_{\langle f \rangle}^{2}$ on $K$ such that
\begin{equation}\label{intro.em}
    \int_{K}g\,d\mu_{\langle f \rangle}^{2} = \mcE_{2}(f, fg) - \frac{1}{2}\mcE_{2}(f^{2}, g), \quad g \in \mcF_{2}.
\end{equation}
(Note that we can define the form $\mcE_{2}(f, g)$ by the polarization: $\mcE_{2}(f, g) \coloneqq \frac{1}{4}\bigl(\mcE_{2}(f + g) - \mcE_{2}(f - g)\bigr)$.)
This measure plays the role of $\abs{\nabla f(x)}^{2}\,dx$ if the underlying space is Euclidean.
On the other hand, for any $f \in \mcF_{2}$ with $\mcE_{2}(f) \neq 0$, the $\mcE_{2}$-energy measure $\mu_{\langle f \rangle}^{2}$ and the $\log_{3}{8}$-dimensional Hausdorff measure $\mu$ on the SC are mutually singular due to the fact that $\beta_{2} > 2$ by a result of Hino \cite{Hin05}.
See \cite{KM20} for an extension of this fact to general metric measure Dirichlet spaces.
This phenomenon is also far different from ``smooth'' settings and motivates the study of $\mcE_{2}$-energy measures on fractals.

For general $p$, due to the lack of  a counterpart of the expression in the right-hand side of \eqref{intro.em}, we will choose to generalize Hino's alternative method of the construction of $\mcE_{2}$-energy measure.
Namely, for any $f \in \mcF_{p}$, we first construct a measure $\mathfrak{m}^{p}_{\langle f \rangle}$ on the shift space $\{ 1, \dots, 8 \}^{\bbN \cup \{ 0 \}}$ associated with the SC and define our $\mcE_{p}$-energy measure $\mu^{p}_{\langle f \rangle}$ as the pushforward measure of $\mathfrak{m}^{p}_{\langle f \rangle}$ under the natural quotient map $\pi\colon\{ 1, \dots, 8 \}^{\bbN \cup \{ 0 \}} \to K$ (see Proposition \ref{prop.pi} for a description of $\pi$), i.e. $\mu^{p}_{\langle f \rangle}(A) = \mathfrak{m}^{p}_{\langle f \rangle}(\pi^{-1}(A))$ for any Borel set $A$ of $K$.
Then our $\mcE_{p}$-energy measure $\mu^{p}_{\langle f \rangle}$ is associated with $\mcE_{p}$ in the sense that $\mu^{p}_{\langle f \rangle}(K) = \mcE_{p}(f)$ (for more details on relations between $\mu^{p}_{\langle f \rangle}$ and $\mcE_{p}$, see Theorem \ref{thm.main3}-(c)).

Furthermore, we will show the \emph{chain rule}: for any $\Phi \in \mcC^{1}(\bbR)$, 
\begin{equation}\label{intro.chain}
    d\mu_{\langle \Phi \circ f \rangle}^{p} = \abs{\Phi' \circ f}^{p}d\mu_{\langle f \rangle}^{p}.
\end{equation}
When $p = 2$, the chain rule \eqref{intro.chain} is proved by using integral expressions of $\mcE_{2}$ (see \cite{FOT}*{(3.2.12)} for example), but such representations take full advantage of the fact that $p = 2$.
Alternatively, we prove \eqref{intro.chain} by introducing a new series of graphs $\{ \bbG_{n} \}_{n \ge 1}$ (see the beginning of subsection \ref{subsec:newgraph}), which is embedded in the SC, and analyzing discrete $p$-energies $\bigl\{ \mcE_{p}^{\bbG_{n}}(\Phi \circ f) \bigr\}_{n \ge 1}$.
This approach is actually valid since our $p$-energies are based on subsequential limits of $\bigl\{ \rho_{p}^{n}\mcE_{p}^{\bbG_{n}} \bigr\}_{n \ge 1}$.

The first result on the existence of suitable $p$-energy on fractals is due to \cite{HPS04}, where the Sierpi\'{n}ski gaskets are considered. 
(Added in revision: for \emph{p.c.f. self-similar sets}, there are also recent studies \cite{BC23, CGQ22, GYZ22}.)
In the very recent paper \cite{Kig21+}, Kigami has established a theory of $(1, p)$-Sobolev space and $p$-energy on \emph{$p$-conductively homogeneous} compact metric spaces (see \cite{Kig21+} for details).
His paper \cite{Kig21+} includes new construction results even if $p = 2$ (see \cite{Kig21+}*{Sections 12 and 13} for a gallery).
Also, a class of highly symmetric p.c.f. self-similar sets called \emph{nested fractals} is also treated in \cite{Kig21+}*{Section 14}.
However, the construction of $\mcE_{p}$-energy measures associated with the $p$-energy $\mcE_{p}$ is not treated in earlier works.

\noindent
\textbf{Outline.}
This paper is organized as follows.
In Section \ref{sec:preli}, we prepare basic frameworks in this paper and state the main results.
In particular, we give the definition of generalized Sierpi\'{n}ski carpets. 
Sections \ref{sec:poincare} and \ref{sec:lowdim} are devoted to extending results of Kusuoka and Zhou to fit our purpose.
Section \ref{sec:poincare} is a collection of basic estimates of $(p, p)$-Poincar\'{e} constants and $\mcR_{p}^{(n)}$.
In Section \ref{sec:lowdim}, we prove powerful results concerning $(p, p)$-Poincar\'{e} constants (uniform H\"{o}lder estimates and a condition called $p$-Knight Move \hyperlink{KMp}{\textup{(KM$_{p}$)}} for example) under Assumption \ref{assum.lowdim} and finish all preparations to construct $p$-energy $\mcE_{p}$ and $(1,p)$-``Sobolev'' space $\mcF_{p}$.
Section \ref{sec:domain} is devoted to investigating detailed properties of $\mcF_{p}$.
Then, in Section \ref{sec:constr}, we introduce another graphical approximation $\{ \bbG_{n} \}_{n \ge 1}$ and construct a canonical $p$-energy $\mcE_{p}$ (see Theorem \ref{thm.main1} for the precise meaning of `canonical').
Section \ref{sec:em} is devoted to discussions on $\mcE_{p}$-energy measures.
Finally, in Section \ref{sec:betap}, we prove $\beta_{p} > p$ (Theorem \ref{thm.betap_strict}) under the assumption that the underling generalized Sierpi\'{n}ski carpet is embedded in $\mathbb{R}^{2}$. 
The appendix contains proofs of some elementary lemmas.

\begin{ntn}
    In this paper, we use the following notation and conventions.
    \begin{itemize}
        \item [(1)] $\bbN \coloneqq \{ n \in \bbZ \mid n > 0\}$ and $\bbZ_{\ge 0} \coloneqq \bbN \cup \{ 0 \}$.
        \item [(2)] We set $a \vee b \coloneqq \max\{ a, b \}$, $a \wedge b \coloneqq \min\{ a, b \}$ for $a, b \in [-\infty, \infty]$.
        \item [(3)] For any countable set $V$, we define $\bbR^{V} \coloneqq \{ f \mid f\colon V \to \bbR \}$.
        \item [(4)]  For $f\colon \bbR \to \bbR$, define $\mathrm{Lip}(f) \coloneqq \sup_{x \neq y \in \bbR}\frac{\abs{f(x) - f(y)}}{\abs{x - y}}$.
        \item [(5)] Let $X$ be a compact topological space.
        We set $\mcC(X) \coloneqq \{ f\colon X \to \bbR \mid \text{$f$ is continuous} \}$ and write its sup norm by $\norm{f}_{\mcC(X)} \coloneqq \sup_{x \in X}\abs{f(x)}$.
        \item [(6)] Let $X$ be a topological space and let $A$ be a subset of $X$. The topological boundary of $A$ is denoted by $\partial A$, that is $\partial A \coloneqq \closure{A}^{X} \setminus \mathrm{int}_{X}A$.
        \item [(7)] Let $(X, d)$ be a metric space.
        The open ball with center $x \in X$ and radius $r > 0$ is denoted by $B_{d}(x, r)$, that is,
        \[
        B_{d}(x, r) \coloneqq \{ y \in X \mid d(x, y) < r \}.
        \]
        If the metric $d$ is clear in context, then we write $B(x, r)$ for short.
        \item [(8)] Let $K$ be a compact metrizable space and let $\mcB(K)$ denote the Borel $\sigma$-algebra of $K$.
        Let $\mu$ be a Borel (regular) measure on $K$.
        For any $A \in \mcB(K)$ with $\mu(A) > 0$ and $f \in L^{1}(K, \mu)$, we define
        \[
        \mint_{A}f\,d\mu \coloneqq \frac{1}{\mu(A)}\int_{A}f\,d\mu.
        \]
       	\item [(9)] We use $\bigsqcup$ to denote disjoint unions. 
        \item [(10)] Let $D \in \mathbb{N}$. Set $\mathbf{0} = \mathbf{0}^{D} \coloneqq (0)_{k = 1}^{D} \in \mathbb{R}^{D}$ and $\mathbf{e}_{j} = \mathbf{e}_{j}^{D} \coloneqq (\delta_{k, j})_{k = 1}^{D} \in \mathbb{R}^{D}$ for each $j \in \{ 1, \dots, D \}$, where $\delta_{j, k}$ is the Dirac delta. For $x = (x_{k})_{k = 1}^{D}, y = (y_{k})_{k = 1}^{D} \in \mathbb{R}^{D}$, we write $\abs{x - y}_{\mathbb{R}^{D}} = \left(\sum_{k = 1}^{D}\abs{x_{k} - y_{k}}^{2}\right)^{1/2}$. 
    \end{itemize}
\end{ntn}

\noindent 
\textbf{Acknowledgments}
This paper is mainly originated from the author's doctoral thesis at Kyoto University. 
The author would like to express his deepest gratitude toward Professor Jun Kigami for guidance during his graduate course, reading drafts, and stimulating conversations.  
He would also like to thank Professor Naotaka Kajino for giving him several valuable comments. 
In particular, Prof. Kajino has suggested using the approach in \cite{Kaj20+} to prove Theorem \ref{thm.betap_strict} and given him detailed comments about the results in \cite{BBKT10}. 
The author would also like to thank Professors Mathav Murugan and Toshiyuki Tanaka for their careful readings of an earlier version of the manuscript. 
In particular, Prof. Murugan has pointed out a gap in the previous proof of Theorem \ref{thm.embed}. 
Finally, he would like to thank Professor Fabrice Baudoin for a comment regarding the critical Besov exponent in \cite{ABCRST21} (Remark \ref{rmk.conj-betap}) and anonymous referees for  helpful comments and suggestions. 


\section{Preliminary and results}\label{sec:preli}

\subsection{Generalized Sierpi\'{n}ski carpets and graphical approximations}\label{subsec.SCappgraph}

We start with the definition of \emph{generalized Sierpi\'{n}ski carpets} and related notations.
The reader is referred to \cite{AOF} for further background and more general framework, namely,  self-similar structure.

Let $D,a \in \mathbb{N}$ with $D \ge 2$, $a \ge 3$ and set $Q_{0} \coloneqq [-1,1]^{D}$. 
Let $S \subsetneq \{ 0, 1, \ldots, a - 1 \}^{D}$ be non-empty, define $f_{i}\colon\mathbb{R}^{D} \to \mathbb{R}^{D}$ by $f_{i}(x) \coloneqq a^{-1}x + 2a^{-1} \cdot i - (a - 1)a^{-1}\sum_{j = 1}^{D}\mathbf{e}_{j}$ for each $i \in \{ 0, 1, \ldots, a - 1 \}^{D}$ and set $Q_{1} \coloneqq \bigcup_{i \in S}f_{i}(Q_{0})$, so that $Q_{1} \subsetneq Q_{0}$.
Let $K$ be the \emph{self-similar set} associated with $\{f_{i}\}_{i\in S}$, i.e., the unique non-empty compact subset of $\mathbb{R}^{D}$ such that $K = \bigcup_{i \in S}f_{i}(K)$, which exists and satisfies $K \subsetneq Q_{0}$ thanks to $Q_{1}\subsetneq Q_{0}$ by \cite{AOF}*{Theorem 1.1.4}.
Define $F_{i} \coloneqq f_{i}|_{K}$ for each $i \in S$ and $\GSC(D, a, S) \coloneqq (K, S, \{ F_{i} \}_{i \in S})$.

By following \cite{Kaj20+}, we will introduce the notion of \emph{generalized Sierpi\'{n}ski carpets}. 
The following definition is essentially due to Barlow and Bass \cite{BB99}*{Section 2}.
The non-diagonality condition in \cite{BB99}*{Hypotheses 2.1} has been modified later in \cite{BBKT10}.
See \cite{BBKT10}*{Remark 2.10-1.} for details of this correction.
\begin{dfn}[Generalized Sierpi\'{n}ski carpet, {\cite{BBKT10}*{Subsection \textup{2.2}}}]\label{dfn.GSC}
$\GSC(D, a, S)$ is called a \emph{generalized Sierpi\'{n}ski carpet}
if and only if the following four conditions are satisfied:
\begin{enumerate}[label=\textup{(GSC\arabic*)},align=left,leftmargin=*,topsep=2pt,parsep=0pt,itemsep=2pt]
\item\label{GSC1} (Symmetry) $f(Q_{1})=Q_{1}$ for any isometry $f$ of $\mathbb{R}^{D}$ with $f(Q_{0}) = Q_{0}$.
\item\label{GSC2} (Connectedness) $Q_{1}$ is connected.
\item\label{GSC3} (Non-diagonality)
	$\mathrm{int}_{\mathbb{R}^{D}}\bigl(Q_{1}\cap \prod_{k = 1}^{D}[2(i_{k} - \varepsilon_{k})a^{-1}, 2(i_{k}+1)a^{-1}]\bigr)$ is either empty or connected for any $(i_{k})_{k = 1}^{D} \in \mathbb{Z}^{D}$ and
	any $(\varepsilon_{k})_{k = 1}^{D} \in \{ 0, 1 \}^{D}$.
\item\label{GSC4} (Borders included) $[-1, 1] \times \{ 1 \}^{D - 1}\subset Q_{1}$.
\end{enumerate}
\end{dfn}
\begin{rmk}
	In \cites{Kaj20+, BBKT10}, generalized Sierpi\'{n}ski carpets are defined as subspaces of $D$-dimensional unit cube $[0, 1]^{D}$. 
	In this paper, we consider GSC as subspaces of $[-1, 1]^{D}$ instead of $[0, 1]^{D}$ to follow \cite{Kig21+}*{Section 11}. 
\end{rmk}

As special cases of Definition \ref{dfn.GSC}, the \emph{standard Sierpi\'{n}ski carpet} (left in Figure \ref{fig.SCs}) and \emph{Menger sponge} (right in Figure \ref{fig.SCs}), are given by $\GSC(2, 3, \{ 0, 1, 2 \}^{2} \setminus \{ (1, 1) \})$ and $\GSC(3, 3, \{ (i_1, i_2, i_3) \in \{ 0, 1, 2 \}^{3} \mid \sum_{k = 1}^{3}\indicator{1}(i_k) \le 1 \})$ respectively. 

In this paper, we suppose that $\GSC(D, a, S) = (K, S, \{ F_i \}_{i \in S})$ is a generalized Sierpi\'{n}ski carpet and that $d \colon K \times K \to [0, +\infty)$ is the normalized Euclidean metric on $K$, i.e. $d(x, y) = \frac{1}{2\sqrt{D}}\abs{x - y}_{\mathbb{R}^{D}}$. 

Next, by following \cites{AOF, Kaj20+, Kig21+}, we introduce useful notations to express the symmetries of $(K, S, \{ F_{i} \}_{i \in S})$ and to describe the topological structure as a self-similar set of $(K, S, \{ F_{i} \}_{i \in S})$. 
\begin{dfn}\label{dfn.hyperplane}
	Define 
	\[
		B_{j, \sigma} = \{ (x_1, \dots, x_{D}) \in Q_{0} \mid x_{j} = \sigma \}
	\]
	for $j \in \{ 1, \dots, D \}$ and $\sigma \in \{ -1, 0, +1 \}$. 
	We also define the hyperplane 
	\[
	\mathcal{H}_{j_{1}, j_{2}}^{+} = \bigl\{ (x_1, \dots, x_{D}) \in \mathbb{R}^{D} \bigm| x_{j_{1}} =  x_{j_{2}} \bigr\}, 
	\]
	and
	\[
	\mathcal{H}_{j_{1}, j_{2}}^{-} = \bigl\{ (x_1, \dots, x_{D}) \in \mathbb{R}^{D} \bigm| x_{j_{1}} = - x_{j_{2}} \bigr\}, 
	\] 
	for $j_{1}, j_{2} \in \{ 1, \dots, D \}$ with $j_{1} \neq j_{2}$.
	Moreover, define 
	\[
	\mathcal{H}_{j_{1}, j_{2}}^{+, \le} = \bigl\{ (x_1, \dots, x_{D}) \in \mathbb{R}^{D} \bigm| x_{j_{1}} \le x_{j_{2}} \bigr\}
	\]
	and
	\[
	\mathcal{H}_{j_{1}, j_{2}}^{+, \ge} = \bigl\{ (x_1, \dots, x_{D}) \in \mathbb{R}^{D} \bigm| x_{j_{1}} \ge x_{j_{2}} \bigr\}. 
	\] 
	We also define $\mathcal{H}_{j_{1}, j_{2}}^{-, \le}$ and $\mathcal{H}_{j_{1}, j_{2}}^{-, \ge}$ in similar ways. 
\end{dfn}
\begin{dfn}\label{dfn.symmetry}
	We define 
	\[
	\mathbb{B}_{D} = \{ T \mid \text{$T \colon \mathbb{R}^{D} \to \mathbb{R}^{D}$ is an isometry such that $T(Q_{0}) = Q_{0}$}  \}, 
	\]
	and
	\begin{equation}\label{GSCsym}
		\mathcal{G}_{0} = \{ f|_{K} \mid f \in \mathbb{B}_{D} \}.   
	\end{equation}
	Then $\mathcal{G}_{0}$ is a finite subgroup of the set of homeomorphism of $K$ by virtue of \ref{GSC1}. 
	Furthermore, define $R_{j} \in \mathbb{B}_{D}$ as the reflection in the hyperplane $B_{j, 0}$ for each $j \in \{ 1, \dots, D \}$ and define $R_{j_{1}, j_{2}}^{\pm} \in \mathbb{B}_{D}$ as the (restriction of the) reflection in the hyperplane $\mathcal{H}_{j_{1}, j_{2}}^{\pm}$ for each $j_{1}, j_{2} \in \{ 1, \dots, D \}$.
	We also use the same symbols $R_{j}$ and $R_{j_{1}, j_{2}}^{\pm}$ to denote these restrictions to $K$, which are elements of $\mathcal{G}_{0}$. 
\end{dfn}

\begin{dfn}\label{dfn.word}

    (1) We set $W_{m} \coloneqq S^{m} = \{ w_{1}\cdots w_{m} \mid \text{$w_{i} \in S$ for $i \in \{ 1, \dots, m \}$}\}$ for $m \in \bbN$ and $W_{\#} \coloneqq \bigcup_{m = 1}^{\infty}W_{m}$.
    For $w = w_{1}\cdots w_{m} \in W_{\#}$, the unique $m \in \bbN$ with $w \in W_{m}$ is denoted by $\abs{w}$ and set $f_{w} \coloneqq f_{w_{1}} \circ \cdots \circ f_{w_{m}}$, $F_{w} \coloneqq F_{w_{1}} \circ \cdots \circ F_{w_{m}}$, $K_{w} \coloneqq F_{w}(K)$, $O_{w} \coloneqq K_{w} \setminus \bigcup_{v \in W_{m}; v \neq w}(K_{w} \cap K_{v})$, and $[w]_{n} \coloneqq w_{1}\cdots w_{n}$ for $n \in \{ 1, \dots, m \}$.
    We define $W_{0} \coloneqq \{ \emptyset \}$ and $[w]_{0} \coloneqq \emptyset$, where $\emptyset$ is an empty word.
    Set $W_{\ast} = W_{0} \cup W_{\#}$ and $F_{\emptyset} \coloneqq \mathrm{id}_{\mathbb{R}^{D}}|_{K}$.
    We also set $i^{m} \coloneqq i\cdots i \in W_{m}$ for each $i \in S$.
    For $n, m \ge 0$, $v = v_{1} \cdots v_{n} \in W_{n}$ and $w = w_{1} \cdots w_{m} \in W_{m}$, define $v \cdot w \in W_{n + m}$ by $v \cdot w = v_{1} \cdots v_{n} w_{1} \cdots w_{m}$. 
    We also write $vw$ for $v \cdot w$ if there is no confusion. 
    For $n \in \bbN$ and non-empty subset $A$ of $W_{n}$, we define $A \cdot W_{m}$ by setting
    \[
    A \cdot W_{m} = \{ vw \mid v \in A, w \in W_{m} \}.
    \]
    When $A = \{ v \}$ for some $v \in W_{n}$, we write $v \cdot W_{m}$ to denote $\{ v \} \cdot W_{m}$ for simplicity.

    (2) The collection of one-sided infinite sequences of symbols $S$ is denoted by $\Sigma$, that is,
    \begin{equation*}
        \Sigma = \{ \omega = \omega_{1}\omega_{2}\omega_{3}\cdots \mid \omega_{i} \in S \text{ for any } i \in \mathbb{N} \},
    \end{equation*}
    which is called the \emph{one-sided shift space} of symbols $S$.
    We define the \emph{shift map} $\sigma:\Sigma \to \Sigma$ by $\sigma(\omega_{1}\omega_{2}\cdots) = \omega_{2}\omega_{3}\cdots$ for each $\omega_{1}\omega_{2}\cdots \in \Sigma$.
    The branches of $\sigma$ are denoted by $\sigma_{i} \, (i \in S)$, namely $\sigma_{i}:\Sigma \to \Sigma$ is defined as $\sigma_{i}(\omega_{1}\omega_{2}\cdots) = i\omega_{1}\omega_{2}\cdots$ for each $i \in S$ and $\omega_{1}\omega_{2}\cdots \in \Sigma$.
    For $w = w_{1} \cdots w_{m} \in W_{\#}$, we write $\sigma_{w} = \sigma_{w_{1}} \circ \cdots \circ \sigma_{w_{m}}$ and $\Sigma_{w} \coloneqq \sigma_{w}(\Sigma)$.
    For $\omega = \omega_{1}\omega_{2}\cdots \in \Sigma$ and $m \in \mathbb{Z}_{\ge 0}$, we define $[\omega]_{m} = \omega_{1} \cdots \omega_{m} \in W_{m}$.
    
    (3) For any $w = w_{1}\cdots w_{m} \in W_{m}$, we define $\xi(w) = w_{1} \cdots w_{m - 1}$. 
    For a subset $A \subseteq W_{m}$, we write $\xi(A)$ for $\{ \xi(w) \mid w \in A \}$. 
    
    (4) For $A \subseteq \mathbb{R}^{D}$ and $n \in \mathbb{Z}_{\ge 0}$, define 
    \[
    W_{n}[A] = \{ w \in W_{n} \mid A \cap K_{w} \neq \emptyset \}. 
    \]
    
    (5) For $T \in \mathcal{G}_{0}$ and $n \in \mathbb{Z}_{\ge 0}$, $\tau[T]\colon W_{n} \to W_{n}$ is the bijection such that $T(K_{w}) = K_{\tau[T](w)}$ for any $w \in W_{n}$. 
\end{dfn}

We consider $\Sigma$ as a topological space equipped with the product topology of $S^{\bbN}$.
Then the following fact is elemental (see \cite{AOF}*{Theorem 1.2.3}).

\begin{prop}\label{prop.pi}
    For any $\omega = \omega_{1}\omega_{2}\cdots \in \Sigma$, the set $\bigcap_{m \ge 1}K_{[\omega]_{m}}$ contains only one point.
    If we define $\pi\colon \Sigma \to K$ by $\{ \pi(\omega) \} = \bigcap_{m \ge 1}K_{[\omega]_{m}}$, then $\pi$ is a continuous surjective map.
    Furthermore, it holds that $\pi \circ \sigma_{i} = F_{i} \circ \pi$ for each $i \in S$.
\end{prop}

Set $N_{\ast} \coloneqq \#S$ and $\alpha \coloneqq \log{N_{\ast}}/\log{a}$. 
Note that $\alpha < D$ by $S \subsetneq \{ 0, 1, \ldots, a - 1 \}^{D}$.
Let $\mu$ be the self-similar probability measure on $K$ with weight $(1/N_{\ast}, \dots, 1/N_{\ast})$, namely $\mu$ is the unique Borel probability measure on $K$ such that $\mu = N_{\ast}(\mu \circ F_{i})$ for any $i \in S$.
It is known that $\alpha$ is the Hausdorff dimension of $(K, d)$ and that $\mu$ is a constant multiple of the $\alpha$-dimensional Hausdorff measure on $(K, d)$; see \cite{AOF}*{Proposition 1.5.8 and Theorem 1.5.7} for example.
In particular, $d$ is \emph{$\alpha$-Ahlfors regular}, that is, there exists a constant $C_{\text{AR}} \ge 1$ such that
\begin{equation}\label{eq.AR}
    C_{\text{AR}}^{-1}\,r^{\alpha} \le \mu(B(x, r)) \le C_{\text{AR}}\,r^{\alpha},
\end{equation}
for any $x \in K$ and $r \in (0, 1)$.
The following lemma on the self-similar measure $\mu$ is standard (see \cite{Kaj20+}*{Lemma 3.3} for example).
\begin{lem}\label{lem.invariant}
    Let $w \in W_{\#}$ and let $f\colon K \to [-\infty, \infty]$ be Borel measurable.
    Then
    \begin{align*}
       \int_{K}\abs{f \circ F_{w}}\,d\mu = N_{\ast}^{\abs{w}}\int_{K_{w}}\abs{f}\,d\mu \quad
       \text{and} \quad \int_{K_{w}}\abs{f \circ F_{w}^{-1}}\,d\mu = N_{\ast}^{-\abs{w}}\int_{K}\abs{f}\,d\mu.
    \end{align*}
\end{lem}

Now, we define some operators that are frequently used in this paper.
\begin{dfn}\label{dfn.mp}
    Let $p \in [1, \infty)$.
    For $w \in W_{\#}$, we define $F_{w}^{\ast}$, $(F_{w})_{\ast}\colon L^{p}(K, \mu) \to L^{p}(K, \mu)$ by setting
    \[
    F_{w}^{\ast}f \coloneqq f \circ F_{w}, \qquad (F_{w})_{\ast}f \coloneqq
    \begin{cases}
        f \circ F_{w}^{-1} \quad &\text{ on $K_{w}$,} \\
        0 \quad &\text{ on $K \setminus K_{w}$,}
    \end{cases}
    \]
    for each $f \in L^{p}(K, \mu)$.
    For $n \in \bbN$, define $M_n\colon L^{p}(K, \mu) \to \bbR^{W_{n}}$ by setting
    \[
    M_{n}f(w) \coloneqq \mint_{K_w}f\,d\mu = N_{\ast}^{n}\int_{K_w}f\,d\mu, \quad w \in W_{n},
    \]
    for each $f \in L^{p}(K, \mu)$.
\end{dfn}

Note that, from Lemma \ref{lem.invariant}, $M_{n}f(w) = \int_{K}F_{w}^{\ast}f\,d\mu$ for any $f \in L^{p}(K, \mu)$, which implies that, for $m \in \bbN$ and $v \in W_{m}$,
\begin{equation}\label{eq.commute}
    M_{n}(F_{v}^{\ast}f)(w) = \int_{K}F_{w}^{\ast}(F_{v}^{\ast}f)\,d\mu = \int_{K}F_{vw}^{\ast}f\,d\mu = M_{n + m}f(vw).
\end{equation}

We introduce graphical approximations of $K$ and related notations by following \cite{Kig20} and \cite{Kig21+}*{section 2}. 
\begin{dfn}\label{dfn.appSC}
    We define $G_{n} \coloneqq (W_{n}, E_{n})$ by setting
    \[
    E_{n} \coloneqq \{ (v, w) \mid \text{$v, w \in W_{n}$ with $v \neq w$ and $K_{v} \cap K_{w} \neq \emptyset$} \}.
    \]
    (This series of graphs $G_{n}$ is called the horizontal networks in \cite{Kig20}.)
    We also define $\widetilde{G}_{n} = \bigl(W_{n}, \widetilde{E}_{n}\bigr)$ by 
    \[
     \widetilde{E}_{n} \coloneqq \bigl\{ (v, w) \in E_{n} \bigm| \text{$F_{v}(Q_{0}) \cap F_{w}(Q_{0})$ is a $(D - 1)$-dimensional hypercube} \bigr\}. 
    \]
\end{dfn}
We use $d_{G_{n}}$ to denote the graph distance of $G_{n}$.
By \hyperlink{GSC2}{\textup{(GSC2)}}, the graph $G_{n}$ is connected. 
Furthermore, by virtue of \hyperlink{GSC3}{\textup{(GSC3)}}, $\bigl(V_{n}, \widetilde{E}_{n}\bigr)$ is also connected (see \cite{Kaj10.Rmk}*{Proposition 2.5}). 
Moreover, by \cite{Kaj10.Rmk}*{Theorem 2.6}, the following result holds. 

\begin{prop}\label{prop.h-network}
	Let $m \in \mathbb{N}$ and let $v, w \in W_{m}$ satisfy $K_{v} \cap K_{w} \neq \emptyset$. 
	Then it holds that $d_{\widetilde{G}_{m}}(v, w) \le D$. 
\end{prop}

For $n, m, k \in \bbN$ and $w \in W_{m}$, we define a subset $\mcB_{n}(w, k)$ of $W_{n + m}$ by setting
\[
\mcB_{n}(w, k) \coloneqq \bigcup_{v \in W_{m};d_{G_{m}}(v, w) \le k}v \cdot W_{n}.
\]
For each $x \in K$ and $s > 0$, we also define a subset $U_{1}(x, s)$ of $K$ by setting
\begin{equation}\label{dfn.quasiball}
    U_{1}(x, s) \coloneqq \bigcup_{w \in \Lambda_{s, 1}(x)}K_{w}, 
\end{equation}
where 
\begin{equation*}
\Lambda_{s, 1}(x) = \left\{ w \;\middle|\;
    \begin{array}{c}
        \text{$w \in W_{n}$ with $x \in K_{w}$ or $(v, w) \in E_{n}$ for some $v \in W_{n}$ } \\
        \text{s.t. $x \in K_{v}$, where $n \in \mathbb{Z}_{\ge 0}$ with $a^{-n} \le s < a^{-n + 1}$}. \\
    \end{array}
    \right\}
\end{equation*}
(See \cite{Kig20}*{Definition 2.3.6}.) 
Then the following proposition says that generalized Sierpi\'{n}ski carpets equipped with the (normalized) Euclidean metrics satisfy the \emph{basic framework} of \cite{Kig20}. 
See also \cite{Kig21+}*{Assumption 2.15 and Proposition 11.4}. 

\begin{prop}\label{prop.BF} 
	Let $(K, S, \{ F_{i} \}_{i \in S}) = \GSC(D, a, S)$ be a generalized Sierpi\'{n}ski carpet. 
	Then the following properties hold: 
    \begin{itemize}
    	\item [\textup{(0)}] \textup{(minimal, strongly finite)} for any $m \ge 0$ and $w \in W_{m}$, 
    	\[
    	O_{w} \coloneqq K_{w} \setminus \left(\bigcup_{v \in W_{m} \setminus \{ w \}}K_{v}\right) \neq \emptyset. 
    	\]
    	Furthermore, if we define 
    	\begin{equation}\label{degree}
    		L_{\ast} \coloneqq \sup_{w \in W_{\#}}\#\{ v \in W_{\abs{w}} \mid (v,w) \in E_{\abs{w}} \}, 
		\end{equation}
		then $L_{\ast} \le 3^{D} - 1$; 
    	\item [\textup{(1)}] for any $w \in W_{\ast}$, $K_{w}$ is connected; 
    	\item [\textup{(2A)}] \textup{($1$-adapted)} there exists a constant $C_{\textup{AD}}$ (depending only on $a, D$) such that for any $x \in K$ and $s \in (0, 1]$, 
    	\begin{equation}\label{1-adapted}
    		U_{1}\bigl(x, C_{\textup{AD}}^{-1}\,s\bigr) \subseteq B_{d}(x, s) \subseteq U_{1}\bigl(x, C_{\textup{AD}}s\bigr);
    	\end{equation}
    	\item [\textup{(2B)}] for any $m \ge 0$ and $w \in W_{m}$, $\diam(K_{w}, d) = a^{-m}$; 
    	\item [\textup{(2C)}] \textup{(thick)} for any $n \ge 1$ and $w \in W_{n}$, there exists $x \in K$ such that 
    	\[
    	K_{w} \supseteq B_{d}\Bigl(x, \frac{1}{\sqrt{D}}a^{-n}\Bigr); 
    	\]
    	\item [\textup{(3)}] for $m \ge 0$ and $v, w \in W_{m}$ with $v \neq w$, 
    	\[
    	\mu(K_{v} \cap K_{w}) = 0 \quad \text{and} \quad \mu(K_{v}) = N_{\ast}^{-m} = a^{-\alpha m};
    	\]
    	\item [\textup{(4)}] for any $m \in \mathbb{N}$ and $w \in W_{m}$, it holds that 
		\[
		\xi\Bigl(\{ v \in W_{m} \mid d_{G_{m}}(v, w) \le 1 \}\Bigr) \subseteq \bigl\{ z \in W_{m - 1} \bigm| d_{G_{m - 1}}\bigl(z, \xi(w)\bigr) \le 1 \bigr\}. 
		\]
    \end{itemize}
\end{prop}
\begin{proof}
	As mentioned in \cite{Kig21+}*{Proposition 11.4}, all statements can be easily verified. 
	Indeed, (0), (1), (2B), (2C) and (4) are immediate from the definition of generalized Sierpi\'{n}ski carpets. 
	A proof of (3) can be found in \cite{Kaj20+}*{Lemma 3.2} for example. 
	Finally, the condition (2A) follows by noting that 
	\begin{equation}\label{adapted-sep}
		\inf_{x \in K_{v}, y \in K_{w}}d(x, y) \ge \frac{1}{\sqrt{D}}a^{-m} \quad \text{if $m \ge 1$, $v, w \in W_{m}$ satisfy $K_{v} \cap K_{w} = \emptyset$.} 
	\end{equation}
\end{proof}

An intrinsic boundary $\partial_{\ast}G_{n}$ of the graph $G_{n}$ is the set of words that the associated $n$-cells intersect with the topological boundary of $[-1, 1]^{D}$, that is,
\[
\partial_{\ast}G_{n} \coloneqq W_{n}\Bigl[\partial [-1, 1]^{D}\Bigr] = \Bigl\{ w \in W_{n} \Bigm| K_{w} \cap \partial [-1, 1]^{D} \neq \emptyset \Bigr\}.
\]
We have the following proposition (see \cite{Kig21+}*{Assumption 2.10 and Proposition 2.16}). 
\begin{prop}\label{prop.boundary}
	Let $(K, S, \{ F_{i} \}_{i \in S}) = \GSC(D, a, S)$ be a generalized Sierpi\'{n}ski carpet. 
	Then $W_{k} \setminus \partial_{\ast}G_{k} \neq \emptyset$ for any $k \ge 2$. 
\end{prop}
\begin{proof}
	By \hyperlink{GSC1}{\textup{(GSC1)}} and \hyperlink{GSC4}{\textup{(GSC4)}}, we have 
	\[
	i_{\ast} \coloneqq (0)_{k = 1}^{D} \in S \quad \text{and} \quad \widehat{i}_{\ast} \coloneqq (a - 1)_{k = 1}^{D} \in S. 
	\] 
	Then we easily see that $i_{\ast}\widehat{i}_{\ast} \in W_{2} \setminus \partial_{\ast}G_{2}$. 
\end{proof}

We conclude this subsection by giving another aspect of \eqref{1-adapted}. 
Define
\begin{equation*}
    n(x, y) \coloneqq \max\left\{ m \in \bbZ_{\ge 0} \;\middle|\;
    \begin{array}{c}
    \text{there exist $v, w \in W_{m}$ such that} \\
    \text{$x \in K_{v}$, $y \in K_{w}$ and $K_{v} \cap K_{w} \neq \emptyset$} \\
    \end{array}
    \right\},
\end{equation*}
where we set $K_{\emptyset} \coloneqq K$.

\begin{lem}\label{lem.1-adapted}
    Let $(K, S, \{ F_{i} \}_{i \in S}) = \GSC(D, a, S)$ be a generalized Sierpi\'{n}ski carpet. 
    Then 
    \[
    \frac{1}{a\sqrt{D}}a^{-n(x, y)} \le d(x, y) \le 2a^{-n(x, y)}, \quad x, y \in K.
    \]
\end{lem}
\begin{proof}
	Let $x, y \in K$ and let $n = n(x, y)$. 
	The definition of $n(x, y)$ immediately implies that there exist $v, w \in W_{n}$ such that $x \in K_{v}$, $y \in K_{w}$ and $K_{v} \cap K_{w}$. 
	Hence 
	\[
	d(x, y) \le \diam\bigl(K_{v} \cup K_{w}, d\bigr) \le 2a^{-n}. 
	\]
	
	We next prove the converse inequality. 
	Let $v, w \in W_{n + 1}$ such that $x \in K_{v}$ and $y \in K_{w}$. 
	Then, by the definition of $n(x, y)$, we have $K_{v} \cap K_{w} = \emptyset$. 
	Therefore, 
	\[
	d(x, y) \ge \inf_{x' \in K_{v}, y' \in K_{w}}d(x', y') \ge \frac{1}{\sqrt{D}}a^{-(n + 1)}, 
	\] 
	where we used \eqref{adapted-sep}. 
	This completes the proof. 
\end{proof}

\subsection{$p$-energies and Poincar\'{e} constants on finite graphs}

In this subsection, we review some basic results and definitions in discrete nonlinear potential theory and introduce $(p, p)$-Poincar\'{e} constants that will play essential roles in this paper.

Let $G = (V, E)$ be a directed, connected, simple finite graph, and let $p > 0$.
We always suppose that $(x,y) \in E$ if and only if $(y, x) \in E$. 

\begin{dfn}
    For $f\colon V \to \bbR$, we define its \emph{$p$-energy} $\mcE_{p}^{G}(f)$ by setting
    \[
    \mcE_{p}^{G}(f) \coloneqq \frac{1}{2}\sum_{(x, y) \in E}\abs{f(x) - f(y)}^{p}.
    \]
\end{dfn}

\begin{dfn}\label{dfn.graphsub}
    For disjoint subsets $A, B$ of $V$, we define their \emph{$p$-conductance} $\mcC_{p}^{G}(A, B)$ by setting
    \[
    \mcC_{p}^{G}(A, B) \coloneqq \inf\bigl\{ \mcE_{p}^{G}(f) \bigm| f|_{A} \equiv 1, f|_{B} \equiv 0 \bigr\}.
    \]
    For a given subset $A$ of $V$, define
    \[
    E^{A} \coloneqq \{ (x, y) \in E \mid x, y \in A \},
    \]
    and
    \[
    \mcE_{p, A}(f) \coloneqq \frac{1}{2}\sum_{(x, y) \in E^{A}}\abs{f(x) - f(y)}^{p}.
    \]
    To clarify the underlying graph, we also write $\mathcal{E}_{p, A}^{G}(f)$ for $\mcE_{p, A}(f)$. 
    We also set
    \[
    \closure{A} \coloneqq \{ x \in V \mid \text{$x \in A$ or $(x, y) \in E$ for some $y \in A$} \}, 
    \]
    and $\bord{A} \coloneqq \closure{A} \setminus A$. 
\end{dfn}

Then the following monotonicity of $p$-conductance is immediate (see \cite{Mthesis}*{Proposition 3.7-(2)} for example).

\begin{prop}\label{prop.mono}
    Let $A, B, A', B' \subseteq V$ with $A \subseteq A'$ and $B \subseteq B'$.
    Then $\mcC_{p}^{G}(A, B) \le \mcC_{p}^{G}(A', B')$.
\end{prop}

The following property states the Markov property of discrete $p$-energy.
(This naming is borrowed from the case $p = 2$.)
This is also immediate from the definition.

\begin{prop}\label{prop.Markov}
    Let $\varphi\colon \bbR \to \bbR$ with $\mathrm{Lip}(\varphi) \le 1$.
    Then $\mcE_{p}^{G}(\varphi \circ f) \le \mcE_{p}^{G}(f)$ for any $f\colon V \to \bbR$.
    In particular, if we define $f^{\#} \coloneqq (f \vee 0) \wedge 1$, then $\mcE_{p}^{G}(f^{\#}) \le \mcE_{p}^{G}(f)$.
\end{prop}

Next we define some types of $(p, p)$-Poincar\'{e} constants.
Let $\nu$ be a non-negative measure on $V$, and let $\partial_{\ast}G \subsetneq V$ be a given non-empty subset.

\begin{dfn}\label{dfn.Poiconst}
    For a non-empty subset $A$ of $V$ and $f\colon A \to \bbR$, define its mean $\mean{f}{A, \nu}$ by setting
    \[
    \mean{f}{A, \nu} \coloneqq \frac{1}{\sum_{x \in A}\nu(x)}\sum_{x \in A}f(x)\nu(x).
    \]
    We define $\lambda_{p}^{(G, \nu)}$ on $(G, \nu)$ by setting
    \[
    \lambda_{p}^{(G, \nu)} \coloneqq \sup\left\{\, \frac{\sum_{x \in V}\abs{f(x) - \mean{f}{V, \nu}}^{p}\nu(x)}{\mcE_{p}^{G}(f)} \;\middle|\; f \in \mathbb{R}^{V}, \mcE_{p}^{G}(f) \neq 0 \,\right\}.
    \]
    We consider its Dirichlet boundary conditioned version $\lambda_{p, \textup{Dir}}^{(G, \nu)}(\partial_{\ast}G)$ defined as
    \[
    \lambda_{p, \textup{Dir}}^{(G, \nu)}(\partial_{\ast}G)
    \coloneqq \sup\left\{\, \frac{\abs{\mean{f}{V, \nu}}^{p}}{\mcE_{p}^{G}(f)} \;\middle|\; f \in \bbR^{V},  \mcE_{p}^{G}(f) \neq 0 \text{ and } f|_{\partial_{\ast}G} \equiv 0 \,\right\}.
    \]
    For disjoint subsets $A, B$ of $V$, we also define
    \[
    \sigma_{p}^{(G, \nu)}(A, B) \coloneqq \sup\left\{\, \frac{\abs{\mean{f}{A,\nu} - \mean{f}{B,\nu}}^{p}}{\mcE_{p, A \cup B}(f)} \;\middle|\; f \in \bbR^{A \cup B},  \mcE_{p, A \cup B}(f) \neq 0 \,\right\}.
    \]
\end{dfn}

By standard arguments in calculus of variations (see \cite{Mthesis}*{proof of Lemma 3.3} for example), one can easily prove the following proposition.

\begin{prop}\label{prop.minimizer}
	Suppose that $p > 1$ and that $G$ is connected. 
    Let $A, B$ be non-empty disjoint subsets of $V$, and let $\partial_{\ast}G \subsetneq V$ be non-empty.
    \begin{itemize}
        \setlength{\parindent}{0pt}
        \item [\textup{(1)}] There exists a unique $f \in \bbR^{V}$ such that $f|_{A} \equiv 1$, $f|_{B} \equiv 0$ and $\mcE_{p}^{G}(f) = \mcC_{p}^{G}(A, B)$.
        \item [\textup{(2)}] There exists a unique $f \in \bbR^{V}$ such that $f|_{\partial_{\ast}G} \equiv 0$, $\mean{f}{V, \nu} = 1$ and $\mcE_{p}^{G}(f)^{-1} = \lambda_{p, \textup{Dir}}^{(G, \nu)}(\partial_{\ast}G)$.
        \item [\textup{(3)}] If both $A$ and $B$ are connected, then there exists $f \in \bbR^{A \cup B}$ such that 
        \[
        \abs{\mean{f}{A, \nu} - \mean{f}{B, \nu}} = 1 \quad \text{and} \quad \mcE_{p, A \cup B}(f)^{-1} = \sigma_{p}^{(G, \nu)}(A, B). 
        \]
        Moreover, such $f$ is unique up to an additive constant and to the multiplication by $-1$. 
    \end{itemize}
\end{prop}

We conclude this subsection by introducing notations of these quantities in specific settings.
We mainly consider $p$-conductance and $(p,p)$-Poincar\'{e} constants on approximating graphs $G_{m}$ introduced in subsection \ref{subsec.SCappgraph}.
Note that, by the self-similarity of $\{ K_{w} \}_{w}$, each subgraph $(w \cdot W_{m}, E^{w \cdot W_{m}})$ is a copy of $G_{m}$ for any $w \in W_{\#}$ and $m \in \bbN$.
Recall that $\mu$ denotes the self-similar probability measure on $K$ with weight $(1/N_{\ast}, \dots, 1/N_{\ast})$.
We consider that $\mu$ is also a measure on $W_{n}$ by setting $\mu(w) \coloneqq \mu(K_{w}) = N_{\ast}^{-n}$ for each $w \in W_{n}$.
Then, for any subset $A$ of $W_{n}$ and $f\colon A \to \bbR$,
\[
\mean{f}{A,\mu} = \frac{1}{\#A}\sum_{w \in A}f(w),
\]
and thus we write $\mean{f}{A}$ to denote $\mean{f}{A, \mu}$ for simplicity.
For $w \in W_{\#}$ and $n \in \bbN$, we define
\[
\mcC_{p}^{(n)} \coloneqq \sup_{w \in W_{\#}}\mcC_{p}^{G_{n + \abs{w}}}\bigl(w \cdot W_{n}, W_{n + \abs{w} }\setminus \mcB_{n}(w, 1)\bigr),
\]
and $\mcR_{p}^{(n)} \coloneqq \bigl(\mcC_{p}^{(n)}\bigr)^{-1}$.
We also set $\lambda_{p}^{(n)} \coloneqq \lambda_{p}^{(G_{n}, \mu)}$ and $\lambda_{p, \textup{Dir}}^{(n)} \coloneqq \lambda_{p, \textup{Dir}}^{(G_{n}, \mu)}(\partial_{\ast}G_{n})$.
Finally, for $v, w \in W_{\#}$ with $\abs{v} = \abs{w}$, define
\[
\sigma_{p}^{(n)}(v, w) \coloneqq \sigma_{p}^{(G_{\abs{v} + n}, \mu)}(v \cdot W_{n}, w \cdot W_{n}),
\]
and
\[
\sigma_{p}^{(n)} \coloneqq \adjustlimits\sup_{m \ge 1}\max_{(v, w) \in \widetilde{E}_{m}}\sigma_{p}^{(n)}(v, w).
\]

\begin{rmk}
    Our definitions of Poincar\'{e} constants are slightly changed from the original definitions adopted in \cite{KZ92}.
    Indeed, $N_{\ast}^{-n}\lambda_{2}^{(n)}$ in our notation is the same as $\lambda_{n}$ in \cite{KZ92}.
    The situations are the same for other Poincar\'{e} constants $\sigma_{2}^{(n)}, \lambda_{2, \textup{Dir}}^{(n)}$.
\end{rmk}

\subsection{Main results}\label{subsec:results}

Now, we are ready to state the main results of this paper.
Let $(K, S, \{ F_{i} \}_{i \in S}) = \GSC(D, a, S)$ be a generalized Sierpi\'{n}ski carpet. 
Then, for $p > 0$, it is well-known that there exists $\rho_{p} > 0$ such that $\lim_{n \to \infty}\bigl(\mathcal{C}_{p}^{(n)}\bigr)^{1/n} = \rho_{p}^{-1}$ (see Theorem \ref{thm.sub-mult}). 

The following two theorems state detailed properties of our $(1,p)$-``Sobolev'' space $\mcF_p$ on $(K, d, \mu)$.

\begin{thm}\label{thm.main1}
    Assume that $p > \dim_{\textup{ARC}}(K, d)$.
    Then a function space $\mcF_{p}$ defined as
    \[
    \mcF_{p} \coloneqq \biggl\{ f \in L^{p}(K, \mu) \biggm| \sup_{n \ge 1}\rho_{p}^{n}\,\mcE_{p}^{G_{n}}(M_{n}f) < \infty \biggr\}
    \]
    is a reflexive and separable Banach space equipped with a norm $\norm{\,\cdot\,}_{\mcF_{p}}$ defined by
    \[
    \norm{f}_{\mcF_{p}} \coloneqq \norm{f}_{L^{p}} + \left(\sup_{n \ge 1}\rho_{p}^{n}\,\mcE_{p}^{G_{n}}(M_{n}f)\right)^{1/p}.
    \]
    Moreover, $\mcF_{p}$ is continuously embedded in a H\"{o}lder space $\mathcal{C}^{0, (\beta_{p} - \alpha)/p}$ on $K$, where $\beta_{p} \coloneqq \log{(N_{\ast}\rho_{p})}/\log{a}$ and
    \[
    \mathcal{C}^{0, (\beta_{p} - \alpha)/p} \coloneqq \biggl\{ f \in \mcC(K) \biggm| \sup_{x \neq y \in K}\frac{\abs{f(x) - f(y)}}{d(x, y)^{(\beta_{p} - \alpha)/p}} < \infty \biggr\}.
    \]
    Furthermore, $\mcF_{p}$ is dense in $\mcC(K)$ with respect to the supremum norm.
\end{thm}

\begin{thm}[Theorem \ref{thm.Besov-charact}]\label{thm.main2}
    Assume that $p > \dim_{\textup{ARC}}(K, d)$.
    Let $\beta_{p}$ be the same constant as in Theorem \ref{thm.main1}.
    Then $\mcF_{p}$ has the following expression: 
    \begin{equation}\label{eq.LB}
        \mcF_{p} = \Biggl\{ f \in L^{p}(K, \mu) \Biggm| \varlimsup_{r \downarrow 0} \int_{K}\mint_{B_{d}(x, r)}\frac{\abs{f(x) - f(y)}^{p}}{r^{\beta_{p}}}\,d\mu(y)d\mu(x) < \infty \Biggr\}.
    \end{equation}
\end{thm}
Note that we will show that $\beta_{p} \ge p$ for any $p > 0$ in Proposition \ref{prop.betap}.

In Section \ref{sec:constr}, we construct a ``canonical'' $p$-energy $\mcE_{p}$ on $(K, d, \mu)$, which satisfies the following properties.
For the definition of Clarkson's inequality, see Definition \ref{dfn.clarkson}.

\begin{thm}\label{thm.main1.2}
	Assume that $p > \dim_{\textup{ARC}}(K, d)$. 
    Then there exists a functional $\mcE_{p}\colon \mcF_{p} \to [0, \infty)$ such that $\mcE_{p}(\,\cdot)^{1/p}$ is a semi-norm satisfying Clarkson's inequality and the associated norm $\norm{\,\cdot\,}_{\mcE_{p}} \coloneqq \norm{\,\cdot\,}_{L^{p}} + \mcE_{p}(\,\cdot\,)^{1/p}$ is equivalent to $\norm{\,\cdot\,}_{\mcF_{p}}$.
    Furthermore, $(\mcE_{p}, \mcF_{p})$ satisfies the following conditions:
    \begin{itemize}
        \item [\textup{(1)}] $\indicator{K} \in \mcF_{p}$, and, for $f \in \mcF_{p}$, $\mcE_{p}(f) = 0$ if and only if $f$ is constant.
        Furthermore, $\mcE_{p}(f + a\indicator{K}) = \mcE_{p}(f)$ for any $f \in \mcF_p$ and $a \in \bbR$;
        \item [\textup{(2)}]\textup{(Regularity)} $\mcF_{p}$ is dense in $\mcC(K)$ with respect to the sup norm;
        \item [\textup{(3)}]\textup{(Markov property)} if $f \in \mcF_{p}$ and $\varphi\colon \bbR \to \bbR$ with $\mathrm{Lip}(\varphi) \le 1$, then $\varphi \circ f \in \mcF_{p}$ and $\mcE_{p}(\varphi \circ f) \le \mcE_{p}(f)$;
        \item [\textup{(4)}]\textup{(Symmetry)} if $f \in \mcF_{p}$ and $T \in \mathcal{G}_{0}$, then $f \circ T \in \mcF_{p}$ and $\mcE_{p}(f \circ T) = \mcE_{p}(f)$;
        \item [\textup{(5)}]\textup{(Self-similarity)} it holds that
        \begin{equation}\label{eq.domain-ss}
            \mcF_{p} = \{ f \in \mcC(K) \mid f \circ F_{i} \in \mcF_{p} \text{ for all $i \in S$} \}
        \end{equation}
        and, for every $f \in \mcF_{p}$,
        \begin{equation}\label{eq.ene-ss}
            \mcE_{p}(f) = \rho_{p}\sum_{i \in S}\mcE_{p}(f \circ F_{i});
        \end{equation}
        \item [\textup{(6)}]\textup{(Strong locality)} if $f, g \in \mcF_{p}$ satisfy $\supp[f] \cap \supp[g - a\indicator{K}] = \emptyset$ for some $a \in \bbR$, then $\mcE_{p}(f + g) = \mcE_{p}(f) + \mcE_{p}(g)$.
    \end{itemize}
\end{thm}

\begin{rmk}
    When $p = 2$, there exists the unique Dirichlet form (up to constant multiples) satisfying all conditions (1)-(5) by \cite{BBKT10}*{Theorem 1.2}, \cite{Hin13}*{Proposition 5.1} and \cite{Kaj13.osc}*{Proposition 5.9}
    \footnote{To be precise, the uniqueness was proved in \cite{BBKT10} in an alternative formulation of (1)-(5). In particular, there is no proof of the self-similarity condition (5) in \cite{BBKT10}. The identity \eqref{eq.domain-ss} was proved in \cite{Hin13}*{Proposition 5.1} and an explicit proof of \eqref{eq.ene-ss} was given in \cite{Kaj13.osc}*{Proposition 5.9}.}.
    This is the reason why we say that a $p$-energy $\mcE_{p}$ satisfying these conditions (1)-(5) is canonical.
    (We will see that the condition (6) is automatically deduced from a combination of (1) and (5).)
    However, we do not know whether or not such uniqueness also holds for $p$-energy.
\end{rmk}

We next introduce $\mcE_{p}$-energy measure $\mu^{p}_{\langle f \rangle}$ for $f \in \mcF_p$ and establish a few properties of it in Section \ref{sec:em} (Theorems \ref{thm.peneMarkov}, \ref{thm.peneSS} and \ref{thm.peneLocal}).
\begin{thm}\label{thm.main3}
	Assume that $p > \dim_{\textup{ARC}}(K, d)$.
    For any $f \in \mcF_{p}$, there exists a Borel finite measure $\mu^{p}_{\langle f \rangle}$ on $K$ with $\mu_{\langle f \rangle}^{p}(K) = \mcE_{p}(f)$ satisfying the following conditions:

    \noindent
    \textup{(1)} if $f, g \in \mcF_p$ and $A \in \mcB(K)$ satisfy $(f - g)|_{A} \equiv \text{const.}$, then $\mu_{\langle f \rangle}^{p}(A) = \mu_{\langle g \rangle}^{p}(A)$;

    \noindent
    \textup{(2)} \textup{(Chain rule)} \,for any $\Phi \in \mcC^{1}(\bbR)$, it holds that $d\mu_{\langle \Phi \circ f \rangle}^{p} = \abs{\Phi' \circ f}^{p}d\mu_{\langle f \rangle}^{p}$;

    \noindent
    \textup{(3)} \textup{(Self-similarity)} \,for any $n \in \bbN$, it holds that
    \begin{equation*}
        \mu_{\langle f \rangle}^{p}(dx) = \rho_{p}^{n}\sum_{w \in W_{n}}(F_{w})_{\ast}\mu_{\langle f \circ F_{w} \rangle}^{p}(dx),
    \end{equation*}
    where $(F_{w})_{\ast}\mu_{\langle f \circ F_{w} \rangle}^{p}(A) \coloneqq \mu_{\langle f \circ F_{w} \rangle}^{p}(F_{w}^{-1}(A))$ for any $A \in \mcB(K)$.
\end{thm}
As mentioned in the introduction, this measure $\mu^{p}_{\langle f \rangle}$ plays the role of $\abs{\nabla f(x)}^{p}\,dx$ in the case of Euclidean spaces.
To treat $\mcE_{2}$-energy measures, there are established frameworks in terms of Dirichlet forms.
For further development of $\mcE_{p}$-energy measures, the lack of $p$-energy form ``$\mcE_{p}(f; g)$'' (formally written as $(f, g) \mapsto \int\langle\abs{\nabla f}^{p - 2}\nabla f, \nabla g \rangle\,dx$) is a big obstacle.
This paper contains no results in this direction.
\begin{rmk}\label{rmk.p-form}
    For the Sierpi\'{n}ski gasket, Herman, Peirone and Strichartz \cite{HPS04} have constructed $p$-energy $\mcE_{p}^{\text{HPS}}(f)$, and Strichartz and Wong \cite{SW04} have suggested an approach to interpret $\mcE_{p}^{\text{HPS}}(f; g)$ as subderivatives of $t \mapsto \frac{1}{p}\mcE_{p}^{\text{HPS}}(f + tg)$ at $t = 0$.
    The notion of $p$-harmonicity and $p$-Laplacian based on this form $\mcE_{p}^{\text{HPS}}(f; g)$ are also considered in \cite{SW04}.
\end{rmk}

Lastly, we prove $\beta_{p} > p$ for \emph{planar} generalized Sierpi\'{n}ski carpets. 
\begin{thm}\label{thm.betap_strict}
	Suppose that $D = 2$. 
    Then $\rho_{p} > N_{\ast}^{-1}a^{p}$.
    In particular, $\beta_{p} > p$ for any $p > 0$.
\end{thm}

\begin{rmk}
	Our proof of Theorem \ref{thm.betap_strict} is inspired by \cite{Kaj20+}, where $\beta_{2} > 2$ is proved for all generalized Sierpi\'{n}ski carpets. 
	However, our argument is limited to the planar case due to the lack of suitable $p$-energy form ``$\mcE_{p}(f; g)$'' (see also Remark \ref{rmk.p-form}). 
	In a forthcoming paper \cite{KS22+}, the required $p$-energy forms will be constructed. 
	Using these $p$-energy forms and following the arguments in \cite{Kaj20+}, we can show $\beta_{p} > p$ for all generalized Sierpi\'{n}ski carpets without assuming $D = 2$. 
	The details will be provided in \cite{KS22+}. 
\end{rmk}

\section{Estimates of Poincar\'{e} constants and conductances}\label{sec:poincare}

In this section and Section \ref{sec:lowdim}, we investigate relations among $(p,p)$-Poincar\'{e} constants $\lambda_{p}^{(n)}$, $\lambda_{p, \textup{Dir}}^{(n)}$, $\sigma_{p}^{(n)}$ and $p$-conductances $\mcC_{p}^{(n)}$ (and its reciprocal $\mcR_{p}^{(n)}$).
Almost all parts of this section are $p$-energy analogs of \cite{KZ92}*{Section 2}.
The ultimate goal is to show that $\lambda_{p}^{(n)}$, $\sigma_{p}^{(n)}$ and $\mcR_{p}^{(n)}$ are comparable without depending on the level $n$.
In particular, the estimate $\sigma_{p}^{(n)} \vee \lambda_{p}^{(n)} \le C\mcR_{p}^{(n)}$ will be needed in later sections (especially Theorem \ref{thm.holder_KM} and Corollary \ref{cor.wm}).
However, we need some hard preparations to this end.
In the case $p = 2$, this was done in \cite{KZ92}*{Theorem 7.16} under two assumptions: \cite{KZ92}*{(B-1) and (B-2)}.
The following conditions are generalizations of these assumptions to fit our $p$-energy context.
\begin{enumerate}[label=\textup{(con)},align=left,leftmargin=*,topsep=4pt,parsep=0pt,itemsep=2pt]
	\item[\hypertarget{Bp}{\textup{(B$_{p}$)}}] There exist $k_{\ast} \in \bbZ_{\ge 0}$ and a positive constant $C_{\ast}$ (that depends only on $p$ and $N_{\ast}$) such that $\sigma_{p}^{(n)} \le C_{\ast}\lambda_{p, \textup{Dir}}^{(n+k_{\ast})}$ for every $n \in \bbN$.
	\item[\hypertarget{KMp}{\textup{(KM$_{p}$)}}] There exists $C_{\text{KM}} > 0$ such that $\lambda_{p}^{(n)} \le C_{\text{KM}}\,\mcR_{p}^{(n)}$ for every $n \in \bbN$.
\end{enumerate}
A proof of (B$_{2}$) for the Sierpi\'{n}ski carpet is given in \cite{KZ92}*{Proposition 8.1}, and we also prove \hyperlink{Bp}{\textup{(B$_{p}$)}} for all $p$ by a similar method to theirs in Section \ref{sec:lowdim} (see Proposition \ref{prop.B}).
The condition \hyperlink{KMp}{\textup{(KM$_{p}$)}} is essential for our goals.
We prove \hyperlink{KMp}{\textup{(KM$_{p}$)}} and show that $\lambda_{p}^{(n)}$, $\sigma_{p}^{(n)}$ and $\mcR_{p}^{(n)}$ are comparable in the next section (see Theorem \ref{thm.KM}).
This section is devoted to a part of preparations toward Theorem \ref{thm.KM}.

\begin{rmk}\label{rmk.CH}
    Kusuoka and Zhou have proved (KM$_{2}$) using the result of Barlow and Bass \cite{BB89} that is called the \emph{Knight Move argument} (see \cite{KZ92}*{Theorem 7.16}).
    The original Knight Move condition \cite{KZ92}*{condition (KM)} is a uniform estimate for discrete harmonic functions with some boundary conditions.
    We can check that a $p$-harmonic analog of \cite{KZ92}*{condition (KM)} is equivalent to \hyperlink{KMp}{\textup{(KM$_{p}$)}} under Assumption \ref{assum.lowdim}, which will be introduced later, and so we call the condition \hyperlink{KMp}{\textup{(KM$_{p}$)}} $p$-Knight Move instead.
    A recent study by Kigami reveals new important aspects of \hyperlink{KMp}{\textup{(KM$_{p}$)}}, and he introduced an important condition that is called \emph{$p$-conductive homogeneity}, which plays a similar role as \hyperlink{KMp}{\textup{(KM$_{p}$)}} (see \cite{Kig21+}*{Theorem 1.1 and 1.3}).
\end{rmk}

In this section, let $(K, S, \{ F_{i} \}_{i \in S}) = \GSC(D, a, S)$ be a generalized Sierpi\'{n}ski carpet and let $p > 0$. 

\subsection{Basic estimates without (B$_{p}$) and (KM$_{p}$)}
Let us start by preparing some basic facts.
The following proposition is easily derived from the definition of $(p, p)$-Poincar\'{e} constants (for $p = 2$, see \cite{KZ92}*{Proposition 1.5}).
\begin{prop}\label{prop.basic}
    Let $n \ge 1$, $m \ge 0$, $w \in W_{m}$ and $f \in \bbR^{W_{n + m}}$.
    \begin{itemize}
        \item [\textup{(1)}] It holds that
        \[
        \sum_{v \in w \cdot W_{n}}\abs{f(v) - \mean{f}{w \cdot W_{n}}}^{p} \le N_{\ast}^{n}\lambda_{p}^{(n)}\mcE_{p, w \cdot W_{n}}^{G_{n + m}}(f).
        \]
        In particular,
        \begin{equation}\label{est.basic1}
            \sum_{v \in W_{n}}\abs{f(v) - \mean{f}{W_{n}}}^{p} \le N_{\ast}^{n}\lambda_{p}^{(n)}\mcE_{p}^{G_n}(f).
        \end{equation}
        \item [\textup{(2)}]
        It holds that
        \begin{equation}\label{est.basic2}
             \abs{\mean{f}{w \cdot W_{n}} - \mean{f}{W_{n + m}}}^{p} \le N_{\ast}^{m}\Bigl(1 \vee N_{\ast}^{-(p - 1)n}\Bigr)\lambda_{p}^{(n + m)}\mcE_{p}^{G_{n + m}}(f).
        \end{equation}
        Moreover, for $w \in W_{n}$, $k \in \{ 1, \dots, n \}$ and $f \in \bbR^{W_{n}}$,
        \begin{align}\label{est.tele}
            \abs{\mean{f}{[w]_{n - k}\cdot W_{k}} - \mean{f}{[w]_{n - k + 1}\cdot W_{k - 1}}}^{p} \le N_{\ast}\lambda_{p}^{(k)}\mcE_{p, [w]_{n - k} \cdot W_k}^{G_{n + m}}(f).
        \end{align}
        \item [\textup{(3)}] For any $(v, w) \in \widetilde{E}_{m}$,
        \begin{equation}\label{est.basic3}
            \abs{\mean{f}{v\cdot W_{n}} - \mean{f}{w\cdot W_{n}}}^{p} \le \sigma_{p}^{(n)}\mcE_{p, \{ v, w \} \cdot W_{n}}^{G_{n + m}}(f).
        \end{equation}
    \end{itemize}
\end{prop}
\begin{proof}
    (1) This is immediate from the definition.

    (2) Note that a simple computation yields that
    \[
    \mean{f}{w \cdot W_{n}} - \mean{f}{W_{n + m}} = N_{\ast}^{-n}\sum_{v \in w \cdot W_{n}}(f(v) - \mean{f}{W_{n + m}}).
    \]
    Applying H\"{o}lder's inequality, we have that
    \begin{align*}
        \abs{\mean{f}{w \cdot W_{n}} - \mean{f}{W_{n + m}}}^{p}
        &\le N_{\ast}^{-pn}\Bigl(N_{\ast}^{(p - 1)n} \vee 1\Bigr)\sum_{v \in w \cdot W_{n}}\abs{f(v) - \mean{f}{W_{n + m}}}^{p} \\
        &\le N_{\ast}^{m}\Bigl(N_{\ast}^{-(n + m)} \vee N_{\ast}^{-(pn + m)}\Bigr)\sum_{v \in W_{n + m}}\abs{f(v) - \mean{f}{W_{n + m}}}^{p} \\
        &\le N_{\ast}^{m}\Bigl(1 \vee N_{\ast}^{-(p - 1)n}\Bigr)\lambda_{p}^{(n + m)}\mcE_{p}^{G_{n + m}}(f),
    \end{align*}
    which proves \eqref{est.basic2}.
    Lastly, by viewing $[w]_{n - k} \cdot W_{k}$ as a copy of $W_{k}$, we see that the estimate \eqref{est.basic2} becomes \eqref{est.tele}.

    (3) It is obvious from the definition.
\end{proof}

For $n \ge 0$ and $m \ge 1$, we define $P_{n + m, n}\colon \bbR^{W_{n + m}} \to \bbR^{W_{n}}$ by setting
\[
P_{n + m, n}f(w) \coloneqq \mean{f}{w \cdot W_{m}}, \quad w \in W_n.
\]
When $m$ is clear in the context, we abbreviate $P_n$ to denote $P_{n + m, n}$.
Note that $\mean{P_{n + m, m}f}{W_n} = \mean{f}{W_{n + m}}$ by a simple calculation.

While the following lemma is also immediate from the definition of $\sigma_{p}^{(n)}$ (for $p = 2$, see \cite{KZ92}*{Lemma 2.12}), this lemma will derive some important properties later.
In particular, the weak monotonicity (Corollary \ref{cor.wm}) comes from this lemma.

\begin{lem}\label{lem.wm}
    For all $n \ge 0$, $m \ge 1$, $f \in \bbR^{W_{n + m}}$ and a subset $A$ of $W_n$,
    \begin{align}\label{est.prewm}
        \mcE_{p, A}^{G_{n}}(P_{n+m, n}f) \le \Bigl(D^{(p - 1)} \vee 1\Bigr)L_{\ast}\sigma_{p}^{(m)}\mcE_{p, A \cdot W_m}^{G_{n + m}}(f).
    \end{align}
\end{lem}
\begin{proof}
    If $f$ is a constant function on $W_{n + m}$, then we have nothing to be proved.
    Let $f\colon W_{n + m} \to \bbR$ be a function that is not constant.
    For each $v, w \in W_n$, define a function $\tilde{f}[v, w]$ on $W_{n + m}$ by setting $\tilde{f}[v, w] \coloneqq \mcE_{p, \{ v, w \} \cdot W_{m}}^{G_{n + m}}(f)^{-1/p}\cdot f$.
    Now, it is a simple computation that 
    \begin{align*}
        \mcE_{p, A}^{G_{n}}(P_{n}f)
        &= \sum_{(v, w) \in E_{n}^{A}}\abs{\mean{f}{v \cdot W_m} - \mean{f}{w \cdot W_m}}^{p} \\
        &= \sum_{(v, w) \in E_{n}^{A}}\abs{\mean{\tilde{f}[v, w]}{v \cdot W_m} - \mean{\tilde{f}[v, w]}{w \cdot W_m}}^{p}\mcE_{p, \{ v, w \} \cdot W_{m}}^{G_{n + m}}(f) \\
        &\le \Bigl(D^{p - 1} \vee 1\Bigr)\sigma_{p}^{(m)}\sum_{(v, w) \in \widetilde{E}_{n}^{A}}\mcE_{p, \{ v, w \} \cdot W_{m}}^{G_{n + m}}(f) \\
        &\le \Bigl(D^{(p - 1)} \vee 1\Bigr)L_{\ast}\sigma_{p}^{(m)}\mcE_{p, A \cdot W_m}^{G_{n + m}}(f), 
    \end{align*}
    where we used Proposition \ref{prop.h-network} in the third line. 
\end{proof}

The following theorem states the \emph{submultiplicative inequality} of $\mcC_{p}^{(n)}$, whose proof can be found in many literatures (e.g. \cite{BK13}*{Proposition 3.6}, \cite{CP13}*{Lemma 3.7}, \cite{Kig20}*{Lemma 4.9.3} or \cite{Kig21+}*{Theorem 4.3}).

\begin{thm}\label{thm.sub-mult}
    There exists $C_{\ref{thm.sub-mult}} > 0$ (depending only on $p, L_{\ast}$) such that
    \begin{equation}\label{eq.sub-mult}
        \mcC_{p}^{(n + m)} \le C_{\ref{thm.sub-mult}}\mcC_{p}^{(n)}\mcC_{p}^{(m)} \quad \text{for any $n, m \in \mathbb{N}$.}
    \end{equation}
    In particular, the limit $\lim_{n \to \infty}\bigl(\mcC_{p}^{(n)}\bigr)^{1/n} \eqqcolon \rho_{p}^{-1} > 0$ exists and
    \[
    \mcC_{p}^{(n)} \ge C_{\ref{thm.sub-mult}}^{-1}\,\rho_{p}^{-n} \quad \text{for any $n \in \bbN$.}
    \]
\end{thm}

The constant $\rho_{p}$ in the above theorem will play indispensable roles in this paper.
The following proposition is an extension of \cite{KZ92}*{Proposition 2.7} and gives an estimate of $\rho_p$.
\begin{prop}\label{prop.betap}
    There exists a positive constant $C_{\ref{prop.betap}}$ depending only on $p, a, D, L_{\ast}$ such that
    \[
    \mcC_{p}^{(n)} \le C_{\ref{prop.betap}}(N_{\ast}a^{-p})^{n} \quad \text{for all $n \in \mathbb{N}$.}
    \]
    In particular, it holds that $\rho_{p} \ge N_{\ast}^{-1}a^{p}$.
\end{prop}
\begin{proof}
    Let $z \in W_{m}$ and set $A \coloneqq z \cdot W_{n}$, $B \coloneqq W_{n + m} \setminus \mcB_{n}(z, 1)$, $K_A \coloneqq \bigcup_{w \in A}K_w$, and $K_B \coloneqq \bigcup_{w \in B}K_w$.
    Then, by Lemma \ref{lem.1-adapted}, we have that $\dist{(K_A, K_B)} \coloneqq \inf\{ d(x, y) \mid x \in K_A, y \in K_B \} \ge ca^{-m}$, where $c$ is a positive constant depending only on $a, D$.
    Define a continuous function $f\colon K \to \bbR$ by setting
    \[
    f(x) \coloneqq \frac{\dist(x, K_A)}{\dist(K_A, K_B)} \wedge 1
    \]
    for each $x \in K$, where $\dist(x, F) \coloneqq \inf_{y \in F}d(x, y)$ for any subset $F$ of $K$.
    Then it is immediate that $f|_{K_A} \equiv 0$ and $f|_{K_B} \equiv 1$, and thus $M_{n + m}f|_{A} \equiv 0$ and $M_{n + m}f|_{B} \equiv 1$.
    This yields that $\mcC_{p}^{G_{n + m}}(A, B) \le \mcE_{p}^{G_{n + m}}(M_{n + m}f)$.

    Next, we will estimate the $p$-energy of $M_{n + m}f$ by estimating distances.
    For $(v, w) \in E_{n + m}$, by the triangle inequality, we have
    \begin{align*}
        \abs{\dist(F_v(x), K_A) - \dist(F_w(x), K_A)} \le d(F_v(x), F_w(x)) \le 2a^{-(n + m)}. 
    \end{align*}
    By Lemma \ref{lem.invariant},
    \begin{align*}
        &\abs{M_{n + m}f(v) - M_{n + m}f(w)} \\
        &\quad=
        \abs{\int_{K}F_{v}^{\ast}f\,d\mu - \int_{K}F_{w}^{\ast}f\,d\mu}  \\
        &\quad\le
        \frac{1}{\dist(K_A, K_B)}\int_{K}\abs{\dist(F_v(x), K_A) - \dist(F_w(x), K_A)}\,d\mu(x) \\
        &\quad\le
        c^{-1}a^{-n}.
    \end{align*}
    Consequently, we conclude that
    \begin{align*}
        \mcC_{p}^{G_{n + m}}(A, B)
        &\le
        \mcE_{p}^{G_{n + m}}(M_{n + m}f) \\
        &\le
        \sum_{w \in \mcB_{n}(z, 1)}\sum_{v \in W_{n + m}}\abs{M_{n + m}f(v) - M_{n + m}f(w)}^{p}\indicator{E_{n + m}}\bigl((v, w)\bigr) \\
        &\le
        c^{-p}L_{\ast}\bigl(\#\mcB_{n}(z, 1)\bigr)a^{-pn} 
        \le
        c^{-p}L_{\ast}(L_{\ast} + 1)N_{\ast}^{n}a^{-pn}. \qedhere
    \end{align*}
\end{proof}

Since ``gluing'' maximizers of $\lambda_{p, \textup{Dir}}^{(n)}$ does not increase energies, the next proposition follows (for $p = 2$, see \cite{KZ92}*{Proposition 2.11}).
\begin{prop}\label{prop.poiD}
    For all $n \ge 1$, 
    \[
    \lambda_{p, \textup{Dir}}^{(n)} \le N_{\ast}\Bigl(1 \vee N_{\ast}^{-(p - 1)n}\Bigr)\lambda_{p}^{(n + 1)}. 
    \]
\end{prop}
\begin{proof}
    Let $f\colon W_n \to \bbR$ satisfy $f|_{\partial_{\ast}G_{n}} \equiv 0$, $\mcE_{p}^{G_{n}}(f) = 1$, and $\mean{f}{W_{n}} = \bigl(\lambda_{p, \textup{Dir}}^{(n)}\bigr)^{1/p}$ (see Proposition \ref{prop.minimizer}-(2)).
    Fix $i \neq j \in S$ and define $f_{\ast}\colon W_{n + 1} \to \bbR$ by
    \[
    f_{\ast}(z) \coloneqq
    \begin{cases}
        f(w) \quad &\text{if $z = iw$ for some $w \in W_n$,} \\
        -f(w) \quad &\text{if $z = jw$ for some $w \in W_n$,} \\
        0 \quad &\text{otherwise.}
    \end{cases}
    \]
    Since $f|_{\partial_{\ast}G_{n}} = 0$, we easily see that $\mean{f_{\ast}}{W_{n + 1}} = 0$ and $\mcE_{p}^{G_{n + 1}}(f_{\ast}) = 2$.
    By H\"{o}lder's inequality,
    \begin{align*}
        &\Bigl(N_{\ast}^{-n} \vee N_{\ast}^{-pn}\Bigr)\sum_{z \in W_{n + 1}}\abs{f_{\ast}(z) - \mean{f_{\ast}}{W_{n + 1}}}^{p} \\
        &=
        2\Bigl((\#W_{n})^{-1} \vee (\#W_{n})^{-p}\Bigr)\sum_{w \in W_{n}}\abs{f(w)}^{p} \\
        &\ge
        2(\#W_{n})^{-p}\abs{\sum_{w \in W_{n}}f(w)}^{p}
        =
        2\abs{\mean{f}{W_{n}}}^{p}
        =
        2\lambda_{p, \textup{Dir}}^{(n)}.
    \end{align*}
    Combining with Proposition \ref{prop.basic}-(1), we finish the proof.
\end{proof}

Next, we see relations between two $(p, p)$-Poincar\'{e} constants $\lambda_{p}^{(n)}$ and $\sigma_{p}^{(n)}$.
The following proposition states that the submultiplicative inequality of $\sigma_{p}^{(n)}$ holds (for $p = 2$, see \cite{KZ92}*{a part of Proposition 2.13}).
\begin{prop}\label{thm.poi2}
    \begin{itemize}
        \item [\textup{(1)}] For any $n, m \in \bbN$, 
        \begin{equation*}
            \lambda_{p}^{(n + m)} \le \bigl(2^{p - 1} \vee 1\bigr)\Bigl\{ \lambda_{p}^{(m)}N_{\ast}^{-n} + L_{\ast}\bigl(D^{p - 1} \vee 1\bigr)\lambda_{p}^{(n)}\sigma_{p}^{(m)} \Bigr\}. 
        \end{equation*}
        \item [\textup{(2)}] For any $n, m \in \bbN$, 
        \begin{equation*}
        	\sigma_{p}^{(n + m)} \le \bigl(D^{p - 1} \vee 1\bigr)L_{\ast}\sigma_{p}^{(n)}\sigma_{p}^{(m)}. 
        \end{equation*}
    \end{itemize}
\end{prop}
\begin{proof}
    (1) Let $f\colon W_{n + m} \to \bbR$ with $\mcE_{p}^{G_{n + m}}(f) = 1$.
    Then we see from Proposition \ref{prop.basic}-(1) and Lemma \ref{lem.wm} that
    \begin{align*}
        & N_{\ast}^{-(n + m)}\sum_{v \in W_n}\sum_{w \in v \cdot W_m}\abs{f(w) - \mean{f}{W_{n + m}}}^{p} \\
        &\le \bigl(2^{p - 1} \vee 1\bigr)N_{\ast}^{-(n + m)}\sum_{\substack{v \in W_n, \\ w \in v \cdot W_m}}\left(\abs{f(w) - \mean{f}{v \cdot W_{m}}}^{p} + \abs{\mean{f}{v \cdot W_{m}} - \mean{P_{n+m,n}f}{W_{n}}}^{p}\right) \\
        &\le \bigl(2^{p - 1} \vee 1\bigr)\lambda_{p}^{(m)}N_{\ast}^{-n} + \bigl(2^{p - 1} \vee 1\bigr)N_{\ast}^{-n}\sum_{v \in W_n}\abs{P_{n+m,n}f(v) - \mean{P_{n+m,n}f}{W_{n}}}^{q} \\
        &\le \bigl(2^{p - 1} \vee 1\bigr)\Bigl\{ \lambda_{p}^{(m)}N_{\ast}^{-n} + L_{\ast}\bigl(D^{p - 1} \vee 1\bigr)\lambda_{p}^{(n)}\sigma_{p}^{(m)} \Bigr\}. 
    \end{align*}
    Since $f$ with $\mcE_{p}^{G_{n + m}}(f) = 1$ is arbitrary, we obtain the desired estimate.

    (2) Let $k \in \bbN$, let $(v, w) \in \widetilde{E}_k$ and let $f\in \bbR^{W_{n + m + k}}$ satisfy $\mcE_{p, \{ v, w \} \cdot W_{n + m}}(f) = 1$.
    Note that $\mean{f}{v \cdot W_{n + m}} = \mean{P_{n + k}f}{v \cdot W_n}$, where $P_{n + k} = P_{n + m + k, n + k}$.
    Indeed,
    \begin{align*}
        \mean{f}{v \cdot W_{n + m}}
        &= N_{\ast}^{-(n + m)}\sum_{v' \in W_n}\sum_{z \in vv' \cdot W_m}f(z) = N_{\ast}^{-n}\sum_{v' \in W_n}\mean{f}{vv' \cdot W_m} \\
        &= N_{\ast}^{-n}\sum_{v' \in W_n}P_{n + k}f(vv') = \mean{P_{n + k}f}{v \cdot W_n}.
    \end{align*}
    Similar computation yields that $\mean{f}{v \cdot W_{n + m}} = \mean{P_{m + k}f}{v \cdot W_m}$.
    From Lemma \ref{lem.wm}, 
    \begin{align*}
        \abs{\mean{f}{v \cdot W_{n + m}} - \mean{f}{w \cdot W_{n + m}}}^{p}
        &= \abs{\mean{P_{n + k}f}{v \cdot W_n} - \mean{P_{n + k}f}{w \cdot W_n}}^{p} \\
        &\le \sigma_{p}^{(n)}\mcE_{p, \{ v, w \} \cdot W_n}^{G_{n + k}}(P_{n + k}f) \\
        &\le \bigl(D^{p - 1} \vee 1\bigr)L_{\ast}\sigma_{p}^{(n)}\sigma_{p}^{(m)}.
    \end{align*}
    The desired result is immediate from this estimate.
\end{proof}

In the rest of this subsection, we prove the following relation with $(p, p)$-Poincar\'{e} constant $\lambda_{p}^{(n)}$ and $p$-conductance $\mcC_{p}^{(n)}$ (see \cite{KZ92}*{Proposition 2.10} for $p = 2$).

\begin{prop}\label{prop.poi1}
	Let $p \ge 1$. 
    For every $n, m, k \in \bbN$ with $W_k \setminus \partial_{\ast}G_k \neq \emptyset$,
    \[
    \frac{\#(W_k \setminus \partial_{\ast}G_k)}{N_{\ast}^{k}\#(\partial_{\ast}G_{k})}\mcR_{p}^{(m)}\lambda_{p}^{(n)} \le C_{\ref{prop.poi1}}\lambda_{p}^{(n + m + k)},
    \]
    where $C_{\ref{prop.poi1}}$ is a positive constant depending only on $p$ and $L_{\ast}$.
\end{prop}

Similar ideas of its proof appear in many contexts (see \cite{BB90}*{proof of Theorem 3.3}, \cite{BK13}*{proof of Proposition 3.6} for example).
In the following two lemmas, we prepare estimates for ``partition of unity'' (for $p = 2$, see \cite{KZ92}*{Lemmas 2.8 and 2.9}).
\begin{lem}\label{lem.unity}
    Let $n, m \in \bbN$, and let $\bigl\{ \varphi_{w}^{(m)} \bigr\}_{w \in W_{n}}$ be a family of $[0, 1]$-valued functions on $W_{n + m}$ such that $\sum_{w \in W_{n}}\varphi_{w}^{(m)} \equiv 1$ and $\bigl.\varphi_{w}^{(m)}\bigr|_{W_{n + m} \setminus \mcB_{m}(w, 1)} \equiv 0$ for each $w \in W_n$.
    If $f \in \bbR^{W_{n}}$, then
    \[
    \mcE_{p}^{G_{n+m}}(f_{\ast}) \le C_{\ref{lem.unity}}\mcE_{p}^{G_n}(f)\max_{w \in W_n}\mcE_{p}^{G_{n+m}}\Bigl(\varphi_{w}^{(m)}\Bigr)
    \]
    where $C_{\ref{lem.unity}} > 0$ is constant depending only on $p, L_{\ast}$, and $f_{\ast} \in \bbR^{W_{n + m}}$ is defined as
    \[
    f_{\ast}(z) \coloneqq \sum_{w \in W_{n}}f(w)\varphi_{w}^{(m)}(z), \quad z \in W_{n + m}.
    \]
\end{lem}
\begin{proof}
    For each $z, z' \in W_{n + m}$, we set
    \[
    A(z, z') \coloneqq \Bigl\{ w \in W_n \Bigm| \varphi_{w}^{(m)}(z) \vee \varphi_{w}^{(m)}(z') > 0 \Bigr\}.
    \]
    Since $\supp\bigl[\varphi_{w}^{(m)}\bigr] \subseteq \mcB_{m}(w, 1)$, we can verify that there exists $M \in \bbN$ depending only on $L_{\ast}$ such that $\#A(z, z') \le M$ for any $n, m \in \bbN$ and $z, z' \in W_{n + m}$.
    Furthermore, we see that
    \begin{align*}
        f_{\ast}(z) - f_{\ast}(z')
        =
        \sum_{w \in A(z, z')}f(w)\Bigl(\varphi_{w}^{(m)}(z) - \varphi_{w}^{(m)}(z')\Bigr)
    \end{align*}
    and $\sum_{w \in A(z, z')}(\varphi_{w}^{(m)}(z) - \varphi_{w}^{(m)}(z')) = 0$.
    From these identities, we have that
    \begin{align}\label{eq.unity.uastpene}
        &\mcE_{p}^{G_{n + m}}(f_{\ast}) \\
        &= \frac{1}{2}\sum_{w \in W_n}\sum_{z \in w \cdot W_m}\sum_{\substack{z' \in W_{n + m}; \\ (z,z') \in E_{n + m}}}\abs{f_{\ast}(z) - f_{\ast}(z')}^{p} \nonumber \\
        &= \frac{1}{2}\sum_{\substack{w \in W_n, \\ z \in w \cdot W_m}}\sum_{\substack{z' \in W_{n + m}; \\ (z,z') \in E_{n + m}}}\abs{\sum_{v \in A(z, z')}(f(v) - f(w))\Bigl(\varphi_{v}^{(m)}(z) -  \varphi_{v}^{(m)}(z')\Bigr)}^{p}. \nonumber
    \end{align}
    
    We first consider the case $p > 1$. 
    By H\"{o}lder's inequality we obtain
    \begin{align}\label{ineq.unity.decomp}
        &\abs{\sum_{v \in A(z, z')}(f(v) - f(w))\Bigl(\varphi_{v}^{(m)}(z) - \varphi_{v}^{(m)}(z')\Bigr)}^{p} \\
        \le &\left(\sum_{v \in A(z, z')}\abs{f(v) - f(w)}^{p}\right)\left(\sum_{v \in A(z, z')}\abs{\varphi_{v}^{(m)}(z) - \varphi_{v}^{(m)}(z')}^{p/(p - 1)}\right)^{p-1}. \nonumber
    \end{align}
    To bound the term $\sum_{v \in A(z, z')}\abs{f(v) - f(w)}^{p}$, for each $v \in A(z, z')$, $w \in W_{n}$ with $z \in w \cdot W_{m}$, and $(z, z') \in E_{n + m}$, we find a path $\bigl[w^1, \dots, w^{l}\bigr]$ in $G_n$ from $v$ to $w$ with $l \le 3$, that is, $(w^{i}, w^{i+1}) \in E_{n}$ or $w^{i} = w^{i + 1}$ for each $i = 1, \dots, l- 1$, and $w^{1} = v, w^{l} = w$.
    Define
    \[
    \Gamma^{(n)}_{\le 3}(w) \coloneqq \left\{ \Bigl[w^1, \dots, w^{l}\Bigr] \;\middle|\;
    \begin{array}{c}
         \text{$l \le 3$, $w^{i} \in W_{n}$, $w^{l} = w$, and}  \\
         \text{$(w^{i}, w^{i+1}) \in E_{n}$ for each $i = 1, \dots, l- 1$}
    \end{array}
    \right\} \neq \emptyset.
    \]
    Then, for any $(z, z') \in E_{n + m}$, we see that
    \begin{align}\label{ineq.unity.covering}
        &\sum_{v \in A(z, z')}\abs{f(v) - f(w)}^{p} \\
        &\quad\le C_{1}\sum_{v \in A(z, z')}\sum_{i = 1}^{l - 1}\abs{f(w^{i}) - f(w^{i + 1})}^{p} \nonumber \\
        &\quad\le
        C_{1}\sum_{[w^1, \dots, w^{l}] \in \Gamma^{(n)}_{\le 3}(w)}\sum_{i = 1}^{l - 1}\abs{f(w^{i}) - f(w^{i + 1})}^{p} \eqqcolon S_{f}(w), \nonumber
    \end{align}
    where $C_{1}$ is a constant depending only on $p, l$.
    Note that the number $\#\Gamma^{(n)}_{\le 3}(w)$ is bounded above by a constant depending only on $L_{\ast}$.
    Thus we conclude that there exists a constant $C_{2}$ depending only on $C_{1}$ and $L_{\ast}$ such that $\sum_{w \in W_n}S_{f}(w) \le C_{2}\mcE_{p}^{G_{n}}(f)$ for any $n \in \bbN$ and $f\colon W_{n} \to \bbR$.
    Combining these estimates \eqref{eq.unity.uastpene}, \eqref{ineq.unity.decomp} and \eqref{ineq.unity.covering}, we obtain
    \begin{align*}
        &\mcE_{p}^{G_{n + m}}(f_{\ast}) \\
        &\le \frac{1}{2}\sum_{w \in W_n}S_{f}(w)\sum_{z \in w \cdot W_m}\sum_{\substack{z' \in W_{n + m}; \\ (z,z') \in E_{n + m}}}\left(\sum_{v \in A(z, z')}\abs{\varphi_{v}^{(m)}(z) - \varphi_{v}^{(m)}(z')}^{p/(p - 1)}\right)^{p-1} \\
        &\le
        M^{p - 1}\sum_{w \in W_n}S_{f}(w)\sum_{z \in w \cdot W_m}\sum_{\substack{z' \in W_{n + m}; \\ (z,z') \in E_{n + m}}}\max_{v \in W_n}\abs{\varphi_{v}^{(m)}(z) - \varphi_{v}^{(m)}(z')}^{p} \\
        &\le
        M^{p - 1}\sum_{w \in W_n}S_{f}(w)\max_{v \in W_n}\mcE_{p, \closure{w \cdot W_m}}^{G_{n + m}}\Bigl(\varphi_{v}^{(m)}\Bigr) \\
        &\le
        C_{2}M^{p - 1}\mcE_{p}^{G_{n}}(f)\max_{v \in W_n}\mcE_{p}^{G_{n + m}}\Bigl(\varphi_{v}^{(m)}\Bigr), 
    \end{align*}
    which finishes the proof when $p > 1$. 
    
    If $p \in (0, 1]$, then we can use 
    \begin{align*}
        &\abs{\sum_{v \in A(z, z')}(f(v) - f(w))\Bigl(\varphi_{v}^{(m)}(z) - \varphi_{v}^{(m)}(z')\Bigr)}^{p} \\
        &\le \left(\sum_{v \in A(z, z')}\abs{f(v) - f(w)}^{p}\right)\left(\sum_{v \in A(z, z')}\abs{\varphi_{v}^{(m)}(z) - \varphi_{v}^{(m)}(z')}^{p}\right) 
    \end{align*}
    instead of \eqref{ineq.unity.decomp} and we obtain the desired estimate in a similar way. 
\end{proof}

\begin{lem}\label{lem.subdiv.fcn}
    Let $n, m, k \in \bbN$ with $W_k \setminus \partial_{\ast}G_k \neq \emptyset$.
    If $f \in \bbR^{W_{n}}$, then there exists a function $f_{\ast} \in \bbR^{W_{n + m + k}}$ satisfying
    \begin{align}\label{subdiv.ext}
        f_{\ast}(v) = f(w) \quad \text{if $w \in W_n$ and $v \in ww' \cdot W_m$ for some $w' \in W_k \setminus \partial_{\ast}G_k$,}
    \end{align}
    and
    \begin{align}\label{ineq.subdiv.fcn}
        \mcE_{p}^{G_{n + m + k}}(f_{\ast}) \le C_{\ref{lem.subdiv.fcn}}\#(\partial_{\ast}G_{k})\mcC_{p}^{(m)}\mcE_{p}^{G_{n}}(f),
    \end{align}
    where $C_{\ref{lem.subdiv.fcn}}$ is a positive constant depending only on $p$ and $L_{\ast}$.
\end{lem}
\begin{proof}
    For each $w \in W_{n + k}$, let $h_{w}^{(m)}\colon W_{n + m + k} \to [0, 1]$ satisfy $h_{w}^{(m)}|_{w \cdot W_{m}} \equiv 1$, $h_{w}^{(m)}|_{W_{n + m + k} \setminus \mcB_{m}(w, 1)} \equiv 0$, and
    \[
    \mcE_{p}^{G_{n + m + k}}\Bigl(h_{w}^{(m)}\Bigr) \le 2\mcC_{p}^{G_{n + m + k}}\bigl(w \cdot W_{m}, W_{n + m + k} \setminus \mcB_{m}(w, 1)\bigr).
    \]
    Define $\Psi \coloneqq \sum_{w \in W_{n + k}}h_{w}^{(m)}$.
    Then it is obvious that $\Psi \ge 1$, and so a family $\bigl\{ \varphi_{w}^{(m)} \bigr\}_{w \in W_{n + k}}$ given by $\varphi_{w}^{(m)} \coloneqq \Psi^{-1}h_{w}^{(m)}$ satisfies the conditions in Lemma \ref{lem.unity}.
    For each $f\colon W_{n} \to \bbR$, define a function $f_{\ast}\colon W_{n + m + k} \to \bbR$ by setting
    \[
    f_{\ast}(v) \coloneqq \sum_{z \in W_{n + k}}f([z]_{n})\varphi_{z}^{(m)}(v), \quad v \in W_{n + m + k}.
    \]
    We will prove that $f_{\ast}$ is the required function.

    First, we check \eqref{subdiv.ext}.
    Define $f_{n + k} \colon W_{n + k} \to \bbR$ by
    \[
    f_{n + k}(w) \coloneqq f([w]_{n}), \quad w \in W_{n + k}.
    \]
    Since $\supp\bigl[\varphi_{w}^{(m)}\bigr] \subseteq \mcB_{m}(w, 1)$, we can write
    \[
    f_{\ast}(v) = \sum_{\substack{\text{$z \in W_{n + k}$:} \\ \text{$v \in \mcB_{m}(z, 1)$}}}f_{n + k}(z)\varphi_{z}^{(m)}(v), \quad v \in W_{n + m + k}.
    \]
    Let $v \in W_{n + m + k}$ and $w \in W_n$ such that $v \in ww' \cdot W_{m}$ for some $w' \in W_k \setminus \partial_{\ast}G_k$.
    From $w' \not\in \partial_{\ast}G_k$, it follows that $\mcB_{m}(ww', 1) \cap \compl{(w\cdot W_{m + k})} = \emptyset$, and thus, for any $z \in W_{n + k}$ with $v \in \mcB_{m}(z, 1)$, we obtain $[z]_n = w$.
    From this observation, it holds that $f_{n + k}(z) = f(w)$, and thus we obtain
    \[
    f_{\ast}(v) = \sum_{\substack{\text{$z \in W_{n + k}$;} \\ \text{$v \in \mcB_{m}(z, 1)$}}}f([z]_n)\varphi_{z}^{(m)}(v) = \sum_{\substack{\text{$z \in W_{n + k}$;} \\ \text{$v \in \mcB_{m}(z, 1)$}}}f(w)\varphi_{z}^{(m)}(v) = f(w),
    \]
    which proves \eqref{subdiv.ext}.

    To prove \eqref{ineq.subdiv.fcn}, it will suffice to show the bound
    \begin{align}\label{est.unityabove}
        \max_{w \in W_{n + k}}\mcE_{p}^{G_{n + m + k}}\Bigl(\varphi_{w}^{(m)}\Bigr) \le c_{1}\mcC_{p}^{(m)},
    \end{align}
    where $c_{1}$ is a positive constant depending only on $p$ and $L_{\ast}$.
    Indeed, by Lemma \ref{lem.unity}, we have
    \begin{align*}
        \mcE_{p}^{G_{n + m + k}}(f_{\ast})
        &\le C_{\ref{lem.unity}}\mcE_{p}^{G_{n + k}}(f_{n + k})\max_{w \in W_{n + k}}\mcE_{p}^{G_{n + m + k}}\Bigl(\varphi_{w}^{(m)}\Bigr) \\
        &\le 2C_{\ref{lem.unity}}\#(\partial_{\ast}G_{k})\mcE_{p}^{G_{n}}(f)\max_{w \in W_{n + k}}\mcE_{p}^{G_{n + m + k}}\Bigl(\varphi_{w}^{(m)}\Bigr).
    \end{align*}
    A combination of this estimate and \eqref{est.unityabove} yields \eqref{ineq.subdiv.fcn}.
    Towards proving \eqref{est.unityabove}, we start by observing that $\#A_{m}(w) \le M$ for some constant $M$ depending only on $L_{\ast}$, where
    \[
    A_{m}(w) \coloneqq \Bigl\{ z \in W_{n + k} \Bigm| \mcB_{m}(w, 1) \cap \closure{\mcB_{m}(z, 1)} \neq \emptyset \Bigr\},
    \]
    for each $w \in W_{n + k}$.
    Indeed, we have $z \in A_{m}(w)$ if $z \in W_{n + k}$ satisfies $h_{w}^{(m)}(v) \wedge h_{z}^{(m)}(v) \neq 0$ for some $v \in W_{n + m + k}$, and thus we obtain (similarly to the bound of $A(z, z')$ in the proof of Lemma \ref{lem.unity}) that $\#A_{m}(w) \le M$ for any $n, m, k \in \bbN$ and $w \in W_{n + k}$.
    Now, it is a simple computation that, for any $v, v' \in W_{n + m + k}$,
    \[
    \varphi_{w}^{(m)}(v) - \varphi_{w}^{(m)}(v') = \frac{1}{\Psi(v)\Psi(v')}\biggl(\Psi(v)\Bigl(h_{w}^{(m)}(v) - h_{w}^{(m)}(v')\Bigr) - h_{w}^{(m)}(v)\bigl(\Psi(v) - \Psi(v')\bigr)\biggr).
    \]
    If we put $h_{w}(v, v') = h_{w}^{(m)}(v) - h_{w}^{(m)}(v')$ for each $w \in W_{n + k}$ and $(v, v') \in E_{n + m + k}$, then we have from the above identity that
    \begin{align*}
        &\mcE_{p}^{G_{n + m + k}}\Bigl(\varphi_{w}^{(m)}\Bigr) \\
        &\le \bigl(2^{p - 1} \vee 1\bigr)\sum_{(v, v') \in E_{n + m + k}}\Biggl(\frac{1}{\abs{\Psi(v')}^{p}}\abs{h_{w}(v, v')}^{p} + \frac{\abs{h_{w}^{(m)}(v)}^{p}}{\abs{\Psi(v)\Psi(v')}^{p}}\abs{\Psi(v) - \Psi(v')}^{p}\Biggr) \\
        &\le \bigl(2^{p - 1} \vee 1\bigr)\left(\mcE_{p}^{G_{n + m + k}}\Bigl(h_{w}^{(m)}\Bigr) + \sum_{(v, v') \in E_{n + m + k}}\abs{\sum_{w' \in A_{m}(w)}h_{w'}(v, v')}^{p}\,\right) \\
        &\underset{\text{($\ast$)}}{\le} \bigl((2M)^{p - 1} \vee 1\bigr)\left(\mcE_{p}^{G_{n + m + k}}\Bigl(h_{w}^{(m)}\Bigr) + \sum_{w' \in A_{m}(w)}\mcE_{p}^{G_{n + m + k}}\Bigl(h_{w'}^{(m)}\Bigr)\right) \\
        &\le \bigl((2M)^{p - 1} \vee 1\bigr)(M + 1)\max_{w \in W_{n + k}}\mcE_{p}^{G_{n + m + k}}\Bigl(h_{w}^{(m)}\Bigr)
        \le 2\bigl((2M)^{p - 1} \vee 1\bigr)(M + 1)\mcC_{p}^{(m)},
    \end{align*}
    where we used H\"{o}lder's inequality in ($\ast$).
    This shows \eqref{est.unityabove}.
\end{proof}

With these preparations in place, we are ready to prove Proposition \ref{prop.poi1}.

\noindent
\textit{Proof of Proposition \ref{prop.poi1}.}\,
    Let $f\colon W_{n} \to \bbR$ with $\mcE_{p}^{G_{n}}(f) = 1$, and let $f_{\ast}\in\bbR^{W_{n + m + k}}$ be a function obtained by applying Lemma \ref{lem.subdiv.fcn} to $f$.
    From Lemma \ref{lem.subdiv.fcn} and $\mcE_{p}^{G_{n}}(f) = 1$, we have $\mcE_{p}^{G_{n + m + k}}(f_{\ast}) \le C_{\ref{lem.subdiv.fcn}}\#(\partial_{\ast}G_{k})\mcC_{p}^{(m)}$.
    On the one hand, Proposition \ref{prop.basic}-(1) yields that
    \begin{align*}
        N_{\ast}^{-(n + m + k)}\sum_{w \in W_{n + m + k}}\abs{f_{\ast}(w) -\mean{f_{\ast}}{W_{n + m + k}}}^{p} \le C_{\ref{lem.subdiv.fcn}}\lambda_{p}^{(n + m + k)}\#(\partial_{\ast}G_{k})\mcC_{p}^{(m)}.
    \end{align*}
    On the other hand, from the property of $f_{\ast}$ in \eqref{subdiv.ext}, we have that
    \begin{align*}
        &N_{\ast}^{-(n + m + k)}\sum_{w \in W_{n + m + k}}\abs{f_{\ast}(w) -\mean{f_{\ast}}{W_{n + m + k}}}^{p} \\
        &\ge N_{\ast}^{-(n + m + k)}\sum_{w \in W_n}\sum_{w' \in W_k \setminus \partial_{\ast}G_k}\sum_{z \in ww'\cdot W_m}\abs{f(w) -\mean{f_{\ast}}{W_{n + m + k}}}^{p} \nonumber\\
        &\ge \#(W_k \setminus \partial_{\ast}G_k)N_{\ast}^{-(n + k)}\sum_{w \in W_n}\abs{f(w) -\mean{f_{\ast}}{W_{n + m + k}}}^{p} \nonumber\\
        &\ge 2^{-p}\frac{\#(W_k \setminus \partial_{\ast}G_k)}{\#W_{k}}N_{\ast}^{-n}\sum_{w \in W_n}\abs{f(w) -\mean{f}{W_{n}}}^{p},
    \end{align*}
    where we used $\sum_{w \in W_{n}}\abs{f(w) - c}^{p} \ge 2^{-p}\sum_{w \in W_n}\abs{f(w) -\mean{f}{W_{n}}}^{p}$ for any $c \in \bbR$ (see \cite{BB.NPT}*{Lemma 4.17} for example) in the last line, which requires $p \ge 1$.
    Since $f$ with $\mcE_{p}^{G_{n}}(f) = 1$ is arbitrary, we obtain the desired estimate. \hspace{\fill}$\square$
    
We conclude this subsection by giving a submultiplicative-like inequality of Poincar\'{e} constants. 
\begin{lem}\label{lem.sub-like}
	Let $p \ge 1$. 
	There exists a positive constant $C_{\ref{lem.sub-like}}$ depending only on $p$, $a$, $D$, $L_{\ast}$, $N_{\ast}$ such that for all $n, m \in \bbN$, 
	\begin{equation*}\label{eq.poi4-0}
    	\lambda_{p}^{(n)} \le C_{\ref{lem.sub-like}}N_{\ast}^{m}\lambda_{p}^{(m)}\sigma_{p}^{(n)}. 
    \end{equation*}
    In particular, 
    \begin{equation}\label{eq.poi4-0}
    	\lambda_{p}^{(n + m)} \le (2D)^{p - 1}L_{\ast}C_{\ref{lem.sub-like}}\lambda_{p}^{(m)}\sigma_{p}^{(n)}. 
    \end{equation}
\end{lem}
\begin{proof}
	Recall that, by Proposition \ref{prop.boundary}, $W_{k} \setminus \partial_{\ast}G_{k} \neq \emptyset$ if $k \ge 2$. 
	By Proposition \ref{prop.poi1}, we have that $\mcR_{p}^{(m - 2)}\lambda_{p}^{(n)} \le c_1\lambda_{p}^{(n + m)}$ for all $n \in \bbN$ and $m \ge 3$, where $c_1$ depends only on $a, D, N_{\ast}$ and $C_{\ref{prop.poi1}}$.
    Combining this estimate with Proposition \ref{thm.poi2}-(1), we obtain
    \begin{align*}
        \mcR_{p}^{(m - 2)}\lambda_{p}^{(n)} \le c_2\Bigl(N_{\ast}^{-m}\lambda_{p}^{(n)} + \lambda_{p}^{(m)}\sigma_{p}^{(n)}\Bigr),
    \end{align*}
    where $c_2 \coloneqq c_1L_{\ast}(2D)^{p - 1} > 0$.
    Since $N_{\ast}^{m}\mcR_{p}^{(m - 2)} \to \infty$ as $m \to \infty$ by Proposition \ref{prop.betap},
    there exists $M_{0} \in \bbN$ such that
    \[
    \inf_{m \ge M_{0}}N_{\ast}^{m}\mcR_{p}^{(m - 2)} \ge c_2 + 1.
    \]
    From these estimates, we have that $(c_2 + 1)N_{\ast}^{-m}\lambda_{p}^{(n)} \le c_2\bigl(N_{\ast}^{-m}\lambda_{p}^{(n)} + \lambda_{p}^{(m)}\sigma_{p}^{(n)}\bigr)$ for all $n \in \bbN$ and $m \ge M_{0}$.
    Hence, we conclude that
    \begin{equation*}
        \lambda_{p}^{(n)} \le c_2N_{\ast}^{m}\lambda_{p}^{(m)}\sigma_{p}^{(n)}. 
    \end{equation*}
    This together with Proposition \ref{thm.poi2}-(1) implies \eqref{eq.poi4-0}.
\end{proof}

\subsection{Comparability under (B$_{p}$) and (KM$_{p}$)}
Now we will prove submultiplicative inequality of $\lambda_{p}^{(\,\cdot\,)}$ under assuming \hyperlink{Bp}{\textup{(B$_{p}$)}} (see \cite{KZ92}*{Theorem 2.1} for $p = 2$). 
The following theorem gives a nonlinear analog of \cite{KZ92}*{Propositions 2.13 and 4.1}. 
\begin{thm}\label{thm.poi4}
	Let $p \ge 1$ and assume that \hyperlink{Bp}{\textup{(B$_{p}$)}} holds. 
    Then there exists a positive constant $C_{\ref{thm.poi4}}$ depending only on $p, a, D, L_{\ast}, N_{\ast}$ and the constants associated with \hyperlink{Bp}{\textup{(B$_{p}$)}} such that the following statements hold:
    \begin{itemize}
        \item [\textup{(1)}] for every $n \in \bbN$,
            \begin{equation}\label{eq.poi4-1}
                C_{\ref{thm.poi4}}^{-1}\sigma_{p}^{(n)} \le \lambda_{p}^{(n)} \le C_{\ref{thm.poi4}}\sigma_{p}^{(n)};
            \end{equation}
        \item [\textup{(2)}] for every $n, m \in \bbN$,
            \begin{equation}\label{eq.poi4-3}
                \lambda_{p}^{(n + m)} \le C_{\ref{thm.poi4}}\lambda_{p}^{(n)}\lambda_{p}^{(m)};
            \end{equation}
        \item [\textup{(3)}] for every $n \in \bbN$,
            \begin{equation}\label{eq.poi4-2}
                C_{\ref{thm.poi4}}^{-1}\lambda_{p}^{(n)} \le \lambda_{p, \textup{Dir}}^{(n)} \le C_{\ref{thm.poi4}}\lambda_{p}^{(n)};
            \end{equation}
        \item [\textup{(4)}] for every $n \in \bbN$,
            \begin{equation}\label{eq.poi4-4}
                \mcR_{p}^{(n)} \le C_{\ref{thm.poi4}}\lambda_{p}^{(n)}.
            \end{equation}
    \end{itemize}
\end{thm}
\begin{proof}
    (1) By \hyperlink{Bp}{\textup{(B$_{p}$)}}, Proposition \ref{prop.poiD} and \eqref{eq.poi4-0} in Lemma \ref{lem.sub-like}, for $n \ge 1$,  
    \[
    \sigma_{p}^{(n)} \le (2D)^{p - 1}C_{\ast}N_{\ast}L_{\ast}C_{\ref{lem.sub-like}}\sigma_{p}^{(k_{\ast} + 1)}\lambda_{p}^{(n)}. 
    \]
    This implies that $\sigma_{p}^{(n)} \le c_{1}\lambda_{p}^{(n)}$ for any $n \in \bbN$, where $c_1$ is a positive constant depending only on $p, a, D, N_{\ast}, L_{\ast}$ and the constants associated with \hyperlink{Bp}{\textup{(B$_{p}$)}}.
    Lemma \ref{lem.sub-like} also implies that $\lambda_{p}^{(n)} \le C_{\ref{lem.sub-like}}N_{\ast}\lambda_{p}^{(1)}\sigma_{p}^{(n)}$ for any $n \in \bbN$.
    Hence we get \eqref{eq.poi4-1}.

    (2) The submultiplicative inequality \eqref{eq.poi4-3} is immediate from Proposition \ref{thm.poi2}-(2) and \eqref{eq.poi4-1}.

    (3) Applying Proposition \ref{thm.poi2}-(2) and \hyperlink{Bp}{\textup{(B$_{p}$)}}, we have that $\sigma_{p}^{(n)} \le C_{\ast}L_{\ast}\sigma_{p}^{(k_{\ast})}\lambda_{p, \textup{Dir}}^{(n)}$ for all $n \ge k_{\ast} + 1$, which together with \eqref{eq.poi4-1} implies that
    \begin{equation}\label{eq.poi4-41}
        \lambda_{p}^{(n)} \le c_2\lambda_{p, \textup{Dir}}^{(n)} \quad \text{for any $n \in \bbN$},
    \end{equation}
    where $c_2$ is a positive constant depending only on $p, a, D, N_{\ast}, L_{\ast}$ and the constants associated with \hyperlink{Bp}{\textup{(B$_{p}$)}}.
    The converse inequality of \eqref{eq.poi4-41} is immediate from Proposition \ref{prop.poiD} and \eqref{eq.poi4-3}.

    (4) We immediately get \eqref{eq.poi4-4} from Proposition \ref{prop.poi1} and \eqref{eq.poi4-3}.
\end{proof}

Next we derive supermultiplicative inequalities of $(p,p)$-Poincar\'{e} constants under assuming both \hyperlink{Bp}{\textup{(B$_{p}$)}} and \hyperlink{KMp}{\textup{(KM$_{p}$)}}. 
\begin{thm}\label{thm.superP}
	Let $p \ge 1$ and assume that both \hyperlink{Bp}{\textup{(B$_{p}$)}} and \hyperlink{KMp}{\textup{(KM$_{p}$)}} hold.
    Then there exists a positive constant $C_{\ref{thm.superP}}$ (depending only on $p, a, D, L_{\ast}, N_{\ast}$ and the constants associated with \hyperlink{Bp}{\textup{(B$_{p}$)}} and \hyperlink{KMp}{\textup{(KM$_{p}$)}}) such that for all $n, m \in \mathbb{N}$ 
    \begin{equation}\label{super-lambda}
        C_{\ref{thm.superP}}^{-1}\lambda_{p}^{(n)}\lambda_{p}^{(m)} \le \lambda_{p}^{(n + m)} \le C_{\ref{thm.superP}}\lambda_{p}^{(n)}\lambda_{p}^{(m)}, 
    \end{equation}
    and 
    \begin{equation}\label{super-res}
    	C_{\ref{thm.superP}}^{-1}\mathcal{R}_{p}^{(n)}\mathcal{R}_{p}^{(m)} \le \mathcal{R}_{p}^{(n + m)} \le C_{\ref{thm.superP}}\mathcal{R}_{p}^{(n)}\mathcal{R}_{p}^{(m)}. 
    \end{equation}
\end{thm}
\begin{proof}
	Thanks to Theorem \ref{thm.poi4}-(2), in order to show \eqref{super-lambda} it will suffice to prove the supermultiplicative inequality: $\lambda_{p}^{(n)}\lambda_{p}^{(m)} \lesssim \lambda_{p}^{(n + m)}$. 
    From Proposition \ref{prop.poi1} and Theorem \ref{thm.poi4}-(2), we have that $\lambda_{p}^{(n)}\mcR_{p}^{(m)} \le c_{1}\lambda_{p}^{(n + m)}$ for any $n, m \in \bbN$, where $c_{1}$ depends only on $p, a, D, L_{\ast}, N_{\ast}$ and the constants associated with \hyperlink{Bp}{\textup{(B$_{p}$)}}.
    By \hyperlink{KMp}{\textup{(KM$_{p}$)}}, we deduce that
    \[
    C_{\text{KM}}^{-1}\,\lambda_{p}^{(n)}\lambda_{p}^{(m)} \le \lambda_{p}^{(n)}\mcR_{p}^{(m)} \le c_{1}\lambda_{p}^{(n + m)}, 
    \]
    and hence we have $\lambda_{p}^{(n)}\lambda_{p}^{(m)} \le c_{1}C_{\text{KM}}\lambda_{p}^{(n + m)}$. 
    
    Since $\lambda_{p}^{(n)}$ and $\mathcal{R}_{p}^{(n)}$ are comparable by Theorem \ref{thm.poi4}-(4) and \hyperlink{KMp}{\textup{(KM$_{p}$)}}, the multiplicative inequality \eqref{super-res} follows from \eqref{super-lambda}. 
\end{proof}

Note that $\sigma_{p}^{(n)}$ and $\lambda_{p}^{(n)}$ are also comparable under the situation of Theorem \ref{thm.superP}. 
From these multiplicative inequalities, we deduce the following behaviors of $\mathcal{R}_{p}^{(n)}$, $\lambda_{p}^{(n)}$ and $\sigma_{p}^{(n)}$: for all $n \in \mathbb{N}$, 
\begin{align}\label{Poi-super}
    c_{\ast}^{-1}\rho_{p}^{n} \le \mathcal{R}_{p}^{(n)} \le c_{\ast}\rho_{p}^{n}, 
    \quad c_{\ast}^{-1}\rho_{p}^{n} \le \lambda_{p}^{(n)} \le c_{\ast}\rho_{p}^{n} 
    \quad\text{and}\quad c_{\ast}^{-1}\rho_{p}^{n} \le \sigma_{p}^{(n)} \le c_{\ast}\rho_{p}^{n}, 
\end{align}
where $\rho_{p} = \lim_{n \to \infty}\bigl(\mcR_{p}^{(n)}\bigr)^{1/n}$ (see Theorem \ref{thm.sub-mult}) and $c_{\ast}$ depends only on $p, a, D,$ $N_{\ast}$, $L_{\ast}$ and the constants associated with \hyperlink{Bp}{\textup{(B$_{p}$)}} and \hyperlink{KMp}{\textup{(KM$_{p}$)}}.

\section{Checking (B$_{p}$) and (KM$_{p}$) for GSCs}\label{sec:lowdim}

This section gives $p$-energy analogs of \cite{KZ92}*{Lemma 3.9, Proposition 3.10, Theorem 7.2, (B-1) and (B-2)}.
The condition \hyperlink{Bp}{\textup{(B$_{p}$)}} can be proved in a combinatorial way. 
In order to prove \hyperlink{KMp}{\textup{(KM$_{p}$)}}, we will assume that $p > \dim_{\textup{ARC}}(K, d)$. 
This assumption is needed to derive (uniform) H\"{o}lder estimates in Theorem \ref{thm.unifhol}.

In this section, let $(K, S, \{ F_{i} \}_{i \in S}) = \GSC(D, a, S)$ be a generalized Sierpi\'{n}ski carpet and let $p > 1$.  

\subsection{Proof of (B$_{p}$)}\label{subsec:lowdim}
The condition \hyperlink{Bp}{\textup{(B$_{p}$)}} plays the converse role of the Proposition \ref{prop.poiD}.
Let us start by providing preparations from \emph{asymptotic geometry} in order to simplify several arguments.
The following definition extends the notion of \emph{rough isometry among graphs} to that among sequences of graphs.

\begin{dfn}\label{dfn.URI}
    For each $i = 1, 2$, let $\bigl\{ G_{n}^{i} = \bigl(V_{n}^{i}, E_{n}^{i}\bigr) \bigr\}_{n \ge 1}$ be a series of finite graphs with
    \begin{equation}\label{eq.bg}
        L_{\ast}^{i} \coloneqq \sup_{n \in \bbN}\max_{x \in V_{n}^{i}}\#\bigl\{ y \in V_{n}^{i} \bigm| (x, y) \in E_{n}^{i} \bigr\} < \infty.
    \end{equation}
    A family of maps $\{ \varphi_{n} \}_{n \ge 1}$, where $\varphi_{n}\colon V_{n}^{1} \to V_{n}^{2}$, is said to be a \emph{uniform rough isometry} from $\bigl\{ G_{n}^{1} \bigr\}_{n \ge 1}$ to $\bigl\{ G_{n}^{2} \bigr\}_{n \ge 1}$ if:
    \begin{itemize}
        \item [(1)] there exist constants $C_{1}, C_{2}$ such that, for every $n \in \bbN$ and $x, y \in V_{n}^{1}$,
        \[
        C_{1}^{-1}d_{G_{n}^{1}}(x, y) - C_{2} \le d_{G_{n}^{2}}(\varphi_{n}(x), \varphi_{n}(y)) \le C_{1}d_{G_{n}^{1}}(x, y) + C_{2};
        \]
        \item [(2)] there exists a constant $C_{3}$ such that, for every $n \in \bbN$,
        \[
        \bigcup_{x \in V_{n}^{1}}B_{d_{G_{n}^{2}}}(\varphi_{n}(x), C_{3}) = V_{n}^{2};
        \]
        \item [(3)] there exists a constant $C_{4}$ such that, for every $n \in \bbN$ and $x \in V_{n}^{1}$,
        \[
        C_{4}^{-1} \le \frac{\#\bigl\{ y' \in V_{n}^{2} \bigm| (\varphi_{n}(x), y') \in E_{n}^{2} \bigr\}}{\#\bigl\{ y \in V_{n}^{1} \bigm| (x, y) \in E_{n}^{1} \bigr\}} \le C_{4}.
        \]
    \end{itemize}
\end{dfn}

\begin{rmk}
    Since each $\varphi_{n}$ is a rough isometry from $G_{n}^{1}$ to $G_{n}^{2}$, there exists a rough isometry $\widetilde{\varphi}_n$ from $G_{n}^{2}$ to $G_{n}^{1}$.
    Moreover, we can choose $\widetilde{\varphi}_n$ so that $\{ \widetilde{\varphi}_n \}_{n \ge 1}$ is a uniform rough isometry from $\{ G_{n}^{2} \}_{n \ge 1}$ to $\{ G_{n}^{1} \}_{n \ge 1}$.
    Consequently, being uniform rough isometry gives an equivalence relation among series of finite graphs satisfying \eqref{eq.bg}.
\end{rmk}

Then the following stability result holds.
Its proof is a straightforward modification of \cite{Mthesis}*{proof of Lemma 8.5}, and so we omit it here (see Appendix \ref{sec:URI} for a proof).

\begin{lem}\label{lem.URI}
    \it
    Let $\bigl\{ G_{n}^{i} = \bigl(V_{n}^{i}, E_{n}^{i}\bigr) \bigr\}_{n \ge 1}$ be a series of finite graphs with
    \[
    L_{\ast}^{i} = \sup_{n \in \bbN}\max_{x \in V_{n}^{i}}\#\bigl\{ y \in V_{n}^{i} \bigm| (x, y) \in E_{n}^{i} \bigr\} < \infty,
    \]
    for each $i = 1, 2$, and let $\varphi_{n}\colon V_{n}^{1} \to V_{n}^{2}$ be a uniform rough isometry from $\bigl\{ G_{n}^{1} \bigr\}_{n \ge 1}$ to $\bigl\{ G_{n}^{2} \bigr\}_{n \ge 1}$.
    Then there exists a positive constant $C_{\textup{URI}}$ (depending only on $C_{1}, C_{2}$ in Definition \ref{dfn.URI}, $L_{\ast}^{1}$ and $p$) such that
    \begin{equation}\label{eq.URI}
        \mcE_{p}^{G_{n}^{1}}(f \circ \varphi_{n}) \le C_{\textup{URI}}\,\mcE_{p}^{G_{n}^{2}}(f),
    \end{equation}
    for every $n \in \bbN$ and $f\colon V_{n}^{2} \to \bbR$.
    In particular,
    \begin{equation}\label{eq.URIcap}
    	\mcC_{p}^{G_{n}^{1}}\bigl(\varphi_{n}^{-1}(A_{n}), \varphi_{n}^{-1}(B_{n})\bigr) \le C_{\textup{URI}}\mcC_{p}^{G_{n}^{2}}(A_{n}, B_{n})
    \end{equation}
    for every $n \in \bbN$ and disjoint subsets $A_{n}, B_{n}$ of $V_{n}^{2}$.
\end{lem}

Next we introduce variants of Poincar\'{e} constants $\widetilde{\lambda}_{p, \textup{Dir}}^{\,(\,\cdot\,)}$ and $\widetilde{\sigma}_{p}^{\,(\,\cdot\,)}$ by setting 
\[
\widetilde{\lambda}_{p, \textup{Dir}}^{\,(n)} \coloneqq \lambda_{p, \textup{Dir}}^{\bigl(\widetilde{G}_{n}, \mu\bigr)}(\partial_{\ast}G_{n}), \quad \text{and} \quad \widetilde{\sigma}_{p}^{\,(n)} \coloneqq \adjustlimits\sup_{m \ge 1}\max_{(v, w) \in \widetilde{E}_{m}}\sigma_{p}^{\bigl(\widetilde{G}_{m + n}, \mu\bigr)}(v \cdot W_{n}, w \cdot W_{n}). 
\]
We can easily verify that there exists a uniform rough isometry between $\{ G_{n} \}_{n \ge 0}$ and $\bigl\{ \widetilde{G}_{n} \bigr\}_{n \ge 0}$ (with $C_{1}, C_{2}, C_{3}, C_{4}$ depend only on $D, L_{\ast}$). 
Our new Poincar\'{e} constants have the following variational expressions: 
\[
\widetilde{\lambda}_{p, \textup{Dir}}^{\,(n)} = \inf\Bigl\{ \mathcal{E}_{p}^{\widetilde{G}_{n}}(f) \Bigm| f \in \mathbb{R}^{W_{n}}, f|_{\partial_{\ast}G_{n}} \equiv 0, \mean{f}{G_{n}} = 1 \Bigr\}^{-1}
\]
and, for $(v, w) \in \widetilde{E}_{m}$,
\begin{align*}
	&\sigma_{p}^{\bigl(\widetilde{G}_{m + n}, \mu\bigr)}(v \cdot W_{n}, w \cdot W_{n}) \\
	&= \inf\biggl\{ \mathcal{E}_{p, \{ v, w \} \cdot W_{n}}^{\widetilde{G}_{n + m}} \biggm| f \in \mathbb{R}^{\{ v, w \} \cdot W_{n}}, \mean{f}{v \cdot W_{n}} - \mean{f}{w \cdot W_{n}} = 1 \biggr\}^{-1}. 
\end{align*}
By noting that average $\mean{f}{A, \mu}$ does not depend on edge sets $E_{n}$ and $\widetilde{E}_{n}$, we have the following lemma as a consequence of Lemma \ref{lem.URI} (see also \cite{Kig21+}*{Section 18}). 

\begin{lem}\label{lem.Poi-URI}
	There exists a constant $C_{\ref{lem.Poi-URI}}$ depending only on $p, D, L_{\ast}$ such that
	\[
	C_{\ref{lem.Poi-URI}}^{-1}\lambda_{p, \textup{Dir}}^{(n)} \le \widetilde{\lambda}_{p, \textup{Dir}}^{\,(n)} \le C_{\ref{lem.Poi-URI}}\lambda_{p, \textup{Dir}}^{(n)} \quad \text{for all $n \in \mathbb{N}$,}
	\]
	and 
	\[
	C_{\ref{lem.Poi-URI}}^{-1}\sigma_{p}^{(n)} \le \widetilde{\sigma}_{p}^{\,(n)} \le C_{\ref{lem.Poi-URI}}\sigma_{p}^{(n)} \quad \text{for all $n \in \mathbb{N}$.}
	\]
\end{lem}

Now, we prove \hyperlink{Bp}{\textup{(B$_{p}$)}}. 
We heavily use the symmetries of the underlying (generalized) Sierpi\'{n}ski carpet, namely the symmetries of $[-1, 1]^{D}$, to prove \hyperlink{Bp}{\textup{(B$_{p}$)}}.
Recall notations in Definitions \ref{dfn.hyperplane}, \ref{dfn.symmetry} and \ref{dfn.word}. 
\begin{prop}\label{prop.B}
    There exists a positive constant $C_{\ast}$ depending only on $p, D, N_{\ast}$, $L_{\ast}$ such that $\sigma_{p}^{(n)} \le C_{\ast}\lambda_{p, \textup{Dir}}^{(n + 2)}$ for any $n \in \bbN$.
    In particular, \hyperlink{Bp}{\textup{(B$_{p}$)}} holds (with $k_{\ast} = 2$).
\end{prop}
\begin{proof}
    The proof is a straightforward (but complicated) generalization of \cite{KZ92}*{Proposition 8.1}, where only the standard Sierpi\'{n}ski carpet is considered. 
    Thanks to Lemma \ref{lem.Poi-URI}, it is enough to show that $\widetilde{\sigma}_{p}^{\,(n)} \lesssim \widetilde{\lambda}_{p, \textup{Dir}}^{\,(n + 2)}$. 
    
    Let $n \in \bbN$.
    Using \hyperlink{GSC1}{\textup{(GSC1)}}, \hyperlink{GSC4}{\textup{(GSC4)}} and the self-similarity of $K$, we easily see that $\widetilde{\sigma}_{p}^{\,(n)} = \widetilde{\sigma}_{p}^{\,\bigl(\widetilde{G}_{n + 1}, \mu\bigr)}(\mathbf{0}\cdot W_{n}, \mathbf{e}_{D} \cdot W_{n})$.
    By Proposition \ref{prop.minimizer}-(3), there exists $f\colon \{ \mathbf{0}, \mathbf{e}_{D} \} \cdot W_{n} \to \bbR$ with $\mathcal{E}_{p, \{ \mathbf{0}, \mathbf{e}_{D} \} \cdot W_{n}}^{\widetilde{G}_{n + 1}}(f) = 1$ such that $\abs{\mean{f}{\mathbf{0} \cdot W_{n}} - \mean{f}{\mathbf{e}_{D} \cdot W_{n}}} = \bigl(\widetilde{\sigma}_{p}^{\,(n)}\bigr)^{1/p}$. 
    Adding a constant if necessary, we may assume that $\mean{f}{\mathbf{0} \cdot W_{n}} = -\mean{f}{\mathbf{e}_{D} \cdot W_{n}}$. 
    Then, by the self-similarity and the uniqueness (Proposition \ref{prop.minimizer}-(3)), we can show that 
    \begin{equation}\label{Bp-sym}
    	f(\mathbf{e}_{D} \cdot w) = -f\bigl(\mathbf{0}\cdot\tau[R_{D}](w)\bigr), \quad w \in W_{n}. 
    \end{equation}
    Indeed, a function $f_{\ast}\colon \{ \mathbf{0}, \mathbf{e}_{D} \} \cdot W_{n} \to \bbR$ given by 
    \[
    f_{\ast}(iw) = 
    \begin{cases}
    	f\bigl(\mathbf{e}_{D} \cdot \tau[R_{D}](w)\bigr) \quad &\text{if $i = \mathbf{0}$ and $w \in W_{n}$,} \\
    	f\bigl(\mathbf{0} \cdot \tau[R_{D}](w)\bigr) \quad &\text{if $i = \mathbf{e}_{D}$ and $w \in W_{n}$,}
    \end{cases}
    \]
    satisfies $\mathcal{E}_{p, \{ \mathbf{0}, \mathbf{e}_{D} \} \cdot W_{n}}^{\widetilde{G}_{n + 1}}(f_{\ast}) = 1$, $\abs{\mean{f_{\ast}}{\mathbf{0} \cdot W_{n}} - \mean{f_{\ast}}{\mathbf{e}_{D} \cdot W_{n}}} = \bigl(\widetilde{\sigma}_{p}^{\,(n)}\bigr)^{1/p}$ and $\mean{f_{\ast}}{\mathbf{0} \cdot W_{n}} = -\mean{f_{\ast}}{\mathbf{e}_{D} \cdot W_{n}}$.
    Moreover, it is immediate that $\mean{f}{\mathbf{0} \cdot W_{n}} = -\mean{f_{\ast}}{\mathbf{0} \cdot W_{n}}$. 
    The uniqueness in Proposition \ref{prop.minimizer}-(3) implies that $\bigl(\sigma_{p}^{(n)}\bigr)^{-1/p}f = \pm\bigl(\sigma_{p}^{(n)}\bigr)^{-1/p}f_{\ast}$. 
   	If $f = f_{\ast}$, then we have $\mean{f_{\ast}}{\mathbf{0} \cdot W_{n}} = -\mean{f_{\ast}}{\mathbf{0} \cdot W_{n}}$, which leads to a contradiction since $\widetilde{\sigma}_{p}^{(n)} = \abs{\mean{f_{\ast}}{\mathbf{0} \cdot W_{n}} - \mean{f_{\ast}}{\mathbf{e}_{D} \cdot W_{n}}}^{p} = 0$. 
   	Hence the case $f = f_{\ast}$ does not happen and thus $f = -f_{\ast}$. 
   	This proves \eqref{Bp-sym}. 
   	
   	Next we will show that $f|_{\mathbf{0} \cdot W_{n}} \ge 0$ or $f_{\ast}|_{\mathbf{0} \cdot W_{n}} \ge 0$ holds. 
   	Suppose that $\mean{f}{\mathbf{0} \cdot W_{n}} \ge 0$. 
   	Obviously, a function $f_{\star}\colon \{ \mathbf{0}, \mathbf{e}_{D} \} \cdot W_{n} \to \bbR$ given by 
    \[
    f_{\star}(iw) = 
    \begin{cases}
    	\abs{f(\mathbf{0} \cdot w)} \quad &\text{if $i = \mathbf{0}$ and $w \in W_{n}$,} \\
    	-\abs{f(\mathbf{e}_{D} \cdot w)} \quad &\text{if $i = \mathbf{e}_{D}$ and $w \in W_{n}$,}
    \end{cases}
    \]
	satisfies 
	\[
	\abs{\mean{f_{\star}}{\mathbf{0} \cdot W_{n}} - \mean{f_{\star}}{\mathbf{e}_{D} \cdot W_{n}}} 
	\ge \abs{\mean{f}{\mathbf{0} \cdot W_{n}} - \mean{f}{\mathbf{e}_{D} \cdot W_{n}}}. 
	\]
   	Furthermore, by noting that $f = - f_{\ast}$ implies $\abs{f_{\star}(x) - f_{\star}(y)} = \abs{f(x) - f(y)}$ for $(x, y) \in \widetilde{E}_{n + 1}$ with $ x \in \mathbf{0} \cdot W_{n}$ and $y \in \mathbf{e}_{D} \cdot W_{n}$, we have 
   	\[
   	\mathcal{E}_{p, \{ \mathbf{0}, \mathbf{e}_{D} \} \cdot W_{n}}^{\widetilde{G}_{n + 1}}(f_{\star}) \le \mathcal{E}_{p, \{ \mathbf{0}, \mathbf{e}_{D} \} \cdot W_{n}}^{\widetilde{G}_{n + 1}}(f) = 1. 
   	\]
   	Since $f$ is an optimizer of $\widetilde{\sigma}_{p}^{\,\bigl(\widetilde{G}_{n + 1}, \mu\bigr)}(\mathbf{0}\cdot W_{n}, \mathbf{e}_{D} \cdot W_{n})$, we deduce from Proposition \ref{prop.minimizer}-(3) that $f = \pm f_{\star}$.
   	The condition $\mean{f}{\mathbf{0} \cdot W_{n}} \ge 0$ implies $f|_{\mathbf{0} \cdot W_{n}} \ge 0$. 
   	The proof in the case $-\mean{f_{\ast}}{\mathbf{0} \cdot W_{n}} = \mean{f}{\mathbf{0} \cdot W_{n}} \le 0$ is similar and we have $f_{\ast}|_{\mathbf{0} \cdot W_{n}} \ge 0$ in this case.
   	
   	By considering $f_{\ast}$ instead of $f$ if necessary, we can assume that $f$ satisfies $f|_{\mathbf{0} \cdot W_{n}} \ge 0$. 
   	Then we have $\mean{f}{\mathbf{0} \cdot W_{n}} = -\mean{f}{\mathbf{e}_{D} \cdot W_{n}} \ge 0$ and hence 
   	\begin{equation}\label{Bp-mean}
   		\mean{f}{\mathbf{0} \cdot W_{n}} = \frac{1}{2}\Bigl(\widetilde{\sigma}_{p}^{\,(n)}\Bigr)^{1/p}. 
   	\end{equation}
   	Furthermore, by the Markov property of $\mathcal{E}_{p}$ (Proposition \ref{prop.Markov}), 
   	\begin{equation}\label{Bp-Econtroll}
   		\mathcal{E}_{p, \{ \mathbf{0}, \mathbf{e}_{D} \} \cdot W_{n}}^{\widetilde{G}_{n + 1}}(f \vee 0) \le \mathcal{E}_{p, \{ \mathbf{0}, \mathbf{e}_{D} \} \cdot W_{n}}^{\widetilde{G}_{n + 1}}(f) = 1. 
   	\end{equation}
   	Note that $(f \vee 0)|_{\mathbf{0} \cdot W_{n}} \equiv f|_{\mathbf{0} \cdot W_{n}}$ and $(f \vee 0)|_{\mathbf{e}_{D} \cdot W_{n}} \equiv 0$. 
   	Roughly speaking, the estimate \eqref{Bp-Econtroll} tells us that values of a function $f_{0}\colon W_{n} \to [0, +\infty)$  given by 
    \[
    f_{0}(w) = f(\mathbf{0} \cdot w), \quad w \in W_{n}, 
    \]
    on $W_{n}\bigl[B_{D, +1}\bigr]$ are small so that its $p$-energy arising from ``boundaries'' can be controlled. 
    Clearly, $\mean{f_{0}}{W_{n}} = \mean{f}{\mathbf{0} \cdot W_{n}} = \frac{1}{2}\bigl(\widetilde{\sigma}_{p}^{\,(n)}\bigr)^{1/p}$. 
    By the uniqueness, we have that $f_{0} \circ \tau\bigl[R_{j}\bigr] = f_{0}$ for any $j = 1, \dots, D - 1$ and that $f_{0} \circ \tau\bigl[R_{j, k}^{\pm}\bigr] = f_{0}$ for any $j, k = 1, \dots, D - 1$. 
	Next, for $j = 1, \dots, D - 1$, we inductively define $f_{j}\colon W_{n} \to [0, +\infty)$ by 
    \[
    f_{j}(w) = 
    \begin{cases}
    	f_{j - 1}(w) \quad &\text{if $w \in W_{n}\bigl[\mathcal{H}_{j, D}^{+, \le}\bigr]$,} \\
    	f_{j - 1}\Bigl(\tau\bigl[R_{j, D}^{+}\bigr](w)\Bigr) \quad &\text{if $w \in W_{n}\bigl[\mathcal{H}_{j, D}^{+, \ge}\bigr]$.}
    \end{cases}
    \]
	This construction yields $f_{j} \circ \tau[R_{k}] = f_{j}$ if $0 \le j < k < D$.
	Moreover, $f_{j} \circ \tau\bigl[R_{k, l}^{+}\bigr] = f_{j}$ for $k, l = 1, \dots, j$ with $k < l$. 
    It is also immediate that $\mathcal{E}_{p}^{\widetilde{G}_{n}}(f_{0}) \le \mathcal{E}_{p, \{ \mathbf{0}, \mathbf{e}_{D} \} \cdot W_{n}}^{\widetilde{G}_{n + 1}}(f) = 1$ and $\mathcal{E}_{p}^{\widetilde{G}_{n}}(f_{j}) \le 2\mathcal{E}_{p}^{\widetilde{G}_{n}}(f_{0}) \le 2$ for each $j = 0, \dots, D - 1$ (we used both \eqref{Bp-mean} and \eqref{Bp-Econtroll}). 
    
    Finally, we construct a function $h\colon W_{n + 2} \to [0, +\infty)$, whose estimates of energies and averages will deduce \hyperlink{Bp}{\textup{(B$_{p}$)}}. 
    To this end, let us introduce `fundamental region' $\Delta^{a,D} = \Delta$ of $[-3a^{-1}, 3a^{-1}]^{D - 1} \times [1 - 2a^{-1}, 1] \eqqcolon \mathcal{R}_{a,D}$ that is given by  
    \[
    \Delta = \bigl\{ (x_1, \dots, x_{D}) \bigm| x_{k} \in [0, 3a^{-1}] \, (1 \le k \le D - 1), x_{D} \in [1 - 2a^{-1}, 1] \bigr\} \cap \bigcap_{i = 1}^{D - 2}\mathcal{H}_{i, i + 1}^{+, \ge}. 
    \]
    Here, we regard $\bigcap_{i = 1}^{D - 2}\mathcal{H}_{i, i + 1}^{+, \ge}$ as $Q_{0}$ when $D = 2$. 
    For each $k \in \{ 0, \dots, D - 1 \}$, define $u_{k} \in \mathbb{R}^{D}$ by $u_{k} \coloneqq \sum_{j = 1}^{k}\mathbf{e}_{j}$, where we set $u_{0} \coloneqq \mathbf{0}^{D}$. 
    Then we have 
    \[
    \Delta \subseteq \bigcup_{k = 0}^{D - 1}f_{u_{k}}(Q_{0}). 
    \]
    Next we inductively define $\{ \Delta_{k} \}_{k = 0}^{D - 2}$ as follows. 
    Define $\Delta_{0} \coloneqq \Delta$. 
    For given $\Delta_{0}, \dots, \Delta_{k - 1}$, define $\Delta_{k}$ as
    \[
    \Delta_{k} \coloneqq \Delta_{k - 1} \cup R_{k, k + 1}^{+}\bigl(\Delta_{k - 1}\bigr). 
    \]
    Note that this construction yields $\Delta_{D - 2} = \bigl[0, 3a^{-1}\bigr]^{D - 1} \times [1 - 2a^{-1}, 1]$ for all $D \ge 2$. 
    Also, we define subsets $\{ \square_{k} \}_{k = 0}^{D - 1}$ of $\mathcal{R}_{a,D}$ as follows. 
    Define $\square_{0} \coloneqq \Delta_{D - 2}$. 
    For given $\square_{0}, \dots, \square_{k - 1}$, define $\square_{k}$ as 
    \[
    \square_{k} \coloneqq \square_{k - 1} \cup R_{k}\bigl(\square_{k - 1}\bigr). 
    \]
    This construction yields $\square_{D - 1} = \mathcal{R}_{a, D}$. 
    
    By $S \subsetneq \{ 0, 1, \dots, a - 1 \}^{D}$, $a \ge 3$, \hyperlink{GSC1}{\textup{(GSC1)}} and \hyperlink{GSC4}{\textup{(GSC4)}}, we can find $v = v_{1}v_{2} \in W_{2}$ such that $v_{1} - \mathbf{e}_{D} \not\in S$, $v \in W_{2}\bigl[ B_{D, -1} \bigr]$ and $\mathcal{N}(v) \subseteq v_{1} \cdot W_{1}$, where $\mathcal{N}(v) \coloneqq \{ z \in W_{2} \mid d_{G_{2}}(z, v) \le 1 \}$.
    Set
    \[
    A_{v} \coloneqq \{ v_{1} \cdot (v_{2} + u_{k}) \mid \text{$k \in \{ 1, \dots, D - 1 \}$ with $v_{2} + u_{k} \in S$} \}
    \]
    and define $\widetilde{h} \colon A_{v} \cdot W_{n} \to [0, +\infty)$ by 
    \[
    \widetilde{h}\bigl(v_{1} \cdot (v_{2} + u_{k}) \cdot z\bigr) = f_{k}(z) \quad \text{for $z \in W_{n}$}. 
    \]
    Note that $A_{v} \cdot W_{n} \supseteq W_{n + 2}\bigl[f_{v_{1}}\bigl(\Delta_{0}\bigr)\bigr]$. 
    The desired function $h\colon W_{n + 2} \to [0, +\infty)$ will be constructed by `unfolding' $\widetilde{h}$ in a suitable way as described below. 
    Define $\widetilde{h}_{0} \coloneqq \widetilde{h}$. 
    Inductively, for $z \in W_{n + 1}$ with $v_{1}z \in W_{n + 2}\bigl[f_{v_{1}}\bigl(\Delta_{k}\bigr)\bigr]$, define $\widetilde{h}_{k} \colon W_{n + 2}\bigl[f_{v_{1}}\bigl(\Delta_{k}\bigr)\bigr] \to [0, +\infty)$ by 
    \[
    \widetilde{h}_{k}(v_{1}z) = 
    \begin{cases}
    	\widetilde{h}_{k - 1}(v_{1}z) \quad &\text{if $v_{1}z \in W_{n + 2}\bigl[f_{v_{1}}\bigl(\Delta_{k - 1}\bigr)\bigr]$,} \\
    	\widetilde{h}_{k - 1}\Bigl(v_{1}\cdot\tau\bigl[R_{k, k + 1}^{+}\bigr](z)\Bigr) &\text{otherwise.}
    \end{cases}
    \]
    Also, define $h_{0} \coloneqq \widetilde{h}_{D - 2} \colon W_{n + 2}\bigl[f_{v_{1}}\bigl(\square_{0}\bigr)\bigr] \to [0, +\infty)$. 
    Inductively, for $z \in W_{n + 1}$ with $v_{1}z \in W_{n + 2}\bigl[f_{v_{1}}\bigl(\square_{k}\bigr)\bigr]$, define $h_{k} \colon W_{n + 2}\bigl[f_{v_{1}}\bigl(\square_{k}\bigr)\bigr] \to [0, +\infty)$ by 
    \[
    h_{k}(v_{1}z) = 
    \begin{cases}
    	h_{k - 1}(v_{1}z) \quad &\text{if $v_{1}z \in W_{n + 2}\bigl[f_{v_{1}}\bigl(\square_{k - 1}\bigr)\bigr]$,} \\
    	h_{k - 1}\Bigl(v_{1}\cdot\tau\bigl[R_{k}\bigr](z)\Bigr) &\text{otherwise.}
    \end{cases}
    \]
    Lastly, we define $h \colon W_{n + 2} \to [0, +\infty)$ by 
    \[
    h(z) = 
    \begin{cases}
    	h_{D - 1}(z) \quad &\text{if $z \in W_{n + 2}\bigl[f_{v_{1}}\bigl(\mathcal{R}_{a, D}\bigr)\bigr]$,} \\
    	0 &\text{otherwise.}
    \end{cases}
    \]
 	
 	Clearly, we have $h|_{W_{n + 2} \setminus \mathcal{N}(v) \cdot W_{n}} \equiv 0$ and $h|_{\partial_{\ast}G_{n + 2}} \equiv 0$ from the construction. 
 	Furthermore, we have from the symmetries of $\{ f_{j} \}_{j = 0}^{D - 1}$ and the definition of $h$ that 
 	\begin{align*}
    	\mathcal{E}_{p}^{\widetilde{G}_{n + 2}}(h) 
    	&\le \sum_{j = 0}^{D - 1}2^{j}\binom{D - 1}{j}\Bigl\{\mathcal{E}_{p}^{\widetilde{G}_{n}}(f_{j}) +(j + 1)\mathcal{E}_{p, \{ \mathbf{0}, \mathbf{e}_{D} \} \cdot W_{n}}^{\widetilde{G}_{n + 1}}(f \vee 0)\Bigr\} \\
    	&\le \sum_{j = 0}^{D - 1}2^{j}(j + 3)\binom{D - 1}{j} \\
    	&= (D + 2)\sum_{j = 0}^{D - 1}2^{j}\binom{D - 1}{j} 
    	= 3^{D - 1}(D + 2), 
    \end{align*}
   	and that 
   	\begin{align*}
   		\mean{h}{W_{n + 2}}
   		\ge N_{\ast}^{-(n + 2)}\sum_{w \in W_{n}}f_{0}(w) 
   		= N_{\ast}^{-2}\mean{f_{0}}{W_{n}}  
   		= \frac{1}{2}N_{\ast}^{-2}\Bigl(\widetilde{\sigma}_{p}^{\,(n)}\Bigr)^{1/p}. 
   	\end{align*}
   	Hence, by putting $h_{\ast} \coloneqq \mcE_{p}^{\widetilde{G}_{n + 2}}(h)^{-1/p} \cdot h$, we conclude that
    \begin{align*}
    	\widetilde{\lambda}_{p, \textup{Dir}}^{\,(n + 2)} 
    	\ge \abs{\mean{h_{\ast}}{W_{n + 2}}}^{p}
    	= \mathcal{E}_{p}^{\widetilde{G}_{n + 2}}(h)^{-1}\abs{\mean{h}{W_{n + 2}}}^{p} 
    	\ge C_{\ast}^{-1}\widetilde{\sigma}_{p}^{\,(n)}
    \end{align*}
    which shows (B$_{p}$), where $k_{\ast} = 2$ and $C_{\ast} = 3^{D - 1}(D + 2) \cdot (2N_{\ast}^{2})^{p}C_{\ref{lem.Poi-URI}}^{2}$.
\end{proof}

\subsection{Uniform H\"{o}lder estimate: $p > \dim_{\textup{ARC}}(K, d)$}\label{subsec:Uhol}
Next we will prove useful H\"{o}lder type estimates. 
In order to obtain these estimates, a ``low-dimensional'' condition: Assumption \ref{assum.lowdim}, which is described in terms of the \emph{Ahlfors regular conformal dimension}, will be essential.
The notion of Ahlfors regular conformal dimension was implicitly introduced by Bourdon and Pajot \cite{BP03}.
The exact definition of this dimension is as follows.

\begin{dfn}\label{dfn.QS}
    Let $X$ be a metrizable space (without isolated points) and let $d_{i}\,(i = 1, 2)$ be compatible metrics on $X$.
    We say that $d_{1}$ and $d_{2}$ are \emph{quasisymmetric} to each other if there exists a homeomorphism $\eta\colon [0, \infty) \to [0, \infty)$ such that
    \[
    \frac{d_{2}(x, a)}{d_{2}(x, b)} \le \eta\left(\frac{d_{1}(x, a)}{d_{1}(x, b)}\right),
    \]
    for every triple $x, a, b \in X$ with $x \neq b$.
    (It is easy to show that being quasisymmetric gives an equivalence relation among metrics.)
    The \emph{Ahlfors regular conformal gauge} $\mcJ_{\text{AR}}(X, d_{1})$ of $(X, d_{1})$ is defined as
    \[
    \mcJ_{\text{AR}}(X, d_{1}) \coloneqq \left\{ d_{2} \;\middle|\;
    \begin{array}{c}
    \text{$d_{2}$ is a metric on $X$, $d_{2}$ is quasisymmetric to $d_{1}$,} \\
    \text{and $d_{2}$ is $\alpha'$-Ahlfors regular for some $\alpha' > 0$.}
    \end{array}
    \right\}.
    \]
    (For the definition of Ahlfors regularity, recall \eqref{eq.AR}.
    Note that $\dim_{\text{H}}(X, d_{2}) = \alpha'$ if $d_{2}$ is $\alpha'$-Ahlfors regular.)
    Then the \emph{Ahlfors regular conformal dimension} (ARC-dimension for short) of $(X, d_{1})$ is
    \begin{equation}\label{dfn.ARCdim}
    \dim_{\text{ARC}}(X, d_{1}) \coloneqq \inf_{d_{2} \in \mcJ_{\text{AR}}(X, d_{1})}\dim_{\text{H}}(X, d_{2}).
    \end{equation}
\end{dfn}

The notion of quasisymmetric and the exact definition of ARC-dimension are not essential in this paper.
We refer the reader to a monograph \cite{MT} and surveys \cites{Bon06, Kle06} for details of the ARC-dimension and related subjects.

The following assumption describes our ``low-dimensional'' setting (see \cite{KZ92}*{the condition (R)} in the context of probability theory).

\begin{ass}\label{assum.lowdim}
    A positive real number $p$ satisfies $p > \dim_{\textup{ARC}}(K, d)$.
\end{ass}

\begin{rmk}
    The following bound concerning the Ahlfors regular conformal dimension of the standard Sierpi\'{n}ski carpet is known:
    \[
    1 + \frac{\log{2}}{\log{3}} \le \dim_{\textup{ARC}}(K, d) < \alpha = \frac{\log{N_{\ast}}}{\log{a}} = \frac{\log{8}}{\log{3}}.
    \]
    The lower bound follows from a general result due to Tyson \cite{Tys00}.
    The strict inequality in the upper bound is proved by Keith and Laakso \cite{KL04}.
\end{rmk}

To promote understanding Assumption \ref{assum.lowdim}, we recall characterization results by Carrasco Piaggio \cite{CP13} and Kigami \cite{Kig20} in our setting.
Recall $\lim_{n \to \infty}\bigl(\mcR_{p}^{(n)}\bigr)^{1/n} = \rho_{p}$ (Theorem \ref{thm.sub-mult}).

\begin{thm}[\cite{CP13}*{Theorem 1.3}, \cite{Kig20}*{Theorems 4.6.9 and 4.7.6}]\label{thm.ARCdim}
	Assumption \ref{assum.lowdim} is equivalent to $\rho_{p} > 1$. 
	In particular, if Assumption \ref{assum.lowdim} holds, then there exist $C_{\ref{thm.ARCdim}} > 0$ and $\theta > 0$ (depending only on $a, \rho_{p}$) such that 
    \begin{equation*}
    	\mathcal{C}_{p}^{(n)} \le C_{\ref{thm.ARCdim}}a^{-\theta n} \quad \text{for all $n \in \mathbb{N}$.}
    \end{equation*}
\end{thm}

Under Assumption \ref{assum.lowdim}, we can show the following powerful H\"{o}lder continuity.
\begin{thm}\label{thm.unifhol}
    Suppose Assumption \ref{assum.lowdim} holds.
    Then there exist constants $\widetilde{C}_{\textup{UH}} > 0$ (depending only on $p, a, D, N_{\ast}, L_{\ast}, \rho_{p}$) and $\theta > 0$ (depending only on $a, \rho_{p}$) such that, for any $n, m \in \bbN$, $z \in W_{m}$,  $v, w \in \mcB_{n}(z, 1)$ and $f \in \bbR^{W_{n + m}}$,
    \begin{align}\label{unif.holder}
        \abs{f(v) - f(w)}^{p} \le \widetilde{C}_{\textup{UH}}\lambda_{p}^{(n + m)}\mcE_{p}^{G_{n + m}}(f)a^{-m\theta}.
    \end{align}
\end{thm}
This theorem is proved by iterating Proposition \ref{prop.basic}-(2). 
Kusuoka and Zhou \cite{KZ92} prepared a general estimate using signed measures (\cite{KZ92}*{Lemma 3.9}) to show H\"{o}lder type estimates, but we need only the case of Dirac measures for our purpose.
Here we give a simplified extension of \cite{KZ92}*{Lemma 3.9 and Proposition 3.10}.
\begin{lem}\label{lem.variance}
    Let $n, m \in \bbN$, let $v \in W_{m}$ and let $f \in \bbR^{W_{n + m}}$.
    Then, for any $w \in v \cdot W_{n}$,
    \begin{equation}\label{eq.variance}
        \abs{f(w) - \mean{f}{v\cdot W_{n}}} \le N_{\ast}^{1/p}\mcE_{p, v \cdot W_{n}}^{G_{n + m}}(f)^{1/p}\sum_{k = 1}^{n}\left(\lambda_{p}^{(k)}\right)^{1/p}.
    \end{equation}
\end{lem}
\begin{proof}
    Let $w \in v \cdot W_{n}$ and set $w_{l} \coloneqq [w]_{l}$ for each $l = m, \dots, n + m$.
    Note that $w_{m} = v$ and $w_{n + m} = w$.
    From Proposition \ref{prop.basic}-(2), we see that
    \begin{align*}
        \abs{f(w) - \mean{f}{v \cdot W_{n}}}
        &\le \sum_{l = m}^{n + m - 1}\abs{\mean{f}{w_{l} \cdot W_{n + m - l}} - \mean{f}{w_{l + 1} \cdot W_{n + m - l - 1}}} \\
        &\le N_{\ast}^{1/p}\mcE_{p, v \cdot W_{n}}(f)^{1/p}\sum_{l = m}^{n + m - 1}\bigl(\lambda_{p}^{(n + m - l)}\bigr)^{1/p} \\
        &= N_{\ast}^{1/p}\mcE_{p, v \cdot W_{n}}(f)^{1/p}\sum_{k = 1}^{n}\bigl(\lambda_{p}^{(k)}\bigr)^{1/p}. \qedhere
    \end{align*}
\end{proof}

\noindent
\textit{Proof of Theorem \ref{thm.unifhol}.}\,
    By Assumption \ref{assum.lowdim} and Theorem \ref{thm.ARCdim}, there exists $\theta > 0$ such that 
    \begin{align}\label{est.super}
        \mcC_{p}^{(n)} \le C_{\ref{thm.ARCdim}}a^{-n\theta} \quad \text{for every $n \in \bbN$.}
    \end{align}
    From \eqref{est.super}, Proposition \ref{prop.poi1} and Theorem \ref{thm.poi4}-(2), we have $\lambda_{p}^{(n)} \le c_1\lambda_{p}^{(n + m)}a^{-m\theta}$ for every $n, m \in \bbN$, where $c_1 > 0$ depends only on $p, a, D, L_{\ast}, N_{\ast}, \rho_{p}$.
    In particular,
    \begin{align}\label{ineq.hol1}
        \lambda_{p}^{(k)} \le c_1\lambda_{p}^{(n + m)}a^{-(n + m - k)\theta} \quad \text{for every $n, m, k \in \bbN$ with $k \le n$.}
    \end{align}
    By Lemma \ref{lem.variance}, for any $z \in W_{m}$, $v \in z \cdot W_{n}$ and $w \in \mcB_{n}(z, 1)$,
    \begin{align*}
        &\abs{f(v) - f(w)} \\
        &\le \abs{f(v) - \mean{f}{z \cdot W_{n}}} + \abs{\mean{f}{z \cdot W_{n}} - \mean{f}{[w]_{m} \cdot W_{n}}} + \abs{f(w) - \mean{f}{[w]_{m} \cdot W_{n}}} \\
        &\le \sum_{i = 1}^{D}\abs{\mean{f}{z(i - 1) \cdot W_{n}} - \mean{f}{z(i) \cdot W_{n}}} + 2N_{\ast}^{1/p}\mcE_{p}^{G_{n + m}}(f)^{1/p}\sum_{k = 1}^{n}\Bigl(\lambda_{p}^{(k)}\Bigr)^{1/p},
    \end{align*}
    where $z(i) \in W_{m} \, (i = 0, \dots, D)$ with $z(0) = z$, $z(D) = [w]_{m}$ satisfy $(z(i - 1), z(i)) \in \widetilde{E}_{m}$ or $z^{i - 1} = z^{i}$ for $i = 1, \dots, D$. 
    (Such $\{ z(i) \}_{i}$ exists due to Proposition \ref{prop.h-network}.)
    From Proposition \ref{prop.basic}-(3), Theorem \ref{thm.poi4}-(1) and \eqref{ineq.hol1}, we see that
    \begin{align*}
        \abs{f(v) - f(w)}
        &\le 2D^{(p - 1)/p}\mcE_{p}^{G_{n + m}}(f)^{1/p}\left(\Bigl(\sigma_{p}^{(n)}\Bigr)^{1/p} + N_{\ast}^{1/p}\sum_{k = 1}^{n}\Bigl(\lambda_{p}^{(k)}\Bigr)^{1/p}\right) \\
        &\le c_{2}\mcE_{p}^{G_{n + m}}(f)^{1/p}\sum_{k = 1}^{n}\Bigl(\lambda_{p}^{(k)}\Bigr)^{1/p} \\
        &\le c_{1}^{1/p}c_{2}\Bigl(\lambda_{p}^{(n + m)}\mcE_{p}^{G_{n + m}}(f)\Bigr)^{1/p}\sum_{k = 1}^{n}a^{-(n + m - k)\theta/p},
    \end{align*}
    where $c_{2}$ is a positive constant depending only on $p, a, D, N_{\ast}, L_{\ast}$.
    Since $\theta > 0$ and
    \begin{align*}
        \sum_{k = 1}^{n}a^{-(n + m - k)\theta/p}
        &= a^{-m\theta/p}\sum_{k = 0}^{n - 1}a^{-k\theta/p} \\
        &= a^{-m\theta/p} \cdot \frac{1 - a^{-n\theta/p}}{1 - a^{-\theta/p}} \le \frac{1}{1 - a^{-\theta/p}} \cdot a^{-m\theta/p},
    \end{align*}
    we have the desired estimate for $v \in z \cdot W_{n}$ and $w \in \mcB_{n}(z, 1)$.
    A combination of this estimate and the triangle inequality proves the desired estimate.\hspace{\fill}$\square$

\subsection{Proof of (KM$_{p}$): $p > \dim_{\textup{ARC}}(K, d)$}\label{subsec:KM}
The aim of this subsection is to prove \hyperlink{KMp}{\textup{(KM$_{p}$)}} under Assumption \ref{assum.lowdim}.
Our strategy for proving \hyperlink{KMp}{\textup{(KM$_{p}$)}} comes from a recent study by Cao and Qiu \cite{CQ21+}, where they give an ``analytic'' proof of (KM$_{2}$) using estimates of Poincar\'{e} constants in \cite{KZ92}.
Although our proof of \hyperlink{KMp}{\textup{(KM$_{p}$)}} is similar to the argument in \cite{CQ21+}*{Section 4}, we give a complete proof of \hyperlink{KMp}{\textup{(KM$_{p}$)}} for the reader's convenience.
Our argument will depend heavily on the uniform H\"{o}lder estimate (Theorem \ref{thm.unifhol}) and on behaviors of ``chain'' type $p$-conductance $\mcC_{p}^{(n, M)}$ (its definition will be given later).

Let us start by introducing a new graph $G_{n, M}$, which is a ``horizontal chain'' consisting of $M$ copies of $G_{n}$.
The exact definition of $G_{n, M}$ is as follows.
Let $n, M \in \bbN$ with $M \ge 2$ and pick $m \in \bbN$ such that $a^{m} \ge M$.
Then, by \hyperlink{GSC4}{\textup{(GSC4)}}, we can find a simple path $\bigl[w(1), \dots, w(M)\bigr]$ in $\bigl(V_{m}, \widetilde{E}_{m}\bigr)$ such that
\[
F_{w(i)}\bigl(\mathcal{K}^{+}\bigr) = F_{w(i + 1)}\bigl(\mathcal{K}^{-}\bigr) \quad \text{ for each $i = 1, \dots, M - 1$,}
\]
where
\begin{equation}\label{dfn.LR}
    \mathcal{K}^{-} \coloneqq K \cap B_{1, -1}, \quad \mathcal{K}^{+} \coloneqq K \cap B_{1, +1}.
\end{equation}
($\mathcal{K}^{\,\pm}$ denotes a couple of opposite faces of $K$.)
We define $G_{n, M} = (V_{n, M}, E_{n, M})$ as a subgraph of $G_{n + m}$ given by
\[
V_{n, M} \coloneqq \bigcup_{i = 1}^{M} w(i) \cdot W_{n} \quad \text{and} \quad E_{n, M} \coloneqq \{ (x, y) \in E_{n + m} \mid x, y \in V_{n, M} \}.
\]
We also consider a horizontal network $\widetilde{E}_{n, M}$ defined by 
\[
\widetilde{E}_{n, M} \coloneqq \Bigl\{ (x, y) \in \widetilde{E}_{n + m} \Bigm| x, y \in V_{n, M} \Bigr\}.
\]
Now, we set
\[
V_{n, M}^{-} \coloneqq w(1) \cdot W_{n} \quad \text{and} \quad V_{n, M}^{+} \coloneqq w(M) \cdot W_{n}.
\]
and define $\mcC_{p}^{(n, M)}$ (see Figure \ref{fig:chain_nM}) by setting
\[
\mcC_{p}^{(n, M)} \coloneqq \mcC_{p}^{G_{n, M}}\bigl(V_{n, M}^{-}, V_{n, M}^{+}\bigr).
\]
We easily see that these definitions do not depend on choices of large $m \in \bbN$ and horizontal chain $\bigl[w(1), \dots, w(M)\bigr]$.

The following lemma describes a key behavior of $\mcC_{p}^{(n, M)}$.

\begin{figure}[t]
    \centering
    \includegraphics[height=90pt]{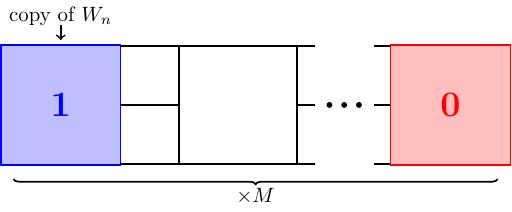}
    \caption{The conductance $\mcC_{p}^{(n, M)}$ (the planar case)}
    \label{fig:chain_nM}
\end{figure}

\begin{lem}\label{lem.chain}
    For every $M \ge 3$ there exists a constant $C(M) \ge 1$ depending only on $p, L_{\ast}, N_{\ast}, M$ such that
    \begin{align}
        C(M)^{-1}\mcC_{p}^{(n)} \le \mcC_{p}^{(n, M)} \le C(M)\mcC_{p}^{(n)} \quad \text{for any $n \in \bbN$.}
    \end{align}
\end{lem}

Its proof will be a straightforward modification of \cite{CQ21+}*{Lemma 4.7}. 
In order to prove Lemma \ref{lem.chain}, we need to show that $\mcC_{p}^{(n)}$ behaves similarly to the conductance that appeared in the work of Barlow and Bass (see the quantity $R_{n}^{-1}$ in \cite{BB90}).
To state rigorously, we define
\[
W_{n}^{-} \coloneqq W_{n}\bigl[B_{1, -1}\bigr], \quad W_{n}^{+} \coloneqq W_{n}\bigl[B_{1, +1}\bigr] \quad \text{and} \quad \mcC_{p, \textup{face}}^{(n)} \coloneqq \mcC_{p}^{G_{n}}\bigl(W_{n}^{-}, W_{n}^{+}\bigr).
\]
The next lemma is proved in \cite{BK13}*{Lemma 4.4} for some special cases of Sierpi\'{n}ski carpets by using $p$-combinatorial modulus instead of $p$-conductance.
We give a simple proof without using $p$-modulus.

\begin{lem}\label{lem.top-bottom}
    There exists a constant $C_{\ref{lem.top-bottom}} \ge 1$ depending only on $p, D, L_{\ast}$ such that
    \[
    C_{\ref{lem.top-bottom}}^{-1}\mcC_{p}^{(n)} \le \mcC_{p, \textup{face}}^{(n)} \le C_{\ref{lem.top-bottom}}\mcC_{p}^{(n)} \quad \text{for any $n \in \bbN$.}
    \]
\end{lem}
\begin{proof}
    By the self-similarity of $K$ and \hyperlink{GSC1}{\textup{(GSC1)}}, there exist $m \in \bbN$ and $w \in W_{m}$ such that $\mcC_{p}^{(n)} = \mcC_{p}^{G_{n + m}}\bigl(w \cdot W_{n}, W_{n + m} \setminus \mcB_{n}(w, 1)\bigr)$.
    Define $A_n \coloneqq \overline{w \cdot W_n}$ and $B_n \coloneqq \overline{W_{n + m} \setminus \mcB_{n}(w, 1)}$.
    It is immediate from Proposition \ref{prop.mono} that $\mcC_{p}^{G_{n + m}}(A_n, B_n) \ge \mcC_{p}^{(n)}$. 
    For $n \in \mathbb{N}$ and $v \in \bord{(w \cdot W_{n})}$ (resp. $v \in \bord{\compl{\mathcal{B}_{n}(w, 1)}}$), we fix $\mathsf{d}(v) \in w \cdot W_{n}$ (resp. $\mathsf{d}(v) \in \compl{\mathcal{B}_{n}(w, 1)}$) such that $(v, \mathsf{d}(v)) \in E_{n + m}$. 
    Define $\varphi_{n}\colon W_{n + m} \to W_{n + m}$ by 
    \[
    \varphi_{n}(v) = 
    \begin{cases}
    	v \quad &\text{if $v \not\in \bord{(w \cdot W_{n})} \cup \bord{\compl{\mathcal{B}_{n}(w, 1)}}$,} \\
    	\mathsf{d}(v) \quad &\text{if $v \in \bord{(w \cdot W_{n})} \cup \bord{\compl{\mathcal{B}_{n}(w, 1)}}$.} 
    \end{cases}
    \]
    Then we easily see that $\varphi_{n}$ is a rough isometry with $C_{1} = 1$, $C_{2} = 2$, $C_{3} = 2$ and $C_{4} = L_{\ast}/2$ in Definition \ref{dfn.URI}.     
    Note that $\mcC_{p}^{G_{n + m}}\bigl(\bord{(w \cdot W_{n})}, \bord{\compl{\mathcal{B}_{n}(w, 1)}}\bigr) = \mcC_{p}^{G_{n + m}}(A_n, B_n)$.
    Applying \eqref{eq.URIcap} in Lemma \ref{lem.URI} and Proposition \ref{prop.mono}, we deduce that there exists $c_{1} > 0$ (depending only on $p, L_{\ast}$) such that, for any $n \in \bbN$,
    \[
    \mcC_{p}^{G_{n + m}}(A_n, B_n) \le c_{1}\mcC_{p}^{(n)}.
    \]
    Let $f\colon W_{n + m} \to \bbR$ satisfy $f|_{A_n} \equiv 1$, $f|_{B_n} \equiv 0$ and $\mcE_{p}^{G_{n + m}}(f) = \mcC_{p}^{G_{n + m}}(A_n, B_n)$.
    If $v \in W_{m}$ satisfies $(v,w) \in \widetilde{E}_{m}$, then we have that
    \[
    \mcC_{p}^{G_{n + m}}(A_n, B_n) \ge \mcE_{p, v \cdot W_{n}}^{G_{n + m}}(f) \ge \mcC_{p, \textup{face}}^{(n)},
    \]
    and thus we conclude that $\mcC_{p, \textup{face}}^{(n)} \le c_{1}\mcC_{p}^{(n)}$.
    
	The converse can be shown in a very similar way as Proposition \ref{prop.B}. 
	Define `fundamental region' $\widetilde{\Delta}^{a,D} = \widetilde{\Delta}$ of $[-3a^{-1}, 3a^{-1}]^{D} \eqqcolon \mathcal{Q}_{a,D}$ by  
    \[
    \widetilde{\Delta} = \bigl[0, 3a^{-1}\bigr]^{D} \cap \bigcap_{i = 1}^{D - 1}\mathcal{H}_{i, i + 1}^{+, \ge}. 
    \]
    For each $k \in \{ 0, \dots, D \}$, define $u_{k} \in \mathbb{R}^{D}$ by $u_{k} \coloneqq \sum_{j = 1}^{k}\mathbf{e}_{j}$, where we set $u_{0} \coloneqq \mathbf{0}^{D}$. 
    Then we have 
    \[
    \widetilde{\Delta} \subseteq \bigcup_{k = 0}^{D}f_{u_{k}}(Q_{0}). 
    \]
    Next we inductively define $\{ \widetilde{\Delta}_{k} \}_{k = 0}^{D - 1}$ as follows. 
    Define $\widetilde{\Delta}_{0} \coloneqq \widetilde{\Delta}$. 
    For given $\widetilde{\Delta}_{0}, \dots, \widetilde{\Delta}_{k - 1}$, define $\widetilde{\Delta}_{k}$ as
    \[
    \widetilde{\Delta}_{k} \coloneqq \widetilde{\Delta}_{k - 1} \cup R_{k, k + 1}^{+}\bigl(\widetilde{\Delta}_{k - 1}\bigr). 
    \]
    Note that this construction yields $\widetilde{\Delta}_{D - 1} = \bigl[0, 3a^{-1}\bigr]^{D}$ for all $D \ge 2$. 
    Also, we define subsets $\bigl\{ \widetilde{\square}_{k} \bigr\}_{k = 0}^{D}$ of $\mathcal{Q}_{a,D}$ as follows. 
    Define $\widetilde{\square}_{0} \coloneqq \widetilde{\Delta}_{D - 1}$. 
    For given $\widetilde{\square}_{0}, \dots, \widetilde{\square}_{k - 1}$, define $\widetilde{\square}_{k}$ as 
    \[
    \widetilde{\square}_{k} \coloneqq \widetilde{\square}_{k - 1} \cup R_{k}\bigl(\widetilde{\square}_{k - 1}\bigr). 
    \]
    This construction yields $\widetilde{\square}_{D} = \mathcal{Q}_{a, D}$. 
    
    Next we will introduce a new graph $\Gamma_{n} \, (n \ge 1)$ as follows. 
    For each $n \ge 1$, define $\widetilde{W}_{n} \subseteq \bigl(\{ 0, \dots, a - 1 \}^{D}\bigr)^{n}$ by 
    \[
    \widetilde{W}_{n} \coloneqq \Bigl\{ (\sigma_{k})_{k = 1}^{D} \cdot w \Bigm| \text{$\sigma_{k} \in \{ -1, 0, 1 \}$ for each $k \in \{ 1, \dots, D \}$}, w \in W_{n - 1} \Bigr\}, 
    \] 
    and $E\bigl(\widetilde{W}_{n}\bigr) \subseteq \widetilde{W}_{n} \times \widetilde{W}_{n}$ by 
    \begin{align*}
    	&E\bigl(\widetilde{W}_{n}\bigr) \coloneqq \\
    	&\hspace*{20pt}\bigl\{ (v, w) \bigm| v, w \in \widetilde{W}_{n},  \text{$f_{v}(Q_{0}) \cap f_{w}(Q_{0})$ is a $(D - 1)$-dimensional hypercube} \bigr\}. 
    \end{align*}
    Then $\Gamma_{n} \coloneqq \Bigl(\widetilde{W}_{n}, E\bigl(\widetilde{W}_{n}\bigr)\Bigr)$. 
    Define subsets $I_{n}, O_{n}$ of $\widetilde{W}_{n}$ by 
    \[
    I_{n} \coloneqq \bigl\{ w \in \widetilde{W}_{n} \bigm| f_{w}(Q_{0}) \cap f_{\mathbf{0}}(Q_{0}) \neq \emptyset \bigr\} = \bigl\{ w \in \widetilde{W}_{n} \bigm| f_{w}(Q_{0}) \cap \bigl[-a^{-1}, a^{-1}\bigr]^{D} \neq \emptyset \bigr\}, 
    \]
    and 
    \[
    O_{n} \coloneqq \bigl\{ w \in \widetilde{W}_{n} \bigm| f_{w}(Q_{0}) \cap \bord{\mathcal{Q}_{a, D}} \neq \emptyset \bigr\}. 
    \]
    By the cutting law of $p$-conductances (see \cite{Mthesis}*{Proposition 3.18} for example), 
    \[
    \mcC_{p}^{\Gamma_{n + 1}}\bigl(I_{n + 1}, O_{n + 1}\bigr) \ge \mcC_{p}^{\widetilde{G}_{n + m}}(A_n, B_n) \ge c\cdot\mcC_{p}^{G_{n + m}}(A_n, B_n) \ge c\cdot\mcC_{p}^{(n)}, 
    \]
    where $c > 0$ depends only on $p, D, L_{\ast}$ (we used Lemma \ref{lem.URI}).  
    Let $f \colon W_{n} \to [0, 1]$ satisfy
    \[
    f|_{W_{n}^{-}} \equiv 1, \quad f|_{W_{n}^{+}} \equiv 0 \quad\text{and}\quad \mcE_{p}^{G_{n}}(f) = \mcC_{p, \textup{face}}^{(n)}.
    \]
    From the uniqueness of the optimizer of $\mathcal{C}_{p, \textup{face}}^{(n)}$, we have $f \circ \tau[R_{j}] = f$ and $f \circ \tau\bigl[R_{j, k}^{\pm}\bigr] = f$ for any $j, k = 2, \dots, D$. 
    Define $h_{1} \colon W_{n} \to [0, 1]$ as $h_{1} \coloneqq f$. 
    Inductively, define $h_{k} \colon W_{n} \to [0, 1]$ by 
    \[
    h_{k}(w) = 
    \begin{cases}
    	h_{k - 1}(w) \quad &\text{if $w \in W_{n}\bigl[\mathcal{H}_{1, k}^{+, \ge}\bigr]$,} \\
    	h_{k - 1}\bigl(\tau\bigl[R_{1, j}^{+}\bigr](w)\bigr) \quad &\text{if $w \in W_{n}\bigl[\mathcal{H}_{1, k}^{+, \le}\bigr]$.}
    \end{cases}
    \]
    Then we have $\mathcal{E}_{p}^{G_{n}}\bigl(h_{k}\bigr) \le 2\mathcal{E}_{p}^{G_{n}}\bigl(h_{k - 1}\bigr)$ and hence 
    \[
    \max_{k \in \{1, \dots, D\}}\mathcal{E}_{p}^{G_{n}}\bigl(h_{k}\bigr) \le 2^{D - 1}\mathcal{E}_{p}^{G_{n}}\bigl(h_{0}\bigr) = 2^{D - 1}\mcC_{p, \textup{face}}^{(n)}. 
    \]
    Let $\widetilde{h} \colon \{ u_{k} \cdot w \mid k \in \{ 1, \dots, D \}, w \in W_{n} \} \to [0, 1]$ such that 
    \[
    \widetilde{h}\bigl(u_{k} \cdot w\bigr) = h_{k}(w) \quad \text{for all $k \in \{ 1, \dots, D \}$ and $w \in W_{n}$. }
    \]
    Note that $\{ u_{k} \cdot w \mid k \in \{ 1, \dots, D \}, w \in W_{n} \}$ contains $\bigl\{ w \in \widetilde{W}_{n + 1} \bigm| f_{w}(Q_{0}) \cap \widetilde{\Delta} \neq \emptyset \bigr\}$ and that 
    \[
    \mathcal{E}_{p}\Bigl(\widetilde{h}\Bigr) = \sum_{k = 1}^{D}\mathcal{E}_{p}^{\widetilde{G}_{n}}\bigl(h_{k}\bigr), 
    \]
    where $\mathcal{E}_{p}\bigl(\widetilde{h}\bigr)$ denotes the $p$-energy of $\widetilde{h}$ on the induced subgraph of $\Gamma_{n + 1}$ whose vertex set is given by $\{ u_{k} \cdot w \mid k \in \{ 1, \dots, D \}, w \in W_{n} \}$. 
    Similar construction by `unfolding' $\widetilde{h}$ as in the proof of Proposition \ref{prop.B} yields a function $h_{\ast} \colon \widetilde{W}_{n + 1} \to [0, 1]$ satisfying 
    \[
    h_{\ast}|_{I_{n + 1}} \equiv 1, \quad h_{\ast}|_{O_{n + 1}} \equiv 0 
    \]
    and 
    \begin{align*}
    	\mathcal{E}_{p}^{\Gamma_{n + 1}}(h_{\ast}) 
    	= \sum_{k = 1}^{D}2^{k}\binom{D}{k}\mathcal{E}_{p}^{\widetilde{G}_{n}}(h_{k})
    	\le 2^{D - 1}\mcC_{p, \textup{face}}^{(n)}\sum_{k = 1}^{D}2^{k}\binom{D}{k} 
    	\le 2^{D - 1}3^{D}\mcC_{p, \textup{face}}^{(n)}. 
    \end{align*}
    Hence we conclude that 
    \begin{equation*}
    	c \cdot \mathcal{C}_{p}^{(n)} \le \mcC_{p}^{\Gamma_{n + 1}}\bigl(I_{n + 1}, O_{n + 1}\bigr) \le \mathcal{E}_{p}^{\Gamma_{n + 1}}(h_{\ast}) \le 2^{D - 1}3^{D}\mcC_{p, \textup{face}}^{(n)}, 
    \end{equation*}
    which completes the proof. 
\end{proof}

Now we are ready to prove Lemma \ref{lem.chain}.

\noindent
\textit{Proof of Lemma \ref{lem.chain}.}\,
    Thanks to Lemma \ref{lem.top-bottom}, it will suffice to compare $\mcC_{p, \textup{face}}^{(n)}$ and $\mcC_{p}^{(n, M)}$.
    Suppose that $G_{n, M}$ is realized as a subgraph of $G_{n + m}$ using a horizontal chain $\bigl[w(1), \dots, w(M)\bigr]$ in $\bigl(V_{m}, \widetilde{E}_{m}\bigr)$, that is, $V_{n, M} = \bigcup_{i = 1}^{M} w(i) \cdot W_{n}$.
    First, we consider the case $M = 3$.
    By the monotonicity of $p$-conductance (Proposition \ref{prop.mono}), we immediately have that $\mcC_{p}^{(n, 3)} \le \mcC_{p, \textup{face}}^{(n)}$.
    We will prove the converse by using Lemma \ref{lem.URI}.
    Let us define a subgraph $\widehat{G}_{n, 3} =  \bigl(\widehat{V}_{n, 3}, \widehat{E}_{n, 3}\bigr)$ of $G_{n, 3}$ by
    \[
    \widehat{V}_{n, 3} \coloneqq \bigl\{ w(1)v \bigm| v \in W_{n}^{+} \bigr\} \cup w(2) \cdot W_{n} \cup \bigl\{ w(3)v \bigm| v \in W_{n}^{-} \bigr\},
    \]
    and 
    \[
    \widehat{E}_{n, 3} \coloneqq \Bigl\{ (v, w) \in E_{n, 3} \Bigm| v, w \in \widetilde{V}_{n, 3} \Bigr\}.
    \]
    Then we easily see that
    \[
    \mcC_{p}^{(n, 3)} = \mcC_{p}^{\widehat{G}_{n, 3}}\Bigl(\bigl\{ w(1)v \mid v \in W_{n}^{+} \bigr\}, \bigl\{ w(3)v \mid v \in W_{n}^{-} \bigr\}\Bigr).
    \]
    Define $\varphi_{n}\colon W_{n} \to \widehat{V}_{n, 3}$ by
    \[
    \varphi_{n}(w) \coloneqq
    \begin{cases}
        w^{1}\tau[R_{1}](w) \quad &\text{if $w \in W_{n}^{-}$,} \\
        w^{2}w \quad &\text{if $w \not\in W_{n}^{\,-} \cup W_{n}^{+}$,} \\
        w^{3}\tau[R_{1}](w) \quad &\text{if $w \in W_{n}^{+}$.}
    \end{cases}
    \]
    Then $\{ \varphi_{n} \}_{n \ge 1}$ is a uniform rough isometry between $\{ G_{n} \}_{n \ge 1}$ and $\bigl\{ \widehat{G}_{n, 3} \bigr\}_{n \ge 1}$ (with $C_1 = 1, C_2 = 2$, $C_{3} = 2$ and $C_{4} = L_{\ast}/2$ in Definition \ref{dfn.URI}).
    Applying Lemma \ref{lem.URI}, we get $\mcC_{p, \textup{face}}^{(n)} \le c_1\mcC_{p}^{(n, 3)}$, where $c_1 > 0$ depends only on $p$ and $L_{\ast}$. 

    Next, let us consider the case $M_{k} \coloneqq 2^{k} + 2$ for $k \in \bbN$.
    Define $\widetilde{\mathcal{C}}_{p}^{\,(n, M)}$ by 
    \[
    \widetilde{\mathcal{C}}_{p}^{\,(n, M)} = \mathcal{C}_{p}^{\widetilde{G}_{n, M}}\bigl(V_{n, M}^{-}, V_{n, M}^{+}\bigr), 
    \]
    where $\widetilde{G}_{n, M} = \bigl(V_{n, M}, \widetilde{E}_{n, M}\bigr)$. 
    We easily see that $G_{n, M}$ and $\bigl(V_{n, M}, \widetilde{E}_{n, M}\bigr)$ are uniformly rough isometric and thus Lemma \ref{lem.URI} implies that there exists a constant $c_{2} > 0$ (depending only on $p, L_{\ast}, D$) such that 
    \[
    c_{2}^{-1}\mathcal{C}_{p}^{\,(n, M)} \le \widetilde{\mathcal{C}}_{p}^{\,(n, M)} \le c_{2}\mathcal{C}_{p}^{\,(n, M)}
    \] 
    Let $f_{n, k}\colon V_{n, M_k} \to \bbR$ satisfy
    \[
    f_{n, k}|_{V_{n, M_k}^{-}} \equiv 1, \quad f_{n, k}|_{V_{n, M_k}^{+}} \equiv 0, \quad\text{and}\quad \mcE_{p}^{G_{n, M_k}}\bigl(f_{n, k}\bigr) = \widetilde{\mcC}_{p}^{\,(n, M_k)}.
    \]
    We shall show that 
    \begin{equation}\label{fnk-value}
        \min\bigl\{ f_{n, k}(v) \bigm| v \in w(i) \cdot W_{n}, 1 \le i \le 2^{k - 1} + 1 \bigr\} \ge \frac{1}{2}.
    \end{equation}
    (See Figure \ref{fig.sym}.)
    To this end, let us consider a function $\widehat{f}_{n,k}$ given by
    \[
    \widehat{f}_{n,k}(v) \coloneqq
    \begin{cases}
        f_{n,k}(v) \vee (1 - f_{n,k}(v)) \quad \text{if $v \in \bigcup_{i = 1}^{2^{k - 1} + 1}w(i) \cdot W_{n}$,} \\
        f_{n,k}(v) \wedge (1 - f_{n,k}(v)) \quad \text{if $v \in \bigcup_{i = 2^{k}}^{2^{k} + 2}w(i) \cdot W_{n}$,}
    \end{cases}
    \]
    which obviously satisfies $\widehat{f}_{n,k}|_{V_{n, M_k}^{-}} \equiv 1$ and $\widehat{f}_{n,k}|_{V_{n, M_k}^{+}} \equiv 0$.
    From the uniqueness of $f_{n, k}$, for $i = 1, \dots, M_{k}$ and $w \in W_{n}$, 
    \[
    f_{n, k}\bigl(w(i)w\bigr) = 1 - f_{n, k}\bigl(w(M_{k} - i + 1)\tau[R_{1}](w)\bigr). 
    \] 
    This yields that 
    \[
    \abs{f_{n, k}(x) - f_{n, k}(y)} = \abs{\widehat{f}_{n, k}(x) - \widehat{f}_{n, k}(y)}, \quad (x, y) \in E_{n}\Bigl(w\bigl(2^{k - 1} + 1\bigr), w\bigl(2^{k - 1} + 2\bigr)\Bigr), 
    \]
    where $E_{n}\bigl(w(2^{k - 1} + 1), w(2^{k - 1} + 2)\bigr)$ denotes 
    \[
    \Bigl\{ (x_{1}, x_{2}) \in E_{n + m} \Bigm| x_{i} \in w\bigl(2^{k - 1} + i\bigr) \cdot W_{n} \text{ for $i = 1, 2$} \Bigr\}. 
    \]
    Hence we have $\mcE_{p}\bigl(\,\widehat{f}_{n,k}\,\bigr) \le \mcE_{p}(f_{n,k})$.
    By the uniqueness of $f_{n, k}$, we have $\widehat{f}_{n,k} = f_{n,k}$, which deduces \eqref{fnk-value}.
    Similarly, we also have
    \begin{equation}\label{fnk-value2}
        \max\bigl\{ f_{n, k}(v) \bigm| v \in w(i) \cdot W_{n}, \,2^{k} \le i \le 2^{k} + 2 \bigr\} \le \frac{1}{2}.
    \end{equation}
    By \eqref{fnk-value}, \eqref{fnk-value2} and the Markov property of $p$-energies on graphs (Proposition \ref{prop.Markov}),
    \begin{align*}
        \widetilde{\mcC}_{p}^{\,(n, M_k)}
        &= \mcE_{p}^{\widetilde{G}_{n, M_{k}}}(f_{n,k}) \\
        &\ge \frac{1}{2}\biggl[\mcE_{p}^{\widetilde{G}_{n, M_{k}}}\Bigl(f_{n,k} \vee \frac{1}{2}\Bigr) + \mcE_{p}^{\widetilde{G}_{n, M_{k}}}\Bigl(f_{n,k} \wedge \frac{1}{2}\Bigr)\biggr]
        = 2^{-p}\widetilde{\mcC}_{p}^{\,(n, M_{k - 1})}.
    \end{align*}
    Iterating this estimate, we conclude that, for any $k \in \mathbb{N}$, 
    \[
    \mathcal{C}_{p}^{(n, M_{k})} \ge 
    c_{2}^{-1}\widetilde{\mcC}_{p}^{\,(n, M_k)} \ge c_{2}^{-1}2^{-pk}\widetilde{\mcC}_{p}^{\,(n, 3)} \ge c_{1}^{-1}c_{2}^{-2}2^{-pk}\mathcal{C}_{p, \textup{face}}^{(n)}
    \]
    Since $\mcC_{p}^{(n, M)} \ge \mcC_{p}^{(n, M')}$ for $M \le M'$, we obtain the desired estimate for all $M$. \hspace{\fill}$\square$
    
\begin{figure}[tb]
    \centering
    \includegraphics{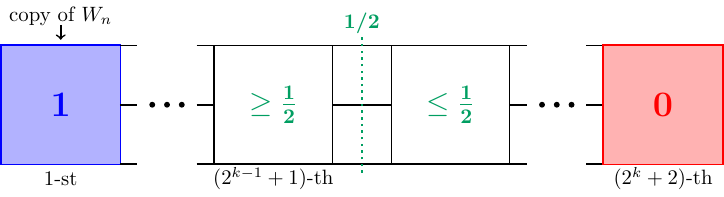}
    \caption{Values of $f_{n, k}$}
    \label{fig.sym}
\end{figure}

Finally, we prove \hyperlink{KMp}{\textup{(KM$_{p}$)}}.
We mainly follow the method in \cite{CQ21+}*{Lemma 4.8}.
\begin{thm}\label{thm.KM}
    Suppose Assumption \ref{assum.lowdim} holds.
    Then \hyperlink{KMp}{\textup{(KM$_{p}$)}} holds.
\end{thm}
\begin{proof}
    Let $f_{n}\colon W_n \to \bbR$ satisfy
    \[
    \mcE_{p}^{G_n}(f_{n}) = 1, \quad f_{n}|_{\partial_{\ast}G_{n}} \equiv 0, \quad\text{and}\quad \mean{f_{n}}{W_{n}} = \Bigl(\lambda_{p, \textup{Dir}}^{(n)}\Bigr)^{1/p}.
    \]
    Note that $f_{n}$ is non-negative.
    Pick $w^{\ast} \in W_{n}$ such that $f_{n}(w^{\ast}) = \max_{w \in W_n}f_{n}(w)$.
    Then we easily see that $f_{n}(w^{\ast}) \ge \bigl(\lambda_{p, \textup{Dir}}^{(n)}\bigr)^{1/p}$.
    Let $\theta, \widetilde{C}_{\textup{UH}} > 0$ be constants in Theorem \ref{thm.unifhol} and choose $l_{\ast} \in \mathbb{N}$ such that
    \begin{equation}\label{choice.bigl}
        \bigl(\widetilde{C}_{\textup{UH}}a^{-l_{\ast}\theta}\bigr)^{1/p} \le \frac{1}{8DN_{\ast}}. 
    \end{equation}
    Then, by Theorem \ref{thm.unifhol}, for any $n \ge l_{\ast}$,
    \begin{align}\label{e.osc}
        \max\bigl\{ \abs{f_{n}(v) - f_{n}(w)} \bigm| k \le n - l_{\ast}, z \in W_{l_{\ast} + k}, v, w \in z \cdot W_{n - l_{\ast} - k} \bigr\} \\
        &\hspace*{-60pt}\le \frac{1}{8DN_{\ast}}\Bigl(\lambda_{p, \textup{Dir}}^{(n)}\Bigr)^{1/p}. \nonumber 
    \end{align}
    Now, we will choose a chain of $(n - 2)$-cells from $\partial_{\ast}G_{n}$ to $w^{\ast}$. 
    Define 
    \[
    S_{\textsf{corner}} \coloneqq \left\{ \sum_{k = 1}^{D}\sigma_{k}\mathbf{e}_{k} \;\middle|\; \sigma_{k} \in \{ -1, +1 \} \right\} 
    \]
    that satisfies $\#S_{\textsf{corner}} = 2^{D}$.  
    Fix a path $[i(1), \dots, i(L)]$ in $G_{1}$, i.e. $i(k) \in S$ and $\bigl(i(k), i(k + 1)\bigr) \in E_{1}$ for any $l$, such that $i(1) \in \partial_{\ast}G_{1}$ and $w^{\ast} \in i(L) \cdot W_{n - 1}$. 
    Note that $L \le \diam G_{1} < N_{\ast}$. 
    We can find ``corners'' $z(1), z(2), \dots, z(2L - 1) \in \{ i^{n - 1} \mid i \in S_{\textsf{corner}} \}$ satisfying $v^{\ast} \coloneqq i(1)z(1) \in \partial_{\ast}G_{n}$ and 
    \[
    \bigl(i(k)z(2k), i(k + 1)z(2k + 1)\bigr) \in E_{n} \quad \text{for $k = 1, \dots, L - 1$.}
    \]
    Since $f_n(v^{\ast}) = 0$ and $\mcE_{p}^{G_n}(f_n) = 1$,
    \begin{align*}
        \Bigl(\lambda_{p, \textup{Dir}}^{(n)}\Bigr)^{1/p} &\le f_{n}(w^{\ast}) = f_{n}(w^{\ast}) - f_{n}(v^{\ast}) \\
        &\le \abs{f_{n}(v^{\ast}) - f_{n}(w^{\ast})} \\
        &\le \sum_{k = 1}^{L - 1}\Bigl(\abs{f_{n}\bigl(i(k)z(2k - 1)\bigr) - f_{n}\bigl(i(k)z(2k)\bigr)} \\
        &\hspace*{50pt}+ \abs{f_{n}\bigl(i(k)z(2k)\bigr) - f_{n}\bigl(i(k + 1)z(2k + 1)\bigr)}\Bigr) \\
        &\le \left\{ \sum_{k = 1}^{L - 1}\abs{f_{n}\bigl(i(k)z(2k - 1)\bigr) - f_{n}\bigl(i(k)z(2k)\bigr)} \right\} + L - 1.
    \end{align*}
    Now, we have $\lambda_{p, \textup{Dir}}^{(n)} \to \infty$ as $n \to \infty$ by Assumption \ref{assum.lowdim} and Theorem \ref{thm.poi4}.
    Hence we may assume that
    \[
    \sum_{k = 1}^{L - 1}\abs{f_{n}\bigl(i(k)z(2k - 1)\bigr) - f_{n}\bigl(i(k)z(2k)\bigr)} \ge
    \frac{1}{2}\Bigl(\lambda_{p, \textup{Dir}}^{(n)}\Bigr)^{1/p},
    \]
    for all large $n$.
    Then there exists $l \in \{ 1, \dots, L - 1 \}$ such that
    \begin{align*}
        \abs{f_{n}\bigl(i(k)z(2k - 1)\bigr) - f_{n}\bigl(i(k)z(2k)\bigr)} 
        \ge \frac{1}{2(L - 1)}\Bigl(\lambda_{p, \textup{Dir}}^{(n)}\Bigr)^{1/p}.
    \end{align*}
    Moreover, by the triangle inequality, there exist $i^{\ast}(0), i^{\ast}(1) \in S_{\textsf{corner}}$ and $j \in \{ 1, \dots, D \}$ with $i^{\ast}(0) - i^{\ast}(1) \in \{ -\mathbf{e}_{j}, \mathbf{e}_{j} \}$ such that 
    \begin{align}\label{e.osc-cell}
        \abs{f_{n}\bigl(i(l)i^{\ast}(0)^{n - 1}\bigr) - f_{n}\bigl(i(l)i^{\ast}(1)^{n - 1}\bigr)} \ge \frac{1}{2D(L - 1)}\Bigl(\lambda_{p, \textup{Dir}}^{(n)}\Bigr)^{1/p}.
    \end{align}
    By using the symmetries of $f_{n}$ if necessary, we may assume that 
    \[
    f_{n}\bigl(i(l)i^{\ast}(1)^{n - 1}\bigr) \ge f_{n}\bigl(i(l)i^{\ast}(0)^{n - 1}\bigr). 
    \]
    Then we see from \eqref{e.osc} and \eqref{e.osc-cell} that
    \begin{align*}
      \delta_{n}
      &\coloneqq \min_{i(l)i^{\ast}(1)^{l_{\ast}}\cdot W_{n - l_{\ast} - 1}}f_{n} - \max_{i(l)i^{\ast}(0)^{l_{\ast}}\cdot W_{n - l_{\ast} - 1}}f_{n} \\
      &\ge \biggl[\frac{1}{2D(L - 1)} - 2\cdot\frac{1}{8DN_{\ast}}\biggr]\Bigl(\lambda_{p, \textup{Dir}}^{(n)}\Bigr)^{1/p}
      \ge \frac{1}{4DN_{\ast}}\Bigl(\lambda_{p, \textup{Dir}}^{(n)}\Bigr)^{1/p},
    \end{align*}
    where we used a bound $L - 1 \le N_{\ast}$ in the last inequality.
    Let us consider the ``horizontal chain of $(n - l_{\ast} - 1)$-cell along $j$-axis'' from $i(l)i^{\ast}(0)^{l_{\ast}}\cdot W_{n - l_{\ast} - 1}$ to $i(l)i^{\ast}(1)^{l_{\ast}}\cdot W_{n - l_{\ast} - 1}$.
    Then, by putting $g_{n} \coloneqq  \bigl((\delta_{n}^{-1}f_{n} + c_{n}) \vee 0\bigr) \wedge 1$, where $c_{n}$ is a constant such that $g_{n}|_{i(l)i^{\ast}(\varepsilon)^{l_{\ast}}\cdot W_{n - l_{\ast} - 1}} \equiv \varepsilon$ for $\varepsilon = 0, 1$, we have that
    \begin{align*}
        \mcC_{p}^{(n - l_{\ast} - 1, a^{l_{\ast}})} \le \mcE_{p}^{G_n}(g_{n})
        \le \delta_{n}^{-p}\mcE_{p}^{G_n}(f_n) \le \bigl(4DN_{\ast}\bigr)^{p}\Bigl(\lambda_{p, \textup{Dir}}^{(n)}\Bigr)^{-1}.
    \end{align*}
    From Lemma \ref{lem.chain}, Theorems \ref{thm.sub-mult} and \ref{thm.poi4}-(3), we conclude that
    \begin{align*}
    	\mathcal{C}_{p}^{(n)} 
    	&\le C_{\ref{thm.sub-mult}}\mathcal{C}_{p}^{(l_{\ast} + 1)}\mathcal{C}_{p}^{(n - l_{\ast} - 1)} \\ 
    	&\le C_{\ref{thm.sub-mult}}C(a^{l_{\ast}})\mathcal{C}_{p}^{(l_{\ast} + 1)}\mathcal{C}_{p}^{(n - l_{\ast} - 1, a^{l_{\ast}})} \\
    	&\le C_{\ref{thm.sub-mult}}C(a^{l_{\ast}})\mathcal{C}_{p}^{(l_{\ast} + 1)}(4DN_{\ast})^{p}\Bigl(\lambda_{p, \textup{Dir}}^{(n)}\Bigr)^{-1} 
    	\le C_{\textup{KM}}\Bigl(\lambda_{p}^{(n)}\Bigr)^{-1}, 
    \end{align*}
    where $C_{\textup{KM}} = C(a^{l_{\ast}})\mathcal{C}_{p}^{(l_{\ast} + 1)}(4DN_{\ast})^{p}C_{\ref{thm.poi4}}^{-1}$, which depends only on $p, a, D, L_{\ast}, N_{\ast}, \rho_{p}$. 
    This proves \hyperlink{KMp}{\textup{(KM$_{p}$)}}.
\end{proof}

We conclude this section by giving some useful estimates to construct $p$-energies. 
Under Assumption \ref{assum.lowdim}, we have multiplicative inequality of $\mathcal{R}_{p}^{(\,\cdot\,)}$ by virtue of Proposition \ref{prop.B}, Theorems \ref{thm.KM} and \ref{thm.superP}.
In particular, $\mathcal{R}_{p}^{(n)} \asymp \rho_{p}^{n}$ for all $n \in \mathbb{N}$.   
Then Theorem \ref{thm.ARCdim}) implies that  
\begin{equation}\label{dfn.betap}
    \beta_{p} \coloneqq \frac{\log{N_{\ast}\rho_{p}}}{\log{a}}
\end{equation}
satisfies $\beta_p - \alpha > 0$, where $\alpha = \log{N_{\ast}}/\log{a}$ is the Hausdorff dimension of $(K, d)$. 

Now, we define the rescaled discrete $p$-energy $\widetilde{\mcE}_{p}^{G_{n}}$ by setting
\[
\widetilde{\mcE}_{p}^{G_{n}}(f) \coloneqq \rho_{p}^{n}\mcE_{p}^{G_{n}}(f), \quad f \in \mathbb{R}^{W_{n}}. 
\]
Similarly to Theorem \ref{thm.unifhol}, we can show the following H\"{o}lder estimate.
\begin{thm}\label{thm.holder_KM}
    For every $n, m \in \bbN$, $z \in W_{m}$,  $v, w \in \mcB_{n}(z, 1)$ and $f \in \bbR^{W_{n + m}}$,
    \begin{equation}\label{rescale.holder}
        \abs{f(v) - f(w)}^{p} \le C_{\textup{UH}}\,\widetilde{\mcE}_{p}^{G_{n + m}}(f)a^{-(\beta_p - \alpha)m},
    \end{equation}
    where $C_{\textup{UH}} > 0$ is a constant depending only on $p, a, D, L_{\ast}, N_{\ast}, \rho_p$.
\end{thm}
\begin{proof}
	The required uniform H\"{o}lder estimate \eqref{rescale.holder} can be proved by using $\mathcal{C}_{p}^{(n)} \lesssim a^{-(\beta_{p} - \alpha)n}$, which is a consequence of supermultiplicative inequality of $\mathcal{C}_{p}^{(\,\cdot\,)}$, instead of \eqref{est.super} in the proof of Theorem \ref{thm.unifhol}. 
\end{proof}

Lastly, we observe a monotonicity result (the so-called \emph{weak monotonicity} in \cite{GY19}).
This is proved in \cite{KZ92}*{Proposition 5.2} for $p = 2$.
\begin{cor}\label{cor.wm}
    For every $n, m \in \bbN$ and $f \in L^{p}(K, \mu)$,
    \begin{equation}\label{est.wm}
        \widetilde{\mcE}_{p}^{G_{n}}(M_{n}f) \le C_{\textup{WM}}\,\widetilde{\mcE}_{p}^{G_{n + m}}(M_{n + m}f),
    \end{equation}
    where $C_{\textup{WM}} > 0$ is a constant depending only on $p, a, D, L_{\ast}, N_{\ast}, \rho_{p}$.
    In particular,
    \begin{equation}\label{est.lim-nehavor}
        \sup_{n \in \bbN}\widetilde{\mcE}_{p}^{G_{n}}(M_{n}f) \le C_{\textup{WM}}\varliminf_{n \to \infty}\widetilde{\mcE}_{p}^{G_{n}}(M_{n}f).
    \end{equation}
\end{cor}
\begin{proof}
    By Proposition \ref{prop.BF}-(3), for any $n, m \in \bbN$, $w \in W_{n}$ and $f \in L^{p}(K, \mu)$,
    \begin{align*}
        P_{n + m, n}(M_{n + m}f)(w)
        = N_{\ast}^{n}\sum_{v \in w \cdot W_{m}}\int_{K_{v}}f\,d\mu
        = M_{n}f(w).
    \end{align*}
    Hence $P_{n + m, n}(M_{n + m}f) = M_{n}f$.
    We get \eqref{est.wm} from Lemma \ref{lem.wm} and \eqref{Poi-super}.
\end{proof}

\begin{rmk}\label{rmk.harnack}
    One can derive a uniform Harnack type estimate for discrete $p$-harmonic functions as an application of \hyperlink{KMp}{\textup{(KM$_{p}$)}}.
    For $p = 2$, this was done by \cite{BB89}*{Theorem 3.1} or \cite{KZ92}*{Lemma 7.8}.
    We expect that such type estimate will be important for future work, but we omit this since its proof does not fit the purpose of this paper.
\end{rmk}

\section{The domain of $p$-energy}\label{sec:domain}

This section aims to prove a part of our main results: Theorems \ref{thm.main1} and \ref{thm.main2}.

Let $(K, S, \{ F_{i} \}_{i \in S}) = \GSC(D, a, S)$ be a generalized Sierpi\'{n}ski carpet. 
Thanks to Corollary \ref{cor.wm}, we know that the following quantity:
\begin{equation}\label{dfn.1p-semi}
	\abs{f}_{\mcF_{p}} \coloneqq \sup_{n \in \bbN}\widetilde{\mcE}_{p}^{G_{n}}(M_{n}f)^{1/p}.
\end{equation}
describes the limit behavior of rescaled $p$-energy $\widetilde{\mcE}_{p}^{G_{n}}(M_{n}f)$.
Then we easily see that $\abs{\,\cdot\,}_{\mcF_{p}}$ defines a ($[0, \infty]$-valued) semi-norm.
We also define a function space $\mcF_{p}$ and its norm $\norm{\,\cdot\,}_{\mcF_{p}}$ by setting
\begin{equation*}\label{dfn.Fp}
    \mcF_{p} \coloneqq \{ f \in L^{p}(K, \mu) \mid \abs{f}_{\mcF_{p}} < \infty \} \quad \text{and} \quad \norm{f}_{\mcF_{p}} \coloneqq \norm{f}_{L^{p}} + \abs{f}_{\mcF_{p}}.
\end{equation*}
Ideally, $\mcF_{p}$ plays the same role as the Sobolev space $W^{1, p}$ in smooth settings like Euclidean spaces.
As stated in \cite{KM21+}*{Section 7}, this $(1,p)$-``Sobolev'' space $\mcF_{p}$ should be \emph{closable} and have \emph{regularity}, that is,
\begin{itemize}
    \item [$\bullet$] any Cauchy sequence $\{ f_n \}_{n \ge 1}$ in $\abs{\,\cdot\,}_{\mcF_{p}}$ with $f_n \to 0$ in $L^{p}$ converges to $0$ in $\mcF_p$;
    \item [$\bullet$] $\mcF_{p} \cap \mcC(K)$ is dense in $\mcC(K)$ with respect to the supremum norm.
\end{itemize}
We prove these properties in subsection \ref{subsec:regular}.
In addition, the \emph{separability} of $\mcF_{p}$ is proved in subsection \ref{subsec:separable}.
The separability will be essential to follow our construction of $p$-energy in section \ref{sec:constr}.
We also see in subsection \ref{subsec:LBtype} that $\mcF_{p}$ has a Besov-like representation, which is an extension of results for $\mcF_{2}$ in \cite{GHL03, Kum00}.

Throughout this section, we suppose Assumption \ref{assum.lowdim} holds.

\subsection{Closability and regularity}\label{subsec:regular}

First, we derive the following H\"{o}lder estimate from the uniform H\"{o}lder estimates on graphical approximations (Theorem \ref{thm.holder_KM}) in the same way as \cite{Kig21+}*{Lemmas 6.10 and 6.13}.

\begin{thm}\label{thm.embed}
    There exists a positive constant $C_{\textup{H\"{o}l}}$ (depending only on $p, a, D, L_{\ast}$, $N_{\ast}$, $\rho_{p}$) such that every $f \in \mathcal{F}_{p}$ has a continuous modification $f_{\ast} \in \mcC(K)$ with
    \[
    \abs{f_{\ast}(x) - f_{\ast}(y)}^{p} \le C_{\textup{H\"{o}l}}\abs{f}_{\mcF_{p}}^{p}d(x, y)^{\beta_{p} - \alpha},
    \]
    for every $x, y \in K$.
    Moreover, the inclusion map $\mcF_{p} \ni f \mapsto f_{\ast} \in \mcC(K)$ is injective.
    In particular, $\mcF_{p}$ is continuously embedded in the H\"{o}lder space $\mcC^{0, (\beta_{p} - \alpha)/p}$.
\end{thm}
\begin{proof}
    Let $f \in \mcF_p$.
    By Proposition \ref{prop.BF}-(3) and Lemma \ref{lem.invariant}, for each $n \ge 1$, we have that $\int_{K}f\,d\mu = N_{\ast}^{-n}\sum_{w \in W_{n}}M_{n}f(w)$.
    From this identity, there exists $w(n) \in W_{n}$ for each $n \in \bbN$ such that $M_{n}\abs{f}(w(n)) \le \int_{K}\abs{f}\,d\mu$.
    Then, by Theorem \ref{thm.holder_KM}, for any $n \in \bbN$ and $v \in W_{n}$,
    \begin{align*}
        \abs{M_{n}f(v)}^{p}
        &\le 2^{p - 1}\abs{M_{n}f(v) - M_{n}f(w(n))}^{p} + 2^{p - 1}\abs{M_{n}f(w(n))}^{p} \nonumber \\
        &\le 2^{p - 1}C_{\text{UH}}\abs{f}_{\mcF_{p}}^{p} + 2^{p - 1}\left(\int_{K}\abs{f}^{p}\,d\mu\right)^{p} \nonumber \\
        &\le 2^{p - 1}C_{\text{UH}}\abs{f}_{\mcF_{p}}^{p} + 2^{p - 1}\norm{f}_{L^p}^{p},
    \end{align*}
    and hence we obtain the following uniform bound of $f \in \mcF_{p}$:
    \begin{align}\label{e.ub}
        \adjustlimits\sup_{n \ge 1}\max_{v \in W_{n}}\abs{M_{n}f(v)} \le c_1\norm{f}_{\mcF_{p}} < \infty ,
    \end{align}
    where $c_1 > 0$ depends only on $p$ and $C_{\text{UH}}$.

    For each $n \in \bbN$, enumerate the elements $W_{n}$ as $W_{n} = \{ w_{n}(1), \dots, w_{n}(N_{\ast}^{n}) \}$ and inductively define $\bigl\{ \widehat{K}_{w_{n}(i)} \bigr\}_{i = 1}^{N_{\ast}^{n}}$ as follows: $\widehat{K}_{w_{n}(1)} \coloneqq K_{w_{n}(1)}$ and
    \[
    \widehat{K}_{w_{n}(i + 1)} \coloneqq K_{w_{n}(i + 1)} \setminus \bigcup_{j \le i}\widehat{K}_{w_{n}(j)}.
    \]
    Note that each $\widehat{K}_{w_{n}(i)}$ is a Borel set of $K$, $\widehat{K}_{w_{n}(i)} \, (i = 1, \dots, N_{\ast}^{n})$ are disjoint, and $K = \bigcup_{i = 1}^{N_{\ast}^{n}}\widehat{K}_{w_{n}(i)}$.
    Also, by Proposition \ref{prop.BF}-(3), we have $\mu\bigl(K_{w} \setminus \widehat{K}_{w}\bigr) = 0$ for any $w \in W_{n}$.
    Next, define a Borel measurable function $f_{n}\colon K \to \bbR$ by setting
    \[
    f_{n} \coloneqq \sum_{w \in W_{n}}M_{n}f(w)\indicator{\widehat{K}_{w}}.
    \]
    Then \eqref{e.ub} yields that
    \begin{equation}\label{eq.UB_AA}
        \sup_{n \ge 1}\sup_{x \in K}\abs{f_{n}(x)} \le c_{1}\norm{f}_{\mcF_{p}}.
    \end{equation}

    Let $n \in \bbN$, let $x \neq y \in K$ and set $n_{\ast} \coloneqq n(x, y) \in \bbZ_{\ge 0}$.
    If $n > n_{\ast}$, then there exist $v, w \in W_{n_{\ast}}$ such that $x \in K_{v}$, $y \in K_{w}$ and $K_{v} \cap K_{w} \neq \emptyset$.
    We can find $v', w' \in W_{n}$ such that $x \in \widehat{K}_{v'}$ and $y \in \widehat{K}_{w'}$.
    Then $v' \in \mcB_{n - n_{\ast}}(v, 1)$, $w' \in \mcB_{n - n_{\ast}}(w, 1)$ and $\mcB_{n - n_{\ast}}(v, 1) \cap \mcB_{n - n_{\ast}}(w, 1) \neq \emptyset$.
    Fix $z' \in \mcB_{n - n_{\ast}}(v, 1) \cap \mcB_{n - n_{\ast}}(w, 1)$.
    Applying Theorem \ref{thm.holder_KM}, we have that
    \begin{align*}
        \abs{f_{n}(x) - f_{n}(y)}^{p}
        &\le 2^{p - 1}\bigl(\abs{M_{n}f(v') - M_{n}f(z')}^{p} + \abs{M_{n}f(z') - M_{n}f(w')}^{p}\bigr) \\
        &\le 2^{p}C_{\textup{UH}}\widetilde{\mcE}_{p}^{G_{n}}(M_{n}f)a^{-(\beta_p - \alpha)n_{\ast}} \\
        &\le 2^{p + 1}a\sqrt{D}C_{\textup{UH}}\abs{f}_{\mcF_p}^{p}d(x, y)^{\beta_p - \alpha},
    \end{align*}
    where we used Lemma \ref{lem.1-adapted} in the last line.
    If $n \le n_{\ast}$, then there exist $v, w \in W_{n}$ such that $x \in K_{v}$, $y \in K_{w}$ and $K_{v} \cap K_{w} \neq \emptyset$.
    Let $v', w' \in W_{n}$ with $x \in \widehat{K}_{v'}$ and $y \in \widehat{K}_{w'}$.
    Then $[v', v, w, w']$ is a path in $G_{n}$, and hence we have that
    \begin{align*}
        &\abs{f_{n}(x) - f_{n}(y)}^{p} \\
        &\le 3^{p - 1}\bigl(\abs{M_{n}f(v') - M_{n}f(v)}^{p} + \abs{M_{n}f(v) - M_{n}f(w)}^{p} + \abs{M_{n}f(w) - M_{n}f(w')}^{p}\bigr) \\
        &\le 3^{p}\mcE_{p}^{G_{n}}(M_{n}f)
        \le 3^{p}\abs{f}_{\mcF_p}^{p}\rho_{p}^{-n}.
    \end{align*}
    As a result of this observation, we conclude that
    \begin{equation}\label{eq.equi_AA}
        \abs{f_{n}(x) - f_{n}(y)}^{p} \le c_{2}\abs{f}_{\mcF_p}^{p}\bigl(d(x, y)^{\beta_p - \alpha} + \rho_{p}^{-n}\bigr), \quad \text{$f \in \mcF_p, n \in \bbN, x, y \in K$,}
    \end{equation}
    where $c_2$ is a positive constant depending only on $p, a, D, C_{\text{UH}}$.

    Thanks to \eqref{eq.UB_AA} and \eqref{eq.equi_AA}, we can apply an Arzel\'{a}--Ascoli type argument for $\{ f_{n} \}_{n \ge 1}$ (see \cite{Kig21+}*{Lemma D.1}).
    For reader's convenience, we provide a complete proof.
    Set $A_{n} \coloneqq \bigl\{ F_{w}\bigl(\sum_{k = 1}^{D}\sigma_{k}\mathbf{e}_{k}\bigr) \bigr\}_{\sigma_{k} \in \{ -1, +1 \}, w \in W_{n}}$ and $A \coloneqq \bigcup_{n \ge 1}A_{n}$.
    Then $A$ is a countable dense subset of $K$.
    Since $\{ f_{n}(x) \}_{n \ge 1}$ is bounded for each $x \in A$ by \eqref{eq.UB_AA}, by a diagonal argument, we obtain a subsequence $\{ n_k \}_{k \ge 1}$ such that $\{ f_{n_k}(x) \}_{k \ge 1}$ converges as $k \to \infty$ for any $x \in A$.
    Define $f_{\ast}(x) \coloneqq \lim_{k \to \infty}f_{n_k}(x)$ for any $x \in A$.
    From \eqref{eq.equi_AA} and Assumption \ref{assum.lowdim}, we see that
    \[
    \abs{f_{\ast}(x) - f_{\ast}(y)}^{p} \le c_{2}\abs{f}_{\mcF_p}^{p}d(x, y)^{\beta_p - \alpha} \quad \text{for any $x, y \in A$.}
    \]
    Since $A$ is dense in $K$, $f_{\ast}$ is extended to a continuous function on $K$, which is again denoted by $f_{\ast} \in \mcC(K)$, and it follows that
    \[
    \abs{f_{\ast}(x) - f_{\ast}(y)}^{p} \le c_{2}\abs{f}_{\mcF_p}^{p}d(x, y)^{\beta_p - \alpha} \quad \text{for any $x, y \in K$.}
    \]
    (We can also show that $\sup_{x \in K}\abs{f_{\ast}(x) - f_{n_k}(x)} \to 0$ as $k \to \infty$. For a proof, see \cite{Kig21+}*{Lemma D.1}.)
    Then, for any $m \in \bbN$ and $w \in W_{m}$, we have $\int_{K_{w}}f_{n_{k}}\,d\mu \to \int_{K_{w}}f_{\ast}\,d\mu \quad \text{as $k \to \infty$}$.
    By Proposition \ref{prop.BF}-(3) and $\mu\bigl(K_{w} \setminus \widehat{K}_{w}\bigr) = 0$,
    \begin{align*}
        \int_{K_{w}}f_{n_{k}}\,d\mu
        &= \int_{K_{w}}\sum_{z \in W_{n_{k}}}M_{n_{k}}f(z)\indicator{\widehat{K}_{z}}\,d\mu
        = \int_{K_{w}}\sum_{z \in w \cdot W_{n_{k} - m}}M_{n_{k}}f(z)\indicator{K_{z}}\,d\mu \\
        &= \sum_{z \in w \cdot W_{n_{k} - m}}\frac{1}{\mu(K_{z})}\left(\int_{K_{z}}f\,d\mu\right)\int_{K_{w}}\indicator{K_{z}}\,d\mu
        = \int_{K_{w}}f\,d\mu,
    \end{align*}
    whenever $w \in W_{m}$ and $n_{k} > m$.
    Letting $k \to \infty$, we obtain $\int_{K_{w}}f_{\ast}\,d\mu = \int_{K_{w}}f\,d\mu$ for all $w \in W_{\#}$.
    By Proposition \ref{prop.BF} and Dynkin's $\pi$-$\lambda$ theorem, we conclude that $f_{\ast}$ is a continuous modification of $f$.
    The injectivity of $f \mapsto f_{\ast}$ is obvious.
    We complete the proof.
\end{proof}

Next, we prove the closability by proving that $(\mcF_{p}, \norm{\,\cdot\,}_{\mcF_{p}})$ is complete.
See also \cite{Kig21+}*{Lemmas 6.15 and 6.16}.

\begin{thm}\label{thm.close}
    $(\mcF_{p}, \norm{\,\cdot\,}_{\mcF_{p}})$ is a Banach space.
\end{thm}
\begin{proof}
    Let $\{ f_n \}_{n \ge 1}$ be a Cauchy sequence in $(\mcF_p, \norm{\,\cdot\,}_{\mcF_p})$.
    Then $\{ f_n \}_{n \ge 1}$ converges to some $f \in L^{p}(K, \mu)$ in $L^{p}$.
    Fix $x_0 \in K$ and set $g_n \coloneqq f_n - f_n(x_0)$.
    Then, by the H\"{o}lder estimate in Theorem \ref{thm.embed}, for all $n, m \ge 1$ and $x \in K$,
    \begin{align*}
        \abs{g_n(x) - g_m(x)}^{p}
        \le C_{\textup{H\"{o}l}}\abs{f_n - f_m}_{\mcF_p}^{p}d(x, x_0)^{\beta_p - \alpha} 
        \le C_{\textup{H\"{o}l}}\abs{f_n - f_m}_{\mcF_p}^{p},
    \end{align*}
    which implies that $\{ g_n \}_{n \ge 1}$ is a Cauchy sequence in $\mcC(K)$.
    Since $\mcC(K)$ is complete, $\{ g_n \}_{n \ge 1}$ converges to some $g \in \mcC(K)$ in the supremum norm.

    It is immediate that $\{ f_n - g_n \}_{n \ge 1}$ converges to $f - g$ in $L^{p}$, and thus we can pick a subsequence $\{ n_k \}_{k \ge 1}$ so that $f_{n_k} - g_{n_k} \to f - g$ for $\mu$-a.e. as $k \to \infty$.
    On the other hand, the definition of $g_n$ implies that $f_n - g_n \equiv f_n(x_0)$.
    Hence the limit $\lim_{k \to \infty}f_{n_k}(x_0) \eqqcolon b$ exists and $f - g = b$ for $\mu$-a.e.
    In particular, $f$ admits a continuous modification.
    We again write $f$ to denote this continuous version.
    Then $f$ is the limit of $f_{n_k}$.
    Indeed, we have
    \[
    \norm{f - f_{n_k}}_{\mcC(K)} \le \norm{g - g_{n_k}}_{\mcC(K)} + \abs{f_{n_k}(x_0) - b} \to 0 \quad\text{as $k \to \infty$}.
    \]
    Since $\{ f_n \}_{n \ge 1}$ is a Cauchy sequence in $\mcF_p$, for any $\varepsilon > 0$ there exists $N(\varepsilon) \ge 1$ such that
    \[
    \sup_{i \wedge j \ge N(\varepsilon)}\sup_{k \ge 1}\widetilde{\mcE}_{p}^{G_k}(M_{k}f_{n_{i}} - M_{k}f_{n_{j}}) \le \varepsilon,
    \]
    which implies that
    \begin{align}\label{e.fatou}
        \sup_{i \ge N(\varepsilon)}\sup_{k \ge 1}\widetilde{\mcE}_{p}^{G_k}(M_{k}f_{n_{i}} - M_{k}f) \le \varepsilon.
    \end{align}
    Therefore, we have that, for large $i \ge 1$ with $n_i \ge N(\varepsilon)$,
    \begin{align*}
        \widetilde{\mcE}_{p}^{G_k}(M_{k}f)^{1/p}
        \le \widetilde{\mcE}_{p}^{G_k}(M_{k}f_{n_{i}} - M_{k}f)^{1/p} + \widetilde{\mcE}_{p}^{G_k}(M_{k}f_{n_{i}})^{1/p}
        \le \varepsilon + \sup_{n \ge 1}\abs{f_{n}}_{\mcF_p},
    \end{align*}
    which implies that $f \in \mcF_p$.
    In addition, \eqref{e.fatou} yields that $\norm{f - f_{n_i}}_{\mcF_p} \to 0$ as $i \to \infty$.

    The convergence $\norm{f - f_n}_{\mcF_p} \to 0$ is easily derived by applying the above arguments for any subsequence of $\{ f_{n} \}_{n \ge 1}$.
    We complete the proof.
\end{proof}

Moreover, we can show that $\mcF_p$ is compactly embedded in $L^{p}(K, \mu)$.

\begin{prop}\label{prop.cptembed}
    The inclusion map from $\mcF_p$ to $L^{p}(K, \mu)$ is a compact operator.
\end{prop}
\begin{proof}
    Let $\{ f_n \}_{n \ge 1}$ be a bounded sequence in $\mcF_p$.
    Since the embedding of $\mcF_p$ in $\mcC^{0, (\beta_{p} - \alpha)/p}$ is continuous, we obtain a subsequence $\{ f_{n_k}\}_{k \ge 1}$ and $f \in \mcC(K)$ such that $f_{n_k}$ converges to $f$ in the supremum norm by applying the Arzel\'{a}--Ascoli theorem.
    This proves our assertion.
\end{proof}

Towards the regularity of $\mcF_{p}$, the following lemma gives a ``partition of unity'' in $\mcF_{p}$.
See also \cite{Kig21+}*{Lemma 6.18}.

\begin{lem}\label{lem.corebase}
    There exists a family $\{ \varphi_{w} \}_{w \in W_{\#}}$ in $\mcF_{p}$ such that
    \begin{itemize}
        \item [\textup{(1)}] for any $w \in W_{\#}$, $0 \le \varphi_{w} \le 1$;
        \item [\textup{(2)}] for any $n \in \bbN$, $\sum_{w \in W_{n}}\varphi_{w} \equiv 1$;
        \item [\textup{(3)}] for any $n \in \bbN$ and $w \in W_{n}$, $\supp[\varphi_{w}] \subseteq U_{1}^{(n)}(w)$, where $U_{1}^{(n)}(w)$ is defined as
        \begin{equation}\label{dfn.U1w}
        	U_{1}^{(n)}(w) \coloneqq \bigcup_{v \in W_{n}; d_{G_{n}}(v,w) \le 1}K_{v};
        \end{equation}
        \item [\textup{(4)}] there exists a constant $C_{\ref{lem.corebase}} > 0$ (depending only on $p, a, D, L_{\ast}, N_{\ast}, \rho_{p}$) such that
        \[
        \abs{\varphi_{w}}_{\mcF_{p}}^{p} \le C_{\ref{lem.corebase}}\rho_{p}^{n} \quad \text{for any $n \in \bbN$ and $w \in W_{n}$.}
        \]
    \end{itemize}
\end{lem}
\begin{proof}
    For $n, m \in \bbN$ and $w \in W_m$, let $\psi_{w}^{(n)}\colon W_{n + m} \to [0, 1]$ satisfy $\psi_{w}^{(n)}|_{w \cdot W_{n}} \equiv 1$, $\psi_{w}^{(n)}|_{W_{n + m} \setminus \mcB_{n}(w, 1)} \equiv 0$ and $\mcE_{p}^{G_{n + m}}\bigl(\psi_{w}^{(n)}\bigr) = \mcC_{p}^{G_{n+ m}}\bigl(w \cdot W_{n}, W_{n + m} \setminus \mcB_{n}(w, 1)\bigr)$.
    Define $\Psi_{m}^{(n)} \coloneqq \left(\sum_{v \in W_{m}}\psi_{v}^{(n)}\right)^{-1}$ and $\varphi_{w}^{(n)} \coloneqq \Psi_{m}^{(n)}\psi_{w}^{(n)}$.
    Note that $\varphi_{w}^{(n)}$ coincides with the function denoted by the same symbol in the proof of Lemma \ref{lem.subdiv.fcn}.
    We also set $\widetilde{\varphi}_{w}^{\,(n)}\colon K \to \bbR$ by setting
    \[
    \widetilde{\varphi}_{w}^{\,(n)} \coloneqq \sum_{z \in W_{n + m}}\varphi_{w}^{(n)}(z)\indicator{\widehat{K}_z},
    \]
    where $\bigl\{ \widehat{K}_z \bigr\}_{z \in W_{\#}}$ is the same as in the proof of Theorem \ref{thm.embed}.
    Then $M_{n + m}\widetilde{\varphi}_{w}^{\,(n)} = \varphi_{w}^{(n)}$ and, from \eqref{est.unityabove} in the proof of Lemma \ref{lem.subdiv.fcn} and \eqref{Poi-super}, we have $\widetilde{\mcE}_{p}^{G_{n + m}}\bigl(\varphi_{w}^{(n)}\bigr) \le c_1\rho_{p}^{m}$ for all $w \in W_m$ and $n \in \bbN$, where $c_1 > 0$ depends only on $p, a, D, N_{\ast}, L_{\ast}$.
    In particular, by Theorem \ref{thm.embed}, we obtain
    \[
    \abs{\widetilde{\varphi}_{w}^{\,(n)}(x) - \widetilde{\varphi}_{w}^{\,(n)}(y)}^{p} \le c_{1}C_{\textup{H\"{o}l}}\,\rho_{p}^{m}d(x, y)^{\beta_p - \alpha},
    \]
    for $x, y \in K$ with $n(x, y) < n + m$.
    Similarly to the Arzel\'{a}--Ascoli type argument in the proof of Theorem \ref{thm.embed}, we can find a subsequence $\{ n_k \}_{k \ge 1}$ and a continuous function $\varphi_{w} \in \mcC(K)$ such that $\lim_{k \to \infty}\widetilde{\varphi}_{w}^{\,(n_k)}(x) = \varphi_{w}(x)$ for any $x \in K$ and
    \[
    \abs{\varphi_w(x) - \varphi_w(y)}^{p} \le c_{1}C_{\textup{H\"{o}l}}\,\rho_{p}^{\abs{w}}d(x, y)^{\beta_p - \alpha} \quad\text{for any $x, y \in K$}.
    \]
    Then the properties (1), (2) and (3) are immediate from this convergence and the associated properties of $\widetilde{\varphi}_{w}^{\,(n)}$, so it will suffice to show (4).
    By the weak monotonicity (Corollary \ref{cor.wm}),
    \begin{align*}
      \widetilde{\mcE}_{p}^{G_{l}}\Bigl(M_{l}\widetilde{\varphi}_{w}^{\,(n_k)}\Bigr)
      &\le C_{\text{WM}}\widetilde{\mcE}_{p}^{G_{n_k + m}}\Bigl(M_{n_k + m}\widetilde{\varphi}_{w}^{\,(n_k)}\Bigr) \\
      &\le C_{\text{WM}}\sup_{n \ge 1}\widetilde{\mcE}_{p}^{G_{n + m}}\Bigl(\varphi_{w}^{(n)}\Bigr) \\
      &\le c_{1}C_{\text{WM}}\rho_{p}^{m},
    \end{align*}
    whenever $l \le n_k + m$.
    Passing to the limit $k \to \infty$ and supremum over $l \in \bbN$ in this estimate, we conclude that $\abs{\varphi_{w}}_{\mcF_p}^{p} \le c_{1}C_{\text{WM}}\rho_{p}^{m}$ for all $m \ge 1$ and $w \in W_m$ and complete the proof. 
\end{proof}

Now, define a subspace $\mcH_{p}^{\star}$ of $\mcF_p$ by setting
\begin{equation}\label{dfn.core}
    \mcH_{p}^{\star} \coloneqq \left\{ \sum_{w \in A}a_{w}\varphi_{w} \;\middle|\; \text{$A$ is a finite subset of $W_{\#}$, $a_{w} \in \bbR$ for each $w \in A$} \right\},
\end{equation}
where $\{ \varphi_{w} \}_{w \in W_{\#}}$ is a family of functions in $\mcF_p$ appeared in Lemma \ref{lem.corebase}.
Then we achieve the regularity of $\mcF_p$ (see also \cite{Kig21+}*{Lemma 6.19}).

\begin{thm}\label{thm.core}
    The space $\mcH_{p}^{\star}$ is dense in $\mcC(K)$ with respect to the sup norm.
    In particular, $\mcF_p$ is dense in $\mcC(K)$.
\end{thm}
\begin{proof}
    Let $f \in \mcC(K)$ and define $f_n$ by setting $f_n \coloneqq \sum_{w \in W_n}M_{n}f(w)\varphi_w \in \mcH_{p}^{\star}$.
    Then
    \begin{align*}
      \abs{f(x) - f_n(x)}
      &\le L_{\ast}^{2}\max_{w \in W_n; x \in \supp[\varphi_w]}\abs{f(x) - M_{n}f(w)} \\
      &\le L_{\ast}^{2}\adjustlimits\max_{w \in W_n}\sup_{x \in U_{1}^{(n)}(w), y \in K_w}\abs{f(x) - f(y)}.
    \end{align*}
    Since $\max_{w \in W_n}\diam\Bigl(U_{1}^{(n)}(w), d\Bigr) \le 2a^{-n} \to 0$ as $n \to \infty$ and $f$ is uniformly continuous, we have $\norm{f - f_n}_{\mcC(K)} \to 0$. 
\end{proof}

\subsection{Separability}\label{subsec:separable}

In this subsection, we prove that $\mcF_{p}$ is separable with respect to $\norm{\,\cdot\,}_{\mcF_{p}}$.
In the case $p = 2$, this is done by applying easy functional analytic arguments since the polarization formula of $\mcE_{2}$ yields a non-negative definite closed quadratic form.
For example, by Proposition \ref{prop.cptembed}, the inclusion map from $\mcF_{2}$ to $L^{2}(K, \mu)$ is a compact operator, and thus there exists a countable complete orthonormal system of $\mcF_{2}$ (see \cite{Dav.ST}*{Exercise 4.2 and Corollary 4.2.3} for example).
One can also give a short proof of the separability of $\mcF_2$ using resolvents (see \cite{FOT}*{proof of Theorem 1.4.2-$\textrm{(i\hspace{-.08em}i\hspace{-.08em}i)}$} for example).
However, it is hopeless to execute similar arguments in our setting.

To overcome this difficulty, we directly show that the space $\mcH_{p}^{\star}$ defined in \eqref{dfn.core} is dense in $\mcF_{p}$ and hence $\bbQ$-hull of $\{ \varphi_{w} \}_{w \in W_{\#}}$ is also dense.
Our strategy is standard in calculus of variations, namely, we extract a strong convergent approximation from $\mcH_{p}^{\star}$ by using Mazur's lemma.
To this end, it will be a key ingredient to ensure the reflexivity of $\mcF_{p}$, which is deduced from a combination of \emph{Clarkson's inequality} and the Milman--Pettis theorem.
We will derive Clarkson's inequality by using $\Gamma$-convergence to find a norm of $\mcF_{p}$ having the required properties.

We start by recalling Clarkson's inequality.

\begin{dfn}[Clarkson's inequality]\label{dfn.clarkson}
    Let $\mathcal{X}$ be a vector space (over $\mathbb{R}$) and let $\norm{\,\cdot\,}$ be a semi-norm on $\mathcal{X}$.
    The semi-norm $\norm{\,\cdot\,}$ satisfies \emph{Clarkson's inequality} if and only if one of the following holds:
    \begin{itemize}
        \item [(1)] There exists $p \in (1, 2]$ such that for every $x, y \in \mcX$, 
        \[
        \norm{x + y}^{\frac{p}{p - 1}} + \norm{x - y}^{\frac{p}{p - 1}} \le 2\bigl(\norm{x}^{p} + \norm{y}^{p}\bigr)^{\frac{1}{p - 1}};
        \]
        \item [(2)] There exists $p \in [2, \infty)$ such that for every $x, y \in \mcX$, 
        \[
        \norm{x + y}^{p} + \norm{x - y}^{p} \le 2^{p - 1}\bigl(\norm{x}^{p} + \norm{y}^{p}\bigr).
        \]
    \end{itemize}
\end{dfn}

\noindent
It is well-known that $L^{p}$-norm on a measurable space satisfies Clarkson's inequality, and that a normed space satisfying Clarkson's inequality is \emph{uniformly convex}.
Recall that Milman--Pettis theorem says that any Banach spaces that possess a uniformly convex norm is reflexive. 

Next, let us recall the definition of $\Gamma$-convergence and its basic properties.
The reader is referred to \cite{Dal} for details on $\Gamma$-convergence.

\begin{dfn}[$\Gamma$-convergence]\label{dfn.gconv}
    Let $\{ \Phi_n \}_{n \ge 1}$ be a sequence of $[-\infty, \infty]$-valued functional on $L^{p}(K, \mu)$.
    We say that a functional $\Phi\colon L^{p}(K, \mu) \to [-\infty,\infty]$ is a $\Gamma$-limit of $\{ \Phi_n \}_{n \ge 1}$ as $n \to \infty$ if the following two inequalities hold;
    \begin{itemize}
        \item [(1)]\textup{(liminf inequality)} If $f_n \to f$ in $L^{p}$, then $\Phi(f) \le \varliminf_{n \to \infty}\Phi_n(f_n)$.
        \item [(2)]\textup{(limsup inequality)} For any $f \in L^{p}(K, \mu)$, there exists a sequence $\{ f_n \}_{n \ge 1}$ such that
        \begin{equation}\label{eq.liminf}
            \text{$f_n \to f$ in $L^{p}$ and $\displaystyle\varlimsup_{n \to \infty}\Phi_n(f_n) \le \Phi(f)$.}
        \end{equation}
        A sequence $\{ f_n \}_{n \ge 1}$ satisfying \eqref{eq.liminf} is called a \emph{recovery sequence} of $f$.
    \end{itemize}
\end{dfn}

Since $(K, d)$ is a separable metric space, the following fact holds.

\begin{thm}[\cite{Dal}*{Theorem 8.5}]\label{thm.gamma}
    Let $\{ \Phi_{n} \}_{n \ge 1}$ be a sequence of functionals on $L^{p}(K, \mu)$.
    Then there exists a subsequence $\{ n_{k} \}_{k \ge 1}$ and a functional $\Phi$ on $L^{p}(K, \mu)$ such that $\Phi$ is a $\Gamma$-limit of $\{ \Phi_{n_{k}} \}_{k \ge 1}$.
\end{thm}

Now, we regard $\widetilde{\mcE}_{p}^{G_n}(\,\cdot\,)$ as a $[0,\infty]$-valued functional on $L^{p}(K, \mu)$ defined by $f \mapsto \widetilde{\mcE}_{p}^{G_n}(M_{n}f)$.
Then, by Theorem \ref{thm.gamma}, there exists a $\Gamma$-convergent subsequence $\Bigl\{ \widetilde{\mcE}_{p}^{\,G_{n_k}}(\,\cdot\,) \Bigr\}_{k \ge 1}$ and we write $\msfE_{p}(\,\cdot\,)$ to denote its $\Gamma$-limit.
We define $\trinorm{\,\cdot\,}_{\mcF_p} \coloneqq \bigl(\norm{\,\cdot\,}_{L^{p}}^{p} + \msfE_{p}(\,\cdot\,)\bigr)^{1/p}$.
This new ``norm'' $\trinorm{\,\cdot\,}_{\mcF_p}$ establishes the reflexivity.
(We also need to show that $\trinorm{\,\cdot\,}_{\mcF_p}$ is a norm.)

\begin{thm}\label{thm.newnorm}
    The norm $\trinorm{\,\cdot\,}_{\mcF_p}$ is equivalent to $\norm{\,\cdot\,}_{\mcF_p}$ and satisfies Clarkson's inequality.
    In particular, the Banach space $\mcF_p$ is reflexive.
\end{thm}
\begin{proof}
    Let $f, g \in L^{p}(K, \mu)$ and let $\{ f_n \}_{n \ge 1}$, $\{ g_n \}_{n \ge 1}$ be their recovery sequences throughout the proof.
    To verify the triangle inequality of $\trinorm{\,\cdot\,}_{\mcF_p}$, define
    \[
    \norm{f}_{p, n} \coloneqq \Bigl(\norm{f}_{L^p}^{p} + \widetilde{\mcE}_{p}^{G_n}(M_{n}f)\Bigr)^{1/p}.
    \]
    Note that the $\Gamma$-limit of $\bigl\{ \norm{\,\cdot\,}_{p, n_k} \bigr\}_{k \ge 1}$ coincides with $\trinorm{\,\cdot\,}_{\mcF_p}$ and that the norm $\norm{\,\cdot\,}_{p, n}$ can be regarded as a suitable $L^p$-norm on $K \sqcup E_{n}$, where $\sqcup$ denotes the disjoint union.
    Using the triangle inequality of $\norm{\,\cdot\,}_{p, n}$, we see that
    \begin{align*}
        \trinorm{f + g}_{\mcF_{p}}
        \le \varliminf_{k \to \infty}\norm{f_{n_k} + g_{n_k}}_{p, n_k}
        &\le \varlimsup_{k \to \infty}\norm{f_{n_k}}_{p, n_k} + \varlimsup_{k \to \infty}\norm{g_{n_k}}_{p, n_k} \\
        &\le \trinorm{f}_{\mcF_{p}}  + \trinorm{g}_{\mcF_{p}},
    \end{align*}
    and thus $\trinorm{\,\cdot\,}_{\mcF_p}$ is an extended norm on $L^{p}(K, \mu)$ (we admit $\trinorm{f}_{\mcF_p} = \infty$).

    Next, we prove $C_{\text{WM}}^{-1}\abs{f}_{\mcF_p}^{p} \le \msfE_{p}(f) \le \abs{f}_{\mcF_p}^{p}$ for every $f \in L^{p}(K, \mu)$ to conclude that $\trinorm{\,\cdot\,}_{\mcF_p}$ and $\norm{\,\cdot\,}_{\mcF_p}$ are equivalent.
    From the liminf inequality, we immediately have that $\msfE_{p}(f) \le \abs{f}_{\mcF_{p}}^{p}$ for $f \in L^{p}(K, \mu)$.
    To prove the converse, note that $M_{n}f_{n_k}(w) \to M_{n}f(w)$ for any $w \in W_{n}$ as $k \to \infty$ by the dominated convergence theorem.
    By the weak monotonicity (Corollary \ref{cor.wm}), we obtain
    \[
    \widetilde{\mcE}_{p}^{G_n}(M_{n}f) \le C_{\text{WM}}\varliminf_{k \to \infty}\widetilde{\mcE}_{p}^{G_{n_k}}(M_{n_k}f_{n_k}) \le C_{\text{WM}}\msfE_{p}(f)
    \]
    for all $n \ge 1$.
    We thus conclude that $\abs{f}_{\mcF_p}^{p} \le C_{\text{WM}}\msfE_{p}(f)$.

    The rest of the proof is mainly devoted to Clarkson's inequalities.
    Let $p \le 2$.
    Recall that $\norm{\,\cdot\,}_{p, n}$ can be regarded a suitable $L^{p}$-norm on $K \sqcup E_{n}$. 
    By Clarkson's inequality for $\norm{\,\cdot\,}_{p, n}$, we have
    \[
    \norm{f + g}_{p, n}^{\frac{p}{p - 1}} + \norm{f - g}_{p, n}^{\frac{p}{p - 1}} \le 2\bigl(\norm{f}_{p, n}^{p} + \norm{g}_{p, n}^{p}\bigr)^{\frac{1}{p - 1}}.
    \]
    Thus, we see that
    \begin{align*}
        \trinorm{f + g}_{\mcF_p}^{\frac{p}{p - 1}} + \trinorm{f - g}_{\mcF_p}^{\frac{p}{p - 1}}
        &\le \varliminf_{k \to \infty}\norm{f_{n_k} + g_{n_k}}_{p, n_k}^{\frac{p}{p - 1}} + \varliminf_{k \to \infty}\norm{f_{n_k} - g_{n_k}}_{p, n_k}^{\frac{p}{p - 1}} \\
        &\le \varlimsup_{k \to \infty}\bigl(\norm{f_{n_k} + g_{n_k}}_{p, n_k}^{\frac{p}{p - 1}} + \norm{f_{n_k} - g_{n_k}}_{p, n_k}^{\frac{p}{p - 1}} \bigr) \\
        &\le 2\varlimsup_{k \to \infty}\bigl(\norm{f_{n_k}}_{p, n_k}^{p} + \norm{g_{n_k}}_{p, n_k}^{p}\bigr)^{\frac{1}{p - 1}} \\
        &\le 2\Bigl(\varlimsup_{k \to \infty}\norm{f_{n_k}}_{p, n_k}^{p} + \varlimsup_{k \to \infty}\norm{g_{n_k}}_{p, n_k}^{p}\Bigr)^{\frac{1}{p - 1}} \\
        &\le 2\bigl(\trinorm{f}_{\mcF_p}^{p} + \trinorm{g}_{\mcF_p}^{p}\bigr)^{\frac{1}{p - 1}},
    \end{align*}
    which is Clarkson's inequality of $\trinorm{\,\cdot\,}_{\mcF_{p}}$ when $p \le 2$.
    Similarly, we get Clarkson's inequality for $p \ge 2$.

    Consequently, we get a new norm $\trinorm{\,\cdot\,}_{\mcF_p}$ of $\mcF_{p}$ satisfying Clarkson's inequality.
    Thus, we see that the Banach space $(\mcF_{p}, \trinorm{\,\cdot\,}_{\mcF_{p}})$ is uniformly convex (see \cite{Brezis}*{proof of Theorem 4.10 and its remark} for example).
    Therefore, we finish the proof by the Milman--Pettis theorem (see \cite{Brezis}*{Theorem 3.31} for example).
\end{proof}

\begin{thm}\label{thm.sep}
    The space $\mcH_{p}^{\star}$ defined in \eqref{dfn.core} is dense in $\mcF_p$.
    Furthermore, $\mcF_p$ is separable.
\end{thm}
\begin{proof}
    Recall the definition of $\widetilde{\varphi}_{w}^{\,(n)}$ in the construction of $\varphi_{w}$ (see the proof of Lemma \ref{lem.corebase}).
    By the diagonal procedure, we can pick a subsequence $\{ n_k \}_{k \ge 1}$ such that $\bigl\{ \widetilde{\varphi}_{w}^{\,(n_k)} \bigr\}_{k \ge 1}$ converges to $\varphi_w$ with respect to the supremum norm for all $w \in W_{\#}$.
    Next, for $f \in \mcF_p$, we define $f_n$ and $f_{n}^{(k)}$ by setting
    \[
    f_n \coloneqq \sum_{w \in W_n}M_{n}f(w)\varphi_{w}, \quad f_{n}^{(k)} \coloneqq \sum_{w \in W_n}M_{n}f(w)\widetilde{\varphi}_{w}^{\,(n_k)}.
    \]
    Similarly to the proof of Theorem \ref{thm.core}, we see that $\{ f_n \}_{n \ge 1}$ converges to $f$ with respect to the supremum norm.
    Also, by Lemma \ref{lem.unity}, we obtain $\widetilde{\mcE}_{p}^{G_{n + n_k}}\bigl(M_{n + n_k}f_{n}^{(k)}\bigr) \le c_1\abs{f}_{\mcF_p}^{p}$, where $c_1 > 0$ depends only on $p, a, D, N_{\ast}, L_{\ast}$.
    Thus, by the weak monotonicity (Corollary \ref{cor.wm}), it holds that $\widetilde{\mcE}_{p}^{G_l}\bigl(M_{l}f_{n}^{(k)}\bigr) \le c_{1}C_{\textup{WM}}\abs{f}_{\mcF_p}^{p}$ whenever $l \le n + n_k$.
    Letting $k \to \infty$, we see that $\{ f _{n} \}_{n \ge 1}$ is bounded in $\mcF_p$.
    Since $\mcF_p$ is reflexive, we may assume that a subsequence $\{ f_{n_k} \}_{k \ge 1}$ converges to $f$ with respect to the weak topology of $\mcF_p$.
    By Mazur's lemma (see \cite{Brezis}*{Corollary 3.8} for example), we can find a sequence $\{ g_l \}_{l \ge 1}$ such that each $g_l$ is a convex combination of $\{ f_{n_k} \}_{k \ge 1}$ and $g_l \to f$ in $\mcF_p$ as $l \to \infty$.
    In particular, we obtain $\overline{\mcH_{p}^{\star}}^{\norm{\,\cdot\,}_{\mcF_{p}}} = \mcF_{p}$.
    Clearly, the $\bbQ$-hull of $\{ \varphi_{w} \}_{w \in W_{\#}}$ also gives this approximation, that is,
    \[
    \overline{\Biggl\{ \sum_{w \in A}a_{w}\varphi_{w} \Biggm| \text{$A$ is a finite subset of $W_{\#}$ and $a_{w} \in \bbQ \, (w \in A)$} \Biggr\}}^{\norm{\,\cdot\,}_{\mcF_{p}}} = \overline{\mcH_{p}^{\star}}^{\norm{\,\cdot\,}_{\mcF_{p}}}.
    \]
    Therefore, $\mcF_p$ is separable.
\end{proof}

\subsection{Lipschitz--Besov type expression}\label{subsec:LBtype}
This subsection is devoted to proving Theorem \ref{thm.main2}.
Let us start by introducing the definition of Lipschitz--Besov spaces on the underlying generalized Sierpi\'{n}ski carpet $(K, d, \mu)$ in some specific cases (see \cite{Bod09} for example).

\begin{dfn}\label{dfn.LB}
    For $s \in (0, \infty)$ and $p \in [1, \infty)$, the \emph{Lipschitz--Besov space} $\Lambda_{p, \infty}^{s}$ is defined as
    \[
    \Lambda_{p, \infty}^{s} \coloneqq \bigl\{ f \in L^{p}(K, \mu) \bigm| \abs{f}_{\Lambda_{p, \infty}^{s}} < \infty \bigr\},
    \]
    where
    \[
    \abs{f}_{\Lambda_{p, \infty}^{s}} \coloneqq \sup_{n \in \bbN}\left(\int_{K}\mint_{B_{d}(x, a^{-n})}\frac{\abs{f(x) - f(y)}^{p}}{a^{-nsp}}\,d\mu(y)d\mu(x)\right)^{1/p}.
    \]
    We also define its norm $\norm{\,\cdot\,}_{\Lambda_{p, \infty}^{s}}$ by setting $\norm{f}_{\Lambda_{p, \infty}^{s}} \coloneqq \norm{f}_{L^{p}} + \abs{f}_{\Lambda_{p, \infty}^{s}}$.
\end{dfn}
Then $(\Lambda_{p, \infty}^{s}, \norm{\,\cdot\,}_{\Lambda_{p, \infty}^{s}})$ is a Banach space.
Furthermore, for any $c \in [1, \infty)$ there exists a positive constant $C_{\text{LB}}(c)$, which depends only on $c, a, N_{\ast}, C_{\text{AR}}$, such that, for any $f \in L^{p}(K, \mu)$,
\begin{equation}\label{LBequiv2}
    \abs{f}_{\Lambda_{p, \infty}^{s}}^{p} \le C_{\text{LB}}(c)\sup_{n \in \bbN}\int_{K}\mint_{B_{d}(x, ca^{-n})}\frac{\abs{f(x) - f(y)}^{p}}{a^{-nsp}}\,d\mu(y)d\mu(x),
\end{equation}
and
\begin{align}\label{LBequiv}
    &C_{\text{LB}}(c)^{-1}\varlimsup_{n \to \infty}\int_{K}\mint_{B_{d}(x, ca^{-n})}\frac{\abs{f(x) - f(y)}^{p}}{a^{-nsp}}\,d\mu(y)d\mu(x) \\
    &\hspace*{25pt} \le \varlimsup_{r \downarrow 0}\int_{K}\mint_{B_{d}(x, r)}\frac{\abs{f(x) - f(y)}^{p}}{r^{sp}}\,d\mu(y)d\mu(x) \nonumber \\
    &\hspace*{50pt} \le C_{\text{LB}}(c)\varlimsup_{n \to \infty}\int_{K}\mint_{B_{d}(x, ca^{-n})}\frac{\abs{f(x) - f(y)}^{p}}{a^{-nsp}}\,d\mu(y)d\mu(x). \nonumber
\end{align}

First, we prove a $(p,p)$-Poincar\'{e} inequality in the sense of Kumagai and Sturm (see \cite{KS05}*{pp. 315}).
Recall that $\alpha = \log{N_{\ast}}/\log{a}$ denotes the Hausdorff dimension of $(K, d)$, and that $\beta_{p} = \log{(N_{\ast}\rho_{p})}/\log{a}$, where $\rho_{p}$ is the resistance scaling factor (see Theorem \ref{thm.sub-mult} and \eqref{Poi-super}).

\begin{lem}\label{lem.PI}
    There exists a positive constant $C_{\textup{PI-KS}}$ (depending only on $p, a, D$, $L_{\ast}$, $N_{\ast}$, $\rho_{p}, C_{\textup{AD}}$) such that
    \begin{equation}\label{PI}
        a^{\beta_{p} n}\sum_{w \in A}\int_{K_{w}}\abs{f(x) - M_{n}f(w)}^{p}\,d\mu(x)
        \le C_{\textup{PI-KS}}\,\varliminf_{l \to \infty}\widetilde{\mcE}_{p, A \cdot W_{l}}^{G_{l + n}}(M_{l + n}f),
    \end{equation}
    for every $n \in \bbN$, $f \in \mcF_{p}$ and subset $A \subseteq W_{n}$.
    In particular, 
    \begin{equation*}
        a^{\beta_{p} n}\sum_{w \in W_{n}}\int_{K_{w}}\abs{f(x) - M_{n}f(w)}^{p}\,d\mu(x)
        \le C_{\textup{PI-KS}}\,\abs{f}_{\mcF_{p}}^{p}.
    \end{equation*}
\end{lem}
\begin{proof}
    Let $f \in \mcF_{p}$ be the continuous version.
    Then, by the mean value theorem, for any $n \in \bbN$ and $w \in W_{n}$ there exists $x_{w} \in K_{w}$ such that $f(x_{w}) = M_{n}f(w)$.
    From the H\"{o}lder estimate (Theorem \ref{thm.embed}), we have that, for any $x \in K_{w}$,
    \begin{align*}
        \abs{f(x) - M_{n}f(w)}^{p}
        &= \abs{F_{w}^{\ast}f(F_{w}^{-1}(x)) - F_{w}^{\ast}f(F_{w}^{-1}(x_{w}))}^{p} \\
        &\le C_{\text{H\"{o}l}}\abs{F_{w}^{\ast}f}_{\mcF_{p}}^{p}d(F_{w}^{-1}(x), F_{w}^{-1}(x_{w}))^{\beta_{p} - \alpha} \\
        &\le C_{\text{H\"{o}l}}\diam(K, d)^{\beta_{p} - \alpha}\abs{F_{w}^{\ast}f}_{\mcF_{p}}^{p}.
    \end{align*}
    Consequently, we obtain
    \[
    \int_{K_{w}}\abs{f(x) - M_{n}f(w)}^{p}\,d\mu(x) \le C_{\text{H\"{o}l}}a^{-\alpha n}\abs{F_{w}^{\ast}f}_{\mcF_{p}}^{p}.
    \]
    Summing over $w \in W_{n}$, we conclude that
    \begin{align*}
        \sum_{w \in A}\int_{K_{w}}\abs{f(x) - M_{n}f(w)}^{p}\,d\mu(x) 
        &\le C_{\text{H\"{o}l}}\,a^{-\alpha n}\sum_{w \in A}\abs{F_{w}^{\ast}f}_{\mcF_{p}}^{p} \\
        &\le C_{\text{H\"{o}l}}C_{\text{WM}}\,a^{-\alpha n}\sum_{w \in A}\varliminf_{l \to \infty}\widetilde{\mcE}_{p}^{G_{l}}\bigl(M_{l}(F_{w}^{\ast}f)\bigr) \\
        &\le C_{\text{H\"{o}l}}C_{\text{WM}}\,a^{-\alpha n}\varliminf_{l \to \infty}\sum_{w \in A}\widetilde{\mcE}_{p}^{G_{l}}\bigl(M_{l}(F_{w}^{\ast}f)\bigr),
    \end{align*}
    where we used the weak monotonicity (Corollary \ref{cor.wm}) in the second line.
    From \eqref{eq.commute}, we see that $\sum_{w \in A}\widetilde{\mcE}_{p}^{G_{l}}\bigl(M_{l}(F_{w}^{\ast}f)\bigr) \le \rho_{p}^{-n}\widetilde{\mcE}_{p, A \cdot W_{l}}^{G_{l + n}}(M_{l + n}f)$.
    In particular, 
    \[
    a^{-\alpha n}\varliminf_{l \to \infty}\sum_{w \in A}\widetilde{\mcE}_{p}^{G_{l}}\bigl(M_{l}(F_{w}^{\ast}f)\bigr) \le a^{-\beta_{p}n}\varliminf_{l \to \infty}\widetilde{\mcE}_{p, A \cdot W_{l}}^{G_{l + n}}(M_{l + n}f),
    \]
    which proves \eqref{PI}.
\end{proof}

Next, we give an extension of \cite{GHL03}*{Theorem 4.11-$\textrm{(i\hspace{-.08em}i\hspace{-.08em}i)}$}.
This is essentially proved in \cite{AB21}*{Theorem 5.1}, so its proof is omitted here.
Since the function space $\mathbf{B}^{p, \alpha}$ in \cite{AB21}*{Theorem 5.1} is defined using heat kernels, we give a direct proof of the lemma in Appendix \ref{sec:LBembed} for reader's convenience.

\begin{lem}\label{lem.LBembed}
    Let $\beta > \alpha$ and $p > 1$.
    Then there exists a positive constant $C_{\ref{lem.LBembed}}$ (depending only on $p, \beta, a, N_{\ast}, C_{\textup{AR}}$) such that
    \begin{align*}
        &\abs{f(x) - f(y)}^{p} \\
        &\quad\le C_{\ref{lem.LBembed}}d(x, y)^{\beta - \alpha}\sup_{r \in (0, 3d(x,y)]}r^{-\beta}\int_{K}\mint_{B_{d}(z, r)}\abs{f(z) - f(z')}^{p}\,d\mu(z')d\mu(z),
    \end{align*}
    for every $f \in \Lambda_{p, \infty}^{\beta/p}$ and $\mu$-a.e. $x, y \in K$.
    In particular, for any $g \in \mcC(K)$ and $x, y \in K$,
    \begin{equation*}
        \abs{g(x) - g(y)}^{p} \le C_{\ref{lem.LBembed}}\abs{g}_{\Lambda_{p, \infty}^{\beta/p}}^{p}d(x, y)^{\beta - \alpha}.
    \end{equation*}
\end{lem}

An important consequence of the above lemma is the following type ``$(p,p)$-Poincar\'{e} inequality''.

\begin{lem}\label{lem.LBPI}
    Let $\beta > \alpha$ and $p > 1$.
    Then there exists a positive constant $C_{\textup{PI-LB}}$ (depending only on $p, \beta, a, D, N_{\ast}, C_{\textup{AR}}$) such that
    \begin{align}\label{eq.LBPI}
        &a^{\beta n}\sum_{w \in W_{n}}\int_{K_{w}}\abs{f(x) - M_{n}f(w)}^{p}\,d\mu(x) \\
        &\le C_{\textup{PI-LB}}\,\sup_{r \in (0, 3a^{-n}]}r^{-\beta}\int_{K}\mint_{B_{d}(x, r)}\abs{f(x) - f(y)}^{p}\,d\mu(y)d\mu(x), \nonumber
    \end{align}
    for every $n \in \bbN$ and $f \in \mcC(K)$.
\end{lem}
\begin{proof}
    We adopt a method in \cite{GY19}*{proof of Theorem 3.5} and generalize it to fit our context.
    Let $f \in \Lambda_{p, \infty}^{\beta/p}$, let $w = w_{1} \cdots w_{n} \in W_{n}$ and fix $k \in \bbN$ that we choose later.
    Then, by the mean value theorem, there exists $x_{w} \in K_{w}$ such that $M_{n}f(w) = f(x_{w})$.
    Let $\omega \in \pi^{-1}(\{ x_{w} \})$ such that $[\omega]_{n} = w$.
    For each $m \in \bbZ_{\ge 0}$, we define $w(m) \coloneqq [\omega]_{n + km} \in W_{n + km}$.
    Then, for $z_{m} \in K_{w(m)} \, (m = 0, \dots, n)$,
    \begin{align}\label{ineq.refine}
        &\abs{f(x_{w}) - f(z_{0})}^{p} \\
        &\le 2^{p - 1}\abs{f(x_{w}) - f(z_{n})}^{p} + 2^{p - 1}\sum_{i = 0}^{n - 1}2^{i(p - 1)}\abs{f(z_{i}) - f(z_{i + 1})}^{p}. \nonumber
    \end{align}
    Integrating \eqref{ineq.refine}, we obtain
    \begin{align}\label{est.iterateword}
        &\mint_{K_{w}}\abs{f(z) - M_{n}f(w)}^{p}\,d\mu(z) \\
        &\le 2^{p - 1}\mint_{K_{w(n)}}\abs{f(x_{w}) - f(z_{n})}^{p}\,d\mu(z_{n}) \nonumber \\
        &\hspace{15pt}+ 2^{2(p - 1)}\sum_{i = 0}^{n - 1}2^{i(p - 1)}\mint_{K_{w(i)}}\mint_{K_{w(i + 1)}}\abs{f(z_{i}) - f(z_{i + 1})}^{p}\,d\mu(z_{i + 1})d\mu(z_{i}). \nonumber
    \end{align}
    Set $c \coloneqq 3\diam(K, d) = 3$ and define
    \begin{equation}\label{eq.LBlocal}
        S_{p,\beta}^{(n)}(f) \coloneqq \sup_{r \in (0, ca^{-n}]}r^{-\beta}\int_{K}\mint_{B_{d}(z, r)}\abs{f(z) - f(z')}^{p}\,d\mu(z')d\mu(z).
    \end{equation}
    By Lemma \ref{lem.LBembed}, the first term of the right-hand side of \eqref{est.iterateword} has a bound:
    \begin{align}\label{ineq.LBPI1word}
        &\mint_{K_{w(n)}}\abs{f(x_{w}) - f(z_{n})}^{p}\,d\mu(z_{n}) \\
        &\le C_{\ref{lem.LBembed}}S_{p,\beta}^{(n)}(f)\,\mint_{K_{w(n)}}d(x_{w}, z_{n})^{\beta - \alpha}\,d\mu(z_{n}) \nonumber \\
        &\le \bigl(C_{\ref{lem.LBembed}}\diam(K, d)^{\beta - \alpha}\bigr)a^{-(n + kn)(\beta - \alpha)}S_{p,\beta}^{(n)}(f). \nonumber
    \end{align}
    For the second term of \eqref{est.iterateword}, we see that
    \begin{align}\label{ineq.LBPI2word}
        &2^{i(p - 1)}\mint_{K_{w(i)}}\mint_{K_{w(i + 1)}}\abs{f(z_{i}) - f(z_{i + 1})}^{p}\,d\mu(z_{i + 1})d\mu(z_{i}) \\
        &\le c_{1}\,2^{i(p - 1)}a^{\alpha k}a^{2\alpha(n + ki)}I_{i}(f), \nonumber 
    \end{align}
    where 
    \[
    I_{i}(f) \coloneqq \int_{K_{w(i)}}\int_{B_{d}(z_{i}, ca^{-(n + ki)})}\abs{f(z_{i}) - f(z_{i + 1})}^{p}\,d\mu(z_{i + 1})d\mu(z_{i})
    \]
    and $c_1 > 0$ depends only on $c$ and $C_{\text{AR}}$.

    Now, we consider $k \in \bbN$ large enough so that
    \begin{equation}\label{eq.k1word}
        k(\beta - \alpha) \ge \alpha \quad \text{and} \quad N_{\ast}a^{-(\beta - \alpha)k} \vee 2^{p - 1}a^{-(\beta - \alpha)k} < 1.
    \end{equation}
    Then, by summing \eqref{ineq.LBPI1word} and \eqref{ineq.LBPI2word} over $w \in W_{n}$, we have from \eqref{eq.k1word} that
    \begin{align*}
        \sum_{w \in W_{n}}\,\mint_{K_{w(n)}}\abs{f(x_w) - f(z_{n})}^{p}\,d\mu(z_{n}) 
        &\le C_{\ref{lem.LBembed}}S_{p,\beta}^{(n)}(f)N_{\ast}^{n}a^{-(n + kn)(\beta - \alpha)} \\
        &\le C_{\ref{lem.LBembed}}S_{p,\beta}^{(n)}(f)\,a^{-(\beta - \alpha)n},
    \end{align*}
    and from Proposition \ref{prop.BF}-(3) and \eqref{eq.k1word} that
    \begin{align*}
        \sum_{i = 0}^{n - 1}&\,\sum_{w \in W_{n}}2^{i(p - 1)}\mint_{K_{w(i)}}\mint_{K_{w(i + 1)}}\abs{f(z_{i}) - f(z_{i + 1})}^{p}\,d\mu(z_{i + 1})d\mu(z_{i}) \\
        &\le c_{1}a^{\alpha k}\sum_{i = 0}^{n - 1}2^{i(p - 1)}a^{2\alpha(n + ki)}\int_{K}\int_{B(x, ca^{-(n + ki)})}\abs{f(x) - f(y)}^{p}\,d\mu(y)d\mu(x) \\
        &\le c_{2}\,a^{\alpha k}S_{p,\beta}^{(n)}(f)\left(\,\sum_{i = 0}^{n - 1}2^{i(p - 1)}a^{-(\beta - \alpha)(n + ki)}\,\right) \\
        &\le c_{2}\,a^{\alpha k}S_{p,\beta}^{(n)}(f)\left(\,\sum_{i = 0}^{\infty}\Bigl(2^{p - 1}a^{-(\beta - \alpha)k}\Bigr)^{i}\,\right)a^{-(\beta - \alpha)n}
        \eqqcolon c_{3}S_{p,\beta}^{(n)}(f)\,a^{-(\beta - \alpha)n},
    \end{align*}
    where $c_{2}, c_{3} > 0$ depend only on $C_{\text{AR}}, c, \beta, a, N_{\ast}, p$.
    From these estimates and \eqref{est.iterateword}, we finish the proof.
\end{proof}

Now we are ready to finish the proof of Theorem \ref{thm.main2}.

\begin{thm}\label{thm.Besov-charact}
  \begin{equation*}\label{eq.Besov-charact}
      \mcF_{p} = \Lambda_{p, \infty}^{\beta_{p}/p} = \Biggl\{ f \in L^{p}(K, \mu) \Biggm| \varlimsup_{r \downarrow 0} \int_{K}\mint_{B_{d}(x, r)}\frac{\abs{f(x) - f(y)}^{p}}{r^{\beta_{p}}}\,d\mu(y)d\mu(x) < \infty \Biggr\}.
  \end{equation*}
\end{thm}
\begin{proof}
  Let $c > 0$ such that $\max_{(v, w) \in E_{n}}\sup_{x \in K_{v}, y \in K_{w}}d(x, y) < ca^{-n}$.
  We can choose such $c$ depending only on $C_{\text{AD}}$ by Lemma \ref{lem.1-adapted} and we can assume that $c \ge 3\diam(K, d)$ without loss of generality.
  Then $y \in B(x, ca^{-n})$ whenever $(v,w) \in E_{n}$ and $x \in K_{v}$, $y \in K_{w}$.
  Let $\beta > \alpha$ and set
  \[
  A_{p, \beta}^{(n)}(f) \coloneqq a^{\beta n}\int_{K}\mint_{B_{d}(x, ca^{-n})}\abs{f(x) - f(y)}^{p}\,d\mu(y)d\mu(x),
  \]
  and define $S_{p,\beta}^{(n)}(f)$ as in \eqref{eq.LBlocal} for each $n \in \bbN$ and $f \in L^{p}(K, \mu)$.
  Then, thanks to Lemma \ref{lem.PI}, it will suffice to show the following two estimates:
  \begin{align}\label{eq.PIKS-upper}
      &A_{p, \beta_p}^{(n)}(f) \\
      &\hspace*{10pt} \le c_{\ref{eq.PIKS-upper}}\biggl(\widetilde{\mcE}_{p}^{G_{n}}(M_{n}f) \nonumber \\
      &\hspace*{55pt} + a^{\beta_{p}n}\sum_{w \in W_{n}}\int_{K_{w}}\abs{f(x) - M_{n}f(w)}^{p}\,d\mu(x)\biggl), \quad f \in L^{p}(K, \mu), \nonumber
  \end{align}
  \begin{equation}\label{eq.PIKS-lower}
      a^{(\beta - \alpha)n}\cdot\mcE_{p}^{G_{n}}(M_{n}f) \le c_{\ref{eq.PIKS-lower}}S_{p,\beta}^{(n)}(f), \quad f \in \mcC(K),
  \end{equation}
  for some positive constants $c_{\ref{eq.PIKS-upper}}, c_{\ref{eq.PIKS-lower}}$ (without depending on $f$ and $n$).
  Indeed, by \eqref{LBequiv2}, \eqref{eq.PIKS-upper} and Lemma \ref{lem.PI}, we immediately see that
  \begin{equation}\label{eq.LBcomp1}
      C_{\textup{LB}}(c)^{-1}\abs{f}_{\Lambda_{p, \infty}^{\beta_p/p}}^{p} \le \sup_{n \in \bbN}A_{p, \beta_p}^{(n)}(f) \le c_{\ref{eq.PIKS-upper}}(1 + C_{\text{PI-KS}})\abs{f}_{\mcF_{p}}^{p}.
  \end{equation}
  Additionally, by the weak monotonicity (Corollary \ref{cor.wm}), \eqref{eq.PIKS-lower} and \eqref{LBequiv}, we have 
  \begin{align}\label{eq.LBcomp2}
      \abs{f}_{\mcF_p}^{p} &\le c_{\ref{eq.PIKS-lower}}C_{\text{WM}}\varlimsup_{r \downarrow 0}r^{-\beta_p}\int_{K}\mint_{B_{d}(x, r)}\abs{f(x) - f(y)}^{p}\,d\mu(y)d\mu(x) \\
      &\le c_{\ref{eq.PIKS-lower}}'\varlimsup_{n \to \infty}a^{\beta_p n}\int_{K}\mint_{B_{d}(x, a^{-n})}\abs{f(x) - f(y)}^{p}\,d\mu(y)d\mu(x), \nonumber
  \end{align}
  where $c_{\ref{eq.PIKS-lower}}' = c_{\ref{eq.PIKS-lower}}C_{\text{WM}}C_{\textup{LB}}(c)C_{\textup{LB}}(1)$. 
  Our assertion follows from \eqref{eq.LBcomp1} and \eqref{eq.LBcomp2}.

  The rest of the proof is devoted to proving \eqref{eq.PIKS-upper} and \eqref{eq.PIKS-lower}.
  First, we will prove \eqref{eq.PIKS-upper}.
  Let $x, y \in K$ with $d(x, y) < ca^{-n}$.
  Then, by the metric doubling property of $(K, d)$ (see \cite{Hei}*{pp. 81} for example), there exists a constant $L \ge 2$ depending only on $C_{\textup{AR}}$ such that, for any $w \in W_{n}$ with $x \in K_{w}$, we can choose $v \in W_{n}$ satisfying $d_{G_{n}}(v, w) \le L$ and $y \in K_{v}$.
  From this observation, we have that
  \begin{align}\label{eq.cover1}
      &A_{p, \beta_p}^{(n)}(f) \\
      &\le a^{\beta_{p}n}\sum_{\substack{w \in W_{n}, v \in W_{n};\\ d_{G_{n}}(v, w) \le L}}\int_{K_{w}}\frac{1}{\mu(B_{d}\bigl(x, ca^{-n})\bigr)}\int_{K_{v}}\abs{f(x) - f(y)}^{p}\,d\mu(y)d\mu(x). \nonumber 
  \end{align}
  To estimate the integral in \eqref{eq.cover1}, let $v, w \in W_{n}$ with $d_{G_{n}}(v, w) \le L$.
  Then we can pick a path $\bigl[w(0), w(1), \dots, w(L)\bigr]$ in $G_{n}$ from $w$ to $v$, that is, $w(i) \, (i = 0, \dots, L)$ satisfy $w(0) = w$, $w(L) = v$ and
  \[
  \text{$w(i - 1) = w(i)$ or $\bigl(w(i - 1), w(i)\bigr) \in E_{n}$ for each $i = 1, \dots, L$.}
  \]
  Let $x_{i} \in K_{w(i)}$ for each $i = 0, \dots, L$.
  Then H\"{o}lder's inequality implies that
  \[
  \abs{f(x_0) - f(x_L)}^{p} \le L^{p - 1}\sum_{i = 1}^{L}\abs{f(x_{i - 1}) - f(x_{i})}^{p}.
  \]
  Now, by integrating this, we deduce that
  \begin{align*}
      &\left(\prod_{i = 1}^{L - 1}\mu\bigl(K_{w(i)}\bigr)\right)\int_{K_{w}}\int_{K_{v}}\abs{f(x) - f(y)}^{p}\,d\mu(x)d\mu(y) \\
      &\le L^{p - 1}\sum_{i = 1}^{L}\frac{\prod_{j = 0}^{L}\mu\bigl(K_{w(j)}\bigr)}{\mu\bigl(K_{w(i - 1)}\bigr)\mu\bigl(K_{w(i)}\bigr)}\int_{K_{w(i - 1)}}\int_{K_{w(i)}}\abs{f(x_{i}) - f(x_{i - 1})}^{p}\,d\mu(x_{i})d\mu(x_{i - 1}).
  \end{align*}
  Since $\mu$ is the self-similar measure with weights $(N_{\ast}^{-1}, \dots, N_{\ast}^{-1})$, it is a simple computation that
  \[
  \frac{\prod_{j = 0}^{L}\mu\bigl(K_{w(j)}\bigr)}{\mu\bigl(K_{w(i - 1)}\bigr)\mu\bigl(K_{w(i)}\bigr)}\frac{1}{\prod_{i = 1}^{L - 1}\mu\bigl(K_{w(i)}\bigr)}
  = \frac{\mu(K_{v})\mu(K_{w})}{\mu\bigl(K_{w(i - 1)}\bigr)\mu\bigl(K_{w(i)}\bigr)} = 1.
  \]
  Furthermore, the Ahlfors regularity of $\mu$ (more precisely, the volume doubling property of $\mu$) implies that there exists a constant $c_1 > 0$ depending only on $C_{\text{AR}}, a, N_{\ast}, c$ such that $\mu(K_{z}) \le c_{1}\mu(B_{d}(x, ca^{-n}))$ for any $n \in \bbN$, $z \in W_{n}$ and $x \in K$.
  Thus, it follows from \eqref{eq.cover1} that
  \begin{align*}
      &A_{p, \beta_p}^{(n)}(f) \\
      &\le \bigl(c_{1}L^{p - 1}L_{\ast}^{L}\bigr)a^{\beta_{p}n}\sum_{(v, w) \in E_{n}}\int_{K_{w}}\mint_{K_{v}}\abs{f(x) - f(y)}^{p}\,d\mu(y)d\mu(x) \\
      &\le c_{\ref{eq.PIKS-upper}}\Biggl(a^{\beta_{p}n}\sum_{v \in W_{n}}\int_{K_{v}}\abs{f(x) - M_{n}f(v)}^{p}\,d\mu(x) \\
      &\hspace{90pt}+ \rho_{p}^{n}\sum_{(v,w) \in E_{n}}\abs{M_{n}f(v) - M_{n}f(w)}^{p}\Biggr),
  \end{align*}
  where $c_{\ref{eq.PIKS-upper}} \coloneqq c_{1}(2L)^{p - 1}L_{\ast}^{L + 1}$.
  This proves \eqref{eq.PIKS-upper}.

  Next let us prove \eqref{eq.PIKS-lower}.
  Let $\beta > \alpha$, let $p > 1$ and let $f \in \mcC(K)$.
  For $n \in \bbN$, $(v, w) \in E_{n}$, $x \in K_{v}$ and $y \in K_{w}$, we see that
  \begin{align*}
      &\abs{M_{n}f(v) - M_{n}f(w)}^{p} \\
      &\le 3^{p - 1}\bigl(\abs{M_{n}f(v) - f(x)}^{p} + \abs{f(x) - f(y)}^{p} + \abs{M_{n}f(w) - f(y)}^{p}\bigr).
  \end{align*}
  Integrating this over $K_{v}$ and $K_{w}$, we obtain
  \begin{align*}
      &a^{(\beta - \alpha)n}\cdot\abs{M_{n}f(v) - M_{n}f(w)}^{p} \\
      &\le 3^{p - 1}\Biggl(\;a^{\beta n}\int_{K_{v}}\abs{M_{n}f(v) - f(x)}^{p}\,d\mu(x)  \\
      &\hspace{10pt}+ a^{\beta n}\int_{K_{w}}\mint_{K_{v}}\abs{f(x) - f(y)}^{p}\,d\mu(x)d\mu(y) + a^{\beta n}\int_{K_{w}}\abs{M_{n}f(w) - f(y)}^{p}\,d\mu(y)\Biggr).
  \end{align*}
  Summing over $(v, w) \in E_{n}$, we obtain
  \begin{align*}
      &a^{(\beta - \alpha)n}\cdot\mcE_{p}^{G_{n}}(M_{n}f) \\
      &\le 2\cdot3^{p - 1}L_{\ast}\Biggl(a^{\beta n}\sum_{v \in W_{n}}\int_{K_{v}}\abs{M_{n}f(v) - f(x)}^{p}\,d\mu(x) \\
      &\hspace{125pt}+ a^{\beta n}\sum_{(v,w) \in E_{n}}\int_{K_{w}}\mint_{K_{v}}\abs{f(x) - f(y)}^{p}\,d\mu(x)d\mu(y)\Biggr).
  \end{align*}
  A bound of the first term in the right-hand side is obtained in Lemma \ref{lem.LBPI}.
  Noting that $K_{v} \subseteq B_{d}(y, ca^{-n})$ for $(v, w) \in E_{n}$ and $y \in K_{w}$, we can estimate the second term as follows:
  \begin{align*}
      &a^{\beta n}\sum_{(v,w) \in E_{n}}\int_{K_{w}}\mint_{K_{v}}\abs{f(x) - f(y)}^{p}\,d\mu(x)d\mu(y) \\
      &\le c_{2}a^{\beta n}\sum_{(v,w) \in E_{n}}\int_{K_{w}}\mint_{B_{d}(y, ca^{-n})}\abs{f(x) - f(y)}^{p}\,d\mu(x)d\mu(y)
      \le c_{2}c^{-\beta}L_{\ast}S_{p, \beta}^{(n)}(f),
  \end{align*}
  where $c_2 > 0$ depends only on $C_{\text{AR}}, c$.
  This proves \eqref{eq.PIKS-lower} and finishes the proof.
\end{proof}

As an immediate consequence of Theorem \ref{thm.Besov-charact}, we have a characterization of $\beta_{p}$ as \emph{critical Besov exponents}.
For details on critical Besov exponents, see \cites{ABCRST21,GHL03} for example.
This result is well-known when $p = 2$ (see \cite{GHL03}*{Theorem 4.6}).

\begin{cor}\label{cor.Lp-Besov}
    It holds that $\beta_{p} = p\cdot\sup\bigl\{ s > 0 \bigm| \Lambda_{p, \infty}^{s} \neq \{ \text{constant} \} \bigr\}$.
\end{cor}
\begin{proof}
    Note that $\Lambda_{p, \infty}^{\beta'/p} \subseteq \Lambda_{p, \infty}^{\beta/p}$ for any $\beta \le \beta'$.
    It is immediate that 
    \[
    \beta_{p} \le p\cdot\sup\bigl\{ s > 0 \bigm| \Lambda_{p, \infty}^{s} \neq \{ \text{constant} \} \bigr\}. 
    \]
    To prove the converse, let $\beta > \beta_{p}$.
    If $f \in \mcC(K)$ is not constant, then there exists $N \in \bbN$ such that $\widetilde{\mcE}_{p}^{G_{N}}(M_{N}f) > 0$.
    By Corollary \ref{cor.wm}, for any $n \ge N$,
    \begin{align*}
        a^{(\beta - \alpha)n}\cdot\mcE_{p}^{G_{n}}(M_{n}f) = a^{(\beta - \alpha)n}\rho_{p}^{-n}\cdot\widetilde{\mcE}_{p}^{G_{n}}(M_{n}f)
        \ge C_{\text{WM}}^{-1}\,a^{(\beta - \alpha)n}\rho_{p}^{-n}\cdot\widetilde{\mcE}_{p}^{G_{N}}(M_{N}f).
    \end{align*}
    Letting $n \to \infty$ in this inequality, we obtain $\varlimsup_{n \to \infty}a^{(\beta - \alpha)n}\cdot\mcE_{p}^{G_{n}}(M_{n}f) = \infty$ since $\rho_{p}^{-1}a^{\beta - \alpha} > 1$.
    By \eqref{eq.PIKS-lower}, we conclude that $\abs{f}_{\Lambda_{p, \infty}^{\beta/p}} = \infty$ whenever $\beta > \beta_{p}$ and $f \in \mcC(K)$ is not constant.
    This proves our assertion.
\end{proof}

\section{Construction of a canonical scaling limit of $p$-energies}\label{sec:constr}
To construct a canonical Dirichlet form on fractals, there is already an established way as appeared in \cite{KZ92}*{proof of Theorem 6.9}.
However, in the original argument of \cite{KZ92}, the Markov property of their ``Dirichlet form'' was not clarified.
In \cite{Kig00}, Kigami has pointed out this gap and filled it. 

After Kigami's work, another very simple way to check the Markov property is given by Barlow, Bass, Kumagai, and Teplyaev \cite{BBKT10}*{proof of Theorem 2.1}.
This method deduces that the Dirichlet forms of Kusuoka and Zhou in \cite{KZ92} have the Markov property, but it very heavily relies on being Dirichlet forms, that is, the use of bilinearity (and locality) is essential to follow \cite{BBKT10}*{proof of Theorem 2.1}.

Regrettably, these ways are insufficient to follow the arguments in Section \ref{sec:em}, where the main results about $\mcE_{p}$-energy measures will be proved.
Indeed, our strategy to prove the chain rule (Theorem \ref{thm.main3}) will require some expression of constructed $p$-energy $\mathcal{E}_{p}$ (see \eqref{representation}) due to the lack of representation formula (see \cite{FOT}*{(3.2.10)} for example) that is very useful in Dirichlet form theory. 
Also, as shown in \cites{CQ21+, GY19, Kig21+}, the usage of $\Gamma$-convergence is very useful to construct energies. 
However, we need to adopt alternative approach because our argument in Theorem \ref{thm.main3} will heavily use the compactness of $f(K)$ for a fixed function $f \in \mathcal{F}_{p}$. 
If we construct $\mathcal{E}_{p}$ by using $\Gamma$-convergence, then we lose this compactness since we have to consider $\bigcup_{n \ge 1}f_{n}(K)$ instead of $f(K)$, where $\{ f_{n} \}_{n \ge 1}$ is a sequence converging to $f$ in $L^{p}$. 

In order to overcome these difficulties, we will introduce a new series of graphs $\bbG_{n}$ approximating the underlying Sierpi\'{n}ski carpet in subsection \ref{subsec:newgraph}.
Then in subsection \ref{subsec:main1} we directly construct $p$-energy $\mcE_{p}$ as a subsequential scaling limit of discrete $p$-energies on this new series of graphs.
Since we already prove \hyperlink{Bp}{\textup{(B$_{p}$)}} and \hyperlink{KMp}{\textup{(KM$_{p}$)}}, Kigami's result \cite{Kig21+}*{Theorem 9.3} gives $p$-energies on GSCs satisfying all properties in Theorem \ref{thm.main1.2} (except Clarkson's inequality). 
We emphasize that a main aim of this section is to construct $p$-energies that will be useful to prove the chain rule in Section \ref{sec:em} (Theorem \ref{thm.main3}-(2)). 

Throughout this section, let $(K, S, \{ F_{i} \}_{i \in S}) = \GSC(D, a, S)$ be a generalized Sierpi\'{n}ski carpet and suppose that Assumption \ref{assum.lowdim} holds. 

\subsection{Behavior of $p$-energies on modified Sierpi\'{n}ski carpet graphs}\label{subsec:newgraph}
For $n \in \mathbb{Z}_{\ge 0}$, define subsets $\mathbb{V}_{n} \subseteq K$ by 
\[
\mathbb{V}_{n} = \left\{ F_{w}\left(\sum_{k = 1}^{D}\sigma_{k}\mathbf{e}_{k}\right) \;\middle|\; w \in W_{n}, \text{$\sigma_{k} \in \{ -1, +1 \}$ for $k = 1, \dots, D$} \right\}. 
\] 
(The conditions \hyperlink{GSC1}{\textup{(GSC1)}} and \hyperlink{GSC4}{\textup{(GSC4)}} ensure that $\sum_{k = 1}^{D}\sigma_{k}\mathbf{e}_{k} \in K$.)
Note that $\#\mathbb{V}_{0} = 2^{D}$. 
Next, we inductively define edge sets $\mathbb{E}_{n}$ by
\[
\mathbb{E}_{0} \coloneqq \{ (x, y) \mid x \neq y, x, y \in \mathbb{V}_{0} \},
\]
and
\[
\mathbb{E}_{n} \coloneqq \bigl\{ \bigl(F_{i}(x), F_{i}(y)\bigr) \bigm| i \in S, (x, y) \in \mathbb{E}_{n - 1} \bigr\}.
\]
Define a new finite graph by $\bbG_{n} \coloneqq (\mathbb{V}_{n}, \mathbb{E}_{n})$ (see Figure \ref{fig:modified}). 
Note that $\mathbb{G}_{0}$ coincides with the complete graph $\mathbb{K}_{2^{D}}$ having $2^{D}$ vertices. 
We will write $d_{\mathbb{G}_{n}}$ for the graph distance of $\mathbb{G}_{n}$. 
By \hyperlink{GSC2}{\textup{(GSC2)}}, we see that $\mathbb{G}_{n}$ is connected.
Furthermore, we easily see that $\{ \mathbb{G}_{n} \}_{n \ge 0}$ is an increasing sequence. 
It is also immediate that
\[
L_{\ast, \textup{modif}} \coloneqq \adjustlimits\sup_{n \in \bbZ_{\ge 0}}\max_{x \in \mathbb{V}_{n}}\#\{ y \in \mathbb{V}_{n} \mid (x, y) \in \mathbb{E}_{n} \} \le L_{\ast}2^{D}.
\]
For any $n, m \in \bbZ_{\ge 0}$ and $w \in W_{n}$, we define a subset $\mathbb{V}^{w}_{m}$ of $\mathbb{V}_{n + m}$ by setting $\mathbb{V}^{w}_{m} \coloneqq \{ F_{w}(x) \mid x \in \mathbb{V}_{m} \}$, and define a subgraph $\bbG_{m}^{w} \coloneqq \bigl(\mathbb{V}^{w}_{m}, \mathbb{E}^{w}_{m}\bigr)$, where 
\[
\mathbb{E}^{w}_{m} = \bigl\{ (x, y) \in \mathbb{E}_{n + m} \bigm| x, y \in \mathbb{V}^{w}_{m} \bigr\}. 
\]
Note that, for $v, w \in W_{\ast}$ with $\abs{v} = \abs{w}$, $\mathbb{V}_{m}^{w} \cap \mathbb{V}_{m}^{v} \neq \emptyset$ if and only if $K_{v} \cap K_{w} \neq \emptyset$. 

For simplicity, we write $\mcR_{p}^{\bbG_{n}}(x, y)$ to denote $\mcC_{p}^{\bbG_{n}}(\{ x \}, \{ y \})^{-1}$, that is,
\[
\mcR_{p}^{\bbG_{n}}(x, y) \coloneqq \sup\left\{ \frac{\abs{f(x) - f(y)}^{p}}{\mcE_{p}^{\bbG_{n}}(f)} \;\middle|\; \text{$f \colon \mathbb{V}_{n} \to \bbR$ is not constant} \right\},
\]
for each $x, y \in \mathbb{V}_{n}$.
Then one of the key ingredients is the next proposition; $\mcR_{p}^{\bbG_{n}}(x, y)$ behaves like $\mcR_{p}^{(n)}$.

\begin{figure}[tb]
    \centering
    \includegraphics[height=85pt]{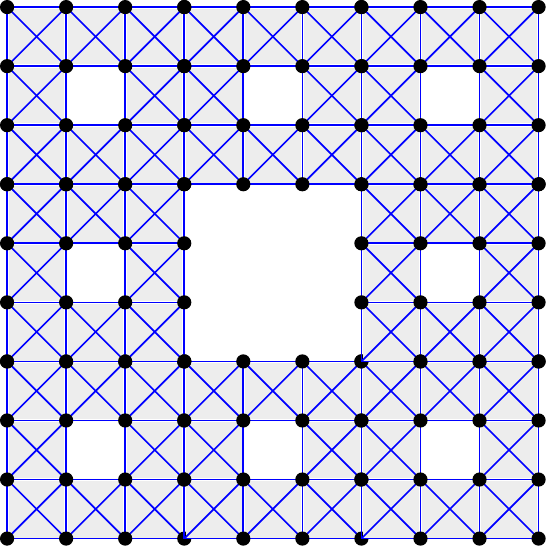}
    \caption{Modified Sierpi\'{n}ski carpet graph $\{ \bbG_{n} \}_{n \ge 1}$ (This figure draws $\bbG_{3}$ in the SC case)}
    \label{fig:modified}
\end{figure}

\begin{prop}\label{prop.ptpt}
    There exists a positive constant $C_{\ref{prop.ptpt}}$ (depending only on $p$, $a$, $D$, $N_{\ast}$, $L_{\ast}$, $\rho_p$) such that, for any $n \in\bbN$ and $x, y \in \mathbb{V}_{0}$,
    \begin{equation}\label{ineq.CH}
        \mcR_{p}^{\bbG_{n}}(x, y) \le C_{\ref{prop.ptpt}}\rho_{p}^{n}.
    \end{equation}
\end{prop}

\begin{rmk}\label{rmk.pttopt}
    For $p = 2$ in the SC case, such point-to-point estimates on a series of Sierpi\'{n}ski carpet graphs are proved in \cite{Sas21+}*{Appendix}, where a uniform Hanarck inequality by Barlow and Bass (see also Remark \ref{rmk.harnack}) is used.
    In \cite{Kig21+}*{Lemma 8.5}, Kigami also shows similar estimates for all $p > \dim_{\text{ARC}}$ assuming $p$-conductive homogeneity, where he also uses some uniform H\"{o}lder estimate.
    Our proof also relies on the uniform H\"{o}lder estimate: Theorem \ref{thm.holder_KM}.
\end{rmk}

To prove this proposition, we need an estimate of $p$-conductance between points on the original graph $\{ G_{n} \}_{n \ge 1}$.
For $n \in \mathbb{Z}_{\ge 0}$ and $x$, we fix $w^{x}(n) \in W_{n}$ such that $x \in K_{w^{x}(n)}$. 
If $x = \sum_{k = 1}^{D}\sigma_{k}\mathbf{e}_{k} \in \mathbb{V}_{0}$, then such $w^{x}(n)$ is uniquely given by 
\[
w^{x}(n) = (\sigma_{1}, \dots, \sigma_{D})^{n} \in W_{n}.
\]
We can show the following lemma in a similar way to the `chain argument' in the proof of Theorem \ref{thm.KM}.

\begin{lem}\label{lem.pttopt}
    There exists a constant $C_{\ref{lem.pttopt}} \ge 1$ (depending only on $p$, $a$, $D$, $N_{\ast}$, $L_{\ast},$ $\rho_p$) such that, for any $n \in \bbN$ and $x, y \in \mathbb{V}_{0}$ with $x \neq y$,
    \begin{equation}\label{ineq.ptpt}
        C_{\ref{lem.pttopt}}^{-1}\rho_{p}^{-n} \le \mcC_{p}^{G_{n}}\bigl(\{ w^{x}(n) \}, \{ w^{y}(n) \}\bigr) \le C_{\ref{lem.pttopt}}\rho_{p}^{-n}. 
    \end{equation}
\end{lem}
\begin{proof}
    An upper bound is easy. 
    Indeed, by Proposition \ref{prop.mono}, Theorems \ref{thm.sub-mult} and \ref{thm.superP}, 
    \begin{align}\label{ptpt.upper}
        \mcC_{p}^{G_{n}}\bigl(\{ w^{x}(n) \}, \{ w^{y}(n) \}\bigr) 
        \le \mathcal{C}_{p}^{(n - 1)} 
        \le c_{\ast}\rho_{p}\cdot\rho_{p}^{-n}, 
    \end{align}
     where $c_{\ast}$ is the constant in \eqref{Poi-super}. 

    In order to prove the converse, we first consider the case $\abs{x - y}_{\mathbb{R}^{D}} = 2$.
    Let $C_{\text{UH}} > 0$ be the constant (depending only on $p, a, D, L_{\ast}, N_{\ast}, \rho_{p}$) in Theorem \ref{thm.holder_KM} and choose $l_{\ast} \in \bbN$ such that 
    \begin{equation}\label{ptpt.choice}
    	C_{\text{UH}}a^{-(\beta_{p} - \alpha)l_{\ast}} \le \frac{1}{4}.	
    \end{equation}
    We also set $v(x) \coloneqq [w^{x}(n)]_{l_{\ast}} \in W_{l_{\ast}}$ for each $x \in \mathbb{V}_{0}$.
    Since $\abs{x - y}_{\mathbb{R}^{D}} = 2$, there exists a horizontal chain $[z(1), \dots, z(L)]$ in $\widetilde{G}_{l_{\ast}}$ such that $z(1) = v(x)$, $z(L) = v(y)$, $z(k) \in \partial_{\ast}G_{l_{\ast}}$ for any $k = 1, \dots, L$ and $(z(k), z(k + 1)) \in \widetilde{E}_{l_{\ast}}$ for each $k = 1, \dots, L - 1$, where $L = a^{l_{\ast}}$. 
    
    Let $f_{n}\colon W_{n} \to \bbR$ satisfy $f_{n}(w^{x}(n)) = 0$, $f_{n}(w^{y}(n)) = 1$ and 
    \[
    \mcE_{p}^{G_{n}}(f_n) = \mcC_{p}^{G_{n}}\bigl(\{ w^{x}(n) \}, \{ w^{y}(n) \}\bigr). 
    \]
    Note that $f_{n}$ is $[0, 1]$-valued.
    From Theorem \ref{thm.holder_KM} and \eqref{ptpt.choice}, 
    \[
    \max_{v(x) \cdot W_{n - l_{\ast}}}f_{n} \le \frac{1}{4} \quad\text{and}\quad \min_{v(y) \cdot W_{n - l_{\ast}}}f_{n} \ge \frac{3}{4}.
    \]
    Now we define $g_{n}$ by setting $g_{n} \coloneqq 2\bigl(\left(f_{n} \vee 1/4\right) \wedge 3/4\bigr)$.
    Then we have that
    \begin{equation*}
        \mcC_{p}^{(n - l_{\ast}, L)} \le \mcE_{p}^{G_{n}}(g_{n}) \le 2^{p}\mcE_{p}^{G_{n}}(f_{n}).
    \end{equation*}
    (Recall the definition of $\mathcal{C}_{p}^{(n, L)}$ in subsection \ref{subsec:KM}.)
    By Lemma \ref{lem.chain} and Theorem \ref{thm.sub-mult}, there exists a constant $C(L) > 0$ depending only on $p, a, L_{\ast}, N_{\ast}$ such that 
    \[
    \mcC_{p}^{(n - l_{\ast}, L)} \ge C(L)^{-1}\mcC_{p}^{(n - l_{\ast})} \ge \bigl(C(L)C_{\ref{thm.sub-mult}}\mathcal{C}_{p}^{(l_{\ast})}\bigr)^{-1}\mcC_{p}^{(n)}. 
    \]
    Hence, by using \eqref{Poi-super}, 
    \[
    \mcC_{p}^{G_{n}}(\{ w^{x}(n) \}, \{ w^{y}(n) \}) \ge \bigl(2^{p}C(L)C_{\ref{thm.sub-mult}}\mathcal{C}_{p}^{(l_{\ast})}\bigr)^{-1}\mcC_{p}^{(n)} \ge c_{\ast}\bigl(2^{p}C(L)C_{\ref{thm.sub-mult}}\mathcal{C}_{p}^{(l_{\ast})}\bigr)^{-1}\rho_{p}^{-n}, 
    \]
    which proves the case $\abs{x - y}_{\mathbb{R}^{D}} = 2$.

    Finally, note that $(v, w) \mapsto \mcC_{p}^{G}(\{ v \}, \{ w \})^{-1/p}$ is a metric on a graph $G$.
    Indeed, this fact immediately follows from the representation:
    \[
    \mcC_{p}^{G}(\{ v \}, \{ w \})^{-1/p} = \max\left\{ \frac{\abs{f(v) - f(w)}}{\mcE_{p}^{G}(f)^{1/p}} \;\middle|\; f\colon V \to \bbR \text{ with } \mcE_{p}^{G}(f) > 0 \right\}.
    \]
    Applying the triangle inequality of this metric, we obtain 
    \begin{align*}
    	\max_{x \neq y \in \mathbb{V}_{0}}\mcC_{p}^{G_{n}}(\{ w^{x}(n) \}, \{ w^{y}(n) \})^{-1/p}
		\le DC_{1}\rho_{p}^{n/p}, 
    \end{align*}
    where $C_{1} = \bigl(2^{p}c_{\ast}^{-1}C(L)C_{\ref{thm.sub-mult}}\mathcal{C}_{p}^{(l_{\ast})}\bigr)^{1/p}$. 
\end{proof}

\begin{rmk}
    It is known that $(x, y) \mapsto \mcC_{p}^{G}(\{ x \}, \{ y \})^{-1/(p - 1)}$ also becomes a metric, and thus $\mcC_{p}^{G}(\,\cdot\,, \,\cdot\,)^{-1/(p - 1)}$ gives a generalization of the \emph{resistance metric}.
    This fact is proved in \cite{ACFPc19}*{Theorem 8} when $G$ is a finite graph.
    One can check the case of infinite graphs in \cite{Mthesis}*{Theorem 4.3}.
\end{rmk}

Now we are ready to prove Proposition \ref{prop.ptpt} by using Lemma \ref{lem.URI}.

\noindent
\textit{Proof of Proposition \ref{prop.ptpt}.}\,
    For each $n \in \mathbb{N}$, let $\varphi_{n}\colon W_{n} \to \mathbb{V}_{n}$ be a function such that $\varphi_{n}(w) \in \mathbb{V}_{0}^{w}$ for $w \in W_{n}$ and
    \[
    \varphi_{n}\bigl((\sigma_{1}, \dots, \sigma_{D})^{n}\bigr) = \sum_{k = 1}^{D}\sigma_{k}\mathbf{e}_{k} \quad \text{if $(\sigma_{k})_{k = 1}^{D} \in \{ -1, +1 \}^{D}$. }
    \]
    Then we easily see that $\varphi_{n}$ is a rough isometry from $G_{n}$ to $\mathbb{G}_{n}$ with $C_{1} = 1, C_{2} = 2, C_{3} = 2, C_{4} = \bigl((2^{D} - 1)/L_{\ast}\bigr) \vee \bigl(L_{\ast, \textup{modif}}/2\bigr)$ (recall Definition \ref{dfn.URI}). 
    Hence $\{ \varphi_{n} \}_{n \ge 1}$ is a uniform rough isometry from $\{ G_{n} \}_{n \ge 1}$ to $\{ \bbG_{n} \}_{n \ge 1}$.
    Note that for $(\sigma_{k})_{k = 1}^{D} \in \{ -1, +1 \}^{D}$,   
    \[
    \varphi_{n}(w) = \sum_{k = 1}^{D}\sigma_{k}\mathbf{e}_{k} \quad \text{if and only if} \quad w = (\sigma_{1}, \dots, \sigma_{D})^{n},  
    \]
    i.e. $\varphi_{n}^{-1}\bigl(\bigl\{ \sum_{k = 1}^{D}\sigma_{k}\mathbf{e}_{k} \bigr\}\bigr) = \{ (\sigma_{1}, \dots, \sigma_{D})^{n} \}$. 
    Thus, for $x \neq y \in \mathbb{V}_{0}$, Lemma \ref{lem.URI} yields that 
    \[
    \mcC_{p}^{G_{n}}\bigl(\{ w^{x}(n) \}, \{ w^{y}(n) \}\bigr) \le C_{\text{URI}}\mcR_{p}^{\bbG_{n}}(x, y)^{-1},
    \]
    which together with Lemma \ref{lem.pttopt} implies \eqref{ineq.CH}.\qed

Next, we define rescaled $p$-energy $\widetilde{\mcE}_{p}^{\,\bbG_{n}}\colon \mcC(K) \to \bbR$ on $\bbG_{n}$ by setting
\[
\widetilde{\mcE}_{p}^{\,\bbG_{n}}(f) \coloneqq \frac{\rho_{p}^{n}}{2}\sum_{(x, y) \in \mathbb{E}_{n}}\abs{f(x) - f(y)}^{p},
\]
for each $f \in \mcC(K)$.
Then the following lemma, especially statement (4) below, is a collection of benefits of the new graphical approximation $\{ \bbG_{n} \}_{n \ge 0}$.

\begin{lem}\label{lem.ss}
	Let $f \in \mathcal{C}(K)$ and let $n \in \mathbb{Z}_{\ge 0}$. 
    The following statements hold.
    \begin{itemize}
        \item [\textup{(1)}] It holds that $\sup_{n \ge 0}\widetilde{\mcE}_{p}^{\,\bbG_{n}}(f) = 0$ if and only if $f$ is constant.
        \item [\textup{(2)}] If $T \in \mathcal{G}_{0}$, then $\widetilde{\mcE}_{p}^{\,\bbG_{n}}(f \circ T) = \widetilde{\mcE}_{p}^{\,\bbG_{n}}(f)$.
        \item [\textup{(3)}] For $m \in \mathbb{Z}_{\ge 0}$, 
        \begin{equation}\label{eq.ss}
            L_{\ast}^{-1}\rho_{p}^{m}\sum_{w \in W_{m}}\widetilde{\mcE}_{p}^{\,\bbG_{n}}(f \circ F_{w}) 
            \le \widetilde{\mcE}_{p}^{\,\bbG_{n + m}}(f) 
            \le \rho_{p}^{m}\sum_{w \in W_{m}}\widetilde{\mcE}_{p}^{\,\bbG_{n}}(f \circ F_{w}).
        \end{equation}
        \item [\textup{(4)}] If $\varphi\colon \bbR \to \bbR$ with $\mathrm{Lip}(\varphi) \le 1$, then $\widetilde{\mcE}_{p}^{\,\bbG_{n}}(\varphi \circ f) \le \widetilde{\mcE}_{p}^{\,\bbG_{n}}(f)$.
    \end{itemize}
\end{lem}
\begin{proof}
    (1) It is immediate that $\sup_{n \in \bbN}\widetilde{\mcE}_{p}^{\,\bbG_{n}}(f) = 0$ if $f \in \mcC(K)$ is constant.
    We easily see that $\closure{\bigcup_{n \in \bbN}\mathbb{V}_{n}}^{K} = K$.
    Thus, if $f \in \mcC(K)$ satisfies $\mcE_{p}^{\,\bbG_{n}}(f) = 0$ for any $n \in \bbZ_{\ge 0}$, then $f$ is constant.

    (2) This immediately follows from the symmetries of $\bbG_{n}$.

    (3) Let $f \in \mcC(K)$ and $n \in \bbZ_{\ge 0}$.
    Since $\bigcup_{w \in W_{m}}\mathbb{V}_{n}^{w} = \mathbb{V}_{n + m}$, we have 
    \begin{align*}
        \widetilde{\mcE}_{p}^{\,\bbG_{n + m}}(f)
        &\le \rho_{p}^{m}\sum_{w \in W_{m}}\left(\frac{\rho_{p}^{n}}{2}\sum_{\substack{(x,y) \in \mathbb{E}_{n + m}; \\ x, y \in K_{w}}}\abs{f(F_{w}(F_{w}^{-1}(x))) - f(F_{w}(F_{w}^{-1}(y)))}^{p}\right) \\
        &= \rho_{p}^{m}\sum_{w \in W_{m}}\widetilde{\mcE}_{p}^{\,\bbG_{n}}(f \circ F_{w}). 
    \end{align*}
    Next, by noting that 
    \[
    \sup_{(x, y) \in \mathbb{E}_{n + m}}\#\bigl\{ w \in W_{m} \bigm| \text{$x, y \in \mathbb{V}_{n}^{w}$} \bigr\} \le L_{\ast}, 
    \]
    we see that 
    \begin{align*}
    	\rho_{p}^{m}\sum_{w \in W_{m}}\widetilde{\mcE}_{p}^{\,\bbG_{n}}(f \circ F_{w}) 
    	&= \frac{\rho_{p}^{n + m}}{2}\sum_{w \in W_{m}}\sum_{\substack{(x, y) \in \mathbb{E}_{n + m}; \\ x, y \in \mathbb{V}_{n}^{w}}}\abs{f(x) - f(y)}^{p} \\
    	&\le L_{\ast}\frac{\rho_{p}^{n + m}}{2}\sum_{(x, y) \in \mathbb{E}_{n + m}}\abs{f(x) - f(y)}^{p} 
    	= L_{\ast}\widetilde{\mathcal{E}}_{p}^{\,\bbG_{n + m}}(f), 
    \end{align*}
    which proves \eqref{eq.ss}. 

    (4) The required estimate immediately follows from the fact that $\abs{\varphi(a) - \varphi(b)} \le \abs{a - b}$ for any $a, b \in \bbR$ whenever $\varphi\colon \bbR \to \bbR$ satisfies $\mathrm{Lip}(\varphi) \le 1$.
\end{proof}

We also have the weak monotonicity of $\widetilde{\mcE}_{p}^{\,\bbG_{n}}$ as follows.
Its proof is similar to \cite{GY19}*{Theorem 7.1}, where $2$-energies are considered.

\begin{lem}\label{lem.newwm}
    There exists a constant $C_{\textup{WM}, \textup{modif}} > 0$ (depending only on $p, a, D$, $N_{\ast}$, $L_{\ast}$, $\rho_p$) such that
    \begin{equation}\label{ineq.newwm}
        \widetilde{\mcE}_{p}^{\,\bbG_{n}}(f) \le C_{\textup{WM}, \textup{modif}}\widetilde{\mcE}_{p}^{\,\bbG_{n + m}}(f),
    \end{equation}
    for every $n, m \in \bbZ_{\ge 0}$ and $f \in \mcC(K)$.
    In particular, for any $f \in \mcC(K)$,
    \begin{equation*}
        \sup_{n \in \bbN}\widetilde{\mcE}_{p}^{\,\bbG_{n}}(f) \le C_{\textup{WM}, \textup{modif}}\varliminf_{n \to \infty}\widetilde{\mcE}_{p}^{\,\bbG_{n}}(f).
    \end{equation*}
\end{lem}
\begin{proof}
    Let $(x, y) \in \mathbb{E}_{n}$.
    Then there exists $w \in W_{n - 1}$ such that $x, y \in K_{w}$.
    Furthermore, there exist $x_{0}, y_{0} \in \mathbb{V}_{0}$ such that $x = F_{w}(x_{0})$ and $y = F_{w}(y_{0})$.
    Now we define
    \begin{align*}
        \mcR_{p}^{\,\bbG_{m}^{w}}(z_{1}, z_{2}) 
        \coloneqq \sup\left\{ \frac{\abs{f(z_{1}) - f(z_{2})}^{p}}{\mcE_{p}^{\,\bbG_{m}^{w}}(f)} \;\middle|\; \text{$f \in \mcC(K)$ with $f|_{\mathbb{V}_{m}^{w}}$ is not constant} \right\}.
    \end{align*}
    From Proposition \ref{prop.ptpt} and the cutting law (\cite{Mthesis}*{Proposition 3.18}), 
    \begin{align*}
        \abs{f(x) - f(y)}^{p}
        &\le \mcR_{p}^{\,\bbG_{m}^{w}}\bigl(F_{w}(x_{0}), F_{w}(y_{0})\bigr)\mcE_{p}^{\,\bbG_{m}^{w}}(f) \\
        &\le \mcR_{p}^{\,\bbG_{m}}(x_{0}, y_{0})\mcE_{p}^{\,\bbG_{m}^{w}}(f) \\
        &\le C_{\ref{prop.ptpt}}\rho_{p}^{m}\mcE_{p}^{\,\bbG_{n + m}}(f).
    \end{align*}
    Summing over $(x,y) \in \mathbb{E}_{n}$, we obtain $\mcE_{p}^{\,\bbG_{n}}(f) \le 2^{-1}C_{\ref{prop.ptpt}}\rho_{p}^{m}\mcE_{p}^{\,\bbG_{n + m}}(f)$, which deduces our assertion with $C_{\textup{WM}, \textup{modif}} = 2^{-1}C_{\ref{prop.ptpt}}$.
\end{proof}

The rest of this subsection is devoted to proving the following lemma.
Recall the definition of $\abs{\,\cdot\,}_{\mathcal{F}_{p}}$ that is a semi-norm of $\mathcal{F}_{p}$ (see \eqref{dfn.1p-semi}). 

\begin{lem}\label{lem.comparable}
    There exists a constant $C_{\ref{lem.comparable}} \ge 1$ depending only on $p, a, D, N_{\ast}, L_{\ast}$, $C_{\textup{AR}}$, $\rho_{p}$ such that
    \begin{equation}\label{ineq.ptcomparable}
        C_{\ref{lem.comparable}}^{-1}\abs{f}_{\mcF_{p}}^{p} \le \sup_{n \in \bbN}\widetilde{\mcE}_{p}^{\,\bbG_{n}}(f) \le C_{\ref{lem.comparable}}\abs{f}_{\mcF_{p}}^{p} \quad \text{for any $f \in L^{p}(K, \mu)$.}
    \end{equation}
\end{lem}

\begin{rmk}\label{rmk.discrete}
    Such discrete characterizations of Lipschitz--Besov space are treated in \cite{Bod09}, but we need some modification as stated in \cite{GL20}*{Remark 1 in Section 3}.
    To be self-contained, we give complete proofs in a similar way as in \cite{GY19}*{Theorems 3.5 and 3.6 and Proposition 11.1}, where they give discrete characterizations of $\mcF_{2}$ on the SC.
\end{rmk}

By virtue of Theorem \ref{thm.main2}, it will suffice to show that $\sup_{n \in \bbN}\widetilde{\mcE}_{p}^{\,\bbG_{n}}(f)$ and $\abs{f}_{\Lambda_{p, \infty}^{\beta_{p}/p}}^{p}$ are comparable.
Similarly to \cite{Bod09}, we will apply an argument using discrete approximations of measure $\mu$ with respect to the weak convergence of probability measures (see Lemma \ref{lem.newPI}-(2)).
The following lemma is elementary (see also \cite{Bod09}*{Lemma 3.12}, \cite{GL20}*{Remark 1 in Section 3}).
Recall the definition of $U_{1}(x, s)$ in \eqref{dfn.quasiball}. 

\begin{lem}\label{lem.weak}
    Let $\{ \mu_{n} \}_{n \ge 1}$ be a sequence of probability measures on $K$ given by
    \[
    \mu_{n} \coloneqq \frac{1}{\#\mathbb{V}_{n}}\sum_{x \in \mathbb{V}_{n}}\delta_{x},
    \]
    where $\delta_{x}$ denotes the Dirac measure with support $\{ x \}$ for each $x \in K$. 
    Then there exist a subsequence $\{ n_k \}_{k \ge 1}$ and a constant $c_{\ref{lem.weak}} > 0$ (depending only on $a, D, N_{\ast}, L_{\ast}, C_{\textup{AR}}$) such that for any $m \in \bbZ_{\ge 0}$, $p > 1$ and $f \in \mcC(K)$,
    \begin{align}\label{eq.weak}
        &\lim_{k \to \infty}\int_{K}\int_{U_{1}(x, a^{-m})}\abs{f(x) - f(y)}^{p}\,d\mu_{n_k}(y)d\mu_{n_k}(x) \\
        &\hspace{10pt}\ge c_{\ref{lem.weak}}\int_{K}\int_{U_{1}(x, a^{-m})}\abs{f(x) - f(y)}^{p}\,d\mu(y)d\mu(x). \nonumber
    \end{align}
\end{lem}
\begin{proof}
    Since $K$ is compact, by Prokhorov's theorem (see \cite{Bil}*{Theorem 5.1} for example), there exist a subsequence $\{ n_k \}_{k \ge 1}$ and a Borel probability measure $\widetilde{\mu}$ on $K$ such that $\mu_{n_k}$ converges weakly to $\widetilde{\mu}$ as $k \to \infty$.
    From the definition of $\bbG_{n}$, by noting that 
    \[
    \sup_{n \ge 0, x \in \mathbb{V}_{n}}\#\bigl\{ w \in W_{n} \bigm| x \in \mathbb{V}_{0}^{w} \bigr\} \le L_{\ast}, 
    \]
    we have 
    \begin{equation}\label{weak.number-pt}
    	L_{\ast}^{-1}2^{D}N_{\ast}^{n} \le \#\mathbb{V}_{n} \le 2^{D}N_{\ast}^{n}, \quad n \in \mathbb{Z}_{\ge 0}. 
    \end{equation}
    
    Next we will see that $\widetilde{\mu}$ is ``$\alpha$-Ahlfors regular'' and $\widetilde{\mu}(\bord{K_{w}}) = 0$ for any $w \in W_{\ast}$. 
    (Recall that $\alpha = \log{N_{\ast}/\log{a}}$ and thus $a^{\alpha} = N_{\ast}$.) 
    For $r \in (0, 1]$, let $n(r) \in \mathbb{Z}_{\ge 0}$ denote the unique non-negative integer such that $a^{-n(r)} \le r < a^{-n(r) + 1}$.
    Let $C_{\textup{AD}} \ge 1$ be the constant in Proposition \ref{prop.BF}-(2A) and let $n_{\ast} \in \mathbb{Z}_{\ge 0}$ such that $a^{-n_{\ast}} \le C_{\textup{AD}}^{-1} < a^{-n_{\ast} + 1}$. 
    Note that $n_{\ast}$ depends only on $C_{\textup{AD}}$ and $a$.
    In addition, note that we can choose $r_{0} \in (0, 1)$ (depending only on $C_{\textup{AD}}$ and $a$) so that $n(r) \ge n_{\ast}$ for all $r \in (0, r_{0}]$. 
    Hereafter, we fix such a small $r_{0} \in (0, 1)$. 
    Let $x \in K$ and $r \in (0, r_{0}]$. 
    By Proposition \ref{prop.BF}-(2A), we know that $U_{1}(x, C_{\textup{AD}}^{-1}r) \subseteq B_{d}(x, r) \subseteq U_{1}(x, C_{\textup{AD}}r)$. 
    From $U_{1}(x, C_{\textup{AD}}^{-1}r) \subseteq B_{d}(x, r)$, we can find $w \in W_{n(r) + n_{\ast}}$ such that $K_{w} \subseteq B_{d}(x, r)$. 
    The portmanteau theorem (see \cite{Bil}*{Theorem 2.1} for example) and \eqref{weak.number-pt} yield that  
    \[
    \widetilde{\mu}\Bigl(\closure{B_{d}(x, r)}\Bigr) 
    \ge \varlimsup_{k \to \infty}\mu_{n_{k}}\Bigl(\closure{B_{d}(x, r)}\Bigr) 
    \ge \varlimsup_{k \to \infty}\mu_{n_{k}}(K_{w}) 
    \ge L_{\ast}^{-1}a^{-(\alpha n_{\ast} + 1)}r^{\alpha}.  
    \]
    Similarly, by $B_{d}(x, r) \subseteq U_{1}(x, C_{\textup{AD}}r)$, portmanteau theorem and \eqref{weak.number-pt}, 
    \[
    \widetilde{\mu}\bigl(B_{d}(x, r)\bigr) 
    \le \varliminf_{k \to \infty}\mu_{n_{k}}\bigl(B_{d}(x, r)\bigr)  
    \le \varliminf_{k \to \infty}\mu_{n_{k}}\bigl(U_{1}(x, C_{\textup{AD}}r)\bigr) 
    \le L_{\ast}^{2}a^{\alpha n_{\ast}}r^{\alpha}. 
    \]    
    Consequently, for any $x \in K$ and $r \in (0, \diam(K, d)]$, we obtain 
    \[
    (aN_{\ast}L_{\ast}C_{\textup{AD}}^{\alpha})^{-1}r^{\alpha} \le \widetilde{\mu}\Bigl(\closure{B_{d}(x, r)}\Bigr), \quad \widetilde{\mu}\bigl(B_{d}(x, r)\bigr) \le L_{\ast}^{2}N_{\ast}C_{\textup{AD}}^{\alpha}r^{\alpha}. 
    \]
	By \cite{Hei}*{Exercise 8.11} (or by following the argument in \cite{Gri10}*{Lemma 2.13}), there exists a constant $C_{1} \ge 1$ depending only on $a, N_{\ast}, L_{\ast}, C_{\textup{AD}}$ such that 
	\[
	C_{1}^{-1}\mathcal{H}^{\alpha}(A) \le \widetilde{\mu}(A) \le C_{1}\mathcal{H}^{\alpha}(A) \quad \text{for any $A \in \mathcal{B}(K)$,}
	\]
	where $\mathcal{H}^{\alpha}$ is the normalized $\alpha$-dimensional Hausdorff measure on $(K, d)$. 
	Since $\mu$ is $\alpha$-Ahlfors regular, the measure $\mu$ is also comparable to $\mathcal{H}^{\alpha}$. 
	Hence Proposition \ref{prop.BF}-(3) yields that $\widetilde{\mu}(\bord{K_{w}}) = 0$ for all $w \in W_{\ast}$. 

    Finally, we will prove \eqref{eq.weak}.
    Let $f \in \mcC(K)$ be not constant.
    Then there exists $N \ge 1$ such that $f|_{\mathbb{V}_{N}}$ is not constant.
    For a Borel measure $\nu$ on $K$ and $n \ge N$, define
    \[
    I(f, \nu) \coloneqq \int_{K}\int_{K}\abs{f(x) - f(y)}^{p}\,d\nu(y)d\nu(x). 
    \]
    Since $f|_{\mathbb{V}_{N}}$ is not constant and $f$ is bounded, we have $I\bigl(f, \widetilde{\mu}\bigr), I(f, \mu_{n}) \in (0, \infty)$.
    The weak convergence of $\mu_{n_{k}} \times \mu_{n_{k}}$ to $\widetilde{\mu} \times \widetilde{\mu}$ (see \cite{Bil}*{Theorem 2.8} for example) implies that $I(f, \mu_{n_k}) \to I\bigl(f, \widetilde{\mu}\bigr)$ as $n \to \infty$.
    Next, we define $\nu_{n}(dx \otimes dy) \coloneqq I(f, \mu_{n})^{-1}\abs{f(x) - f(y)}^{p}\,d\mu_{n}(x)d\mu_{n}(y)$.
    Since $f$ is continuous, we easily see that $\nu_{n_{k}}$ converges weakly to the probability measure $\nu$ on $K \times K$ given by 
    \[
    \nu(dx \otimes dy) \coloneqq I\bigl(f, \widetilde{\mu}\bigr)^{-1}\abs{f(x) - f(y)}^{p}\,d\widetilde{\mu}(x)d\widetilde{\mu}(y). 
    \]
    Thanks to $\widetilde{\mu}(\bord{K_{w}}) = 0$, it is immediate that 
    \[
    \nu\bigl(\partial(K \times U_{1}(x, a^{-m}))\bigr) 
    \le \nu(\bord{K} \times K) + \nu\bigl(K \times \bord{U_{1}(x, a^{-m})}\bigr) = 0. 
    \]
    (Note that $\bord{U_{1}(x, a^{-m})} \subseteq \bigcup_{w \in W_{m}}\bord{K_{w}}$.)
    The portmanteau theorem and the fact that $\mu \asymp \widetilde{\mu}$ deduce \eqref{eq.weak}.
\end{proof}

We next prove some $(p, p)$-Poincar\'{e} type inequalities (in the sense of Kumagai and Sturm), which plays the same role as Lemma \ref{lem.LBPI}.

\begin{lem}\label{lem.newPI}
    Suppose that $\beta > \alpha$ and $p > 1$.
    Then there exists a positive constant $C_{\ref{lem.newPI}}$ (depending only on $p, a, D, N_{\ast}$, $\beta$, $C_{\textup{AR}}$ and $\rho_p$) such that for every $n \in \bbZ_{\ge 0}$ and $f \in L^{p}(K, \mu)$,
        \begin{equation}\label{ineq.newPI}
        a^{\beta n}\sum_{w \in W_{n}}\sum_{x \in \mathbb{V}_{0}^{w}}\,\int_{K_{w}}\abs{f(x) - f(z)}^{p}\,d\mu(z) \le C_{\ref{lem.newPI}}\abs{f}_{\Lambda_{p, \infty}^{\beta/p}}^{p}. 
        \end{equation}
\end{lem}
\begin{proof}
    The proof is essentially the same as that of Lemma \ref{lem.LBPI}, so we omit the proof. 
    (Consider $x \in \mathbb{V}_{0}^{w}$ instead of $x_{w}$ in Lemma \ref{lem.LBPI} and note that $\max_{w \in W_{\ast}}\#\mathbb{V}_{0}^{w} = 2^{D}$.)
\end{proof}

Now, we are ready to prove Lemma \ref{lem.comparable}.

\noindent
\textit{Proof of Lemma \ref{lem.comparable}.}\,
    Let $\beta > \alpha$.
    We will prove the following two bounds:
    \begin{align*}
        &\text{Upper bound:}\; \sup_{n \in \bbN}a^{(\beta - \alpha)n}\cdot\mcE_{p}^{\,\bbG_{n}}(f) \le C_{\textup{upper}}\abs{f}_{\Lambda_{p, \infty}^{\beta/p}}^{p} \quad \text{for every $f \in L^{p}(K, \mu)$}, \\
        &\text{Lower bound:}\; \sup_{n \in \bbN}a^{(\beta_{p} - \alpha)n}\cdot\mcE_{p}^{\,\bbG_{n}}(f) \ge C_{\textup{lower}}\abs{f}_{\Lambda_{p, \infty}^{\beta/p}}^{p} \quad \text{for every $f \in \mcC(K)$},
    \end{align*}
    for some positive constants $C_{\textup{upper}}, C_{\textup{lower}}$ (without depending on $f$).
    Note that, from Theorem \ref{thm.main2}, the case $\beta = \beta_p$ in these bounds includes our assertion.

    \noindent
    \textbf{Upper bound.} 
    For $f \in L^{p}(K, \mu)$, we easily see that
    \[
    a^{(\beta - \alpha)n}\cdot\mcE_{p}^{\,\bbG_{n}}(f) \le 2^{p - 2}a^{\beta n}L_{\ast, \textup{modif}}\sum_{w \in W_{n}}\sum_{x \in \mathbb{V}_{0}^{w}}\int_{K_{w}}\abs{f(x) - f(z)}^{p}\,d\mu(z).
    \]
    Applying Lemma \ref{lem.newPI}, we get the desired bound with $C_{\textup{upper}} = 2^{p - 2}L_{\ast, \textup{modif}}C_{\ref{lem.newPI}}$ that depends only on $p, a, D, N_{\ast}, L_{\ast}, \beta, C_{\textup{AR}}, \rho_{p}$.

    \noindent
    \textbf{Lower bound.} 
    The desired bound can be obtained by applying \cite{Bod09}*{Theorem 3.18}, but we give a complete proof for reader's convenience. 
    Our proof is inspired by the arguments in those of \cite{Bod09}*{Lemma 4.2} and \cite{GY19}*{Theorem 3.6}. 
    
    Let $f \in \mcC(K)$ and let $n, m \in \bbZ_{\ge 0}$.
    Then we see that
    \begin{align}\label{est.lowernew}
        &\int_{K}\int_{U_{1}(x, a^{-n})}\abs{f(x) - f(y)}^{p}\,d\mu_{n + m}(y)d\mu_{n + m}(x) \\
        &\le L_{\ast}\sum_{w \in W_{n}}\int_{K_{w}}\int_{U_{1}(x, a^{-n})}\abs{f(x) - f(y)}^{p}\,d\mu_{n + m}(y)d\mu_{n + m}(x) \nonumber \\
        &\le L_{\ast}\sum_{\substack{v, w \in W_{n}; \\ d_{G_{n}}(v,w) \le 2}}\int_{K_{w}}\int_{K_{v}}\abs{f(x) - f(y)}^{p}\,d\mu_{n + m}(y)d\mu_{n + m}(x). \nonumber
    \end{align}
    Now, for $v, w \in W_{n}$, we define 
    \begin{align*}
    	I_{m}(f; v, w) 
    	&\coloneqq \int_{K_{w}}\int_{K_{v}}\abs{f(x) - f(y)}^{p}\,d\mu_{n + m}(y)d\mu_{n + m}(x) \\
    	&= (\#\mathbb{V}_{n + m})^{-2}\sum_{x \in \mathbb{V}_{m}^{w}}\sum_{y \in \mathbb{V}_{m}^{v}}\abs{f(x) - f(y)}^{p}. 
    \end{align*}
    For $v, w \in W_{n}$ with $d_{G_{n}}(v, w) \le 2$, we fix $\mathsf{t}[v, w] \in W_{n}$ such that 
    \[
    \max_{z \in \{ v, w \}}d_{G_{n}}\bigl(z, \mathsf{t}[v, w]\bigr) \le 1, 
    \]
    i.e. $K_{v} \cap K_{\mathsf{t}[v, w]} \neq \emptyset$ and $K_{w} \cap K_{\mathsf{t}[v, w]} \neq \emptyset$. 
    We also fix $x_{0} \in \mathbb{V}_{0}^{w} \cap \mathbb{V}_{0}^{\mathsf{t}[v, w]}$ and $x_{1} \in \mathbb{V}_{0}^{v} \cap \mathbb{V}_{0}^{\mathsf{t}[v, w]}$. 
    Then, by H\"{o}lder's inequality, 
    \begin{align*}
    	&I_{m}(f; v, w) \\
    	&\le 3^{p - 1}(\#\mathbb{V}_{n + m})^{-2}\sum_{x \in \mathbb{V}_{m}^{w}}\sum_{y \in \mathbb{V}_{m}^{v}}\bigl\{ \abs{f(x) - f(x_{0})}^{p} + \abs{f(x_{0}) - f(x_{1})}^{p} \\
    	&\hspace*{220pt}+ \abs{f(y) - f(x_{1})}^{p} \bigr\}. 
    \end{align*}
    Summing over suitable $v, w \in W_{n}$ and using \eqref{weak.number-pt}, we get 
    \begin{align}\label{Im-1}
    	&\sum_{\substack{v, w \in W_{n}; \\ d_{G_{n}}(v,w) \le 2}}I_{m}(f; v, w) \\
    	&\le 3^{p - 1}L_{\ast}^{2}(\#\mathbb{V}_{n + m})^{-2}(\#\mathbb{V}_{m})\sum_{w \in W_{n}}\max_{y \in \mathbb{V}_{0}^{w}}\sum_{x \in \mathbb{V}_{m}^{w}}\abs{f(x) - f(y)}^{p} \nonumber \\
    	&\le 2^{-D}3^{p - 1}L_{\ast}^{4}a^{-\alpha(2n + m)}\sum_{w \in W_{n}}\max_{y \in \mathbb{V}_{0}^{w}}\sum_{x \in \mathbb{V}_{m}^{w}}\abs{f(x) - f(y)}^{p}. \nonumber 
    \end{align}
    To obtain estimates of the sums in \eqref{Im-1}, we will find ``good sequences in $\mathbb{V}_{n + m}$''. 
    Let $z \in w \cdot W_{m}$ and define $z(k) \coloneqq [z]_{k} \in W_{k}$ for $k = n, \dots, n + m$. 
    Let $q_{k} \in \mathbb{V}_{0}^{z(k)}$ for $k = n, \dots, n + m$. 
    Then, by using H\"{o}lder's inequality repeatedly, 
    \begin{align}\label{LPBI.new-repeat}
    	&\abs{f(q_{n}) - f(q_{n + m})}^{p} \\
    	&\le 2^{p - 1}\bigl(\abs{f(q_{n}) - f(q_{n + 1})}^{p} + \abs{f(q_{n + 1}) - f(q_{n + m})}^{p}\bigr) \nonumber \\
    	&\le 2^{p - 1}\abs{f(q_{n}) - f(q_{n + 1})}^{p} \nonumber \\
    	&\hspace*{30pt}+ 2^{2(p - 1)}\bigl(\abs{f(q_{n + 1}) - f(q_{n + 2})}^{p} + \abs{f(q_{n + 2}) - f(q_{n + m})}^{p}\bigr) \nonumber \\
    	&\le \cdots \le 2^{p - 1}\sum_{k = 0}^{m - 1}2^{k(p - 1)}\abs{f(q_{n + k}) - f(q_{n + k + 1})}^{p}. \nonumber 
    \end{align}
    Since $q_{k}, q_{k + 1} \in \mathbb{V}_{1}^{z(k)}$ for any $k$, there exists $\bigl[x_{k}(0), \dots, x_{k}(L)\bigr]$ such that $x_{k}(j) \in \mathbb{V}_{1}^{z(k)}$, $x_{k}(0) = q_{k}$, $x_{k}(L) = q_{k + 1}$ and $d_{\mathbb{G}_{k + 1}}\bigl(x_{k}(j), x_{k}(j + 1)\bigr) \le 1$, where $L \coloneqq \diam(\mathbb{G}_{1}, d_{\mathbb{G}_{1}}) \le \#W_{1} = N_{\ast}$. 
    Again by H\"{o}lder's inequality, for $k = n, \dots, n + m$, 
    \[
    \abs{f(q_{n + k}) - f(q_{n + k + 1})}^{p} 
    \le L^{p - 1}\sum_{j = 0}^{L - 1}\abs{f\bigl(x_{n + k}(j)\bigr) - f\bigl(x_{n + k}(j + 1)\bigr)}^{p}. 
    \]
    We can assume that $\bigl[x_{k}(0), \dots, x_{k}(L)\bigr]$ is a simple path in the sense that 
    \[
    \max_{\substack{k = n, \dots, n + m; \\ (q, q') \in \mathbb{E}_{1}^{z(k)}}}\#\Bigl\{ j \Bigm| \text{$(q, q') = \bigl(x_{k}(j), x_{k}(j + 1)\bigr)$} \Bigr\} = 1. 
    \]
    Then we have 
    \[
    \abs{f(q_{n + k}) - f(q_{n + k + 1})}^{p} 
    \le 2L^{p - 1}\mathcal{E}_{p}^{\mathbb{G}_{1}^{z(n + k)}}(f) 
    \le 2N_{\ast}^{p - 1}\mathcal{E}_{p}^{\mathbb{G}_{1}^{z(n + k)}}(f).  
    \]
    Combining with \eqref{LPBI.new-repeat}, we get 
    \begin{equation}\label{repeat-arbitrary}
    	\abs{f(q_{n}) - f(q_{n + m})}^{p} \le 2^{p}N_{\ast}^{p - 1}\sum_{k = 0}^{m - 1}2^{k(p - 1)}\mathcal{E}_{p}^{\mathbb{G}_{1}^{z(n + k)}}(f).
    \end{equation}
    Note that choices of $w \in W_{n}$, $q_{n}$ and $q_{n + m}$ are arbitrary.
    Hence, by noting that 
    \[
    \max_{v \in W_{n + k}}\#\{ z \in w \cdot W_{m} \mid [z]_{n + k} = v \} = N_{\ast}^{m - k} = a^{\alpha(m - k)} \quad (k = 0, \dots, m)
    \]
    and using \eqref{repeat-arbitrary}, we obtain the following estimate of the sum in \eqref{Im-1}: 
    \begin{align*}
    	&\sum_{w \in W_{n}}\max_{y \in \mathbb{V}_{0}^{w}}\sum_{x \in \mathbb{V}_{m}^{w}}\abs{f(x) - f(y)}^{p} \\
    	&\le \sum_{w \in W_{n}}\max_{y \in \mathbb{V}_{0}^{w}}\sum_{z \in w \cdot W_{m}}\sum_{x \in \mathbb{V}_{0}^{z}}\abs{f(x) - f(y)}^{p} \\
    	&\le 2^{p + D}N_{\ast}^{p - 1}\sum_{w \in W_{n}}\sum_{z \in w \cdot W_{m}}\sum_{k = 0}^{m - 1}2^{k(p - 1)}\mathcal{E}_{p}^{\mathbb{G}_{1}^{z(n + k)}}(f) \\
    	&\le 2^{p + D}N_{\ast}^{p - 1}\sum_{k = 0}^{m - 1}2^{k(p - 1)}a^{\alpha(m - k)}\sum_{w \in W_{n}}\sum_{z' \in w \cdot W_{k}}\mathcal{E}_{p}^{\mathbb{G}_{1}^{z'}}(f) \\
    	&\le 2^{p + D}N_{\ast}^{p - 1}L_{\ast}a^{\alpha m}\sum_{k = 0}^{m - 1}2^{k(p - 1)}a^{-\alpha k}\mathcal{E}_{p}^{\mathbb{G}_{n + k + 1}}(f). 
    \end{align*}
    Hence \eqref{Im-1} becomes 
    \begin{align}\label{Im-2}
    	\sum_{\substack{v, w \in W_{n}; \\ d_{G_{n}}(v,w) \le 2}}I_{m}(f; v, w)  
    	\le C_{1}a^{-2\alpha n}\sum_{k = 1}^{m}2^{k(p - 1)}a^{-\alpha k}\mathcal{E}_{p}^{\mathbb{G}_{n + k}}(f). 
    \end{align}
    where $C_{1} = 2a(3N_{\ast})^{p - 1}L_{\ast}^{5}$. 
     
    From \eqref{est.lowernew}, \eqref{Im-2} and \eqref{weak.number-pt}, we see that 
    \begin{align}\label{Im-3}
    	&a^{(\alpha + \beta_{p})n}\int_{K}\int_{U_{1}(x, a^{-n})}\abs{f(x) - f(y)}^{p}\,d\mu_{n + m}(y)d\mu_{n + m}(x) \\
    	&\le C_{1}L_{\ast}a^{(\beta_{p} - \alpha)n}\sum_{k = 1}^{m}2^{k(p - 1)}a^{-\alpha k}\mathcal{E}_{p}^{\mathbb{G}_{n + k}}(f) \nonumber \\
    	&\le C_{1}L_{\ast}\left(\sup_{n \ge 0}a^{(\beta_{p} - \alpha)n}\mathcal{E}_{p}^{\mathbb{G}_{n}}(f)\right)\sum_{k = 0}^{m - 1}2^{k(p - 1)}a^{-\beta_{p} k} 
    	\le C_{2}\sup_{n \ge 0}\widetilde{\mathcal{E}}_{p}^{\,\mathbb{G}_{n}}(f). \nonumber 
    \end{align}
    where $C_{2} = C_{1}L_{\ast}\sum_{k = 0}^{\infty}a^{(p - 1 - \beta_{p})k} < +\infty$. 
    (Note that $p - \beta_{p} \le 0$ by Proposition \ref{prop.betap}.) 
    
    Letting $m \to \infty$ in \eqref{Im-3}, we see from Lemma \ref{lem.weak} that
    \[
    a^{(\alpha + \beta_{p})n}\int_{K}\int_{U_{1}(x, a^{-n})}\abs{f(x) - f(y)}^{p}\,d\mu(y)d\mu(x)
    \le c_{\ref{lem.weak}}^{-1}C_{2}\sup_{n \ge 0}\widetilde{\mathcal{E}}_{p}^{\,\mathbb{G}_{n}}(f),
    \]
    for any $n \in \bbZ_{\ge 0}$.
    Since $U_{1}(x, a^{-n}) \supseteq B_{d}(x, C_{\textup{AD}}^{-1}a^{-n})$ by Lemma \ref{lem.1-adapted}, we conclude that $\abs{f}_{\Lambda_{p, \infty}^{\beta_{p}/p}} \le C_{\textup{lower}}^{-1}\sup_{n \in \bbN}a^{(\beta_{p} - \alpha)n}\cdot\mcE_{p}^{\,\bbG_{n}}(f)$ for all $f \in \mcC(K)$, where $C_{\textup{lower}}$ depends only on $p, a, D, N_{\ast}, L_{\ast}, \rho_{p}, C_{\textup{AR}}$. \hspace{\fill}$\square$ 

\subsection{Proof of Theorem \ref{thm.main1.2}}\label{subsec:main1}

Now, we construct the desired $p$-energy $\mcE_{p}$ on $K$.

\noindent
\textit{Proof of Theorem \ref{thm.main1.2}.}\,
    From \eqref{eq.ss} and Lemma \ref{lem.comparable}, we immediately conclude that
    \begin{equation}\label{ss-domain}
    	\mcF_{p} = \{ f \in \mcC(K) \mid f \circ F_{i} \in \mcF_{p} \text{ for any $i \in S$} \}.
    \end{equation}
    
    Since $\mcF_{p}$ is separable by Theorem \ref{thm.sep}, there exists a countable dense subset $\mcF_{p}^{0} = \{ f_j \}_{j \ge 1}$ of $\mcF_{p}$.
    By virtue of Lemma \ref{lem.comparable}, we know that 
    \begin{equation}\label{new.comparable}
    	C_{\ref{lem.comparable}}^{-1}\abs{f}_{\mcF_{p}}^{p} 
    	\le \sup_{n \ge 0}\widetilde{\mathcal{E}}_{p}^{\,\mathbb{G}_{n}}(f) 
    	\le C_{\ref{lem.comparable}}\abs{f}_{\mcF_{p}}^{p}. 
    \end{equation}
    In particular, for each $f \in \mathcal{F}_{p}$, a sequence $\bigl\{ \widetilde{\mathcal{E}}_{p}^{\,\mathbb{G}_{n}}(f) \bigr\}_{n \ge 0}$ is bounded. 
    By the diagonal argument, we can take a subsequence $\{ n_k \}_{k \ge 1}$ such that 
    \[
    \text{$\Bigl\{ \widetilde{\mcE}_{p}^{\,\mathbb{G}_{n_{k}}}(f_j) \Bigr\}_{k \ge 1}$ converges for each $j \ge 1$ as $k \to \infty$.}
    \]
    For simplicity, we write $\mathscr{E}_{p, n}$ to denote $\widetilde{\mathcal{E}}_{p}^{\,\mathbb{G}_{n}}$. 
    Note that $\mathscr{E}_{p, n}(\,\cdot\,)^{1/p}$ is a semi-norm on $\mathbb{R}^{\mathbb{V}_{n}}$. 
    Let $f \in \mcF_p$, let $\varepsilon > 0$ and let $f_{\ast} \in \mcF_{p}^{0}$ such that $\norm{f - f_{\ast}}_{\mathcal{F}_p} < \varepsilon$.
    For $k, l \ge 1$, by using the triangle inequality of $\mathscr{E}_{p, n}(\,\cdot\,)^{1/p}$, 
    \begin{align*}
        &\abs{\mathscr{E}_{p, n_k}(f)^{1/p} - \mathscr{E}_{p, n_l}(f)^{1/p}} \\
        &\le \abs{\mathscr{E}_{p, n_k}(f)^{1/p} - \mathscr{E}_{p, n_k}(f_{\ast})^{1/p}}  \\
        &\hspace{30pt}+ \abs{\mathscr{E}_{p, n_k}(f_{\ast})^{1/p} - \mathscr{E}_{p, n_l}(f_{\ast})^{1/p}} + \abs{\mathscr{E}_{p, n_l}(f)^{1/p} - \mathscr{E}_{p, n_l}(f_{\ast})^{1/p}} \\
        &\le 2C_{\ref{lem.comparable}}\abs{f - f_{\ast}}_{\mcF_{p}} + \abs{\mathscr{E}_{p, n_k}(f_{\ast})^{1/p} - \mathscr{E}_{p, n_l}(f_{\ast})^{1/p}},
    \end{align*}
    and hence we obtain $\varlimsup_{k \wedge l \to \infty}\abs{\mathscr{E}_{p, n_k}(f)^{1/p} - \mathscr{E}_{p, n_l}(f)^{1/p}} \le 2C_{\ref{lem.comparable}}\varepsilon$.
    This shows that a sequence $\bigl\{ \mathscr{E}_{p, n_k}(f) \bigr\}_{k \ge 1}$ is Cauchy for any $f \in \mathcal{F}_{p}$.
    We conclude that $\lim_{k \to \infty}\mathscr{E}_{p, n_k}(f)$ exists for all $f \in \mathcal{F}_p$.
    Now, define 
    \[
    \mathscr{E}(f) \coloneqq \lim_{k \to \infty}\mathscr{E}_{p, n_k}(f), \quad f \in \mathcal{F}_{p}. 
    \]
    
    Next, for $f \in \mathcal{C}(K)$ and $m \in \mathbb{N}$, we define 
    \begin{equation*}
        \overline{\mcE}_{p, m}(f) \coloneqq \rho_{p}^{m}\sum_{w \in W_{m}}\mathscr{E}(f \circ F_{w}), 
    \end{equation*}
    and 
    \begin{equation*}
        \widehat{\mcE}_{p, m}(f) \coloneqq \frac{1}{m}\sum_{l = 0}^{m - 1}\overline{\mcE}_{p, l}(f), 
    \end{equation*}
    where we set $\overline{\mathcal{E}}_{p, 0}(f) \coloneqq \mathscr{E}(f)$.  
    Note that 
    \begin{align}\label{eq.originalEp}
    	\overline{\mcE}_{p, m}(f) 
    	&= \lim_{k \to \infty}\rho_{p}^{m}\sum_{w \in W_{m}}\widetilde{\mathcal{E}}_{p}^{\,\mathbb{G}_{n_{k}}}(f \circ F_{w}) \\
    	&= \lim_{k \to \infty}\frac{\rho_{p}^{m + n_{k}}}{2}\sum_{w \in W_{m}}\sum_{(x, y) \in \mathbb{E}_{n_{k}}^{w}}\abs{f(x) - f(y)}^{p}, \nonumber 
    \end{align}
    and that 
    \begin{equation}\label{eq.originalEp2}
    	\widehat{\mcE}_{p, m}(f) 
    	= \lim_{k \to \infty}\frac{1}{2m}\sum_{l = 0}^{m - 1}\rho_{p}^{l + n_{k}}\sum_{w \in W_{l}}\sum_{(x, y) \in \mathbb{E}_{n_{k}}^{w}}\abs{f(x) - f(y)}^{p}. 
    \end{equation}
    
    From \eqref{eq.originalEp}, Lemmas \ref{lem.comparable}, \ref{lem.newwm} and \ref{lem.ss}-(3), for any $m \in \mathbb{Z}_{\ge 0}$, we have 
    \begin{equation*}
    	(C_{\ref{lem.comparable}}C_{\textup{WM,modif}})^{-1}\abs{f}_{\mathcal{F}_{p}}^{p}
    \le \overline{\mathcal{E}}_{p, m}(f) 
    \le L_{\ast}C_{\ref{lem.comparable}}\abs{f}_{\mathcal{F}_{p}}^{p}, \quad f \in \mathcal{F}_{p}. 
    \end{equation*}
    From this comparability and the definition of $\widehat{\mathcal{E}}_{p, m}$, it is immediate that 
    \begin{equation}\label{pre-comparable}
    	(C_{\ref{lem.comparable}}C_{\textup{WM,modif}})^{-1}\abs{f}_{\mathcal{F}_{p}}^{p}
    	\le \widehat{\mathcal{E}}_{p, m}(f) 
    	\le L_{\ast}C_{\ref{lem.comparable}}\abs{f}_{\mathcal{F}_{p}}^{p}, \quad f \in \mathcal{F}_{p}. 
    \end{equation}
    Hence $\bigl\{ \widehat{\mathcal{E}}_{p, m}(f)  \bigr\}_{m \ge 1}$ is bounded for each $f \in \mathcal{F}_{p}$. 
    We also know that $\widehat{\mathcal{E}}_{p, m}(\,\cdot\,)^{1/p}$ satisfies the triangle inequality from the expression \eqref{eq.originalEp2}. 
    From the argument used to define $\mathscr{E}$, we conclude that there exists a subsequence $\{ m_{j} \}_{j \ge 1}$ such that 
    \[
    \text{$\Bigl\{ \widehat{\mathcal{E}}_{p, m_{j}}(f) \Bigr\}_{j \ge 1}$ converges for any $f \in \mathcal{F}_{p}$.}
    \]
    Define 
    \begin{align}\label{representation}
    	 \mathcal{E}_{p}(f) 
    	 &\coloneqq \lim_{j \to \infty}\widehat{\mathcal{E}}_{p, m_{j}}(f) \\
    	 &= \lim_{j \to \infty}\lim_{k \to \infty}\frac{1}{2m_{j}}\sum_{l = 0}^{m_{j} - 1}\rho_{p}^{l + n_{k}}\sum_{w \in W_{l}}\sum_{(x, y) \in \mathbb{E}_{n_{k}}^{w}}\abs{f(x) - f(y)}^{p}, \quad f \in \mathcal{F}_{p}. \nonumber 
    \end{align}
    We shall show that $\mathcal{E}_{p}$ have the desired properties. 
    
    First, we will see that $\mathcal{E}_{p}(\,\cdot\,)^{1/p}$ is a semi-norm, which is comparable to $\abs{\,\cdot\,}_{\mathcal{F}_{p}}$, and satisfies Clarkson's inequality (see Definition \ref{dfn.clarkson}). 
    Obviously, we have from \eqref{pre-comparable} that 
    \begin{equation}\label{comparable}
    	(C_{\ref{lem.comparable}}C_{\textup{WM,modif}})^{-1}\abs{f}_{\mathcal{F}_{p}}^{p}
    	\le \mathcal{E}_{p}(f) 
    	\le L_{\ast}C_{\ref{lem.comparable}}\abs{f}_{\mathcal{F}_{p}}^{p}, \quad f \in \mathcal{F}_{p}. 
    \end{equation}
    By virtue of the expression \eqref{eq.originalEp2}, we can regard $\widehat{\mathcal{E}}_{p, m}(\,\cdot\,)^{1/p}$ as a limit of $\ell^{p}$-norms on suitable finite sets, i.e. 
    \[
    I_{m, k}(f) \coloneqq \left(\frac{1}{2m}\sum_{l = 0}^{m - 1}\rho_{p}^{l + n_{k}}\sum_{w \in W_{l}}\sum_{(x, y) \in \mathbb{E}_{n_{k}}^{w}}\abs{f(x) - f(y)}^{p}\right)^{1/p}, \quad f \in \mathcal{F}_{p}, 
    \]
    coincides with suitable $\ell^{p}$-norm on $\bigsqcup_{l = 0}^{m - 1}\bigsqcup_{w \in W_{l}}\mathbb{E}_{n_{k}}^{w}$ and 
    \[
    \widehat{\mathcal{E}}_{p, m}(f)^{1/p} = \lim_{k \to \infty}I_{m, k}(f). 
    \]
    By Clarkson's inequality for $\ell^{p}$ spaces, for $m, k \in \mathbb{N}$ and $f, g \in \mathcal{F}_{p}$, 
    \begin{itemize}
    	\item if $p \le 2$, then 
    	\begin{equation}\label{pre-clarkson1}
    		I_{m, k}(f + g)^{\frac{p}{p - 1}} + I_{m, k}(f - g)^{\frac{p}{p - 1}} 
    		\le 2\bigl(I_{m, k}(f)^{p} + I_{m, k}(g)^{p}\bigr)^{\frac{1}{p - 1}}; 
    	\end{equation}
    	\item if $p \ge 2$, then 
    	\begin{equation}\label{pre-clarkson2}
    		I_{m, k}(f + g)^{p} + I_{m, k}(f - g)^{p} 
    		\le 2^{p - 1}\bigl(I_{m, k}(f)^{p} + I_{m, k}(g)^{p}\bigr). 
    	\end{equation}
    \end{itemize}
    Passing limits in \eqref{pre-clarkson1} and \eqref{pre-clarkson2}, we get Clarkson's inequalities of $\mathcal{E}_{p}(\,\cdot\,)^{1/p}$. 
    
    Next, we will show that $\mathcal{E}_{p}$ satisfies the properties (1)-(6). 

    \noindent
    (1)
    Let $f \in \mcC(K)$ be constant.
    Then we easily see that $\widetilde{\mcE}_{p}^{\,\bbG_{n}}(f \circ F_{w}) = 0$ for any $n \in \bbZ_{\ge 0}$ and $w \in W_{\ast}$.
    From \eqref{eq.originalEp} and \eqref{eq.originalEp2}, we have $\overline{\mathcal{E}}_{p, m}(f) = \widehat{\mathcal{E}}_{p, m}(f) = 0$ and thus $\mathcal{E}_{p}(f) = 0$.
    Conversely, if $\mcE_{p}(f) = 0$, then, by \eqref{pre-comparable} and Lemma \ref{lem.comparable}, $\widetilde{\mcE}_{p}^{\,\bbG_{n}}(f) = 0$ for any $n \in \bbZ_{\ge 0}$.
    Lemma \ref{lem.ss}-(1) yields that $f|_{\bigcup_{n \ge 0}\mathbb{V}_{n}}$ is constant.
    Since $\bigcup_{n \ge 0}\mathbb{V}_{n}$ is dense in $K$, we conclude that $f$ is constant.
    Next, let $f \in \mcF_p$ and let $a \in \bbR$.
    Then it is immediate that $\widetilde{\mcE}_{p}^{\,\bbG_{n}}(f) = \widetilde{\mcE}_{p}^{\,\bbG_{n}}(f + a\indicator{K})$ for any $n \in \bbZ_{\ge 0}$.
    Hence $\mcE_{p}(f) = \mcE_{p}(f + a\indicator{K})$.

    \noindent
    (2)
    This is proved in Theorem \ref{thm.core}.

    \noindent
    (3)
    Note that $\norm{\varphi \circ f}_{L^{p}}^{p} \le 2^{p - 1}(\abs{\varphi(0)}^{p} + \norm{f}_{L^{p}}^{p})$ whenever $\varphi\colon\bbR \to \bbR$ satisfies $\mathrm{Lip}(\varphi) \le 1$.
    Since $\widetilde{\mcE}_{p}^{\,\bbG_n}$ has the Markov property (Lemma \ref{lem.ss}-(4)), we see that $\mathscr{E}, \overline{\mcE}_{p, m}$ and $\widehat{\mathcal{E}}_{p, m}$ also have the same property.
    This immediately deduces the Markov property of $\mcE_{p}$.

    \noindent
    (4)
    For the same reason as (3), $\mcE_p$ has the required symmetries (see Lemma \ref{lem.ss}-(2)).

    \noindent
    (5)
    For any $m \in \mathbb{Z}_{\ge 0}$, the definition of $\overline{\mathcal{E}}_{p, m}$ implies that 
    \begin{align*}
    	\overline{\mathcal{E}}_{p. m + 1}(f) 
    	&= \rho_{p}^{m + 1}\sum_{w \in W_{m + 1}}\mathscr{E}(f \circ F_{w}) \\
    	&= \rho_{p}^{m + 1}\sum_{i \in S}\sum_{w \in W_{m}}\mathscr{E}(f \circ F_{i} \circ F_{w}) 
    	= \rho_{p}\sum_{i \in S}\overline{\mathcal{E}}_{p. m}(f \circ F_{i}). 
    \end{align*}
    Hence, we see that 
    \begin{align*}
    	\rho_{p}\sum_{i \in S}\mathcal{E}_{p}(f \circ F_{i}) 
    	&= \rho_{p}\sum_{i \in S}\lim_{j \to \infty}\widehat{\mathcal{E}}_{p, m_{j}}(f \circ F_{i}) \\
    	&= \lim_{j \to \infty}\frac{1}{m_{j}}\sum_{l = 0}^{m_{j} - 1}\rho_{p}\sum_{i \in S}\overline{\mathcal{E}}_{p, l}(f \circ F_{i}) \\
    	&= \lim_{j \to \infty}\frac{1}{m_{j}}\sum_{l = 0}^{m_{j} - 1}\overline{\mathcal{E}}_{p, l + 1}(f) \\
    	&= \lim_{j \to \infty}\frac{1}{m_{j}}\left(\sum_{l = 0}^{m_{j} - 1}\overline{\mathcal{E}}_{p, l}(f) + \overline{\mathcal{E}}_{p, m_{j}}(f) - \overline{\mathcal{E}}_{p, 0}(f) \right) 
    	= \mathcal{E}_{p}(f). 
    \end{align*}
    
    \noindent
    (6)
    Let $A_{1} \coloneqq \supp[f]$ and let $A_{2} \coloneqq \supp[g - a\indicator{K}]$.
    Since $\mathrm{dist}(A_{1}, A_{2}) > 0$, there exists $N \in \bbN$ such that $\sup_{n \ge N}\max_{w \in W_{n}}\diam(K_{w}, d) < \mathrm{dist}(A_{1}, A_{2})$.
    Then $f \circ F_{w}$ or $(g - a\indicator{K}) \circ F_{w}$ is equal to $0$ for any $w \in W_n$ and $n \ge N$.
    From the self-similarity, we deduce that, for $n \ge N$,
    \begin{align*}
        \mcE_{p}(f + g)
        &= \mcE_{p}(f + g - a\indicator{K}) \\
        &= \rho_{p}^{n}\sum_{w \in W_{n}}\mcE_{p}\bigl(f \circ F_{w} + (g - a\indicator{K}) \circ F_{w}\bigr) \\
        &= \rho_{p}^{n}\sum_{w \in W_{n}[A_{1}]}\mcE_{p}\bigl(f \circ F_{w}\bigr) + \rho_{p}^{n}\sum_{w \in W_{n}[A_{2}]}\mcE_{p}\bigl((g - a\indicator{K}) \circ F_{w}\bigr) \\
        &= \mcE_{p}(f) + \mcE_{p}(g). 
    \end{align*}
    We complete the proof. \qed

\begin{rmk}\label{rmk.beyond}
    \begin{itemize}
        \item [(1)] Since $\mcE_{p}(\,\cdot\,)^{1/p}$ satisfies Clarkson's inequality,  $\mcE_{p}(\,\cdot\,)^{1/p}$ is strictly convex, that is, if $\lambda \in (0, 1)$ $f, g \in \mcF_{p}$ with $f - g$ is not constant, then
        \[
        \mcE_{p}\bigl(\lambda f + (1-\lambda)g\bigr)^{1/p} < \lambda\,\mcE_{p}(f)^{1/p} + (1 - \lambda)\mcE_{p}(g)^{1/p}.
        \]
        The convexity of $\mcE_{p}(\,\cdot\,)$ is also immediate from the construction.
        Moreover, from the convexity of $x \mapsto \abs{x}^{p}$, we can show that $\mcE_{p}(\,\cdot\,)$ is strictly convex.
        \item [(2)] The framework in \cite{KZ92} includes not only the standard planar Sierpi\'{n}ski carpet but also Sierpi\'{n}ski gaskets and other self-similar sets (nested fractals for example).
        A recent paper by Kigami \cite{Kig21+} gives a more general framework to construct canonical $p$-energy on $p$-conductively homogeneous compact metric spaces, which includes new results even when $p = 2$.
    \end{itemize}
\end{rmk}

For future work, it is useful to provide the following estimate concerning products of functions in $\mcF_{p}$.
When $p = 2$, this result is standard (see \cite{FOT}*{Theorem 1.4.2-(ii)} for example).
See also \cite{Kig21+}*{Lemma 6.17-(2)}.

\begin{prop}\label{prop.prod}
    For any $f, g \in \mcF_p$,
    \begin{equation*}
        \mcE_{p}(f \cdot g) \le 2^{p - 1}\bigl(\norm{f}_{\mcC(K)}^{p}\mcE_{p}(g) + \norm{g}_{\mcC(K)}^{p}\mcE_{p}(f)\bigr).
    \end{equation*}
    In particular, $f \cdot g \in \mcF_{p}$.
\end{prop}
\begin{proof}
    For any $n \in \bbN$, we have
    \begin{align*}
        \mcE_{p}^{\bbG_{n}}(f \cdot g)
        &\le 2^{p - 1}\frac{1}{2}\sum_{(x, y) \in \mathbb{E}_{n}}\bigl(\abs{g(x)}^{p}\abs{f(x) - f(y)}^{p} + \abs{f(y)}^{p}\abs{g(x) - g(y)}^{p}\bigr) \\
        &\le 2^{p - 1}\bigl(\norm{g}_{\mcC(K)}^{p}\mcE_{p}^{\bbG_{n}}(f) + \norm{f}_{\mcC(K)}^{p}\mcE_{p}^{\bbG_{n}}(g)\bigr).
    \end{align*}
    In view of the proof of Theorem \ref{thm.main1.2}, this immediately implies our assertion.
\end{proof}

\section{Construction and basic properties of $\mcE_{p}$-energy measures}\label{sec:em}
In this section, we construct $\mcE_{p}$-energy measures in the same way as Hino's work \cite{Hin05}*{Lemma 4.1}.
We also investigate some properties, especially the chain rule of $\mcE_{p}$-energy measures.
Let $(K, S, \{ F_{i} \}_{i \in S}) = \GSC(D, a, S)$ be a generalized Sierpi\'{n}ski carpet.  
We always suppose that Assumption \ref{assum.lowdim} holds in this section. 

Let $(\mcE_{p}, \mcF_{p})$ be the $p$-energy in subsection \ref{subsec:main1}, and let $f \in \mcF_{p}$.
For each $n \ge 0$, we define a measure $\mathfrak{m}_{\langle f \rangle}^{p, n}$ on $W_{n}$ (equipped with the $\sigma$-algebra $2^{W_{n}}$) by setting
\[
\mathfrak{m}_{\langle f \rangle}^{p, n}(A) \coloneqq \rho_{p}^{n}\sum_{w \in A}\mcE_{p}\bigl(f \circ F_{w}\bigr), \quad A \subseteq W_{n}.
\]
Then we easily see that the total mass of $\mathfrak{m}_{\langle f \rangle}^{p, n}$ is equal to $\mcE_{p}(f) < \infty$.
Furthermore, it follows from the self-similarity of $\mcE_{p}$ that, for any $A \subseteq W_{n}$,
\begin{align*}
    \mathfrak{m}_{\langle f \rangle}^{p, n + 1}(A \cdot W_{1})
    = \rho_{p}^{n}\sum_{w \in A}\rho_{p}\sum_{i \in S}\mcE_{p}\bigl(f \circ F_{wi}\bigr)
    = \rho_{p}^{n}\sum_{w \in A}\mcE_{p}\bigl(f \circ F_{w}\bigr)
    = \mathfrak{m}_{\langle f \rangle}^{p, n}(A).
\end{align*}
Therefore, $\bigl\{ \mathfrak{m}_{\langle f \rangle}^{p, n} \bigr\}_{n \ge 0}$ satisfies the consistency condition, and hence Kolmogorov's extension theorem (see \cite{Dud}*{Theorem 12.1.2} for example) yields a unique Borel finite measure $\mathfrak{m}_{\langle f \rangle}^{p}$ on $\Sigma$ such that $\mathfrak{m}_{\langle f \rangle}^{p}(\Sigma_{w}) = \mathfrak{m}_{\langle f \rangle}^{p, \abs{w}}\bigl(\{ w \}\bigr)$ for every $w \in W_{\ast}$.
Then we define $\mu_{\langle f \rangle}^{p} \coloneqq \pi_{\ast}\mathfrak{m}_{\langle f \rangle}^{p}$, where $\pi$ is the natural projection (recall Proposition \ref{prop.pi}).
Note that $\mu_{\langle f \rangle}^{p}$ is Borel regular (see \cite{Dud}*{Theorem 7.1.3} for example).

\begin{rmk}
	We used only the self-similar property of $\mathcal{E}_{p}$ to verify the consistency condition of Kolmogorov's extension theorem. 
	Therefore, our definition of $\mu^{p}_{\langle \,\cdot\, \rangle}$ works if we have a self-similar $p$-energy on a self-similar set.  
	In other words, specific structures of generalized Sierpi\'{n}ski carpets except their self-similarities are irrevalent for the above approach. 
\end{rmk}

Before proving Theorem \ref{thm.main3}, we observe two fundamental properties of $\mu_{\langle f \rangle}^{p}$.

\begin{prop}
    Let $f \in \mcF_{p}$.
    Then $\mu_{\langle f \rangle}^{p} \equiv 0$ if and only if $f$ is constant.
\end{prop}
\begin{proof}
    It is immediate from $\mu_{\langle f \rangle}^{p}(K) = \mcE_{p}(f)$ and Theorem \ref{thm.main1.2}-(1).
\end{proof}

\begin{prop}\label{prop.peneapp}
    For every $f, g \in \mcF_{p}$ and $A \in \mcB(K)$, it holds that
    \begin{equation}\label{ineq.pene-semi}
        \abs{\mu_{\langle f \rangle}^{p}(A)^{1/p} - \mu_{\langle g \rangle}^{p}(A)^{1/p}} \le \mu_{\langle f - g \rangle}^{p}(A)^{1/p}.
    \end{equation}
    In particular, if $f_{n} \in \mcF_{p}$ converges to $f$ in $\mcF_{p}$, then $\mu_{\langle f_{n} \rangle}^{p}(A) \to \mu_{\langle f \rangle}^{p}(A)$ for every $A \in \mcB(K)$.
\end{prop}
\begin{proof}
    Since $\mu_{\langle f \rangle}^{p}$ is Borel regular, it will suffice to prove \eqref{ineq.pene-semi} when $A$ is a closed set.
    Let $A$ be a closed set of $K$ and define $C_{l} \coloneqq \{ w \in W_{l} \mid \Sigma_{w} \cap \pi^{-1}(A) \neq \emptyset \}$ for each $l \in \bbN$.
    Then, as proved in \cite{Hin05}*{proof of Lemma 4.1}, one can show that $\bigl\{ \Sigma_{C_{l}} \bigr\}_{l \ge 1}$ is a decreasing sequence and $\bigcap_{l \in \bbN}\Sigma_{C_{l}} = \pi^{-1}(A)$, where $\Sigma_{C_l} \coloneqq \{ \omega \in \Sigma \mid [\omega]_{l} \in C_{l} \}$.

    Recall that $\mcE_{p}$ is obtained as a subsequential limit of $\bigl\{ \widehat{\mcE}_{p, n} \bigr\}_{n \ge 1}$, where $\widehat{\mcE}_{p, n}$ is given in \eqref{eq.originalEp2}.
    We may assume that $\mcE_{p}(f) = \lim_{n \to \infty}\widehat{\mcE}_{p, n}(f)$ for every $f \in \mcF_{p}$.
    For each $l, n \in \bbN$, we can regard $\left(\sum_{w \in C_{l}}\widehat{\mcE}_{p, n}\bigl(f \circ F_{w}\bigr)\right)^{1/p}$ as a limit of suitable $\ell^{p}$-norms on finite sets.
    Consequently, we have that
    \begin{align*}
        &\abs{\left(\rho_{p}^{l}\sum_{w \in C_{l}}\widehat{\mcE}_{p, n}\bigl(f \circ F_{w}\bigr)\right)^{1/p} - \left(\rho_{p}^{l}\sum_{w \in C_{l}}\widehat{\mcE}_{p, n}\bigl(g \circ F_{w}\bigr)\right)^{1/p}} \\
        &\le \left(\rho_{p}^{l}\sum_{w \in C_{l}}\widehat{\mcE}_{p, n}\bigl((f - g) \circ F_{w}\bigr)\right)^{1/p}.
    \end{align*}
    Letting $n \to \infty$ in this inequality, we conclude that
    \[
    \abs{\mathfrak{m}_{\langle f \rangle}^{p}\bigl(\Sigma_{C_{l}}\bigr)^{1/p} - \mathfrak{m}_{\langle g \rangle}^{p}\bigl(\Sigma_{C_{l}}\bigr)^{1/p}} \le \mathfrak{m}_{\langle f - g \rangle}^{p}\bigl(\Sigma_{C_{l}}\bigr)^{1/p},
    \]
    for any $l \in \bbN$.
    Letting $l \to \infty$, we obtain \eqref{ineq.pene-semi} for any closed set $A$.
\end{proof}

First, we prove Theorem \ref{thm.main3}-(2) and (3)
\begin{thm}[Theorem \ref{thm.main3}-(2)]\label{thm.peneMarkov}
  For any $\Phi \in \mcC^{1}(\bbR)$ and $f \in \mcF_p$,
  \begin{equation}\label{eq.pene-chain}
      d\mu_{\langle \Phi \circ f \rangle}^{p} = \abs{\Phi' \circ f}^{p}d\mu_{\langle f \rangle}^{p}.
  \end{equation}
\end{thm}
\begin{proof}
  Note that $\Phi \circ f \in \mcF_p$ for any $f \in \mcF_p$ by the Markov property of $\mcE_p$ (Theorem \ref{thm.main1.2}-(c)) and the compactness of $f(K)$.

  First, we prove \eqref{eq.pene-chain} when $\Phi$ is a polynomial.
  Let $\Phi$ be a polynomial and let $f \in \mcF_{p}$.
  Since $f(K)$ is compact, for any $\varepsilon > 0$ there exists $N(\varepsilon) \in \bbN$ such that
  \[
  \abs{\,\abs{\frac{\Phi(f(y)) - \Phi(f(y'))}{f(y) - f(y')}}^{p} - \abs{\Phi'(f(x))}^{p}} < \varepsilon
  \]
  whenever $x, y, y' \in K_{w}$ for some $w \in W_{n}$ with $n \ge N(\varepsilon)$.
  Thus, for any $m \in \bbN$, $n \ge N(\varepsilon)$, $(y, y') \in \mathbb{E}_{m}$, $w \in W_{n}$ and $x \in K_{w}$,
  \begin{align*}
      &\abs{\abs{\Phi(f(F_{w}(y))) - \Phi(f(F_{w}(y')))}^{p} - \abs{\Phi'(x)}^{p}\abs{f(F_{w}(y)) - f(F_{w}(y'))}^{p}} \\
      &\le \varepsilon\abs{f(F_{w}(y)) - f(F_{w}(y'))}^{p},
  \end{align*}
  and we conclude that
  \[
      \abs{\widetilde{\mcE}_{p}^{\,\bbG_{m}}\bigl((\Phi \circ f) \circ F_{w}\bigr) - \abs{\Phi'(f(x))}^{p}\widetilde{\mcE}_{p}^{\,\bbG_{m}}(f \circ F_{w})}
      \le \varepsilon\,\widetilde{\mcE}_{p}^{\,\bbG_{m}}(f\circ F_{w}).
  \]
  Taking an appropriate limit, we have
  \begin{equation*}
      \abs{\mathscr{E}_{p}\bigl((\Phi \circ f) \circ F_{w}\bigr) - \abs{\Phi'(f(x))}^{p}\mathscr{E}_{p}(f\circ F_{w})} \le \varepsilon\,\mathscr{E}_{p}(f\circ F_{w}),
  \end{equation*}
  whenever $w \in W_{n}$, $n \ge N(\varepsilon)$ and $f \in \mathcal{F}_{p}$, where $\mathscr{E}$ is the same as in the proof of Theorem \ref{thm.main1.2}.
  Combining with the definition of $\mathcal{E}_{p}$, we obtain 
  \begin{equation}\label{est.chainss}
      \abs{\mathcal{E}_{p}\bigl((\Phi \circ f) \circ F_{w}\bigr) - \abs{\Phi'(f(x))}^{p}\mathcal{E}_{p}(f\circ F_{w})} \le \varepsilon\,\mathcal{E}_{p}(f\circ F_{w}),
  \end{equation}
  For any $m \in \bbN$, $w \in W_{m}$ and $n \ge N(\varepsilon)$, we see from the self-similarity of $\mcE_{p}$ that
  \begin{align*}
      &\abs{\mathfrak{m}_{\langle \Phi(f) \rangle}^{p}(\Sigma_{w}) - \int_{\Sigma_{w}}\abs{\Phi'(f(\pi(\omega)))}^{p}\,d\mathfrak{m}_{\langle f \rangle}^{p}(\omega)} \\
      &\le \sum_{v \in W_{n}}\int_{\Sigma_{wv}}\abs{\frac{\mcE_{p}\bigl(F_{wv}^{\ast}(\Phi(f))\bigr)}{\mcE_{p}\bigl(F_{wv}^{\ast}f\bigr)} - \abs{\Phi'(f(\pi(\omega)))}^{p}}\,d\mathfrak{m}_{\langle f \rangle}^{p}(\omega) \\
      &\le \varepsilon\,\mathfrak{m}_{\langle f \rangle}^{p}(\Sigma_{w}).
  \end{align*}
  Hence, for any $w \in W_{\#}$,
  \[
  \mathfrak{m}_{\langle \Phi(f) \rangle}^{p}(\Sigma_{w}) = \int_{\Sigma_{w}}\abs{\Phi'(f(\pi(\omega)))}^{p}\,d\mathfrak{m}_{\langle f \rangle}^{p}(\omega).
  \]
  By Dynkin's $\pi$-$\lambda$ theorem, we get $\mathfrak{m}_{\langle \Phi(f) \rangle}^{p}(d\omega) = \abs{\Phi'(f(\pi(\omega)))}\mathfrak{m}_{\langle f \rangle}^{p}(d\omega)$.
  By the change of variable formula (see \cite{Dud}*{Theorem 4.1.11} for example), we have \eqref{eq.pene-chain} in this case.

  Next, let $\Phi \in \mcC^{1}(\bbR)$.
  Then, by applying Weierstrass' approximation theorem for $\Phi'$, we can obtain a sequence of polynomials $\{ \Phi_{k} \}_{k \ge 1}$ with $\Phi_{k}(0) = \Phi(0)$ such that $\Phi_{k} \to \Phi$ and $\Phi_{k}' \to \Phi'$ uniformly on $f(K)$.
  By the argument in the last paragraph, we know that
  \begin{equation}\label{eq.chainapp}
      \mu_{\langle \Phi_{k}(f) \rangle}^{p}(dx) = \abs{\Phi_{k}'(f(x))}^{p}\mu_{\langle f \rangle}^{p}(dx),
  \end{equation}
  for every $k \in \bbN$.
  For any $\widetilde{\Phi} \in \mcC^{1}(\bbR)$, it is immediate that
  \[
  \abs{\widetilde{\Phi}(f(F_{w}(y))) - \widetilde{\Phi}(f(F_{w}(y')))} \le \sup_{s \in f(K_{w})}\abs{\widetilde{\Phi}'(s)}\abs{f(F_{w}(y)) - f(F_{w}(y'))},
  \]
  and hence,
  \[
  \widetilde{\mcE}_{p}^{\,\bbG_{m}}\bigl(F_{w}^{\ast}\bigl(\widetilde{\Phi} \circ f\bigr)\bigr) \le \sup_{s \in f(K_{w})}\abs{\widetilde{\Phi}'(s)}^{p}\widetilde{\mcE}_{p}^{\,\bbG_{m}}(f\circ F_{w}).
  \]
  From the construction in subsection \ref{subsec:main1} and the self-similarity of $\mcE_{p}$, we get
  \begin{equation}\label{ineq.localPhi}
      \mcE_{p}\bigl(\widetilde{\Phi} \circ f\bigr) \le \rho_{p}^{n}\,\sum_{w \in W_{n}}\sup_{s \in f(K_{w})}\abs{\widetilde{\Phi}'(s)}^{p}\widetilde{\mcE}_{p}^{\,\bbG_{m}}(f\circ F_{w}),
  \end{equation}
  for every $n \in \bbN$.
  From \eqref{ineq.localPhi} and the self-similarity of $\mcE_{p}$, since the convergence $\Phi_{k}' \to \Phi'$ is uniform, we obtain $\lim_{k \to \infty}\mcE_{p}\bigl(\Phi \circ f - \Phi_{k} \circ f\bigr) = 0$.
  We deduce our assertion by letting $k \to \infty$ in \eqref{eq.chainapp} and applying Proposition \ref{prop.peneapp}.
\end{proof}

\begin{thm}[Theorem \ref{thm.main3}-(3)]\label{thm.peneSS}
  For any $n \in \bbN$ and $f \in \mcF_p$,
  \begin{equation}\label{eq.peness}
      \mu_{\langle f \rangle}^{p}(dx) = \rho_{p}^{n}\sum_{w \in W_{n}}(F_{w})_{\ast}\mu_{\langle F_{w}^{\ast}f \rangle}^{p}(dx).
  \end{equation}
\end{thm}
\begin{proof}
  Let $n, m \in \bbN$, let $w \in W_{m}$ and let $f \in \mcF_{p}$.
  If $m \le n$, then we see that
  \begin{align*}
      \rho_{p}^{n}\,\sum_{v \in W_{n}}(\sigma_{v})_{\ast}\mathfrak{m}_{\langle F_{v}^{\ast}f \rangle}^{p}(\Sigma_{w})
      &= \rho_{p}^{n}\,\sum_{v \in w \cdot W_{n - m}}(\sigma_{v})_{\ast}\mathfrak{m}_{\langle F_{v}^{\ast}f \rangle}^{p}(\Sigma_{w}) \\
      &= \rho_{p}^{n}\,\sum_{v \in w \cdot W_{n - m}}\mathfrak{m}_{\langle F_{v}^{\ast}f \rangle}^{p}(\Sigma) \\
      &= \rho_{p}^{n + m}\sum_{v \in w \cdot W_{n - m}}\mcE_{p}\bigl(F_{v}^{\ast}f\bigr)
      = \mathfrak{m}_{\langle f \rangle}^{p}(\Sigma_{w}).
  \end{align*}
  If $m \ge n$, then we have that
  \begin{align*}
      \rho_{p}^{n}\,\sum_{v \in W_{n}}(\sigma_{v})_{\ast}\mathfrak{m}_{\langle F_{v}^{\ast}f \rangle}^{p}(\Sigma_{w})
      &= \rho_{p}^{n}\,(\sigma_{[w]_{n}})_{\ast}\mathfrak{m}_{\langle F_{[w]_{n}}^{\ast}f \rangle}^{p}(\Sigma_{w}) \\
      &= \rho_{p}^{m}\,\mcE_{p}\bigl(F_{w}^{\ast}f\bigr)
      = \mathfrak{m}_{\langle f \rangle}^{p}(\Sigma_{w}).
  \end{align*}
  Therefore, by Dynkin's $\pi$-$\lambda$ theorem, we deduce that
  \[
  \mathfrak{m}_{\langle f \rangle}^{p}(d\omega)
  = \rho_{p}^{n}\,\sum_{w \in W_{n}}(\sigma_{w})_{\ast}\mathfrak{m}_{\langle F_{w}^{\ast}f \rangle}^{p}(d\omega),
  \]
  for every $n \in \bbN$.
  By Proposition \ref{prop.pi}, we have the desired result.
\end{proof}

As an immediate consequence of Theorems \ref{thm.peneMarkov} and \ref{thm.peneSS}, we can prove the following \emph{energy image density property} (we borrow this naming from \cite{BH.DF}*{Theorem I.7.1.1}).

\begin{prop}\label{prop.EID}
    For any $f \in \mcF_{p}$, it holds that the image measure of $\mu_{\langle f \rangle}^{p}$ by $f$ is absolutely continuous with respect to the one-dimensional Lebesgue measure $\mathscr{L}^{1}$ on $\bbR$.
    In particular, $\mu_{\langle f \rangle}^{p}(\{ x \}) = 0$ for any $x \in K$.
\end{prop}
\begin{proof}
    We follow \cite{CF}*{Theorem 4.3.8}.
    It will suffice to show that $f_{\ast}\mu_{\langle f \rangle}^{p}(F) = 0$ whenever $f \in \mcF_{p}$ and $F$ is a compact subset of $\bbR$ with $\mathscr{L}^{1}(F) = 0$.
    We can choose a sequence $\{ \varphi_{n} \}_{n \ge 1}$ from continuous functions on $\bbR$ with compact supports such that $\abs{\varphi_{n}} \le 1$, $\lim_{n \to \infty}\varphi_{n}(x) = \indicator{F}(x)$ for each $x \in \bbR$, and
    \[
    \int_{0}^{\infty}\varphi_{n}(t)\,dt = \int_{-\infty}^{0}\varphi_{n}(t)\,dt = 0,
    \]
    for each $n \in \bbN$.
    Define $\Phi_{n}(x) \coloneqq \int_{0}^{x}\varphi_{n}(t)\,dt$ for each $x \in \bbR$ and $n \in \bbN$.
    Then we easily see that $\Phi_{n} \in \mcC^{1}(\bbR)$ with compact support, $\Phi_{n}(0) = 0$, and $\abs{\Phi_{n}'} \le 1$ for each $n \in \bbN$.
    By the dominated convergence theorem, it is immediate that $\lim_{n \to \infty}\Phi_{n}(x) = 0$ for each $x \in \bbR$ and $\Phi_{n} \circ f$ converges to $0$ in $L^{p}(K, \mu)$.
    Since $\mcE_{p}(\Phi_{n} \circ f) \le \mcE_{p}(f)$ by the Markov property of $\mcE_{p}$, we deduce that $\{ \Phi_{n} \}_{n \ge 1}$ is $\mcF_{p}$-bounded.
    Therefore, there exists a subsequence $\{ n_k \}_{k \ge 1}$ such that $\{ \Phi_{n_k} \circ f \}_{k \ge 1}$ converges to $0$ weakly in $\mcF_{p}$.
    By Mazur's lemma, there exist $N(l) \in \bbN$ and $\{ a(l)_{k} \}_{k = l}^{N(l)}$ with $a(l)_{k} \ge 0$ and $\sum_{k = l}^{N(l)}a(l)_{k} = 1$ such that $\Psi_{l} \circ f \coloneqq \sum_{k = l}^{N(l)}a(l)_{k}\Phi_{n_k} \circ f$ converges to $0$ in $\mcF_{p}$ as $l \to \infty$.
    Then, by Fatou's lemma and the change of variable formula, we conclude that
    \begin{align*}
        f_{\ast}\mu_{\langle f \rangle}^{p}(F)
        &= \int_{\bbR}\lim_{l \to \infty}\abs{\sum_{k = l}^{N(l)}a(l)_{k}\Phi_{n_k}'(t)}^{p}\,f_{\ast}\mu_{\langle f \rangle}^{p}(dt) \\
        &\le \varliminf_{l \to \infty}\int_{K}\abs{\Psi_{l}'(f(x))}^{p}\,\mu_{\langle f \rangle}^{p}(dx) \\
        &= \varliminf_{l \to \infty}\mu_{\langle \Psi_{l} \circ f \rangle}^{p}(K)
        = \varliminf_{l \to \infty}\mcE_{p}\bigl(\Psi_{l} \circ f\bigr) = 0. \qedhere
    \end{align*}
\end{proof}

Finally, we prove Theorem \ref{thm.main1.2}-(1).
\begin{thm}[Theorem \ref{thm.main3}-(1)]\label{thm.peneLocal}
  Let $f, g \in \mcF_p$.
  If $(f - g)|_{A}$ is constant for some Borel subset $A$ of $K$, then $\mu_{\langle f \rangle}^{p}(A) = \mu_{\langle g \rangle}^{p}(A)$.
\end{thm}
\begin{proof}
  Let $f \in \mcF_p$ and let $A \in \mcB(K)$.
  Suppose that $f|_{A} = c$ for some $c \in \bbR$.
  Then, by Proposition \ref{prop.EID}, we have $\mu_{\langle f \rangle}^{p}(f^{-1}(\{ c \})) = 0$, which implies that $\mu_{\langle f \rangle}^{p}(A) = 0$.
  Combining this result and Proposition \ref{prop.peneapp}, we finish the proof.
\end{proof}

We conclude this section by showing a consequence of the symmetries of $\mcE_p$ that will be used to prove Theorem \ref{thm.betap_strict} in the next section.
\begin{prop}\label{prop.pemeas-sym}
    For any $f \in \mcF_p$ and $T \in \mathcal{G}_{0}$, it holds that $T_{\ast}\mu_{\langle f \rangle}^{p} = \mu_{\langle T^{\ast}f \rangle}^{p}$.
\end{prop}
\begin{proof}
    Let $A \in \mcB(K)$ be a closed set and let $T \in \mathcal{G}_{0}$.
    For each $l \in \bbN$, define
    \[
    C_{l} \coloneqq \{ w \in W_{l} \mid \Sigma_{w} \cap \pi^{-1}(A) \neq \emptyset \} \quad \text{and} \quad C_{l}^{T} \coloneqq \{ w \in W_{l} \mid \Sigma_{w} \cap \pi^{-1}(T^{-1}(A)) \neq \emptyset \}.
    \]
    Then we easily see that $\tau[T]|_{W_{l}}$ gives a bijection between $C_{l}$ and $C_{l}^{T}$.
    Hence, for any $n \in \bbN$,
    \begin{align*}
        \sum_{w \in C_{l}}\widetilde{\mcE}_{p}^{\bbG_{n}}(F_{w}^{\ast}(T^{\ast}f))
        = \sum_{w \in C_{l}}\widetilde{\mcE}_{p}^{\bbG_{n}}((T \circ F_{w})^{\ast}f)
        = \sum_{w \in C_{l}^{T}}\widetilde{\mcE}_{p}^{\bbG_{n}}(F_{w}^{\ast}f).
    \end{align*}
    From \eqref{eq.originalEp2}, we get $\sum_{w \in C_{l}}\widehat{\mcE}_{p, n}(F_{w}^{\ast}(T^{\ast}f)) = \sum_{w \in C_{l}^{T}}\widehat{\mcE}_{p, n}(F_{w}^{\ast}f)$, and thus
    \begin{align*}
    	\mathfrak{m}_{\langle T^{\ast}f \rangle}^{p}\bigl(\Sigma_{C_{l}}\bigr) 
    	&= \rho_{p}^{l}\sum_{w \in C_{l}}\mcE_{p}\bigl(T^{\ast}f \circ F_{w}\bigr) \\
    	&= \rho_{p}^{l}\lim_{k \to \infty}\sum_{w \in C_{l}}\widehat{\mcE}_{p, n_k}\bigl(F_{w}^{\ast}(T^{\ast}f)\bigr) \\ 
    	&= \rho_{p}^{l}\lim_{k \to \infty}\sum_{w \in C_{l}^{T}}\widehat{\mcE}_{p, n}\bigl(F_{w}^{\ast}f\bigr) \\
    	&= \rho_{p}^{l}\sum_{w \in C_{l}^{T}}\mcE_{p}\bigl(f \circ F_{w}\bigr) 
    	= \mathfrak{m}_{\langle f \rangle}^{p}\Bigl(\Sigma_{C_{l}^{T}}\Bigr). 
    \end{align*}
    Letting $l \to \infty$, we obtain $\mu_{\langle T^{\ast}f \rangle}^{p}(A) = T_{\ast}\mu_{\langle f \rangle}^{p}(A)$ because $\bigcap_{l \in \bbN}\Sigma_{C_{l}} = \pi^{-1}(A)$ and $\bigcap_{l \in \bbN}\Sigma_{C_{l}^{T}} = \pi^{-1}(T^{-1}(A))$ as seen in the proof of Proposition \ref{prop.peneapp}.
    Since both of these measures $\mu_{\langle T^{\ast}f \rangle}^{p}$ and $T_{\ast}\mu_{\langle f \rangle}^{p}$ are Borel regular, we complete the proof.
\end{proof}

\section{Proof of Theorem \ref{thm.betap_strict}}\label{sec:betap}
We conclude this paper by proving Theorem \ref{thm.betap_strict}: a strict inequality $\beta_p > p$.
Our argument is similar to \cite{Kaj20+}*{Section 3}.
A key to prove Theorem \ref{thm.betap_strict} is the notion of \emph{$\mcE_p$-harmonicity} (see Definition \ref{dfn.p-harm}).
In this section, we suppose that Assumption \ref{assum.lowdim} holds (except Theorem \ref{thm.betap_strict}) and we write $(\mathcal{E}_{p}, \mathcal{F}_{p})$ for the $p$-energy on a generalized Sierpi\'{n}ski carpet $(K, S, \{ F_{i} \}_{i \in S}) = \GSC(D, a, S)$ in Theorem \ref{thm.main1.2}. 
Recall that $\mathcal{F}_{p} \subseteq \mathcal{C}(K)$ in this case. 

\begin{dfn}
    Let $U$ be a non-empty open subset of $K$.
    We define
    \begin{equation}\label{dfn.bourd}
        \mcC^{U} \coloneqq \{ f \in \mcF_{p} \mid \supp[f] \subseteq U \}, \quad\text{and}\quad \quad \mcF_{p}^{U} \coloneqq \closure{\mcC^{U}}^{\,\norm{\,\cdot\,}_{\mcF_{p}}}.
    \end{equation}
\end{dfn}

\begin{prop}\label{prop.FpU}
	It holds that 
    \begin{equation}\label{eq.FpU}
        \mcF_{p}^{U} = \{ f \in \mcF_{p} \mid \text{$f(x) = 0$ for any $x \in K \setminus U$} \}.
    \end{equation}
\end{prop}
\begin{proof}
    It is easy to show that $\mcF_{p}^{U} \subseteq \{ f \in \mcF_{p} \mid \text{$f(x) = 0$ for any $x \in K \setminus U$} \} \eqqcolon \widetilde{\mathcal{F}}_{p}^{U}$, so we will prove the converse. 
    Let $f \in \widetilde{\mathcal{F}}_{p}^{U}$ be non-negative. 
  	For $n \ge 1$, since $f$ is uniformly continuous on $K$, we can choose $r_{n} > 0$ such that 
  	\[
  	f(x) < n^{-1} \text{ for all } x \in O_{n} \coloneqq \bigcup_{x \in K \setminus U}B(x, r_{n}). 
  	\]
  	We can assume that $\{ r_{n} \}_{n \ge 1}$ is non-increasing. 
  	Also, we define 
  	\[
  	f_{n} \coloneqq \bigl(f - n^{-1}\bigr) \vee 0 \in \mathcal{F}_{p}. 
  	\]
  	Then it is clear that $\supp[f_{n}] \subseteq K \setminus U$ and $f_{n} \to f$ in $\mathcal{C}(K)$ as $n \to \infty$. 
  	In particular, $f_{n} \in \mathcal{C}^{U}$ for all $n \ge 1$. 
  	By the Markov property of $\mathcal{E}_{p}$, we have 
  	\[
  	\mathcal{E}_{p}(f_{n}) \le \mathcal{E}_{p}\bigl(f - n^{-1}\bigr) = \mathcal{E}_{p}(f), 
  	\]
  	and hence $\{ f_{n} \}_{n \ge 1}$ is bounded in $\mathcal{F}_{p}$. 
  	Since $\mathcal{F}_{p}$ is reflexive, there exists a subsequence $\{ n_k \}_{k \ge 1}$ such that $f_{n_{k}}$ converges weakly to $f$. 
  	By Mazur's lemma, we can find a sequence $\{ g_{m} \}_{m \ge 1}$ from 
  	\[
  	\Biggl\{ \sum_{k = 1}^{N}a_{k}f_{n_{k}} \Biggm| N \in \mathbb{N}, \text{$a_{k} \ge 0$ and $\sum_{k = 1}^{N}a_{k} = 1$} \Biggr\} \subseteq \mathcal{C}^{U}
  	\]
  	such that $g_{m} \to f$ in $\mathcal{F}_{p}$, which completes the proof. 
\end{proof}

\begin{dfn}\label{dfn.p-harm}
    Let $U$ be a non-empty open subset of $K$.
    For $f \in \mathcal{F}_{p}$, we set 
    \[
    f + \mathcal{F}_{p}^{U} \coloneqq \bigl\{ f + g \bigm| g \in \mathcal{F}_{p}^{U} \bigr\}. 
    \]
    A function $h \in \mcF_p$ is \emph{$\mcE_p$-harmonic} on $U$ if     
    \begin{equation}\label{eq.p-harm}
        \mcE_{p}(h) = \inf\bigl\{ \mcE_{p}(f) \bigm| f \in h + \mathcal{F}_{p}^{U} \bigr\}.
    \end{equation}
\end{dfn}

\begin{prop}\label{prop.p-harm}
    Let $U$ be a non-empty open subset of $K$ with $U \neq K$ and let $g \in \mcF_p$.
    Then there exists a unique function $h \in \mcF_{p}$ that is $\mcE_{p}$-harmonic on $U$ and $h|_{K \setminus U} \equiv g|_{K \setminus U}$.
\end{prop}
\begin{proof}
    If $g|_{K \setminus U} \equiv a$ for some $a \in \bbR$, then $h \coloneqq a$ is the required function.
    Suppose that $g \in \mcF_p$ is not constant on $K \setminus U$.
    Since $g$ is bounded and $\mcE_{p}(f + a\indicator{K}) = \mcE_{p}(f)$ for any $f \in \mcF_p$ and $a \in \bbR$, we may assume that $0 \le g \le 1$.
    Clearly, $\{ f \in \mcF_p \mid f|_{K \setminus U} \equiv g |_{K \setminus U} \}$ is non-empty.
    For each $\lambda \ge 0$, define
    \[
    \mathsf{c}_{\lambda} \coloneqq \inf\{ \mcE_{p}(f) + \lambda\norm{f}_{L^{p}}^{p} \mid \text{$f \in \mcF_{p}$ with $f|_{K \setminus U} \equiv g |_{K \setminus U}$} \}.
    \]
    Note that $\mathsf{c}_{\lambda} < \infty$.
    Let $f \in \mcF_{p}$ satisfy $f|_{K \setminus U} \equiv g|_{K \setminus U}$.
    Set $f^{\#} \coloneqq (f \vee 0) \wedge 1 \in \mcF_{p}$.
    Then, it follows from $0 \le g \le 1$ that $f^{\#}|_{K \setminus U} = g |_{K \setminus U}$.
    Thus,
    \[
    \mcE_{p}(f) \ge \mcE_{p}(f^{\#}) + \lambda\norm{f^{\#}}_{L^{p}}^{p} - \lambda \ge \mathsf{c}_{\lambda} - \lambda,
    \]
    which implies that $\mathsf{c}_{0} \ge \mathsf{c}_{\lambda} - \lambda$ for any $\lambda \ge 0$.
    For each $n \in \bbN$, we can choose $f_n \in \mcF_{p}$ with $f_{n}|_{K \setminus U} \equiv g |_{K \setminus U}$ such that
    \[
    \mcE_{p}(f_n) + n^{-1}\norm{f_{n}}_{L^p}^{p} < \mathsf{c}_{n^{-1}} + n^{-1}.
    \]
    Then $\mcE_{p}(f_{n}^{\#}) \le \mathsf{c}_{0} + 2n^{-1}$ for each $n \in \bbN$, where $f_{n}^{\#} \coloneqq (f_n \vee 0) \wedge 1 \in \mcF_{p}$.
    Since $\norm{f_{n}^{\#}}_{L^p}^{p} \le 1$ for any $n \in \bbN$, there exist $h \in L^{p}(K, \mu)$ and a subsequence $\{ n_k \}_{k \ge 1}$ such that $\{ f_{n_k}^{\#} \}_{k \ge 1}$ converges weakly to $h$ in $L^{p}$.
    Applying Mazur's lemma, we find convex combinations $u_{k} = \sum_{j = k}^{N_{k}}a_{k, j}f_{n_j}^{\#}$ (i.e. $N_k \in \bbN$, $a_{k, j} \ge 0$ and $\sum_{j = k}^{N_k}a_{k, j} = 1$) such that $u_{k}$ converges to $h$ in $L^{p}$ as $k \to \infty$.
    Note that $f_{n}^{\#}|_{K \setminus U} \equiv g |_{K \setminus U}$ and thus $u_{k}|_{K \setminus U} \equiv g |_{K \setminus U}$.
    Also, we obtain $h|_{K \setminus U} = g |_{K \setminus U} \, \mu$-a.e. since $u_{k} \to h$ in $L^{p}$ as $k \to \infty$.
    By the triangle inequality of $\mcE_{p}(\,\cdot\,)^{1/p}$, we have that $\mcE_{p}(u_k) \in [\mathsf{c}_{0}, \mathsf{c}_{0} + 2n_{k}^{-1})$, which together with Clarkson's inequality implies that $\lim_{k \wedge l \to \infty}\mcE_{p}(u_k - u_l) = 0$.
    Indeed, when $p < 2$,
    \begin{align*}
        \mcE_{p}(u_k - u_l)^{\frac{1}{p - 1}}
        &\le 2\bigl(\mcE_{p}(u_k) + \mcE_{p}(u_l)\bigr)^{\frac{1}{p - 1}} - \mcE_{p}(u_k + u_l)^{\frac{1}{p - 1}} \\
        &\le 2\bigl(2\mathsf{c}_{0} + 2n_{k}^{-1} + 2n_{l}^{-1}\bigr)^{\frac{1}{p - 1}} - 2^{\frac{p}{p - 1}}\mathsf{c}_{0}^{\frac{1}{p - 1}} \\
        &\underset{k \wedge l \to \infty}{\rightarrow}
        2(2\mathsf{c}_{0})^{\frac{1}{p - 1}} - 2^{\frac{p}{p - 1}}\mathsf{c}_{0}^{\frac{1}{p - 1}} = 0.
    \end{align*}
    The case $p \ge 2$ is similar to the above.
    Therefore, $\{ u_k \}_{k \ge 1}$ is a Cauchy sequence in $\mcF_{p}$.
    Since $\mcF_{p}$ is a Banach space, we see that $h \in \mcF_p$ and $u_k$ converges to $h$ in $\mcF_{p}$.
    Moreover, by $\lim_{k \to \infty}\norm{u_{k}}_{L^p} = \norm{h}_{L^{p}}$, we conclude that $\mcE_{p}(h) = \lim_{k \to \infty}\mcE_{p}(u_k) = \mathsf{c}_{0}$, that is, $h$ is a minimizer of $\inf\{ \mcE_{p}(f) \mid \text{$f \in \mcF_{p}$ with $f|_{K \setminus U} \equiv g|_{K \setminus U}$} \}$.

    Lastly, we prove the uniqueness.
    Let $h_{i} \in \mcF_p \, (i = 1, 2)$ be $\mcE_p$-harmonic on $U$ with $h_{i}|_{K \setminus U} \equiv g|_{K \setminus U}$.
    When $p < 2$, by Clarkson's inequality of $\mcE_{p}$,
    \begin{align*}
        \mcE_{p}(h_1 - h_2)^{\frac{1}{p - 1}}
        &\le 2\bigl(\mcE_{p}(h_1) + \mcE_{p}(h_2)\bigr)^{\frac{1}{p - 1}} - \mcE_{p}(h_1 + h_2)^{\frac{1}{p - 1}} \\
        &\le 2^{1 + \frac{1}{p - 1}}\mathsf{c}_{0}^{\frac{1}{p - 1}} - 2^{\frac{p}{p - 1}}\mathsf{c}_{0}^{\frac{1}{p - 1}}
        = 0.
    \end{align*}
    Thus $h_1 - h_2$ is constant by Theorem \ref{thm.main1.2}-(1).
    Since $K \setminus U$ is not empty, we have that $h_1 = h_2$.
    The case $p \le 2$ is similar.
\end{proof}

Recall the definitions of $\mathcal{K}^{-}, \mathcal{K}^{+}$ (see \eqref{dfn.LR}). 
It is immediate from Theorems \ref{thm.main1} and \ref{thm.main1.2}-(3) that $\{ f \in \mcF_{p} \mid f|_{\mathcal{K}^{\,-}} \equiv 0, f|_{\mathcal{K}^{\,+}} \equiv 1 \} \neq \emptyset$.
Thus we have the following lemma by applying Proposition \ref{prop.p-harm} and using the symmetries of Sierpi\'{n}ski carpets.

\begin{lem}\label{lem.p-harm_LR}
    There exists a function $h_{0} \in \mcF_{p}$ such that $h_{0}|_{\mathcal{K}^{-}} \equiv 0$, $h_{0}|_{\mathcal{K}^{+}} \equiv 1$ and $h_{0}$ is $\mcE_{p}$-harmonic on $K \setminus (\mathcal{K}^{-} \cup \mathcal{K}^{+})$.
    Moreover, it holds that $h_{0} \circ R_{j} = h_{0}$ for all $j = 2, \dots, D$. 
\end{lem}

Let $h_0$ be the $\mathcal{E}_{p}$-harmonic function given in Lemma \ref{lem.p-harm_LR}.
Since $\mcE_p(f) = 0$ if and only if $f$ is constant, we immediately have that $\mcE_{p}(h_0) > 0$.
Inductively, we define
\begin{align}\label{hn}
    &h_{n} \coloneqq \sum_{i = (i_{k})_{k = 1}^{D} \in S}(F_{i})_{\ast}\bigl(a^{-1}(h_{n - 1} + i_{1}\indicator{K})\bigr). 
\end{align}
The following proposition is clear by its definition and the self-similarity of $\mcE_p$.

\begin{lem}\label{lem.hn}
    For any $n \in \bbZ_{\ge 0}$, it holds that $h_{n} \in \mcF_p$, $h_{n}|_{\mathcal{K}^{-}} \equiv 0$, and $h_{n}|_{\mathcal{K}^{+}} \equiv 1$.
\end{lem}

Hereafter we suppose that $D = 2$. 
The following lemma is a key to prove Theorem \ref{thm.betap_strict}.

\begin{lem}\label{lem.not-harm}
	Suppose that $D = 2$.
    Then the function $h_2$ is not $\mcE_p$-harmonic on $K \setminus (\mathcal{K}^{-} \cup \mathcal{K}^{+})$.
\end{lem}
\begin{proof}
	Set 
	\[
	\mathcal{K}^{\,-}_{\ast} \coloneqq K \cap B_{2, -1}, \quad \mathcal{K}^{\,+}_{\ast} \coloneqq K \cap B_{2, +1}. 
	\]
	(Recall notations in Definition \ref{dfn.hyperplane}.)
    Suppose to the contrary that $h_2$ were $\mcE_p$-harmonic on $K \setminus (\mathcal{K}^{-} \cup \mathcal{K}^{+})$.
    We claim that then a contradiction that $h_{0}|_{\mathcal{K}^{+}_{\ast}} \equiv 0$ would be implied.
    Let $\varphi \in \mcF_{p}^{K \setminus (\mathcal{K}^{-} \cup \mathcal{K}^{-}_{\ast})}$, i.e. $\varphi \in \mathcal{F}_{p}$ satisfies $\varphi(x) = 0$ for any $x \in \mathcal{K}^{-} \cup \mathcal{K}^{-}_{\ast}$.
    
    From $S \subsetneq \{ 0, \dots, a - 1 \}^{2}$, $a \ge 3$ and \hyperlink{GSC4}{\textup{(GSC4)}}, there exist $v_{1}, v_{2} \in S \subseteq \mathbb{R}^{2}$ such that $v_{1} + \mathbf{e}_{1} \not\in S$, $v_{2} \in W_{1}\bigl[\mathcal{K}^{+}\bigr]$ and $v_{2} + \mathbf{e}_{2} \in S$. 
    Now, we define $\varphi_{\star} \in \mathcal{C}(K)$ by setting 
    \[
    \varphi_{\star}(x) \coloneqq 
    \begin{cases}
    	a^{-2}(F_{v_{1}v_{2}})_{\ast}\varphi(x) \, &\text{if $x \in K_{v_{1}v_{2}}$,} \\
    	a^{-2}(F_{v_{1}(v_{2} + \mathbf{e}_{2})})_{\ast}\bigl(\varphi \circ R_{2}\bigr)(x) \, &\text{if $x \in K_{v_{1}(v_{2} + \mathbf{e}_{2})}$,} \\
    	0 \, &\text{otherwise.}
    \end{cases}
    \] 
    Note that $\varphi_{\star}$ is well-defined since 
    \[
    (F_{v_{1}v_{2}})_{\ast}\varphi(x)|_{F_{v_{1}v_{2}}(\mathcal{K}^{\,-} \cup \mathcal{K}^{\,+}_{\ast})} = 0, \quad (F_{v_{1}(v_{2} + \mathbf{e}_{2})})_{\ast}\varphi(x)|_{F_{v_{1}(v_{2} + \mathbf{e}_{2})}(\mathcal{K}^{\,-} \cup \mathcal{K}^{\,-}_{\ast})} = 0, 
    \]
    and
    \[
    \bord{\bigl(K_{v_{1}v_{2}} \cup K_{v_{1}(v_{2} + \mathbf{e}_{2})}\bigr)} = F_{v_{1}v_{2}}\bigl(\mathcal{K}^{-} \cup \mathcal{K}^{+}_{\ast}\bigr) \cup  F_{v_{1}(v_{2} + \mathbf{e}_{2})}\bigl(\mathcal{K}^{-} \cup \mathcal{K}^{-}_{\ast}\bigr) \cup \bigcup_{\mathbf{s} \in \{ \mathbf{0}, \mathbf{e}_{2} \}}F_{v_{1}(v_{2} + \mathbf{s})}(\mathcal{K}^{+}). 
    \]
    Moreover, it follows that $\varphi_{\star} \in \mathcal{F}_{p}^{K \setminus (\mathcal{K}^{-} \cup \mathcal{K}^{+})}$. 
    Since $h_{2}$ is $\mathcal{E}_{p}$-harmonic on $K \setminus (\mathcal{K}^{-} \cup \mathcal{K}^{+})$, we have $\mcE_p(h_2 + \varphi_{\star}) > \mcE_p(h_2)$ unless $\varphi \equiv 0$.
    Using Theorem \ref{thm.main1.2}-(4), (5) and $\varphi \circ R_{2} = \varphi$, we see that
    \begin{align*}
        &\mcE_p(h_2 + \varphi_{\star}) - \mcE_p(h_2) \\
        &= \rho_{p}^{2}\sum_{w \in W_{2}}\left(\mcE_{p}(F_{w}^{\ast}h_2 + F_{w}^{\ast}\varphi_{\star}) - \mcE_{p}(F_{w}^{\ast}h_2)\right) \\
        &= \rho_{p}^{2}\sum_{\mathbf{s} \in \{ \mathbf{0}, \mathbf{e}_{2} \}}\Bigl(\mcE_{p}(F_{v_{1}(v_{2} + \mathbf{s})}^{\ast}h_2 + F_{v_{1}(v_{2} + \mathbf{s})}^{\ast}\varphi_{\star}) - \mcE_{p}(F_{v_{1}(v_{2} + \mathbf{s})}^{\ast}h_2)\Bigr) \\
        &= 2\rho_{p}^{2}a^{-2p}\left(\mcE_{p}(h_0 + \varphi) - \mcE_{p}(h_0)\right).
    \end{align*}
    Hence, we conclude that $\mcE_{p}(h_0 + \varphi) > \mcE_{p}(h_0)$ for any $\varphi \in \mcF_{p}^{K \setminus (\mathcal{K}^{-} \cup \mathcal{K}^{-}_{\ast})} \setminus \{ 0 \}$.
    This implies that $h_0$ is the minimizer of $\inf\bigl\{ \mcE_p(f) \bigm| \text{$f \in \mcF_p$ with $f|_{\mathcal{K}^{\,-} \cup \mathcal{K}^{\,-}_{\ast}} \equiv h_{0}|_{\mathcal{K}^{\,-} \cup \mathcal{K}^{\,-}_{\ast}}$} \bigr\}$.

    Next, we define $\widetilde{h}_0\in \mathcal{C}(K)$ by
    \begin{align}\label{dfn.h0-sym}
        \widetilde{h}_0(x)\coloneqq
    \begin{cases}
        h_0(x) \quad &\text{if $x \in \mathcal{H}_{1, 2}^{\,-, \le}$,} \\
        \bigl(h_{0} \circ R_{1, 2}^{-}\bigr)(x) &\text{if $x \in \mathcal{H}_{1, 2}^{\,-, \ge}$.}
    \end{cases}
    \end{align}
    Then it is clear that $\widetilde{h}_0 \in \mcF_{p}$ and $\widetilde{h}_{0}|_{\mathcal{K}^{\,-} \cup \mathcal{K}^{-}_{\ast}} \equiv h_{0}|_{\mathcal{K}^{-} \cup \mathcal{K}^{-}_{\ast}}$.
    Moreover, we can show that $\mcE_{p}\bigl(\widetilde{h}_0\bigr) = \mcE_p(h_0)$. 
    To prove this, let 
    \[
    A_{1} \coloneqq \mathcal{H}_{1, 2}^{-, \le} \setminus \mathcal{H}_{1, 2}^{-}, \quad A_{2} \coloneqq \mathcal{H}_{1, 2}^{-} \quad \text{and} \quad A_{3} \coloneqq \mathcal{H}_{1, 2}^{\,-, \ge} \setminus \mathcal{H}_{1, 2}^{-}. 
    \]
    Since $h_{0}|_{A_{1} \sqcup A_{2}} = \widetilde{h}_{0}|_{A_{1} \sqcup A_{2}}$, Theorem \ref{thm.peneLocal} yields $\mu^{p}_{\langle h_{0} \rangle}(A_{1} \sqcup A_{2}) = \mu^{p}_{\langle \widetilde{h}_{0} \rangle}(A_{1} \sqcup A_{2})$. 
    Moreover, since $R_{1, 2}^{\,+}(A_{3}) = A_{3}$ and $\bigl(\widetilde{h}_{0} \circ R_{1, 2}^{+}\bigr)|_{A_3} = (1 - h_0)|_{A_3}$, Proposition \ref{prop.pemeas-sym} and Theorem \ref{thm.peneLocal} imply 
    \[
    \mu_{\langle \widetilde{h}_0 \rangle}^{p}(A_3) = \mu_{\bigl\langle (R_{1, 2}^{\,+})^{\ast}\widetilde{h}_0 \bigr\rangle}^{p}(A_3) = \mu_{\langle 1 - h_0 \rangle}^{p}(A_3) = \mu_{\langle h_0 \rangle}^{p}(A_3).
    \]
    Therefore, we obtain $\mu^{p}_{\langle h_{0} \rangle}(K) = \mu^{p}_{\langle \widetilde{h}_{0} \rangle}(K)$, which implies that $\mcE_{p}\bigl(\widetilde{h}_0\bigr) = \mcE_p(h_0)$. 

    By Proposition \ref{prop.p-harm}, we have $h_0 = \widetilde{h}_0$.
    Hence $h_{0}|_{\mathcal{K}^{+}_{\ast}} \equiv 0$, which contradicts the fact that $h_{0}\bigl((1, 1)\bigr) = 1$.
    We complete the proof.
\end{proof}

Now we are ready to prove Theorem \ref{thm.betap_strict}.

\noindent
\textit{Proof of Theorem \ref{thm.betap_strict}.}\,
    By Lemmas \ref{lem.hn} and \ref{lem.not-harm}, we obtain $\mcE_{p}(h_2) > \mcE_{p}(h_0)$.
    Since
    \[
    \mcE_{p}(h_2) = \rho_{p}^{2}\sum_{w \in W_{2}}\mcE_{p}(F_{w}^{\ast}h_2) = \rho_{p}^{2}a^{-2p}\sum_{w \in W_{2}}\mcE_{p}(h_0),
    \]
    we conclude that $\rho_{p}^{2}a^{-2p}N_{\ast}^{2} > 1$, which proves our assertion for $p > \dim_{\textrm{ARC}}(K, d)$. 
    We know that $\beta_p/p$ is monotonically non-increasing by \cite{Kig20}*{Lemma 4.7.4}, and thus we obtain the desired result. \qed

\begin{rmk}\label{rmk.conj-betap}
    A recent study of $L^{p}$ Besov critical exponent in \cite{ABCRST21} implies a partial result of Theorem \ref{thm.betap_strict}. 
    (Note that sub-Gaussian type heat kernels estimates on GSCs are obtained by Barlow and Bass in \cite{BB99}*{Theorem 1.3}.)
    Indeed, $\beta_p$ is characterized as the $L^{p}$ Besov critical exponent in Corollary \ref{cor.Lp-Besov}.
    Thus a critical exponent $\alpha_{p}^{\#}$ in \cite{ABCRST21}*{equation (7)} coincides with $\beta_{p}/(p\beta_{2})$.
    Therefore, \cite{ABCRST21}*{Theorem 3.11} gives a lower bound of $\beta_{p}$:
    \begin{itemize}
        \item $\beta_{p} \ge \frac{p\beta_{2}}{2}$ for $p \in (\dim_{\textup{ARC}}(K, d), 2)$;
        \item $\beta_{p} \ge (p - 2)(\beta_{2} - \alpha) + \beta_{2}$ for $p \ge 2$.
    \end{itemize}
    This bound implies $\beta_p > p$ for $p \in \bigl(\dim_{\textup{ARC}}(K, d), (2\alpha - \beta_2)/(\alpha - \beta_2 + 1)\bigr)$.
\end{rmk}

\appendix
\section{Miscellaneous facts}\label{sec:appendix}
\subsection{Proof of Lemma \ref{lem.URI}}\label{sec:URI}

This lemma is obtained by observing that the estimates in \cite{Mthesis}*{Lemma 8.4} depend only on the constants controlling rough isometries.
For the reader's convenience, we give a complete proof.

\begin{lem}
    \it
    For each $i = 1, 2$, let $\{ G_{n}^{i} = (V_{n}^{i}, E_{n}^{i}) \}_{n \ge 1}$ be a series of finite graphs with
    \[
    L_{\ast}^{i} \coloneqq \sup_{n \in \bbN}\max_{x \in V_{n}^{i}}\#\{ y \in V_{n}^{i} \mid (x, y) \in E_{n}^{i} \} < \infty,
    \]
    and let $\varphi_{n}\colon V_{n}^{1} \to V_{n}^{2}$ be a uniform rough isometry from $\{ G_{n}^{1} \}_{n \ge 1}$ to $\{ G_{n}^{2} \}_{n \ge 1}$.
    Then there exists a positive constant $C_{\textup{URI}}$ (depending only on $C_{1}, C_{2}$ in Definition \ref{dfn.URI}, $L_{\ast}^{1}$ and $p$) such that
    \begin{equation*}
        \mcE_{p}^{G_{n}^{1}}(f \circ \varphi_{n}) \le C_{\textup{URI}}\,\mcE_{p}^{G_{n}^{2}}(f),
    \end{equation*}
    for every $n \in \bbN$ and $f\colon V_{n}^{2} \to \bbR$.
    In particular,
    \[
    \mcC_{p}^{G_{n}^{1}}(\varphi_{n}^{-1}(A_{n}), \varphi_{n}^{-1}(B_{n})) \le C_{\textup{URI}}\mcC_{p}^{G_{n}^{2}}(A_{n}, B_{n})
    \]
    for every $n \in \bbN$, where $A_{n}, B_{n}$ are disjoint subsets of $V_{n}^{2}$.
\end{lem}
\begin{proof}
    Let $n \in \bbN$, let $f\colon V_{n}^{2} \to \bbR$ and let $(x, y) \in E_{n}^{1}$.
    We set $x' = \varphi_{n}(x)$ and $y' = \varphi_{n}(y)$.
    Let $C_{i} \, (i = 1, \dots 4)$ be constants in the definition of uniform rough isometry.
    Then we get
    \begin{align*}
        0 \leq d_{G_{n}^{2}}(x', y')
        \leq
        C_{1} + C_{2}.
    \end{align*}
    We set $L \in \bbN$ such that $L - 1 < C_{1} + C_{2} \le L$.
    Since $d_{G_{n}^{2}}(x', y') \le L$, there exist $l \le L$ and a path $[z_{0}, z_{1}, \dots, z_{l}]$ in $G_{n}^{2}$ from $x'$ to $y'$, that is $z_{0} = x'$, $z_{l} = y'$ and $(z_{i - 1}, z_{i}) \in E_{n}^{2}$ for each $i = 1, \dots, l$.
    Now, by H\"{o}lder's inequality, we have that
    \begin{align*}
        \abs{f \circ \varphi_{n}(y) - f \circ \varphi_{n}(x)}
        =
        \abs{f(x') - f(y')}
        &\leq
        \sum_{i = 1}^{l}\abs{f(z_{i - 1}) - f(z_{i})} \\
        &\leq
        L^{(p - 1)/p}\left(\sum_{i = 1}^{l}\abs{f(z_{i - 1}) - f(z_{i})}^{p}\right)^{1/p}.
    \end{align*}
    In particular, it follows that
    \begin{equation}\label{ineq.edge}
        \abs{f \circ \varphi_{n}(y) - f \circ \varphi_{n}(x)}^{p} \le L^{p - 1}\sum_{i = 1}^{l}\abs{f(z_{i - 1}) - f(z_{i})}^{p}.
    \end{equation}

    For each $(x, y) \in E_{n}^{1}$, fix a path $\gamma_{xy}' \coloneqq [z_{0}', \dots, z_{l}']$ in $G_{n}^{2}$ from $\varphi_{n}(x)$ to $\varphi_{n}(y)$ with $l \le L$.
    For each $(v, w) \in E_{n}^{2}$, we set
    \begin{equation*}
        M_{(v,w)} \coloneqq \#\{ (x,y) \in E_{n}^{1} \mid \text{path $\gamma_{xy}'$ contains $(v,w)$} \}.
    \end{equation*}
    We also define $\mcA_{v}^{(n)} \coloneqq \{ x \in V_{n}^{1} \mid \varphi_{n}(x) \in B_{d_{G_{n}^{2}}}(v, L) \}$ for each $y \in V_{n}^{2}$.
    Then, for any $a, b \in \mcA_{v}^{(n)}$,
    \begin{align*}
        C_{1}^{-1}d_{G_{n}^{1}}(a, b) - C_{2}
        \leq
        d_{G_{n}^{2}}(\varphi_{n}(a), \varphi_{n}(b))
        \leq
        d_{G_{n}^{2}}(\varphi_{n}(a), v) + d_{G_{n}^{2}}(\varphi_{n}(b), v)
        \leq
        2L.
    \end{align*}
    Therefore, we have $\diam(\mcA_{v}^{(n)}, d_{G_{n}^{1}}) \le C_{1}(2L + C_{2}) \eqqcolon C_{\star}$, which implies that $\#\mcA_{v}^{(n)} \le C_{\star}L_{\ast}^{1}$ for any $n \in \bbN$ and $v \in V_{n}^{2}$.
    Now, since the length of $\gamma_{\varphi_{n}(x)\varphi_{n}(y)}$ is less than $L$, it follows that
    \[
    \{ (x, y) \in E_{n}^{1} \mid \text{a path $\gamma'_{xy}$ contains a edge $(v, w) \in E_{n}^{2}$} \} \subseteq \mcA_{v}^{(n)} \times \mcA_{v}^{(n)}.
    \]
    This yields that $\#M_{(v,w)} \le (C_{\star}L_{\ast}^{1})^{2}$ for any $n \in \bbN$ and $(v,w) \in E_{n}^{2}$.
    Using this bound and summing \eqref{ineq.edge} over $(x,y) \in E_{n}^{1}$, we conclude that
    \begin{equation*}
        \mcE_{p}^{G_{n}^{1}}(f \circ \varphi_{n}) \le L^{p - 1}(C_{\star}D_{\ast}^{1})^{2}\,\mcE_{p}^{G_{n}^{2}}(f). \qedhere
    \end{equation*}
\end{proof}

\subsection{Proof of Lemma \ref{lem.LBembed}}\label{sec:LBembed}

We prove Lemma \ref{lem.LBembed} in a metric measure space setting by extending \cite{GHL03}*{Theorem 4.11-$\textrm{(i\hspace{-.08em}i\hspace{-.08em}i)}$}.
Let $(\mathcal{X}, d, \mu)$ be $\alpha$-Ahlfors regular, that is, $(\mcX, d)$ is a non-empty metric space, $\mu$ is a Borel regular measure on $(\mcX, d)$ without point mass, and there exist $\alpha > 0$ and $C_{\text{AR}} \ge 1$ such that
\[
C_{\text{AR}}^{-1}r^{\alpha} \le \mu(B_{d}(x, r)) \le C_{\text{AR}}r^{\alpha}
\]
for every $x \in X$ and $r \in (0, \diam(\mcX, d))$.
Note that $\dim_{\text{H}}(\mcX, d) = \alpha$.

\begin{lem}
    Let $\beta > \alpha$ and $p > 1$.
    Then there exists a positive constant $C_{\ref{lem.LBembed}}$ (depending only on $p, \beta, \alpha, C_{\textup{AR}}$) such that
    \begin{align*}
        &\abs{f(x) - f(y)}^{p} \\
        &\le C_{\ref{lem.LBembed}}d(x, y)^{\beta - \alpha}\sup_{r \in (0, 3d(x,y)]}r^{-\beta}\int_{\mathcal{X}}\mint_{B_{d}(z, r)}\abs{f(z) - f(z')}^{p}\,d\mu(z')d\mu(z),
    \end{align*}
    for every $f \in \Lambda_{p, \infty}^{\beta/p}$ and $\mu$-a.e. $x, y \in \mcX$.
\end{lem}
\begin{proof}
    For $f \in L^{1}_{\textrm{loc}}(\mathcal{X}, \mu)$, $x \in \mathcal{X}$ and $r > 0$, we set $f_{B_{d}(x, r)} \coloneqq \mint_{B_{d}(x, r)}f(z)\,d\mu(z)$ and 
    \[
    A_{r}(f) \coloneqq \sup_{\rho \in (0, 3r]}\rho^{-\beta}\int_{\mathcal{X}}\mint_{B_{d}(z, \rho)}\abs{f(z) - f(z')}^{p}\,d\mu(z')d\mu(z). 
    \]
    Let $f \in \Lambda_{p, \infty}^{\beta/p}$, let $x \neq y \in \mcX$ and $r > 0$ such that $d(x, y) \le r$.
    By Fubini's theorem, 
    \begin{align*}
        f_{B_{d}(x, r)}
        = \frac{1}{\mu(B_{d}(x, r))\mu(B_{d}(y, r))}\int_{B_{d}(x, r)}\int_{B_{d}(y, r)}f(z)\,d\mu(z')d\mu(z).
    \end{align*}
    Also, we have
    \begin{equation*}
        f_{B_{d}(y, r)} = \frac{1}{\mu(B_{d}(x, r))\mu(B_{d}(y, r))}\int_{B_{d}(x, r)}\int_{B_{d}(y, r)}f(z')\,d\mu(z')d\mu(z).
    \end{equation*}
    From these identities, we have that
    \begin{align}\label{LBembed1}
        &\abs{f_{B_{d}(x, r)} - f_{B_{d}(y, r)}}^{p} \\
        &= \abs{\frac{1}{\mu(B_{d}(x, r))\mu(B_{d}(y, r))}\int_{B_{d}(x, r)}\int_{B_{d}(y, r)}(f(z) - f(z'))\,d\mu(z')d\mu(z)}^{p} \nonumber \\
        &\le \frac{1}{\mu(B_{d}(x, r))\mu(B_{d}(y, r))}\int_{B_{d}(x, r)}\int_{B_{d}(y, r)}\abs{f(z) - f(z')}^{p}\,d\mu(z')d\mu(z) \nonumber \\
        &\le \frac{1}{\mu(B_{d}(x, r))\mu(B_{d}(y, r))}\int_{\mathcal{X}}\int_{B_{d}(z, 3r)}\abs{f(z) - f(z')}^{p}\,d\mu(z')d\mu(z)  \nonumber \\
        &\le c_{1}r^{-\alpha}\int_{\mathcal{X}}\mint_{B_{d}(z, 3r)}\abs{f(z) - f(z')}^{p}\,d\mu(z')d\mu(z)  
        \le c_{1}r^{-\alpha + \beta}A_{r}(f), 
    \end{align}
    where we used H\"{o}lder's inequality in the third line and the Ahlfors regularity in the fifth line.
    ($c_1$ is a positive constant depending only on $C_{\text{AR}}$.)
    Similarly, we obtain
    \begin{align}\label{LBembed2}
        \abs{f_{B_{d}(x, 2r)} - f_{B_{d}(x, r)}}^{p} \le c_{1}r^{-\alpha + \beta}A_{r}(f).
    \end{align}

    Next, let $\mathcal{X}_{\ast}$ be the set of Lebesgue points with respect to $f$ and let $r > 0$.
    By Lebesgue's differential theorem on Ahlfors regular metric measure space (see \cite{Hei}*{Theorem 1.8} for example), it holds that $\mu(\mcX\setminus\mcX_{\ast}) = 0$.
    Set $r_{k} \coloneqq 2^{-k}r$ for any $k \in \mathbb{Z}_{\ge 0}$.
    Then, for any $x \in \mathcal{X}_{\ast}$ and any $\varepsilon > 0$ there exists $K \in \mathbb{N}$ such that $\abs{f(x) - f_{B_{d}(x, r_{k})}} < \epsilon$ for all $k \geq K$.
    Now we have that
    \begin{align*}
        &\abs{f(x) - f_{B_{d}(x, r)}} \\
        &\le \abs{f(x) - f_{B_{d}(x, r_{k})}} + \abs{f_{B_{d}(x, r_{K})} - f_{B_{d}(x, r_{0})}} \le \varepsilon + \sum_{k = 0}^{\infty}\abs{f_{B_{d}(x, r_{k})} - f_{B_{d}(x, r_{k + 1})}}.
    \end{align*}
    Since $\varepsilon > 0$ is arbitrary, we get
    \begin{equation*}
        \abs{f(x) - f_{B_{d}(x, r)}} \le \sum_{k = 0}^{\infty}\abs{f_{B_{d}(x, r_{k})} - f_{B_{d}(x, r_{k + 1})}},
    \end{equation*}
    for any $x \in \mcX_{\ast}$.
    From this inequality and \eqref{LBembed2}, we see that
    \begin{align}\label{LBembed3}
        \abs{f(x) - f_{B_{d}(x, r)}} 
        \le \sum_{k = 0}^{\infty}\abs{f_{B_{d}(x, r_{k})} - f_{B_{d}(x, 2r_{k})}} 
        \le c_{2}r^{(\beta - \alpha)/p}A_{r}(f)^{1/p}, 
    \end{align}
    where $c_{2} \coloneqq c_{1}^{1/p}\sum_{k = 0}^{\infty}2^{-k(\beta - \alpha)/p}$.

    Since $\mu$ has no point mass, it holds that $\mathcal{X}_{\ast} \setminus \{ x \} \neq \emptyset$ for any $x \in \mathcal{X}_{\ast}$.
    Let $y \in \mathcal{X}_{\ast} \setminus \{ x \}$ and set $r \coloneqq d(x, y) > 0$.
    From \eqref{LBembed1} and \eqref{LBembed3}, we conclude that
    \begin{align*}
        \abs{f(x) - f(y)}
        &\le \abs{f(x) - f_{B_{d}(x, r)}} + \abs{f_{B_{d}(x, r)} - f_{B_{d}(y, r)}} + \abs{f(y) - f_{B_{d}(y, r)}} \\
        &\hspace{-5pt}\le c_{3}r^{(\beta - \alpha)/p}A_{r}(f)^{1/p},
    \end{align*}
    where $c_{3} \coloneqq c_{1}^{1/p} + 2c_{2}$.
    This proves our assertion.
\end{proof}


\end{document}